\documentclass[noams]{compositio}
\usepackage[displaymath, mathlines, right, pagewise]{lineno}
\let\openbox\relax

\usepackage{amsmath,amsthm,amsfonts,amssymb,mathrsfs,bm,bbm}
\usepackage[usenames]{color}
\usepackage{hyperref}
\usepackage{amssymb}
\usepackage{appendix}
\renewcommand{\contentsname}{Contents}
\usepackage{amsmath}

\renewcommand{\epsilon}{\varepsilon}
\renewcommand{\phi}{\varphi}

\newtheorem{theorem}{Theorem}[section]
\newtheorem{remark}[theorem]{Remark}
\newtheorem{lemma}[theorem]{Lemma}
\newtheorem{corollary}[theorem]{Corollary}
\newtheorem{conjecture}[theorem]{Conjecture}
\newtheorem{proposition}[theorem]{Proposition} % no italics in definition
\newtheorem{definition}[theorem]{Definition} % [section] for reset counting in each sections 

\newtheorem{claim}[theorem]{Claim}

\usepackage{tikz}
\usepackage{tikz-cd}
\usetikzlibrary{arrows}
\usetikzlibrary{cd}

\begin{document}

\numberwithin{equation}{section}

\title{The local twisted Gan-Gross-Prasad conjecture for $\text{GL(V)/U(V)}$}
\author{Nhat Hoang Le}
\email{lnhoang@nus.edu.sg}
\address{Department of Mathematics, Block S17, National University of Singapore, 10 Lower Kent Ridge Drive, 119076.}

\thanks{The author would like to thank his advisor Wee Teck Gan for suggesting this problem and many useful comments. The author also thanks Lei Zhang for his guidance during when the project was being conducted. He would like to thank Raphael Beuzart-Plessis and Chen-Bo Zhu for serving his thesis committee. He especially thanks Raphael Beuzart-Plessis for his hospitality as well as many helpful advices during the time the author visited Aix-Marseille University. The author thanks Rui Chen, Chen Wan and Jialiang Zou for many helpful discussions and encouragements. This project is supported by the NUS President's Graduate Fellowship and the NUS Overseas Research Immersion Award.}

\begin{abstract}
In this paper, following the method developed by J.-L. Waldspurger and R. Beuzart-Plessis for Bessel models, we study two local relative trace formulas for the local twisted Gan-Gross-Prasad conjecture. By obtaining spectral expansions and a partial geometric comparision between the two formulas, we prove the local twisted Gan-Gross-Prasad conjecture for $\text{GL(V)/U(V)}$ over nonarchimedean fields for tempered representations.
\end{abstract}

\maketitle

\vspace*{6pt}\tableofcontents

\section{Introduction}\label{sec1}

The Gross-Prasad and the Gan-Gross-Prasad conjectures \cite{GP92,GP94,GGP12a} considered a family of restriction problems for classical groups and proposed precise answers to these problems using the local and global Langlands correspondence. These restriction problems were formulated in terms of a pair $W\subset V$ of orthogonal, hermitian, symplectic or skew-hermitian spaces. For Bessel models, the Gan-Gross-Prasad conjecture predicts the behavior of 
$
\text{Hom}_{H}\left(\pi,\xi\right),
$
where $\pi$ is an irreducible representation in a generic L-packet of $G$ (the
product of isometric groups of $V$ and $W$) and $\xi$ is a certain character of a certain subgroup $H$ of $G$. The above conjecture has been largely solved by a series of works
by J.-L. Waldspurger, C. Moeglin, R. Beuzart-Plessis, W.T. Gan, A.
Ichino, H.Xue, Z. Luo, C. Chen \cite{Wal10,Wal12a,Wal12c,Wal12d,Wal12e,MW12,BP14,BP15,BP16,BP20,GI16,Xue23a,Luo21,CL22,Ch21,Ch23}. For Fourier-Jacobi models, the conjectures considered (skew-hermitian) unitary groups or symplectic-metaplectic groups, and attached the Weil representation to these models. Over p-adic fields, they were resolved by W.T.Gan and A.Ichino in skew-hermitian case \cite{GI16}, and H. Atobe in symplectic-metaplectic
case \cite{Ato18}. In archimedean case, H. Xue \cite{Xue23b} proved the conjecture for tempered representations of unitary groups. The symplectic-metaplectic case is an ongoing work by C. Chen, R. Chen and J. Zou.

In \cite{GGP23}, W.T. Gan, B.H. Gross and D. Prasad proposed a twisted variant of the Fourier-Jacobi case for unitary groups. For simplicity, let us consider the case when $F$ is a $p$-adic field. Let $E/F$ be a quadratic extension of $F$ and $V$ be an $n$-dimensional skew-hermitian space over $E$. Let $\psi$ be a nontrivial additive character of $F$ and $\mu$ be a conjugate-symplectic character of $E^{\times}$ (i.e. the restriction
of $\mu$ to $F^{\times}$ is the quadratic character $\omega_{E/F}$
associated to $E/F$ by the local class field theory). Let $\omega_{V,\psi,\mu}$
be the Weil representation of the isometry group $U\left(V\right)$.
Let $K/F$ be a quadratic extension of $F$ (we may take $K=E$). The three authors consider a branching problem associated to the triple $\left(G,H,\omega_{V,\psi,\mu}\right)$, where $G=\text{Res}_{K/F}U\left(V_{K}\right)$ and $H_{V}=U\left(V\right)$.

To be more precise, for an irreducible generic representation $\pi$
of $G\left(F\right)$, they consider the problem of determining 
$$m_{V}\left(\pi\right)=\dim\text{Hom}_{H_{V}}\left(\pi,\omega_{V,\psi,\mu}\right).$$
We remark that a partial result when $K \neq E$ has been achieved in \cite{CG22}. We focus on the case when $K=E$, in which case $G=\text{GL}(V)$. We state \cite[Conjecture 2.1]{GGP23}.
\begin{conjecture}\label{1.1}
$ $
\begin{enumerate}
\item If $\pi\in\text{Irr}\left(G\left(F\right)\right)$ is generic, then
$$
\underset{V}{\sum}\ \dim\text{Hom}_{H_{V}}\left(\pi,\omega_{V,\psi,\mu}\right)=1,
$$
where $V$ runs over equivalence classes of $n$-dimensional skew-hermitian forms.
\item For generic $\pi\in\text{Irr}\left(G\left(F\right)\right)$, the unique
skew-hermitian space $V$ which gives nonzero contribution to the
above sum satisfies
$$
\mu\left(\det V\right)=\epsilon\left(1/2,\pi\times{}^{\sigma}\pi^{\vee}\times\mu^{-1},\psi_{E}\right)\cdot\omega_{\pi}\left(-1\right)^{n}\cdot\omega_{E/F}\left(-1\right)^{n\left(n-1\right)/2},
$$
where $\omega_\pi$ is the central character of $\pi$.
\end{enumerate}
\end{conjecture}
In \cite{GGP23}, the three authors resolve the above conjecture for principal series representations, and in general reduce the conjecture to the case of discrete series by using Mackey theory. In addition, the case of Steinberg representation was proved. The main result of this paper is to prove that Conjecture \ref{1.1} holds for tempered representations. 
\begin{theorem}\label{thm1.2}
\begin{enumerate}
\item If $\pi\in\text{Irr}\left(G\left(F\right)\right)$ is tempered, then
$$
\underset{V}{\sum}\ \dim\text{Hom}_{H_{V}}\left(\pi,\omega_{V,\psi,\mu}\right)=1,
$$
where $V$ runs over equivalence classes of $n$-dimensional skew-hermitian forms.
\item For tempered $\pi\in\text{Irr}\left(G\left(F\right)\right)$, the unique
skew-hermitian space $V$ which gives nonzero contribution to the
above sum satisfies
$$
\mu\left(\det V\right)=\epsilon\left(1/2,\pi\times{}^{\sigma}\pi^{\vee}\times\mu^{-1},\psi_{E}\right)\cdot\omega_{\pi}\left(-1\right)^{n}\cdot\omega_{E/F}\left(-1\right)^{n\left(n-1\right)/2}.
$$
\end{enumerate}
\end{theorem}
We now give a brief description of our strategy. Our proof follows the local trace formula approach developed by Waldspurger, Beuzart-Plessis and Wan in the case of Bessel models. We are however dealing with the Fourier-Jacobi model here, and a new feature in this situation is the presence of the Weil representation of the subgroup $H$ (as opposed to a $1$-dimensional character of $H$ in the Bessel case). As far as we are aware, our work is the first instance where a local trace formula approach has been successfully  developed for a Fourier-Jacobi type model. As in \cite{Wal10,Wal12a,Wal12c,BP14,BP15,BP16,BP20,Wan19} and many following works related to geometric multiplicity formulas, Theorem \ref{thm1.2} should follow from a comparison of a formula for the multiplicity $m_{V}\left(\pi\right)$, which express $m_{V}\left(\pi\right)$ in terms of the Harish-Chandra character of $\pi$, and its twisted variant, which carries the information of epsilon factors. Recall there exists a locally integrable smooth function $\theta_{\pi}$ on the regular semisimple locus $G_{\text{reg}}\left(F\right)$ of $G\left(F\right)$ such that 
$$
\text{Trace }\pi\left(f\right)=\int_{G\left(F\right)}\theta_{\pi}\left(x\right)f\left(x\right)dx,
$$
for all $f\in C_{c}^{\infty}\left(G\left(F\right)\right)$. The function $\theta_\pi$ is unique and is called the Harish-Chandra character of $\pi$. One can regularize the character $\theta_{\pi}$ on $G_{\text{reg}}\left(F\right)$ to a function $c_{\theta_{\pi}}$ on the semisimple locus $G_{\text{ss}}\left(F\right)$ of $G\left(F\right)$. If $x\in G_{\text{reg}}\left(F\right)$, we have $c_{\theta_{\pi}}\left(x\right)=\theta_{\pi}\left(x\right)$.
For general $x\in G_{\text{ss}}\left(F\right)$, we have $c_{\theta_{\pi}}\left(x\right)$
is the main coefficient of a certain local expansion of $\theta_{\pi}$
near $x$. For the precise definition of $c_{\theta_{\pi}}$, we refer
the reader to section \ref{sec3.3}.

We focus on part 1 of Theorem \ref{thm1.2}. Taking inspirations from \cite{Wal10,BP16,BP20},
we consider the left regular representation $L^{2}\left(H\left(F\right)\backslash G\left(F\right),\omega_{V,\psi,\mu}\right)$ of $G\left(F\right)$. Recall that the Weil representation $\omega_{V,\psi,\mu}$ can be realised as the space of Schwartz functions on a Lagrangian space $X$ of the symplectic space $\text{Res}_{E/F}V$. A function $f\in C_{c}^{\infty}\left(G\left(F\right)\right)$ acts
on $L^{2}\left(H\left(F\right)\backslash G\left(F\right),\omega_{V,\psi,\mu}\right)$ by 
$$
\left(R\left(f\right)\phi\right)\left(x,w\right)=\int_{G\left(F\right)}f\left(g\right)\phi\left(xg,w\right)dg,
$$
where $\phi\in L^{2}\left(H\left(F\right)\backslash G\left(F\right),\omega_{V,\psi,\mu}\right)$,
$x\in G\left(F\right)$ and $w\in X$. We have 
$$
\left(R\left(f\right)\phi\right)\left(x,w\right)=\int_{H\left(F\right)\backslash G\left(F\right)}K_{f}\left(x,y\right)\phi\left(y,w\right)dy,
$$
for any $\phi\in L^{2}\left(H\left(F\right)\backslash G\left(F\right),\omega_{V,\psi,\mu}\right)$
and $x,y\in G\left(F\right)$, where 
$$
K_{f}\left(x,y\right)=\int_{H\left(F\right)}f\left(x^{-1}hy\right)\omega_{V,\psi,\mu}\left(h\right)dh,\ \ x,y\in G\left(F\right)
$$
is the $\text{End}\left(\omega_{V,\psi,\mu}\right)$-valued kernel
function associated to $f$. In order to study the representation
$L^{2}\left(H\left(F\right)\backslash G\left(F\right),\omega_{V,\psi,\mu}\right)$,
we compute the trace of $R\left(f\right)$, which is the following linear form 
$$
J_{V}\left(f\right)=\int_{Z_{G}\left(F\right)H\left(F\right)\backslash G\left(F\right)}\text{Trace}\left(K_{f}\left(x,x\right)\right)dx.
$$
$$
=\int_{Z_{G}\left(F\right)H\left(F\right)\backslash G\left(F\right)}\underset{i}{\sum}\int_{H\left(F\right)}f\left(g^{-1}hg\right)\left\langle \phi_{i},\omega_{V,\psi,\mu}\left(h\right)\phi_{i}\right\rangle dhdg,
$$
where $f\in{\mathcal{C}}_{\text{scusp}}\left(G\left(F\right)\right)$ and
$\left\{ \phi_{i}\right\} _{i}$ is an orthonormal basis for $\omega_{V,\psi,\mu}$.
Following the method in \cite{BP20}, we obtain a spectral expansion
of $J_{V}$, which involves a space  of virtual tempered representations ${\mathcal{X}}\left(G\right)$. This corresponds to Theorem \ref{6.1} in Section \ref{sec6}.
\begin{theorem}
For any strongly cuspidal function $f\in{\mathcal{C}}_{\text{scusp}}\left(G\left(F\right)\right)$,
we have 
$$
J_{V}\left(f\right)=\int_{{\mathcal{X}}\left(G\right)}\hat{\theta}_{f}\left(\pi\right)m_{V}\left(\bar{\pi}\right)d\pi.
$$
\end{theorem}

Here $\hat{\theta}_{f}\left(\pi\right)$ is a certain weighted character
associated to $f$, which is defined explicitly in Section \ref{sec3.6}.

One novelty in our treatment is that unlike \cite{Wal10,BP16,BP20}, we do not need to obtain a geometric expansion of $J_V$ in order to prove our main results. In particular, we do not need to obtain a geometric multiplicity formula for $m_V(\pi)$. Instead of considering a single linear form $J_{V}$, we first linearize the problem to a linear form of quasi-characters on $G\left(F\right)$ (denoted by $QC(G(F))$), and then obtain a ``geometric expansion'' for the sum of such linear forms over isomorphism classes of skew-hermitian spaces. To be more precise, we define the following linear form on $QC\left(G\left(F\right)\right)$
$$
J_{\text{qc}}\left(\theta\right)=c_\theta(1)+\underset{\pi\in{\Pi}_{2}\left(G\right)}{\sum}\ \left(m_{V}\left(\bar{\pi}\right)+m_{V^\prime}(\bar{\pi})-1\right)\int_{\Gamma_{\text{ell}}\left(G\right)}D^{G}\left(x\right)\theta\left(x\right)\theta_{\pi}\left(x\right)dx.
$$
In order to prove the strong multiplicity one property, we prove the following theorem, which is Theorem \ref{7.1} in section \ref{sec7}, by induction.
\begin{theorem}
For any $\theta\in QC\left(G\left(F\right)\right)$, we have 
$$
J_{\text{qc}}\left(\theta\right)=c_{\theta}\left(1\right).
$$
\end{theorem}

We now consider part 2 of Theorem \ref{thm1.2}. The first step is formulating a twisted variant of the trace formula considered in part 1. We fix a choice of $E$-basis for $V$. Let $M=\text{Res}_{E/F}\text{GL}_{n}\times\text{GL}_{n}$ and $\theta_{n}:\left(g,h\right)\mapsto\left(J_{n}{}^{t}\bar{h}^{-1}J_{n}^{-1},J_{n}{}^{t}\bar{g}^{-1}J_{n}^{-1}\right)$ be an involution on $M$, where 
$$
J_{n}=\left(\begin{array}{cccc}
0 & \cdots & 0 & -1\\
\vdots & 0 & 1 & 0\\
0 & \cdots & 0 & \vdots\\
\left(-1\right)^{n} & 0 & \cdots & 0
\end{array}\right).
$$
The restriction of $\theta_{n}$ to $G=\text{Res}_{E/F}\text{GL}_{n}$
(realised by the diagonal embedding to $M$) deduces an involution
on $G$, noting that here $G$ can be identified with $\text{Res}_{E/F}\text{GL}_{n}$ by our choice of basis for $V$. Let $\tilde{M}=M\theta_{n}$ and $\tilde{G}=G\theta_{n}$.
We define the Weil representation $\omega_{\mu}$ of $G\left(F\right)$
realised on ${\mathcal{S}}\left(V\right)$ by 
$$
\left(\omega_{\mu}\left(g\right)\phi\right)\left(v\right)=\left|\det g\right|^{\frac{1}{2}}\mu\left(\det g\right)\phi\left(vg\right),
$$
for any $g\in G\left(F\right)$ and $\phi\in{\mathcal{S}}\left(V\right)$.
Then $\omega_{\mu}$ can be extended to a representation $\tilde{\omega}_{\psi,\mu}$
of $\tilde{G}\left(F\right)$ by letting $\tilde{\omega}_{\psi,\mu}\left(\theta_{n}\right)\phi=\widehat{\overline{\phi}\left(\cdot J_{n}\right)}$,
where $\bar{\phi}\left(v\right)=\phi\left(\bar{v}\right)$ induced by the action of $\text{Gal}(E/F)$ on $\text{Res}_{E/F}V$ (depending on a choice of $E$-basis) and $\hat{\phi}$
is the Fourier transform of $\phi$ with respect to $\psi_{E}$. Let
$(\pi,\tilde{\pi})$ be an irreducible tempered representation of $\tilde{M}\left(F\right)$.
We consider the $1$-dimensional space $\text{Hom}_{G}\left(\pi,\omega_{\mu}\right)$ (cf. \cite[Theorem B]{Sun12}).
Let $\ell\in\text{Hom}_{G}\left(\pi,\omega_{\mu}\right)$ be a nonzero element.
There exists a constant $c$ such that $\tilde{\omega}_{\psi,\mu}\left(\tilde{y}\right)\circ\ell\circ\tilde{\pi}\left(\tilde{y}\right)^{-1}=c\ell$,
for any $\tilde{y}\in\tilde{G}\left(F\right)$. This constant measures
the failure of $\ell$ from being an $\tilde{G}\left(F\right)$-equivariant
homomorphism. Moreover, by a classical local Rankin-Selberg $L$-function argument, one can show that $c$ is equal to an elementary factor up to $\epsilon\left(\frac{1}{2},\pi\times\mu^{-1},\psi_{E}\right)$.

We follow the setting in \cite{BP14}. By Frobenius reciprocity, we have 
$$
\text{Hom}_{\tilde{G}}\left(\tilde{\pi},\tilde{\omega}_{\psi,\mu}\right)=\text{Hom}_{\tilde{M}}\left(\tilde{\pi},\text{Ind}_{\tilde{G}}^{\tilde{M}}\tilde{\omega}_{\psi,\mu}\right).
$$
This motivates us to study the unitary representation $L^{2}\left(\tilde{G}\left(F\right)\backslash\tilde{M}\left(F\right),\tilde{\omega}_{\psi,\mu}\right)$
of $\tilde{M}\left(F\right)$, whose action is given by right translation.
Similar to part 1, a function $\tilde{f}\in C^{\infty}_{c}\left(\tilde{M}\left(F\right)\right)$
acts on this space by a kernel operator 
$$
K_{\tilde{f}}\left(x,y\right)=\int_{\tilde{G}\left(F\right)}f\left(y^{-1}\tilde{h}x\right)\tilde{\omega}_{\psi,\mu}\left(\tilde{h}\right)d\tilde{h},\ \ x,y\in M\left(F\right).
$$
We write the trace of $\text{Ind}_{\tilde{G}}^{\tilde{M}}\left(\tilde{f}\right)$
as the integration over $G\left(F\right)\backslash M\left(F\right)$
of $\text{Tr}\left(K_{\tilde{f}}\left(x,x\right)\right)$. More precisely,
we define the following linear form 
$$
\tilde{J}\left(\tilde{f}\right)=\int_{A_{\tilde{M}}\left(F\right)G\left(F\right)\backslash M\left(F\right)}\underset{i}{\sum}\int_{\tilde{G}\left(F\right)}\tilde{f}\left(g^{-1}\tilde{h}g\right)\left\langle \phi_{i},\tilde{\omega}_{\psi,\mu}\left(\tilde{h}\right)\phi_{i}\right\rangle d\tilde{h}dg,
$$
where $\left\{ \phi_{i}\right\} _{I}$ is an orthonormal basis for
$\omega_{\mu}$. Unlike \cite{Wal12c,BP14,BP15}, we do not
obtain the geometric side of $\tilde{J}(\tilde{f})$. Our alternative approach is
to compare the elliptic part of $\tilde{J}\left(\tilde{f}\right)$
and one in $J\left(f\right)=\mu\left(\det V\right)J_{V}\left(f\right)+\mu\left(\det V^{\prime}\right)J_{V^{\prime}}\left(f\right)$
directly. By using a spectral expansion
of $\tilde{J}\left(\tilde{f}\right)$, combining with induction and
Harish-Chandra descent, it suffices to compare the two linear forms
near the identity. Then one can use a homogeneity statement and an
instance for the Steinberg representation as in \cite[Corollary 5.10]{GGP23}
to prove Theorem \ref{thm1.2}.

We conclude this introduction with an outline of this paper. The second section is intended to set up the notations and recall some well-known results. In section \ref{sec4}, we introduce the local twisted GGP conjecture and its basic objects. They include a Weil representation and its functorial property, as well as spherical variety structure and some estimates. 

Sections 4-6 are devoted to proving the first part of Theorem \ref{thm1.2}. In section \ref{sec5}, we show that there exists an explicit tempered intertwining, which
is essential in the proof of a spectral side expansion. Section \ref{sec6}
contains an introduction of the distribution $J_{V}\left(f\right)$
and its spectral expansion. In section \ref{sec7}, we prove a geometric
expansion of the distribution $J_{\text{qc}}$, which is a linearization of the distribution $J_V\left(f\right)$ studied in the previous section, and then use it to
prove Theorem \ref{thm1.2}(i).

Sections 7-9 are devoted to proving part (ii) of Theorem \ref{thm1.2}. In section \ref{sec8}, we recall twisted endoscopy, base change, Whittaker model and a simple proof for the twisted endoscopic character identity in our case of interest. We also introduce a twisted local trace formula and some fundamental settings, including the relation between tempered intertwinings and the $\epsilon$-factors of our interest. Section \ref{sec9} is devoted to proving a spectral expansion for our twisted trace formula. In the last section, which is section \ref{sec10}, we make a comparison between the elliptic part of our twisted trace formula and the original one (i.e. the linear form $J_{\text{qc}}$), and then use it to prove Theorem \ref{thm1.2}(ii).

Finally, in Appendix \ref{app}, we give a proof for the finite multiplicity of the local twisted GGP conjecture.

\section{Preliminaries}\label{sec2}
\subsection{Groups, measures and notations}\label{sec2.1}

Let $F$ be a $p$-adic field for which we fix an algebraic closure
$\bar{F}$. We denote by $\left|\,\cdot\,\right|_{F}$ the canonical absolute
value on $F$ as well as its unique extension to $\bar{F}$. Let $G$
be a connected reductive group defined over $F$. We denote by $A_{G}$
the split component of the connected component of the center
$Z_{G}$ of $G$. Let $X^{*}\left(G\right)$ be the group of algebraic characters of $G$
defined over $F$ and ${\mathcal{A}}_{G}^{*}=X^{*}\left(G\right)\otimes_{\mathbb{Z}}\mathbb{R}$
and ${\mathcal{A}}_{G}=\text{Hom}\left(X^{*}\left(G\right),\mathbb{R}\right)$.
We define the homomorphism 
$$
\begin{array}{ccccc}
H_{G} & : & G\left(F\right) & \longrightarrow & {\mathcal{A}}_{G}\\
 &  & g & \mapsto & \left(\chi\mapsto\log\left|\chi\left(g\right)\right|_{F}\right)
\end{array}.
$$
Let ${\mathcal{A}}_{G,F}$ and $\tilde{{\mathcal{A}}}_{G,F}$ be images of $G\left(F\right)$
and $A_{G}\left(F\right)$ in ${\mathcal{A}}_{G}$ via $H_{G}$. They are
lattices in ${\mathcal{A}}_{G}$. We set ${\mathcal{A}}_{G,F}^{\vee}=\text{Hom}\left({\mathcal{A}}_{G,F},2\pi\mathbb{Z}\right)$
and $\tilde{{\mathcal{A}}}_{G,F}^{\vee}=\text{Hom}\left(\tilde{{\mathcal{A}}}_{G,F},2\pi\mathbb{Z}\right)$
to be lattices in ${\mathcal{A}}_{G}^{*}$. For a maximal torus $T$ of $G$,
let $\delta\left(G\right)=\dim G-\dim T$, noting that it does not depend on
choices of $T$.

We denote by $\mathfrak{g}$ the Lie algebra of $G$ and $\left(g,X\right) \mapsto gXg^{-1}: G\times\mathfrak{g} \rightarrow \mathfrak{g}$ the adjoint action. For $x\in G$, we denote by $Z_{G}\left(x\right)$
the centralizer of $x$ in $G$ and by $G_{x}$ its identity component.
We call an element $x$ in $G$ semisimple if it is contained in a
maximal torus of $G$. We denote by $G_{ss}$ the subset of $G$ containing
its semisimple elements. For $x\in G_{ss}$, we set $D^{G}\left(x\right)=\left|\det\left(1-\text{Ad}\left(x\right)\right)_{\mid\mathfrak{g}/\mathfrak{g}_{x}}\right|_{F}$.
An element $x\in G$ is called regular if $Z_{G}\left(x\right)$ is
abelian and $G_{x}$ is a torus. We denote by $G_{reg}$ the subset
of regular elements in $G$. Let ${\mathcal{T}}\left(G\right)$ be a set
of representatives for the conjugacy classes of maximal tori in $G$.
A maximal torus $T$ of $G$ is elliptic if $A_{T}=A_{G}$. An element
$x\in G\left(F\right)$ is said to be elliptic if it belongs to some
elliptic maximal torus. We set $G\left(F\right)_{\text{ell}}$ and
$G_{\text{reg}}\left(F\right)_{\text{ell}}$ the subsets of elliptic
elements in $G\left(F\right)$ and $G_{\text{reg}}\left(F\right)$.

Let us fix a minimal parabolic subgroup $P_{\min}$ of $G$ and a
Levi component $M_{\min}$. We fix a maximal compact subgroup $K$
of $G\left(F\right)$ in good relative position to $M_{\min}$. Let $P=MU$
be a parabolic subgroup of $G$. We have the Iwasawa decomposition
$G\left(F\right)=M\left(F\right)U\left(F\right)K$. We can choose
maps 
$$m_{P}:G\left(F\right)\rightarrow M\left(F\right),\ \ \ u_{P}:G\left(F\right)\rightarrow U\left(F\right),
\ \ \ k_{P}:G\left(F\right)\rightarrow K$$
such that $g=m_{P}\left(g\right)u_{P}\left(g\right)k_{P}\left(g\right)$,
for all $g\in G\left(F\right)$. Then we extend the homomorphism
$H_{M}$, which is defined similarly to $H_G$, to $H_{P}(g)=H_{M}\left(m_{P}\left(g\right)\right)$, for $g\in G(F)$. The above map depends on the maximal compact subgroup $K$ but its
restriction to $P\left(F\right)$ does not and is given by $H_P(mu)=H_M(m)$ for all $m\in M(F)$ and $u \in U(F)$. For a Levi subgroup $M$
of $G$, we denote by ${\mathcal{P}}\left(M\right)$, ${\mathcal{L}}\left(M\right)$
and ${\mathcal{F}}\left(M\right)$ the finite sets of parabolic subgroups
admitting $M$ as their Levi component, of Levi subgroups containing $M$
and of parabolic subgroups containing $M$ respectively. If $M\subset L$
are two Levi subgroups, we set ${\mathcal{A}}_{M}^{L}={\mathcal{A}}_{M}/{\mathcal{A}}_{L}$.

For every Levi subgroup $M$ and maximal torus $T$ of $G$, we denote
by $W\left(G,M\right)$ and $W\left(G,T\right)$ the Weyl groups of
$M\left(F\right)$ and $T\left(F\right)$ respectively, that is 
$$
W\left(G,M\right)=\text{Norm}_{G\left(F\right)}\left(M\right)/M\left(F\right)\text{ and }W\left(G,T\right)=\text{Norm}_{G\left(F\right)}\left(T\right)/T\left(F\right).
$$
We have the Weyl integration formula 
$$
\int_{G\left(F\right)}f\left(g\right)dg=\underset{T\in{\mathcal{T}}\left(G\right)}{\sum}\left|W\left(G,T\right)\right|^{-1}\int_{T\left(F\right)}D^{G}\left(t\right)\left(\int_{T\left(F\right)\backslash G\left(F\right)}f\left(g^{-1}tg\right)dg\right)dt,
$$
for any $f\in C_{c}^{\infty}\left(G\left(F\right)\right)$, where the measure on $T\left(F\right)\backslash G\left(F\right)$ (which we also denote by $dg$) arises from the quotient of ones on $G(F)$ and $T(F)$.

A twisted group is a pair $\left(G,\tilde{G}\right)$, where $G$
is a connected reductive group defined over $F$ and $\tilde{G}$
is a $G$-bitorsor, i.e. an algebraic variety defined over $F$ with
two left and right commutative actions, each of them making $\tilde{G}$ into a principal homogeneous space
under $G$. The underlying group $G$ is usually omitted and we denote the twisted group $\left(G,\tilde{G}\right)$ by $\tilde{G}$. Note that when $\tilde{G}=G$ and $G$-actions are group actions,
we have $\left(G,\tilde{G}\right)$ coincides with $G$.

Let $\tilde{G}$ be a twisted group. For any $\tilde{x}\in\tilde{G}$,
there exists a unique automorphism $\theta_{\tilde{x}}$ of $G$ such
that $\tilde{x}g=\theta_{\tilde{x}}\left(g\right)\tilde{x}$ for all
$g\in G$. This induces automorphisms on $X^{*}\left(G\right)$,
$A_{G}$ and ${\mathcal{A}}_{G}$, which are independent of choices of $\tilde{x}$.
For simplicity, we denote the three automorphisms by $\theta_{\tilde{G}}$. 

Assume that $\theta_{\tilde{G}}$ is of finite order. Denote 
$$
A_{\tilde{G}}=\left(A_{G}^{\theta_{\tilde{G}}=1}\right)^{0},\ {\mathcal{A}}_{\tilde{G}}={\mathcal{A}}_{G}^{\theta_{\tilde{G}}=1},\ {\mathcal{A}}_{\tilde{G}}^{*}=\left({\mathcal{A}}_{G}^{*}\right)^{\theta_{\tilde{G}}=1},\,a_{\tilde{G}}=\dim\left({\mathcal{A}}_{\tilde{G}}\right).
$$
Similar to the untwisted setting, we define the homomorphism $H_{\tilde{G}}: G\left(F\right) \rightarrow {\mathcal{A}}_{\tilde{G}}$. The group $G$ admits a conjugation action on $\tilde{G}$ by $\left(g,\tilde{x}\right)=g\tilde{x}g^{-1}$.
For a subset $\tilde{X}$ of $\tilde{G}$, we denote by $\text{Norm}_{G}\left(\tilde{X}\right)$
resp. $Z_{G}\left(\tilde{X}\right)$ resp. $G_{\tilde{X}}$ the normalizer
resp. the centralizer resp. the identity component of the centralizer
of $\tilde{X}$. For a subset $X$ of $G$, we denote by $N_{\tilde{G}}\left(X\right)$
and $Z_{\tilde{G}}\left(X\right)$ the normalizer and centralizer
of $X$ in $\tilde{G}$ via the action $\left(\tilde{x},g\right)\mapsto\theta_{\tilde{x}}\left(g\right)$.

We call an element $\tilde{x}$ in $\tilde{G}$ semisimple if there
exists a pair $\left(B,T\right)$ consisting of a Borel subgroup $B$
of $G$ and a maximal torus $T$ of $B$ defined over $\bar{F}$ such
that $\tilde{x}$ normalizes $B$ and $T$. We denote by $\tilde{G}_{ss}$
the subset of $\tilde{G}$ containing its semisimple elements. For
$\tilde{x}\in\tilde{G}_{ss}$, let $D^{\tilde{G}}\left(\tilde{x}\right)=\left|\det\left(1-\theta_{\tilde{x}}\right)_{\mid\mathfrak{g}/\mathfrak{g}_{\tilde{x}}}\right|_{F}$. An element $\tilde{x}\in\tilde{G}$ is called regular if $Z_{G}\left(\tilde{x}\right)$ is abelian and $G_{\tilde{x}}$ is a torus. We denote by $\tilde{G}_{reg}$ the subset of regular elements.

We denote a twisted parabolic subgroup of $\tilde{G}$ by a pair $\left(P,\tilde{P}\right)$,
where $P$ is a parabolic subgroup of $G$ defined over $F$ and $\tilde{P}$
is the normalizer of $P$ in $\tilde{G}$ such that $\tilde{P}\left(F\right)\neq0$.
For such pair, $\tilde{P}$ completely determines $P$, so we often
call $\tilde{P}$ as a twisted parabolic subgroup. 
A twisted Levi component of $\tilde{P}$ is a pair $\left(M,\tilde{M}\right)$
consisting of a Levi component $M$ of $P$ (defined over $F$) and
the normalizer $\tilde{M}$ of $M$ in $\tilde{P}$ such that $\tilde{M}\left(F\right)\neq\emptyset$.
As the second term completely determines the first term, we call $\tilde{M}$
a twisted Levi component of $\tilde{P}$. Let $\tilde{P}=\tilde{M}U$. We can
naturally extend the modulus character $\delta_{P}$ to $\tilde{P}\left(F\right)$. Namely, for any $\tilde{x}=\tilde{m}u \in \tilde{P}(F)$, we set $\delta_P(\tilde{x})=\det(\theta_{\tilde{m}\mid\mathfrak{u}})$. For a twisted Levi subgroup
$\tilde{M}$, we denote by ${\mathcal{P}}\left(\tilde{M}\right),\,{\mathcal{F}}\left(\tilde{M}\right),\,{\mathcal{L}}\left(\tilde{M}\right)$
the finite sets of twisted parabolic subgroups admitting
$\tilde{M}$ as a twisted Levi component, of twisted parabolic subgroups
containing $\tilde{M}$ and of twisted Levi subgroups containing $\tilde{M}$,
respectively. For twisted Levi subgroups $\tilde{M},\tilde{L}$ and
a twisted parabolic subgroup $\tilde{P}$, we notice that $\tilde{M}\subset\tilde{L}$
and $\tilde{M}\subset\tilde{P}$ imply $M\subset L$ and $M\subset P$,
respectively. Let $\tilde{Q}$ be a twisted parabolic subgroup. Then
$\tilde{Q}=\tilde{L}U$, where $\tilde{L}$ is a twisted Levi component
of $\tilde{Q}$ and $U$ is the unipotent radical of $Q$.
Twisted Levi subgroups are characterized as centralizers in $\tilde{G}$
of split tori. In other words, if $A$ is a split subtorus of $G$
such that $Z_{\tilde{G}}\left(A\right)\left(F\right)\neq\emptyset$,
then $Z_{\tilde{G}}\left(A\right)$ is a twisted Levi of $\tilde{G}$.
Conversely, if $\tilde{M}$ is a twisted Levi of $\tilde{G}$, then
$\tilde{M}=Z_{\tilde{G}}\left(A_{\tilde{M}}\right)$.

Let $W^{G}=\text{Norm}_{G\left(F\right)}\left(M_{\min}\right)/M_{\min}\left(F\right)$.
We denote $\tilde{P}_{\min}=N_{\tilde{G}}\left(P_{\min}\right)$ and $\tilde{M}_{\min}=N_{\tilde{G}}\left(P_{\min},M_{\min}\right)$.
Then $\tilde{P}_{\min}$ is a minimal twisted parabolic subgroup and
$\tilde{M}_{\min}$ is a minimal Levi component. Let ${\mathcal{L}}^{\tilde{G}}={\mathcal{L}}\left(\tilde{M}_{\min}\right)$. Let $\tilde{M}\in{\mathcal{L}}^{\tilde{G}}$ and $\tilde{P}=\tilde{M}U\in{\mathcal{P}}\left(\tilde{M}\right)$. Similar to the untwisted setting, we can also define a map $H_{\tilde{P}} : G\left(F\right) \rightarrow {\mathcal{A}}_{\tilde{M}}$.

A twisted maximal torus of $\tilde{G}$ is a pair $\left(T,\tilde{T}\right)$
consisting of a maximal torus $T$ of $G$ defined over $F$ and a
subvariety $\tilde{T}$ of $\tilde{G}$ (defined over $F$), which
is the intersection of normalizers of $T$ and a Borel subgroup
$B$ (defined over $\bar{F}$) containing $T$ in $\tilde{G}$, such that $\tilde{T}\left(F\right)\neq\emptyset$.
For such pair, the restriction to $T$ of automorphisms $\theta_{\tilde{x}}$
for $\tilde{x}\in\tilde{T}$ does not depend on $\tilde{x}$. We denote
this restriction by $\theta_{\tilde{T}}$, or simply $\theta$ if
there is no confusion. Let $T_{\theta}$ be the connected component
of the subgroup of fixed points $T^{\theta}$. We denote by $\tilde{T}\left(F\right)/\theta$
the set of orbits of the action of $T\left(F\right)$ on $\tilde{T}\left(F\right)$
by conjugation. It is naturally an $F$-analytic manifold and for
all $\tilde{t}\in\tilde{T}\left(F\right)/\theta$, the map $t \mapsto t\tilde{t}:T_{\theta}\left(F\right) \rightarrow \tilde{T}\left(F\right)/\theta$ is a local isomorphism. If $T$ is split, we define a Haar measure of $T\left(F\right)$ such that the volume of its maximal compact subgroup is equal to $1$. In general, we provide $A_{T}\left(F\right)$ with this measure and choose a measure for $T\left(F\right)$ such that $\text{vol}\left(T\left(F\right)/A_{T}\left(F\right)\right)=1$.
There exists a unique measure of $\tilde{T}\left(F\right)/\theta$
such that the above map preserves local measure for any $\tilde{t}\in\tilde{T}\left(F\right)/\theta$.
We equip $\tilde{T}\left(F\right)/\theta$ with this measure. Moreover, the principal homogeneous space $\tilde{G}(F)$ inherits the measure of $G(F)$. We have the Weyl integration formula 
$$
\int_{\tilde{G}\left(F\right)}\tilde{f}\left(\tilde{x}\right)d\tilde{x}=\underset{\tilde{T}\in{\mathcal{T}}\left(\tilde{G}\right)}{\sum}\left|W\left(G,\tilde{T}\right)\right|^{-1}\left|T^{\theta}\left(F\right):T_{\theta}\left(F\right)\right|^{-1}\int_{\tilde{T}\left(F\right)/\theta}D^{\tilde{G}}\left(\tilde{t}\right)\int_{T_{\theta}\left(F\right)\backslash G\left(F\right)}f\left(g^{-1}\tilde{t}g\right)dgd\tilde{t},
$$
for any $\tilde{f}\in C_{c}^{\infty}\left(\tilde{G}\left(F\right)\right)$, where $dg$ is the measure defined by the quotient of ones on $G(F)$ and $T_\theta(F)$.

We denote $\mathcal{A}_{\tilde{G},F}=H_{\tilde{G}}\left(G\left(F\right)\right)$,
${\mathcal{A}}_{A_{\tilde{G}},F}=H_{\tilde{G}}\left(A_{\tilde{G}}\left(F\right)\right)$,
${\mathcal{A}}_{\tilde{G},F}^{\vee}=\text{Hom}\left({\mathcal{A}}_{\tilde{G},F},2\pi\mathbb{Z}\right)$,
${\mathcal{A}}_{A_{\tilde{G}},F}^{\vee}=\text{Hom}\left({\mathcal{A}}_{A_{\tilde{G}},F},2\pi\mathbb{Z}\right)$.
Then ${\mathcal{A}}_{\tilde{G},F}$ and ${\mathcal{A}}_{A_{\tilde{G}},F}$ are
lattices in ${\mathcal{A}}_{\tilde{G}}$, whereas ${\mathcal{A}}_{\tilde{G},F}^{\vee}$ and ${\mathcal{A}}_{A_{\tilde{G}},F}^{\vee}$ are lattices in ${\mathcal{A}}_{\tilde{G}}^{*}$.
We equip these lattices with the counting measure. We set Haar measures
on ${\mathcal{A}}_{\tilde{G}}$ and ${\mathcal{A}}_{\tilde{G}}^{*}$ such that
volumes of ${\mathcal{A}}_{\tilde{G}}/{\mathcal{A}}_{A_{\tilde{G}},F}$ and
${\mathcal{A}}_{\tilde{G}}^{*}/{\mathcal{A}}_{A_{\tilde{G}},F}^{\vee}$ are
equal to $1$, respectively.

For an affine algebraic variety $X$ over $\bar{F}$, let ${\mathcal{O}}\left(X\right)$
be the ring of regular functions. We choose a finite set of generators
$\left\{ f_{1},\ldots,f_{m}\right\} $ of ${\mathcal{O}}\left(X\right)$ as an $\bar{F}$-algebra. Define 
$$
\sigma_{X}\left(x\right)=1+\log\left(\max\left\{ 1,\left|f_{1}\left(x\right)\right|,\ldots,\left|f_{m}\left(x\right)\right|\right\} \right),\ \text{for } x\in X.
$$
Two such functions $\sigma_{X}$ and $\tilde{\sigma}_{X}$ are called
equivalent if $\sigma_{X}\sim\tilde{\sigma}_{X}$, i.e. there exists
$C_{1},C_{2}>0$ such that $C_{1}\tilde{\sigma}_{X}<\sigma_{X}<C_{2}\tilde{\sigma}_{X}$.
A log-norm on $X$ is a particular function $\sigma_{X}$ inside its
equivalence class. Generally, for any algebraic variety $X$ over
$\bar{F}$, we choose a finite open affine covering $\left(U_{i}\right)_{i\in I}$
of $X$ and fix log-norms $\sigma_{U_{i}}$ on $U_{i}$, for $i\in I$.
We define a log-norm on $X$ by letting 
$$
\sigma_{X}\left(x\right)=\inf\left\{ \sigma_{U_{i}}\left(x\right)\mid\ i\in I\text{ and }x\in U_{i}\right\} .
$$
We denote by $\Xi^G$ the Harish-Chandra function on $G(F)$ (see \cite[Section 1.5]{BP20} for a precise definition). By using the function $\Xi^{G}$, as in \cite[Section 1.5]{BP20}, we can define the (weak) Harish-Chandra Schwartz space ${\mathcal{C}}\left(G\left(F\right)\right)$ (resp. $\mathcal{C}^w(G(F))$).

Let us fix an element $\tilde{x}\in\tilde{G}\left(F\right)$. We define the Harish-Chandra-Schwartz
space ${\mathcal{C}}\left(\tilde{G}\left(F\right)\right)$ as the space of functions $\tilde{f}\in C\left(\tilde{G}\left(F\right)\right)$ such that $f(g):=\tilde{f}(g\tilde{x})$ lies in $\mathcal{C}(G(F))$.

\subsection{Tempered representations}\label{sec2.3}

A unitary representation of $G\left(F\right)$ is a continuous representation
$\left(\pi,V_{\pi}\right)$ of $G\left(F\right)$ on a Hilbert space
$V_{\pi}$ such that for any $g\in G\left(F\right)$, the operator
$\pi\left(g\right)$ is unitary. There is an action of $i{\mathcal{A}}_{G}^{*}$
on unitary representations given by $\left(\lambda,\pi\right)\mapsto\pi_{\lambda}$,
where $\pi_{\lambda}\left(g\right)=e^{\lambda\left(H_{G}\left(g\right)\right)}\pi\left(g\right)$
for any $g\in G\left(F\right)$. We denote by $i{\mathcal{A}}_{G,\pi}^{*}$
the stabilizer of $\pi$ for this action. Let $\text{End}\left(\pi\right)$ be the space of continuous endomorphisms of the space of $\pi$ and $\text{End}\left(\pi\right)^\infty$ be its subspace containing smooth vectors. As in \cite[Section 2.2]{BP20}, we endow $\text{End}\left(\pi\right)^\infty$ with its finest locally convex topology. From now, we assume any representations that we consider are of finite length. Let $\text{Temp}(G)$ and $\Pi_2(G)$ be the sets of isomorphism classes of irreducible tempered representations and irreducible square-integrable representations, respectively. 

Square-integrable representations are preserved by unramified twists. Let $\Pi_{2}\left(G\right)/i\mathcal{A}^*_{G,F}$ be the set of orbits in $\Pi_{2}\left(G\right)$ via this action. Let ${\mathcal{X}}_{\text{temp}}\left(G\right)$ be the set of isomorphism classes of tempered representations of $G\left(F\right)$ of the form $i_{M}^{G}\left(\sigma\right)$, where $M$ is a Levi subgroup of $G$ and $\sigma$ is an irreducible square-integrable representation of $M\left(F\right)$. Let ${\mathcal{M}}$ be a set of representatives for the conjugacy classes
of Levi subgroups of $G$. Then ${\mathcal{X}}_{\text{temp}}\left(G\right)$ is naturally a quotient of 
$$
\tilde{{\mathcal{X}}}_{\text{temp}}\left(G\right)=\underset{M\in{\mathcal{M}}}{\bigsqcup}\ \ \underset{{\mathcal{O}}\in\Pi_{2}\left(M\right)/i{\mathcal{A}}_{M,F}^{*}}{\bigsqcup}{\mathcal{O}}.
$$
Since each orbit $\mathcal{O}\in \{\Pi_{2}\left(M\right)\}$ is a quotient of $i{\mathcal{A}}_{M,F}^{*}$ by a finite subgroup, ${\mathcal{X}}_{\text{temp}}\left(G\right)$ is a real smooth manifold. We equip ${\mathcal{X}}_{\text{temp}}\left(G\right)$
with the quotient topology.

Let $V$ be a locally convex topological vector space. A function
$f:{\mathcal{X}}_{\text{temp}}\left(G\right)\rightarrow V$ is smooth if
the pullback of $f$ to $\tilde{{\mathcal{X}}}_{\text{temp}}\left(G\right)$
is a smooth function. We denote by $C^{\infty}\left({\mathcal{X}}_{\text{temp}}\left(G\right),V\right)$
the space of smooth functions on ${\mathcal{X}}_{\text{temp}}\left(G\right)$
taking values in $V$. For simplicity, we set $C^{\infty}\left({\mathcal{X}}_{\text{temp}}\left(G\right)\right)=C^{\infty}\left({\mathcal{X}}_{\text{temp}}\left(G\right),\mathbb{C}\right)$.

We define a regular Borel measure $d\pi$ on ${\mathcal{X}}_{\text{temp}}\left(G\right)$
by 
$$
\int_{{\mathcal{X}}_{\text{temp}}\left(G\right)}\phi\left(\pi\right)d\pi=\underset{M\in{\mathcal{M}}}{\sum}\left|W\left(G,M\right)\right|^{-1}\underset{{\mathcal{O}}\in\Pi_{2}\left(M\right)/i{\mathcal{A}}_{M,F}^{*}}{\sum}\left[i{\mathcal{A}}_{M,\sigma}^{\vee}:i{\mathcal{A}}_{M,F}^{\vee}\right]^{-1}\int_{i{\mathcal{A}}_{M,F}^{*}}\phi\left(i_{P}^{G}\left(\sigma_{\lambda}\right)\right)d\lambda,
$$
for any $\phi\in C_{c}^{\infty}\left({\mathcal{X}}_{\text{temp}}\left(G\right)\right)$,
where for any $M\in{\mathcal{M}}$ we have fixed $P\in{\mathcal{P}}\left(M\right)$
and for any ${\mathcal{O}}\in\Pi_{2}\left(M\right)/i{\mathcal{A}}_{M,F}^{*}$
we have fixed a base-point $\sigma\in{\mathcal{O}}$. For any $\pi=i^G_P(\sigma)\in \mathcal{X}_\text{temp}(G)$, we set $\mu(\pi)=d(\sigma)j(\sigma)^{-1}$. This quantity only depends on $\pi$.

We now define a space of smooth functions $C^{\infty}\left({\mathcal{X}}_{\text{temp}}\left(G\right),{\mathcal{E}}\left(G\right)\right)$. The elements in $C^{\infty}\left({\mathcal{X}}_{\text{temp}}\left(G\right),{\mathcal{E}}\left(G\right)\right)$ are defined as functions $\pi\in\text{Temp}\left(G\right)\mapsto T_{\pi}\in\text{End}\left(\pi\right)^\infty$ such that for any parabolic subgroup $P=MU$ and $\sigma\in\Pi_{2}\left(M\right)$,
setting $\pi_{K}=i_{K\cap P}^{K}\left(\sigma\mid_{P\cap K}\right)$
and $\pi_{\lambda}=i_{P}^{G}\left(\sigma_{\lambda}\right)$ for all
$\lambda\in i{\mathcal{A}}_{M}^{*}$, the function $\lambda\in i{\mathcal{A}}_{M}^{*}\mapsto T_{\pi_{\lambda}}\in\text{End}\left(\pi_{\lambda}\right)^\infty\simeq\text{End}\left(\pi_{K}\right)^\infty$ is smooth. Let 
$$
\mathcal{C}\left({\mathcal{X}}_{\text{temp}}\left(G\right),{\mathcal{E}}\left(G\right)\right)=\left\{ T\in C^{\infty}\left({\mathcal{X}}_{\text{temp}}\left(G\right),{\mathcal{E}}\left(G\right)\right)\mid\,\text{Supp}\left(T\right)\text{ is compact}\right\} .
$$
For later use, we state the matricial Paley-Wiener theorem, which
is a strong version of the Harish-Chandra Plancherel formula \cite{Wal03}.
\begin{theorem}\label{2.2}
$ $
\begin{enumerate}
\item The map $f\in{\mathcal{C}}\left(G\left(F\right)\right)\mapsto\left(\pi\in\text{Temp}\left(G\right)\mapsto\pi\left(f\right)\in\text{End}\left(\pi\right)^\infty\right)$
induces a topological isomorphism ${\mathcal{C}}\left(G(F)\right)\simeq \mathcal{C}\left({\mathcal{X}}_{\text{temp}}\left(G\right),{\mathcal{E}}\left(G\right)\right)$.
\item The inverse of that isomorphism is given by sending $T\in \mathcal{C}\left({\mathcal{X}}_{\text{temp}}\left(G\right),{\mathcal{E}}\left(G\right)\right)$
to the function $f_{T}$ defined by 
$$
f_{T}\left(g\right)=\int_{{\mathcal{X}}_{\text{temp}}\left(G\right)}\text{Trace}\left(\pi\left(g^{-1}\right)T_{\pi}\right)\mu\left(\pi\right)d\pi,
$$
where $\mu(\pi)$ is defined in \cite[page 59]{BP20}.
\end{enumerate}
\end{theorem}

We now consider representations of a twisted group $\tilde{G}(F)$. A representation of $\tilde{G}\left(F\right)$ is a triple $\left(\pi,\tilde{\pi},E_{\pi}\right)$,
where $\pi$ is a smooth representation of $G\left(F\right)$ with an underlying space $E_{\pi}$ and $\tilde{\pi}:\tilde{G}\left(F\right)\rightarrow\text{Aut}_{\mathbb{C}}\left(E_{\pi}\right)$
satisfying $\tilde{\pi}\left(g\tilde{x}g^{\prime}\right)=\pi\left(g\right)\tilde{\pi}\left(\tilde{x}\right)\pi\left(g^{\prime}\right)$,
for any $g,g^{\prime}\in G\left(F\right)$ and $\tilde{x}\in\tilde{G}\left(F\right)$.
Two representations $\left(\pi_{1},\tilde{\pi}_{1},E_{\pi_{1}}\right)$
and $\left(\pi_{2},\tilde{\pi}_{2},E_{\pi_{2}}\right)$ are equivalent
if there exists linear isomorphisms $A:E_{\pi_{1}}\rightarrow E_{\pi_{2}}$
and $B:E_{\pi_{1}}\rightarrow E_{\pi_{2}}$ which intertwine $\pi_{1}$
and $\pi_{2}$ and satisfy $B\tilde{\pi}_{1}\left(\tilde{x}\right)=\tilde{\pi}_{2}\left(\tilde{x}\right)A$,
for any $\tilde{x}\in\tilde{G}\left(F\right)$. We say a representation
$\left(\pi,\tilde{\pi},E_{\pi}\right)$ of $\tilde{G}\left(F\right)$
is admissible if $\pi$ is, and unitary if there exists a positive
definite hermitian product which is invariant under the image of $\tilde{\pi}$.
A representation $\left(\pi,\tilde{\pi},E_{\pi}\right)$ is tempered
if it is unitary, and $\pi$ is of finite length and any irreducible
subrepresentations of $\pi$ are tempered. In general, we omit the
term $\left(\pi,E_{\pi}\right)$ and denote by $\tilde{\pi}$ a representation
of $\tilde{G}\left(F\right)$.

\subsection{Elliptic representations and the spaces ${\mathcal{X}}\left(G\right)$ and $E_\text{disc}(\tilde{G})$}\label{sec2.4}

Let $R_{\text{temp}}\left(G\right)$ be the space of complex virtual
tempered representations of $G\left(F\right)$, i.e. the complex vector
space with basis $\text{Temp}\left(G\right)$ consisting of irreducible
tempered representations of $G\left(F\right)$.

In \cite{Art93}, Arthur defines a set ${\mathcal{X}}_{\text{ell}}\left(G\right)$
of virtual tempered representations of $G\left(F\right)$ called elliptic
representations, which are actually well-defined up to scalar of module
1, i.e. ${\mathcal{X}}_{\text{ell}}\left(G\right)\subset R_{\text{temp}}\left(G\right)/\mathbb{S}^{1}$. Let $R_{\text{ell}}\left(G\right)$
be the subspace of $R_{\text{temp}}\left(G\right)$ generated by ${\mathcal{X}}_{\text{ell}}\left(G\right)$
and denote by $R_{\text{ind}}\left(G\right)$ the subspace of $R_{\text{temp}}\left(G\right)$
generated by the image of all the linear maps $R_{\text{temp}}\left(M\right)\rightarrow R_{\text{temp}}\left(G\right)$ given by parabolic inductions, where $M$ is a proper Levi subgroup of $G$. Then $R_{\text{temp}}\left(G\right)$ can be decomposed into the direc sum $R_{\text{ind}}\left(G\right)\oplus R_{\text{ell}}\left(G\right)$. The set ${\mathcal{X}}_{\text{ell}}\left(G\right)$ is invariant under unramified twists. Let ${\mathcal{X}}_{\text{ell}}\left(G\right)/i{\mathcal{A}}_{G,F}^{*}$
be the set of unramified orbits in ${\mathcal{X}}_{\text{ell}}\left(G\right)$.
Let $\underline{{\mathcal{X}}}_{\text{ell}}\left(G\right)$ be the inverse
image of ${\mathcal{X}}_{\text{ell}}\left(G\right)$ in $R_{\text{temp}}\left(G\right)$.
This set is invariant under multiplication by $\mathbb{S}^{1}$.

We denote by ${\mathcal{X}}\left(G\right)$ the subset of $R_{\text{temp}}\left(G\right)/\mathbb{S}^{1}$
consisting of virtual representations of the form $i_{M}^{G}\left(\sigma\right)$,
where $M$ is a Levi subgroup of $G$ and $\sigma\in{\mathcal{X}}_{\text{ell}}\left(M\right)$.
Also, let $\underline{{\mathcal{X}}}\left(G\right)$ be the inverse image
of ${\mathcal{X}}\left(G\right)$ in $R_{\text{temp}}\left(G\right)$.
The fibers of the natural projection $\underline{{\mathcal{X}}}\left(G\right)\rightarrow{\mathcal{X}}\left(G\right)$
are all isomorphic to $\mathbb{S}^{1}$. Let ${\mathcal{M}}$ be a set
of representatives for the conjugacy classes of Levi subgroups of
$G$. Then ${\mathcal{X}}\left(G\right)$ is naturally a quotient of $\underset{M\in{\mathcal{M}}}{\bigsqcup}\ \ \underset{{\mathcal{O}}\in{\mathcal{X}}_{\text{ell}}\left(M\right)/i{\mathcal{A}}_{M,F}^{*}}{\bigsqcup}{\mathcal{O}}$. This defines a structure of topological space on ${\mathcal{X}}\left(G\right)$.
Let us define a regular Borel measure $d\pi$ on ${\mathcal{X}}\left(G\right)$
by 
$$
\int_{{\mathcal{X}}\left(G\left(F\right)\right)}\phi\left(\pi\right)d\pi=\underset{M\in{\mathcal{M}}}{\sum}\left|W\left(G,M\right)\right|^{-1}\underset{{\mathcal{O}}\in{\mathcal{X}}_{\text{ell}}\left(M\right)/i{\mathcal{A}}_{M,F}^{*}}{\sum}\left[i{\mathcal{A}}_{M,\sigma}^{\vee}:i{\mathcal{A}}_{M,F}^{\vee}\right]^{-1}\int_{i{\mathcal{A}}_{M,F}^{*}}\phi\left(i_{M}^{G}\left(\sigma_{\lambda}\right)\right)d\lambda,
$$
for any continuous and compactly supported function $\phi$ on ${\mathcal{X}}\left(G\right)$,
where a base point $\sigma\in{\mathcal{O}}$ is fixed for every orbit ${\mathcal{O}}\in{\mathcal{X}}_{\text{ell}}\left(M\right)/i{\mathcal{A}}_{M,F}^{*}$. Finally, we extend the function $\pi\mapsto D\left(\pi\right)$ to ${\mathcal{X}}\left(G\right)$ by setting $D\left(\pi\right)=D\left(\sigma\right)$ for any $\pi=i_{M}^{G}\left(\sigma\right)$, where $M$ is a Levi
subgroup and $\sigma\in{\mathcal{X}}_{\text{ell}}\left(M\right)$.

We now move to the setting of twisted groups. Let $\text{Temp}\left(\tilde{G}\right)$
be the set of $G\left(F\right)$-irreducible tempered representations
of $\tilde{G}\left(F\right)$. Let $R_{\text{temp}}\left(\tilde{G}\right)$
be the space of complex virtual tempered representations of $G\left(F\right)$,
i.e. the complex vector space with basis $\text{Temp}\left(\tilde{G}\right)$. We recall the subsets $E_\text{disc}(\tilde{G})$ and $E_\text{ell}(\tilde{G})$ of $\text{Temp}\left(\tilde{G}\right)/conj$ defined in \cite[Section 2.8]{BW23}.
We equip $E_\text{disc}(\tilde{G})$ with the unique measure such that for every $\tau \in E_{disc}(\tilde{G})$, the action map $\lambda \in i\mathcal{A}_{\tilde{G}}^* \mapsto \lambda \cdot r$ is locally measure preserving. For every sufficiently nice function $\varphi : E_\text{disc}(\tilde{G})\rightarrow \mathbb{C}$, we have
$$\int_{E_\text{disc}(\tilde{G})}\varphi(\tau) d\tau = \underset{\tau \in E_\text{disc}(\tilde{G})/i\mathcal{A}_{\tilde{G},F}^*}{\sum} |\text{Stab}(i\mathcal{A}_{\tilde{G},F}^*,\tau)|^{-1} \int_{i\mathcal{A}_{\tilde{G},F}^*} \varphi(\lambda \cdot \tau)d\lambda,$$
where $\text{Stab}(i\mathcal{A}_{\tilde{G},F}^*,\tau)$ is the stabilizer of $\tau$ in $i\mathcal{A}_{\tilde{G},F}^*$.

\subsection{$(G,M)$-families and $(\tilde{G},\tilde{M})$-families}\label{sec2.5}

We recall some facts about $(G,M)$-families in \cite[Section 17]{Art05}. Let $M$ be a Levi subgroup of $G$ and $V$ be a locally convex topological space. A $\left(G,M\right)$-family with values in $V$ is a family $\left(c_{P}\right)_{P\in{\mathcal{P}}\left(M\right)}$ of smooth functions
on $i{\mathcal{A}}_{M}^{*}$ taking values in $V$ such that for all adjacent
parabolic subgroups $P,P^{\prime}\in{\mathcal{P}}\left(M\right)$, the
functions $c_{P}$ and $c_{P^{\prime}}$ coincide on the hyperplane
supporting the wall that separates the positive chambers for $P$
and $P^{\prime}$. For any $\left(G,M\right)$-family $\left(c_{P}\right)_{P\in{\mathcal{P}}\left(M\right)}$,
Arthur associates an element $c_{M}$ of $V$ as follows. The function $c_{M}\left(\lambda\right)=\underset{P\in{\mathcal{P}}\left(M\right)}{\sum}c_{P}\left(\lambda\right)\theta_{P}\left(\lambda\right)^{-1}$ extends to a smooth function on $i{\mathcal{A}}_{M}^{*}$ where 
$$
\theta_{P}\left(\lambda\right)=\text{meas}\left({\mathcal{A}}_{M}^{G}/\mathbb{Z}\Delta_{P}^{\vee}\right)^{-1}\underset{\alpha\in\Delta_{P}}{\prod}\lambda\left(\alpha^{\vee}\right),\text{ for }P\in{\mathcal{P}}\left(M\right)
$$
and set $c_{M}=c_{M}\left(0\right)$. Here $\Delta_{P}$ is the set
of simple roots of $A_{M}$ in $P$, $\Delta_{P}^{\vee}$ is the corresponding
set of simple coroots, and for every $\alpha\in\Delta_{P}$, $\alpha^{\vee}$
is denoted as the corresponding simple coroot.

A $\left(G,M\right)$-orthogonal set is a family $\left(Y_{P}\right)_{P\in{\mathcal{P}}\left(M\right)}$
of points in ${\mathcal{A}}_{M}$ such that for any adjacent parabolic
subgroups $P,P^{\prime}\in{\mathcal{P}}\left(M\right)$ there exists a
real number $r_{P,P^{\prime}}$ such that $Y_{P}-Y_{P^{\prime}}=r_{P,P^{\prime}}\alpha^{\vee}$,
where $\alpha$ is the unique root of $A_{M}$ that is positive for
$P$ and negative for $P^{\prime}$. If moreover we have $r_{P,P^{\prime}}\geq0$
for any adjacent $P,P^{\prime}\in{\mathcal{P}}\left(M\right)$, then we
say that the family is positive. Clearly if $\left(Y_{P}\right)_{P\in{\mathcal{P}}\left(M\right)}$
is a $\left(G,M\right)$-orthogonal set, then the family $\left(c_{P}\right)_{P\in{\mathcal{P}}\left(M\right)}$
defined by $c_{P}\left(\lambda\right)=e^{\lambda\left(Y_{P}\right)}$
is a $\left(G,M\right)$-family. If the family $\left(Y_{P}\right)_{P\in{\mathcal{P}}\left(M\right)}$
is positive, then $c_{M}$ is the volume in ${\mathcal{A}}_{M}^{G}$ of
the convex hull of the set $\left\{ Y_{P}:\ P\in{\mathcal{P}}\left(M\right)\right\} $.

Let $\tilde{M}$ be a twisted Levi subgroup of $\tilde{G}$. As in Section 2.3 in \cite{Wal12b}, we extend previous notions to the setting of twisted groups, which are $\left(\tilde{G},\tilde{M}\right)$-families and $\left(\tilde{G},\tilde{M}\right)$-orthogonal
sets. A family $\mathcal{Y}=(Y_{\tilde{P}})_{\tilde{P}\in \mathcal{P}(\tilde{M})}$ of points in $\mathcal{A}_{\tilde{M}}$ is a $(\tilde{G},\tilde{M})$-orthogonal set if for every adjacent twisted parabolic subgroups $\tilde{P},\tilde{Q}\in \mathcal{P}(\tilde{M})$, we have
$Y_{\tilde{P}}-Y_{\tilde{Q}} \in \mathbb{R}\alpha_{\tilde{P},\tilde{Q}}^\vee$, where $\alpha_{\tilde{P},\tilde{Q}}^\vee$ is the unique root of $A_{\tilde{M}}$ that is positive for $\tilde{P}$ and negative for $\tilde{Q}$. Moreover, if $Y_{\tilde{P}}-Y_{\tilde{Q}}$ lies in $\mathbb{R}_{>0}\alpha_{\tilde{P},\tilde{Q}}^\vee$, for every pair of adjacent twisted parabolic subgroups $\tilde{P},\tilde{Q}\in \mathcal{P}(\tilde{M})$, we say $\mathcal{Y}$ is positive. Similar to the group case, if $\mathcal{Y}$ is positive, we can take $v^{\tilde{Q}}_{\tilde{L}}(\mathcal{Y})$ to be the volume of the convex hull of $(Y_{\tilde{P}})_{\tilde{P}\in \mathcal{P}(\tilde{L}),\tilde{P}\subset \tilde{Q}}$, for every $\tilde{L}\in \mathcal{L}(\tilde{M})$ and $\tilde{Q}\in \mathcal{F}(\tilde{L})$. We drop the superscript when $\tilde{Q}=\tilde{G}$.
\subsection{Harish-Chandra descent}\label{sec3.1}

Let $x\in G_{ss}\left(F\right)$. Let $\Omega_{x}\subseteq G_{x}\left(F\right)$ be a $G$-good open neighborhood of $x$ (see \cite[Section 3.2]{BP20} for this definition) and we set $\Omega=\Omega_{x}^{G}$. The following
integration formula holds for any $f$ which is integrable on $\Omega$
$$
\int_{\Omega}f\left(y\right)dy=\int_{Z_{G}\left(x\right)\left(F\right)\backslash G\left(F\right)}\int_{\Omega_{x}}f\left(g^{-1}yg\right)\eta_{x}^{G}\left(y\right)dydg
$$
$$
=\left[Z_{G}\left(x\right)\left(F\right):G_{x}\left(F\right)\right]^{-1}\int_{G_{x}\left(F\right)\backslash G\left(F\right)}\int_{\Omega_{x}}f\left(g^{-1}yg\right)\eta_{x}^{G}\left(y\right)dydg.
$$
For a function $f$ on $\Omega$, let $f_{x,\Omega_{x}}$ be the
function on $\Omega_{x}$ given by $f_{x,\Omega_{x}}\left(y\right)=\eta_{x}^{G}\left(y\right)^{1/2}f\left(y\right)$.
The map $f\mapsto f_{x,\Omega_{x}}$ induces topological isomorphisms
$$
C^{\infty}\left(\Omega\right)^{G}\simeq C^{\infty}\left(\Omega_{x}\right)^{Z_{G}\left(x\right)}\text{ and }C^{\infty}\left(\Omega_{\text{rss}}\right)^{G}\simeq C^{\infty}\left(\Omega_{x,\text{ rss}}\right)^{Z_{G}\left(x\right)}.
$$

Let $\tilde{x}\in\tilde{G}_{ss}\left(F\right)$ and $\theta_{\tilde{x}}$
be the unique automorphism of $G\left(F\right)$ defined by $\tilde{x}$.
A subset $\Omega\subset G\left(F\right)$ is called completely $\theta_{\tilde{x}}$-twisted
$G\left(F\right)$-invariant if it is $\theta_{\tilde{x}}$-twisted
$G\left(F\right)$-conjugation invariant (i.e $g^{-1}\Omega\theta_{\tilde{x}}\left(g\right)=\Omega$,
for any $g\in G\left(F\right)$) and for any $x\in\Omega$, its semisimple
part $x_{s}$ also lies in $\Omega$. Let us fix a completely $\theta_{\tilde{x}}$-twisted
$G\left(F\right)$-invariant open subset $\Omega\subset G\left(F\right)$.
Let $C^{\infty}\left(\Omega\tilde{x}\right)^{G}$ be the space of
smooth and $\theta_{\tilde{x}}$-twisted $G\left(F\right)$-invariant
functions on $\Omega$. It is a closed subset of $C^{\infty}\left(\Omega\tilde{x}\right)$
and we endow it with the induced locally convex topology. Let ${\mathcal{D}}^{\prime}\left(\Omega\tilde{x}\right)^{G}$
be the space of $\theta_{\tilde{x}}$-twisted $G\left(F\right)$-invariant
distributions on $\Omega\tilde{x}$. Similarly, let us denote by $C_{c}^{\infty}\left(\Omega\tilde{x}\right)$
the space of all functions $\tilde{f}\in C_{c}^{\infty}\left(G\left(F\right)\tilde{x}\right)$
such that $\overline{\text{Supp}\left(\tilde{f}\right)^{G}}\subset\Omega\tilde{x}$.

Let $\tilde{x}\in\tilde{G}_{ss}\left(F\right)$. An open subset $\Omega_{\tilde{x}}\subset G_{\tilde{x}}\left(F\right)$
is called $G$-good if it is completely $Z_{G}\left(\tilde{x}\right)\left(F\right)$-invariant
and the map 
$$
\begin{array}{ccc}
\Omega_{\tilde{x}}\times^{Z_{G}\left(\tilde{x}\right)\left(F\right)}G\left(F\right) & \longrightarrow & \tilde{G}\left(F\right)\\
\left(y,g\right) & \mapsto & g^{-1}y\theta_{\tilde{x}}\left(g\right)\tilde{x}
\end{array}
$$
induces an $F$-analytic isomorphism between $\Omega_{\tilde{x}}\times^{Z_{G}\left(\tilde{x}\right)\left(F\right)}G\left(F\right)$
and $\left(\Omega_{\tilde{x}}\tilde{x}\right)^{G}$. The Jacobian
of the above map at $\left(y,g\right)\in\Omega_{\tilde{x}}\times^{Z_{G}\left(\tilde{x}\right)\left(F\right)}G\left(F\right)$
is equal to $\eta_{\tilde{x}}^{\tilde{G}}\left(y\right)=D^{\tilde{G}}\left(\tilde{x}\right)\left|\det\left(1-\text{Ad}\left(y\right)\right)_{\mathfrak{g}/\mathfrak{g}_{\tilde{x}}}\right|$. An open subset $\Omega_{\tilde{x}}\subset G_{\tilde{x}}\left(F\right)$ is $G$-good if and only if the following conditions are satisfied 
\begin{enumerate}
\item $\Omega_{\tilde{x}}$ is completely $Z_{G}\left(\tilde{x}\right)\left(F\right)$-invariant;
\item For any $y\in\Omega_{\tilde{x}}$, we have $\eta_{\tilde{x}}^{\tilde{G}}\left(y\right)\neq0$;
\item For any $g\in G\left(F\right)$, $g^{-1}\Omega_{\tilde{x}}\theta_{\tilde{x}}\left(g\right)\cap\Omega_{\tilde{x}}$
is nonempty if and only if $g\in Z_{G}\left(\tilde{x}\right)\left(F\right)$.
\end{enumerate}
Let $\Omega_{\tilde{x}}\subset G_{\tilde{x}}\left(F\right)$ be a
$G$-good open neighborhood of $\tilde{x}$ and set $\Omega=\left(\Omega_{\tilde{x}}\tilde{x}\right)^{G}$.
The following integration formula holds for any $\tilde{f}$ which
is integrable on $\Omega$.
$$
\int_{\Omega}\tilde{f}\left(\tilde{y}\right)d\tilde{y}=\int_{Z_{G}\left(\tilde{x}\right)\left(F\right)\backslash G\left(F\right)}\int_{\Omega_{\tilde{x}}}\tilde{f}\left(g^{-1}y\theta_{\tilde{x}}\left(g\right)\tilde{x}\right)\eta_{\tilde{x}}^{\tilde{G}}\left(y\right)dydg
$$
$$
=\left[Z_{G}\left(\tilde{x}\right)\left(F\right):G_{\tilde{x}}\left(F\right)\right]^{-1}\int_{G_{\tilde{x}}\left(F\right)\backslash G\left(F\right)}\int_{\Omega_{\tilde{x}}}\tilde{f}\left(g^{-1}y\theta_{\tilde{x}}\left(g\right)\tilde{x}\right)\eta_{\tilde{x}}^{\tilde{G}}\left(y\right)dydg.
$$
For any function $\tilde{f}$ defined on $\Omega$, let $\tilde{f}_{\tilde{x},\Omega_{\tilde{x}}}$
be the function on $\Omega_{\tilde{x}}$ given by $\tilde{f}_{\tilde{x},\Omega_{\tilde{x}}}\left(y\right)=\eta_{\tilde{x}}^{\tilde{G}}\left(y\right)^{1/2}\tilde{f}\left(y\tilde{x}\right)$.
The map $\tilde{f}\mapsto\tilde{f}_{\tilde{x},\Omega_{\tilde{x}}}$
induces topological isomorphisms 
$$
C^{\infty}\left(\Omega\right)^{G}\simeq C^{\infty}\left(\Omega_{\tilde{x}}\right)^{Z_{G}\left(\tilde{x}\right)},\ \ \ C^{\infty}\left(\Omega_{\text{rss}}\right)^{G}\simeq C^{\infty}\left(\Omega_{\tilde{x},\text{ rss}}\right)^{Z_{G}\left(\tilde{x}\right)}.
$$

Let $\omega\subseteq\mathfrak{g}\left(F\right)$ be a $G$-excellent
open subset (see \cite[Section 3.3]{BP20} for this definition) and set $\Omega=\exp\left(\omega\right)$. The Jacobian of the exponential map at $X\in\omega_{ss}$ is given by $j^{G}\left(X\right)=D^{G}\left(e^{X}\right)D^{G}\left(X\right)^{-1}$. Hence, the following integration formula holds for any $f$ which is integrable on $\Omega$
$$
\int_{\Omega}f\left(g\right)dg=\int_{\omega}f\left(e^{X}\right)j^{G}\left(X\right)dX.
$$
For any function $f$ on $\Omega$, we set $f_{\omega}$ the function
on $\omega$ defined by $f_{\omega}\left(X\right)=j^{G}\left(X\right)^{1/2}f\left(e^{X}\right)$.
The map $f\rightarrow f_{\omega}$ induces topological isomorphisms
$$
C^{\infty}\left(\Omega\right)\simeq C^{\infty}\left(\omega\right)\text{ and }C^{\infty}\left(\Omega_{\text{rss}}\right)\simeq C^{\infty}\left(\omega_{\text{rss}}\right).
$$
\subsection{Quasi-characters}\label{sec3.3}

Let $\omega\subseteq\mathfrak{g}\left(F\right)$ be a completely $G\left(F\right)$-invariant
open subset. A quasi-character on $\omega$ is a $G\left(F\right)$-invariant
smooth function $\theta:\omega_{\text{reg}}\rightarrow\mathbb{C}$
satisfying the following condition: for any $X\in\omega_{ss}$, there
exists a $G$-good open neighborhood $\omega_{X}\subseteq\mathfrak{g}_{X}\left(F\right)$
of $X$ satisfying $\omega_{X}^{G}\subseteq\omega$ and coefficients
$c_{\theta,{\mathcal{O}}}\left(X\right)$, where ${\mathcal{O}}$ is in the set $\text{Nil}\left(\mathfrak{g}_{X}\right)$ containing nilpotent orbits of $\mathfrak{g}_X$,
such that 
$$
\theta\left(Y\right) = \underset{\mathcal{O} \in \text{Nil} \left(g_{X}\right)}{\sum}c_{\theta,{\mathcal{O}}}\left(X\right) \hat{j}\left(\mathcal{O},Y\right),\ \forall\,Y\in\omega_{X,\text{reg}}.
$$
Here $\hat{j}\left(\mathcal{O},\cdot\right)$ is the Fourier transform of the orbital integral of $\mathcal{O}$. If $\theta$ is a quasi-character on $\omega$ and $f\in C^{\infty}\left(\omega\right)^{G}$,
then $f\theta$ is also a quasi-character on $\omega$. We denote
by $QC\left(\omega\right)$ the space of all quasi-characters on $\omega$
and by $QC_{c}\left(\omega\right)$ the subspace of quasi-characters
on $\omega$ whose support is compact modulo conjugation.

Let $\Omega\subseteq G\left(F\right)$ be a completely $G\left(F\right)$-invariant
open subset. Similar to the setting of Lie algebras, a quasi-character on $\Omega$ is a $G\left(F\right)$-invariant
smooth function $\theta:\Omega_{\text{reg}}\rightarrow\mathbb{C}$
satisfying the following condition: for any $x\in\Omega_{ss}$, there
exists $\omega_{x}\subseteq\mathfrak{g}_{x}\left(F\right)$ a $G_{x}$-excellent
open neighborhood of $0$ satisfying $\left(x\exp\left(\omega_{x}\right)\right)^{G}\subseteq\Omega$
and coefficients $c_{\theta,{\mathcal{O}}}\left(x\right)$, where ${\mathcal{O}}\in\text{Nil}\left(\mathfrak{g}_{x}\right)$,
such that 
$$
\theta\left(xe^{Y}\right)=\underset{{\mathcal{O}}\in\text{Nil}\left(g_{x}\right)}{\sum}c_{\theta,{\mathcal{O}}}\left(x\right)\hat{j}\left({\mathcal{O}},X\right),\ \forall\,X\in\omega_{x,\text{reg}}.
$$
Then
$$
D^{G}\left(xe^{X}\right)^{1/2}\theta\left(xe^{X}\right)=D^{G}\left(xe^{X}\right)^{1/2}\underset{{\mathcal{O}}\in\text{Nil}_{\text{reg}}\left(g_{x}\right)}{\sum}c_{\theta,{\mathcal{O}}}\left(x\right)\hat{j}\left({\mathcal{O}},X\right)+O\left(\left|X\right|\right),
$$
for all $X\in\mathfrak{g}_{x,\text{reg}}\left(F\right)$ sufficiently
near $0$. By the homogeneity property of the functions $\hat{j}\left({\mathcal{O}},\cdot\right)$
and their linear independence, we can see that coefficients $c_{\theta,{\mathcal{O}}}\left(x\right)$,
where ${\mathcal{O}}\in\text{Nil}_{\text{reg}}\left(\mathfrak{g}_{x}\right)$,
are uniquely defined. We set 
$$
c_{\theta}\left(x\right)=\frac{1}{\left|\text{Nil}_{\text{reg}}\left(g_{x}\right)\right|}\underset{{\mathcal{O}}\in\text{Nil}_{\text{reg}}\left(g_{x}\right)}{\sum}c_{\theta,{\mathcal{O}}}\left(x\right),
$$
for all $x\in G_{ss}\left(F\right)$. This gives us a complex-valued
function $c_{\theta}$ on $G_{ss}\left(F\right)$. Similarly, for
any quasi-character $\theta$ on $\mathfrak{g}\left(F\right)$, we
can associate to it a function $c_{\theta}$ on $\mathfrak{g}_{ss}\left(F\right)$.
We recall \cite[Proposition 4.5.1]{BP20} for later use.
\begin{proposition}\label{3.2}
Let $\theta$ be a quasi-character on $G\left(F\right)$ and let $x\in G_{ss}\left(F\right)$.
We have the following properties. 
\begin{enumerate}
\item Assume $G_{x}$ is quasi-split. Let $B_{x}\subset G_{x}$ be a Borel
subgroup and $T_{\text{qd,}x}\subset B_{x}$ be a maximal torus (both
defined over $F$). Then
$$
D^{G}\left(x\right)^{1/2}c_{\theta}\left(x\right)=\left|W\left(G_{x},T_{\text{qd,}x}\right)\right|^{-1}\underset{x^{\prime}\in T_{\text{qd,}x}\left(F\right)\rightarrow x}{\lim}D^{G}\left(x^{\prime}\right)^{1/2}\theta\left(x^{\prime}\right).
$$
\item The function $\left(D^{G}\right)^{1/2}c_{\theta}$ is locally bounded
on $G\left(F\right)$. To be more precise, for any invariant and compact
modulo conjugation subset $L$ of $G\left(F\right)$, there exists
a continuous semi-norm $\nu_{L}$ on $QC\left(G\left(F\right)\right)$
such that 
$$
\underset{x\in L_{ss}}{\sup}D^{G}\left(x\right)^{1/2}\left|c_{\theta}\left(x\right)\right|\leq\nu_{L}\left(\theta\right),
$$
for all $\theta\in QC\left(G\left(F\right)\right)$.
\item Let $\Omega_{x}\subseteq G_{x}\left(F\right)$ be a $G$-good open
neighborhood of $x$. Then
$$
D^{G}\left(y\right)^{1/2}c_{\theta}\left(y\right)=D^{G_{x}}\left(y\right)^{1/2}c_{\theta_{x,\Omega_{x}}}\left(y\right),
$$
for any $y\in\Omega_{x,ss}$.
\end{enumerate}
\end{proposition}

Next, we define elliptic germs of quasi-characters.
\begin{definition}\label{3.3}
Let $G$ be a reductive group defined over $F$. Two quasi-characters
$\theta$ and $\theta^{\prime}$ of $G\left(F\right)$ are said to
be equivalent if there exists an open neighborhood $\Omega$ of $G_{\text{ell}}\left(F\right)$
such that $\theta=\theta^{\prime}$ on $\Omega_{\text{reg}}$. We
denote the set of equivalent classes of quasi-characters of $G\left(F\right)$
by ${\mathcal{E}}\left(G\right)$. We call ${\mathcal{E}}\left(G\right)$ the
set of elliptic germs of quasi-characters of $G\left(F\right)$.
\end{definition}
\begin{remark}\label{3.4}
    We can easily extend the above setting to twisted groups, including the definition of quasi-characters, Proposition \ref{3.2} and Definition \ref{3.3}.
\end{remark}
\subsection{Strongly cuspidal functions}\label{sec3.4}
Let $P=MU$ be a parabolic subgroup of $G$. For $f\in{\mathcal{C}}\left(G\left(F\right)\right)$,
we define 
$$f^{P}\left(m\right)=\delta_{P}\left(m\right)^{1/2}\int_{U\left(F\right)}f\left(mu\right)du,$$
where $m\in M\left(F\right)$, as a function in ${\mathcal{C}}\left(M\left(F\right)\right)$. A function $f\in{\mathcal{C}}\left(G\left(F\right)\right)$ is called strongly cuspidal if $f^{P}=0$ for any proper parabolic subgroup of $P$ of $G$. We denote by ${\mathcal{C}}_{\text{scusp}}\left(G\left(F\right)\right)$
the subspace of strongly cuspidal functions in ${\mathcal{C}}\left(G\left(F\right)\right)$.
More generally, for a completely $G\left(F\right)$-invariant open
subset $\Omega\subseteq G\left(F\right)$, we set ${\mathcal{C}}_{\text{scusp}}\left(\Omega\right)={\mathcal{C}}\left(\Omega\right)\cap{\mathcal{C}}_{\text{scusp}}\left(G\left(F\right)\right)$.

Recall that in section \ref{sec2.3}, we have defined a topological space ${\mathcal{C}}\left({\mathcal{X}}_{\text{temp}}\left(G\right),{\mathcal{E}}\left(G\right)\right)$
of smooth sections $\pi\in{\mathcal{X}}_{\text{temp}}\left(G\right)\mapsto T_{\pi}\in\text{End}\left(\pi\right)$
and the map that associates to $f\in{\mathcal{C}}\left(G\left(F\right)\right)$
its Fourier transform $\pi\in{\mathcal{X}}_{\text{temp}}\left(G\right)\mapsto\pi\left(f\right)$,
which induces a topological isomorphism ${\mathcal{C}}\left(G\left(F\right)\right)\simeq{\mathcal{C}}\left({\mathcal{X}}_{\text{temp}}\left(G\right),{\mathcal{E}}\left(G\right)\right)$.
We denote by ${\mathcal{C}}_{\text{scusp}}\left({\mathcal{X}}_{\text{temp}}\left(G\right),{\mathcal{E}}\left(G\right)\right)$
the image of ${\mathcal{C}}_{\text{scusp}}\left(G\left(F\right)\right)$
under this isomorphism .

Likewise, we can extend the definition of strongly cuspidal functions to twisted groups and we denote the space of strongly cuspidal functions on $\tilde{G}\left(F\right)$
by ${\mathcal{C}}_{\text{scusp}}\left(\tilde{G}\left(F\right)\right)$.

\subsection{Weighted orbital integrals of strongly cuspidal functions}\label{sec3.5}

Let $M$ be a Levi subgroup of $G$. As in Section \ref{sec2.1}, by fixing
a maximal compact subgroup $K$ of $G\left(F\right)$, we have the
following map $H_{P}:G\left(F\right)\rightarrow{\mathcal{A}}_{M},$ where $P\in{\mathcal{P}}\left(M\right)$. For every $g\in G\left(F\right)$,
the family $\left(H_{P}\left(g\right)\right)_{P\in{\mathcal{P}}\left(M\right)}$
is a positive $\left(G,M\right)$-orthogonal set. By the previous
section, this gives us a $\left(G,M\right)$-family $\left(v_{p}\left(g,\cdot\right)\right)_{P\in{\mathcal{P}}\left(M\right)}$
and the number $v_{M}\left(g\right)$ associated to this $\left(G,M\right)$-family,
i.e. the volume in ${\mathcal{A}}_{M}^{G}$ of the convex hull determined
by $\left(H_{P}\left(g\right)\right)_{P\in{\mathcal{P}}\left(M\right)}$.
The function $g\mapsto v_{M}\left(g\right)$ is left $M\left(F\right)$
and right $K$-invariant.

Let $x\in M\left(F\right)\cap G_{\text{rss}}\left(F\right)$. For
any $f\in{\mathcal{C}}\left(G\left(F\right)\right)$, we define the weighted
orbital integral of $f$ at $x$ by setting 
$$
J_{M}\left(x,f\right)=D^{G}\left(x\right)^{1/2}\int_{G_{x}\left(F\right)\backslash G\left(F\right)}f\left(g^{-1}xg\right)v_{M}\left(g\right)dg.
$$
This integral is absolutely convergent and defines a tempered distribution
$J_{M}\left(x,\cdot\right)$ on $G\left(F\right)$. More generally,
for the $\left(G,M\right)$-family $\left(v_{P}\left(g,\cdot\right)\right)_{P\in{\mathcal{P}}\left(M\right)}$,
it is possible to associate to it a complex number $v_{L}^{Q}\left(g\right)$,
for any $L\in{\mathcal{L}}\left(M\right)$ and $Q\in{\mathcal{F}}\left(L\right)$.
This allows us to define a tempered distribution $J_{L}^{Q}\left(x,\cdot\right)$
on $G\left(F\right)$, for any $L\in{\mathcal{L}}\left(M\right)$ and $Q\in{\mathcal{F}}\left(L\right)$,
by setting 
$$
J_{L}^{Q}\left(x,f\right)=D^{G}\left(x\right)^{1/2}\int_{G_{x}\left(F\right)\backslash G\left(F\right)}f\left(g^{-1}xg\right)v_{L}^{Q}\left(g\right)dg,\ \ f\in{\mathcal{C}}\left(G\left(F\right)\right).
$$
When $L=M$ and $Q=G$, this recovers the definition of $J_M$. For a regular semisimple element $x\in G_{\text{rss}}\left(F\right)$, let $M\left(x\right)$
be the centralizer of $A_{G_{x}}$ in $G$. For any $f\in{\mathcal{C}}_{\text{scusp}}\left(G\left(F\right)\right)$, we set 
$$
\theta_{f}\left(x\right)=\left(-1\right)^{a_{G}-a_{M\left(x\right)}}\nu\left(G_{x}\right)^{-1}D^{G}\left(x\right)^{-1/2}J_{M\left(x\right)}^{G}\left(x,f\right),
$$
where $x\in G_{\text{rss}}\left(F\right)$. By \cite[Lemma 5.2.1]{BP20}, the function
$\theta_{f}$ is invariant under conjugation. Moreover, it is a quasi-character
of $G\left(F\right)$.

We consider the case of twisted groups. Similar to the untwisted setting, for $\tilde{f}\in \mathcal{C}_{scusp}\left(\tilde{G}\left(F\right)\right)$, we can define the orbital integral $J_{\tilde{M}}\left(\,\cdot\,,\tilde{f}\right)$, where $\tilde{M}\in{\mathcal{L}}^{\tilde{G}}$, and the quasi-character $\theta_{\tilde{f}}$.

Let $\left(G_{n}\right)_{n}$ be an increasing sequence of compact
subsets modulo $A_{\tilde{G}}\left(F\right)$ of $G\left(F\right)$. We say $\left(G_{n}\right)_{n}$ is of moderate growth if there exists $c_{1},c_{2}>0$ such that for any $n\geq1$, we have  
$$
\left\{ g\in G:\bar{\sigma}\left(g\right)\leq c_{1}n\right\} \subseteq G_{n}\subseteq\left\{ g\in G:\bar{\sigma}\left(g\right)\leq c_{2}n\right\},
$$
where $\bar{\sigma}=\sigma_{A_{\tilde{G}}\backslash G}$. In the remainder of this section, we prove the following technical proposition.
\begin{proposition}\label{3.8}
Let $\left(G_{n}\right)_{n}$ be an increasing sequence of compact
subsets modulo $A_{\tilde{G}}\left(F\right)$ of $G\left(F\right)$ which is of moderate growth. We assume further
$\underset{n}{\bigcup}\ G_{n}=G\left(F\right)$ and $KG_{n}K=G_{n}$. For any $\tilde{f}\in{\mathcal{C}}_{\text{scusp}}\left(\tilde{G}\right)$
and $\tilde{x}\in\tilde{G}_{\text{reg}}\left(F\right)$, one has 
$$
\theta_{\tilde{f}}\left(\tilde{x}\right)=\underset{n}{\lim}\int_{A_{\tilde{G}}\left(F\right)\backslash G_{n}}\tilde{f}\left(g^{-1}\tilde{x}g\right)dg.
$$
\end{proposition}

\begin{proof}
We choose $c_{1},c_{2}>0$ such that for any $n\geq1$,
we have 
$$
\left\{ g\in G:\bar{\sigma}\left(g\right)\leq c_{1}n\right\} \subseteq G_{n}\subseteq\left\{ g\in G:\bar{\sigma}\left(g\right)\leq c_{2}n\right\}.
$$
Let $I_{n}\left(\tilde{f}\right)=\int_{A_{\tilde{G}}\left(F\right)\backslash G_{n}}\tilde{f}\left(g^{-1}\tilde{x}g\right)dg$. Since $\tilde{x}$ is regular, $T=G_{\tilde{x}}$ is a torus. Observe
$$
I_{n}\left(\tilde{f}\right)=\int_{T\left(F\right)\backslash G\left(F\right)}\tilde{f}\left(g^{-1}\tilde{x}g\right)\kappa_{n}\left(g\right)dg,
$$
where $\kappa_{n}\left(g\right)=\int_{A_{\tilde{G}}\left(F\right)\backslash T\left(F\right)}\mathbf{1}_{G_{n}}\left(tg\right)dt$. Without loss of generality, assume $\tilde{M}=\tilde{M}\left(\tilde{x}\right)\in{\mathcal{L}}^{\tilde{G}}$.
Note that we have fixed a minimal twisted parabolic subgroup $\tilde{P}_{\min}$
with a minimal Levi component $\tilde{M}_{\min}$, together with a
maximal compact subgroup $K$ of $G\left(F\right)$ with relative
position to $M_{\min}$. Let $Y\in{\mathcal{A}}_{M_{\min}}$ such that
$\alpha\left(Y\right)>0$ for any $\alpha\in\Delta$ (we denote by
$\Delta$ the set of simple roots of $G$ associated to $P_{\min}$
and $M_{\min}$). For any $P\in{\mathcal{P}}\left(M_{\min}\right)$, there
exists a unique $w\in W\left(G,M_{\min}\right)$ such that $wP_{\min}w^{-1}=P$.
Let $Y_{P}=w\cdot Y$. For $\tilde{Q}\in{\mathcal{P}}\left(\tilde{M}\right)$,
we denote by $Y_{\tilde{Q}}$ the projection of $Y_{P}$ onto ${\mathcal{A}}_{\tilde{M}}$,
where $P\in{\mathcal{P}}\left(M_{\min}\right)$ satisfying $P\subset Q$.
For $g\in G\left(F\right)$ and $\tilde{Q}\in{\mathcal{P}}\left(\tilde{M}\right)$,
let $Y_{\tilde{Q}}\left(g\right)=Y_{\tilde{Q}}-H_{\overline{\tilde{Q}}}\left(g\right)$,
where $\overline{\tilde{Q}}$ is the opposite parabolic group of $\tilde{Q}$
with respect to $\tilde{M}$. Then ${\mathcal{Y}}\left(g\right)=\left(Y_{\tilde{Q}}\left(g\right)\right)_{\tilde{Q}\in{\mathcal{P}}\left(\tilde{M}\right)}$
is a $\left(\tilde{G},\tilde{M}\right)$-orthogonal family. Moreover,
there exists $c>0$ such that for all $g\in G\left(F\right)$ and
$Y\in{\mathcal{A}}_{M_{\min}}^{+}$ satisfying $\underset{m\in M\left(F\right)}{\inf}\sigma\left(mg\right)\leq c\underset{\alpha\in\Delta}{\inf}\alpha\left(Y\right)$, the $\left(\tilde{G},\tilde{M}\right)$-orthogonal set ${\mathcal{Y}}\left(g\right)$
is positive. 

Assume this is the case. Let $\sigma_{\tilde{M}}^{\tilde{G}}\left(\cdot,{\mathcal{Y}}\left(g\right)\right)$
be the characteristic function in ${\mathcal{A}}_{\tilde{M}}$ of the convex
hull of the family $\left(Y_{\tilde{Q}}\left(g\right)\right)_{\tilde{Q}\in{\mathcal{P}}\left(\tilde{M}\right)}$
and $v\left(g,Y\right)=\int_{A_{\tilde{G}}\left(F\right)\backslash T\left(F\right)}\sigma_{\tilde{M}}^{\tilde{G}}\left(H_{\tilde{M}}\left(t\right),{\mathcal{Y}}\left(g\right)\right)dt$. For $g\in G\left(F\right)$, define $\sigma_{T\backslash G}\left(g\right)=\underset{t\in T\left(F\right)}{\inf}\sigma\left(tg\right)$. Since $\tilde{f}$ is compactly supported and the projection $G$ onto $T\backslash G$ has the norm descent property, for $n$ sufficiently large, one has
$$
\sigma_{T\backslash G}\left(g\right)<\log n\ \ \ \text{ if }\ \ \ \tilde{f}\left(g^{-1}\tilde{x}g\right)\neq0.
$$
We prove the following proposition.
\begin{proposition}\label{3.9}
There exists $c_{3}>1/c$ sufficiently large and $n_{0}\geq1$ such
that if $n\geq n_{0}$ and $\underset{\alpha\in\Delta}{\inf}\alpha\left(Y\right)>c_{3}\log n$, one has the following equality 
$$
\int_{T\left(F\right)\backslash G\left(F\right)}\tilde{f}\left(g^{-1}\tilde{x}g\right)\kappa_{n}\left(g\right)dg=\int_{T\left(F\right)\backslash G\left(F\right)}\tilde{f}\left(g^{-1}\tilde{x}g\right)v\left(g,Y\right)dg.
$$
\end{proposition}

\begin{proof}
We take $c_{3}>1/c$ and $Y\in{\mathcal{A}}_{M_{\min}}$ such that $\underset{\alpha\in\Delta}{\inf}\alpha\left(Y\right)>c_{3}\log n$. Then for $n$ sufficiently large, ${\mathcal{Y}}\left(g\right)$ is positive $\left(\tilde{G},\tilde{M}\right)$-orthogonal, for any $g$ such
that $\tilde{f}\left(g^{-1}\tilde{x}g\right)\neq0$. Let 
$$
I_{Y}\left(\tilde{f}\right)=\int_{T\left(F\right)\backslash G\left(F\right)}\tilde{f}\left(g^{-1}\tilde{x}g\right)v\left(g,Y\right)dg.
$$
Let $Z\in{\mathcal{A}}_{M_{\min}}$ satisfying $\underset{\alpha\in\Delta}{\inf}\alpha\left(Z\right)\geq\frac{1}{c}\log n$. Similar to $\mathcal{Y}$, one can define a positive $\left(\tilde{G},\tilde{M}\right)$-orthogonal family ${\mathcal{Z}}\left(g\right)=\left(Z_{\tilde{Q}}\left(g\right)\right)_{\tilde{Q}\in{\mathcal{P}}\left(\tilde{M}\right)}$.
Let $\tilde{Q}=\tilde{L}U_{Q}\in{\mathcal{F}}\left(\tilde{M}\right)$,
we denote by $\sigma_{\tilde{M}}^{\tilde{Q}}\left(\cdot,{\mathcal{Z}}\left(g\right)\right)$
(resp. $\tau_{\tilde{Q}}$) the characteristic functions in ${\mathcal{A}}_{\tilde{M}}$
of the sum of ${\mathcal{A}}_{\tilde{L}}$ and the convex hull of $\left(Z_{\tilde{P}}\left(g\right)\right)_{\tilde{P}\subset\tilde{Q}}$
(resp. ${\mathcal{A}}_{\tilde{M}}^{\tilde{L}}\oplus{\mathcal{A}}_{\tilde{Q}}^{+}$).
We have $\underset{\tilde{Q}\in{\mathcal{F}}\left(\tilde{M}\right)}{\sum}\sigma_{\tilde{M}}^{\tilde{Q}}\left(\lambda,{\mathcal{Z}}\left(g\right)\right)\tau_{\tilde{Q}}\left(\lambda-Z\left(g\right)_{\tilde{Q}}\right)=1$, for any $\lambda\in{\mathcal{A}}_{\tilde{M}}$. Let 
$$
v\left(g,Y,\tilde{Q}\right)=\int_{A_{\tilde{G}}\left(F\right)\backslash T\left(F\right)}\sigma_{\tilde{M}}^{\tilde{G}}\left(H_{\tilde{M}}\left(t\right),{\mathcal{Y}}\left(g\right)\right)\sigma_{\tilde{M}}^{\tilde{Q}}\left(H_{\tilde{M}}\left(t\right),{\mathcal{Z}}\left(g\right)\right)\tau_{\tilde{Q}}\left(H_{\tilde{M}}\left(t\right)-Z\left(g\right)_{\tilde{Q}}\right)dt
$$
and 
$$
I_{Y}^{\tilde{Q}}\left(\tilde{f}\right)=\int_{T\left(F\right)\backslash G\left(F\right)}\tilde{f}\left(g^{-1}\tilde{x}g\right)v\left(g,Y,\tilde{Q}\right)dg
$$
and
$$
\kappa_{n}\left(g,\tilde{Q}\right)=\int_{A_{\tilde{G}}\left(F\right)\backslash T\left(F\right)}\mathbf{1}_{G_{n}}\left(tg\right)\sigma_{\tilde{M}}^{\tilde{Q}}\left(H_{\tilde{M}}\left(t\right),{\mathcal{Z}}\left(g\right)\right)\tau_{\tilde{Q}}\left(H_{\tilde{M}}\left(t\right)-Z\left(g\right)_{\tilde{Q}}\right)dt
$$
and 
$$
I_{n}^{\tilde{Q}}\left(\tilde{f}\right)=\int_{T\left(F\right)\backslash G\left(F\right)}\tilde{f}\left(g^{-1}\tilde{x}g\right)\kappa_{n}\left(g,\tilde{Q}\right)dg.
$$
One has 
$$
I_{Y}\left(\tilde{f}\right)=\underset{\tilde{Q}\in{\mathcal{F}}\left(\tilde{M}\right)}{\sum}I_{Y}^{\tilde{Q}}\left(\tilde{f}\right)\text{ and }I_{n}\left(\tilde{f}\right)=\underset{\tilde{Q}\in{\mathcal{F}}\left(\tilde{M}\right)}{\sum}I_{n}^{\tilde{Q}}\left(\tilde{f}\right).
$$
We prove the following claim.
\begin{claim}\label{3.10}
If $Z$ satisfies the following inequality 
$$
\underset{\alpha\in\Delta}{\sup}\,\alpha\left(Z\right)\leq\min\left\{ \underset{\alpha\in\Delta}{\inf}\alpha\left(Y\right),\left(\log n\right)^{2}\right\} ,
$$
one has $I_{Y}^{\tilde{G}}\left(\tilde{f}\right)=I_{n}^{\tilde{G}}\left(\tilde{f}\right)$
for $n$ sufficiently large.
\end{claim}

\begin{proof}
The above inequality implies the convex hull of $\left(Z_{\tilde{P}}\left(g\right)\right)_{\tilde{P}\subset\tilde{Q}}$
is contained in one of of $\left(Y_{\tilde{P}}\left(g\right)\right)_{\tilde{P}\subset\tilde{Q}}$,
for any $g\in G\left(F\right)$ satisfying $\sigma_{T\backslash G}\left(g\right)<\log n$.
This gives us 
$
v\left(g,Y,\tilde{G}\right)=v\left(g,Z\right)$ whenever $\sigma_{T\backslash G}\left(g\right)<\log n$. Moreover, the hypothesis in Claim \ref{3.10} implies $\sigma\left(tg\right)\ll\left(\log n\right)^{2}$
whenever $\sigma_{\tilde{M}}^{\tilde{Q}}\left(H_{\tilde{M}}\left(t\right),{\mathcal{Z}}\left(g\right)\right)=1$. Since $\sigma\left(tg\right)<c_{1}n$ implies $tg\in G_{n}$,
one has $\kappa_{n}\left(g,\tilde{G}\right)=v\left(g,Z\right)$ for $n$ sufficiently large. This gives us the proof of Claim \ref{3.10}.
\end{proof}
We fix a proper twisted parabolic subgroup $\tilde{Q}=\tilde{L}U_{Q}\in{\mathcal{F}}\left(\tilde{M}\right)$.
We prove the following claim.
\begin{claim}\label{3.11}
There exists $c_{4}>0$ such that if $\underset{\alpha\in\Delta}{\inf}\alpha\left(Z\right)\geq c_{4}\log n$
and $c_{3}$ and $n$ sufficiently large, one has 
$$
I_{n}^{\tilde{Q}}\left(\tilde{f}\right)=I_{Y}^{\tilde{Q}}\left(\tilde{f}\right)=0.
$$
\end{claim}

\begin{proof}
Let $U_{\bar{Q}}$ be the unipotent radical of $\bar{Q}$. We write
$$
I_{n}^{\tilde{Q}}\left(\tilde{f}\right)=\int_{K}\int_{T\left(F\right)\backslash L\left(F\right)}\int_{U_{\bar{Q}}\left(F\right)}\tilde{f}\left(k^{-1}l^{-1}\bar{u}^{-1}\tilde{x}\bar{u}lk\right)\kappa_{n}\left(\bar{u}lk,\tilde{Q}\right)d\bar{u}\delta_{Q}\left(l\right)dldk
$$
and 
$$
I_{Y}^{\tilde{Q}}\left(\tilde{f}\right)=\int_{K}\int_{T\left(F\right)\backslash L\left(F\right)}\int_{U_{\bar{Q}}\left(F\right)}\tilde{f}\left(k^{-1}l^{-1}\bar{u}^{-1}\tilde{x}\bar{u}lk\right)v\left(\bar{u}lk,Y,\tilde{Q}\right)d\bar{u}\delta_{Q}\left(l\right)dldk.
$$
There exists $c^{\prime}>0$ such that for any $\bar{u}\in U_{\bar{Q}}\left(F\right),l\in L\left(F\right)$
and $k\in K$ satisfying $\sigma_{T\backslash G}\left(\bar{u}lk\right)<\log n$,
one has 
$$
\sigma\left(\bar{u}\right),\sigma_{T\backslash G}\left(l\right),\sigma_{T\backslash G}\left(\bar{u}lk\right)\leq\log n.
$$
Suppose for such elements, we have $\kappa_{n}\left(\bar{u}lk,\tilde{Q}\right)=\kappa_{n}\left(lk,\tilde{Q}\right)$
and $v\left(\bar{u}lk,Y,\tilde{Q}\right)=v\left(lk,Y,\tilde{Q}\right)$.
Since $\tilde{f}$ is strongly cuspidal, one has 
$$
\int_{U_{\bar{Q}}\left(F\right)}\tilde{f}\left(k^{-1}l^{-1}\bar{u}^{-1}\tilde{x}\bar{u}lk\right)d\bar{u}=0.
$$
This gives us the proof of Claim \ref{3.11}. Hence, it suffices to prove
$\kappa_{n}\left(\bar{u}lk,\tilde{Q}\right)=\kappa_{n}\left(lk,\tilde{Q}\right)$
and $v\left(\bar{u}lk,Y,\tilde{Q}\right)=v\left(lk,Y,\tilde{Q}\right)$.
To be more precise, we prove there exists $c_{4}>0$ such that if $\underset{\alpha\in\Delta}{\inf}\alpha\left(Z\right)\geq c_{4}\log n$ and $c_{3}$ and $n$ are sufficiently large, one has 
$$
\kappa_{n}\left(\bar{u}g,\tilde{Q}\right)=\kappa_{n}\left(g,\tilde{Q}\right)\text{ and }v\left(\bar{u}g,Y,\tilde{Q}\right)=v\left(g,Y,\tilde{Q}\right),
$$
for any $\bar{u}\in U_{\bar{Q}}\left(F\right)$ and $g\in G\left(F\right)$
satisfying $\sigma\left(\bar{u}\right),\sigma\left(g\right),\sigma\left(\bar{u}g\right)\leq c^{\prime}\log n$.

We first show the statement for $v\left(\cdot,Y,\tilde{Q}\right)$.
By definition, one has 
$$
v\left(g,Y,\tilde{Q}\right)=\int_{A_{\tilde{G}}\left(F\right)\backslash T\left(F\right)}\sigma_{\tilde{M}}^{\tilde{G}}\left(H_{\tilde{M}}\left(t\right),{\mathcal{Y}}\left(g\right)\right)\sigma_{\tilde{M}}^{\tilde{Q}}\left(H_{\tilde{M}}\left(t\right),{\mathcal{Z}}\left(g\right)\right)\tau_{\tilde{Q}}\left(H_{\tilde{M}}\left(t\right)-Z\left(g\right)_{\tilde{Q}}\right)dt.
$$
The functions $\sigma_{\tilde{M}}^{\tilde{Q}}\left(\cdot,{\mathcal{Z}}\left(g\right)\right)$
and $\tau_{\tilde{Q}}\left(\cdot-Z\left(g\right)_{\tilde{Q}}\right)$
only depends on $\left(Z_{\tilde{P}}\left(g\right)\right)_{\tilde{P}\subset\tilde{Q}}$.
Consequently, they are invariant under left translation of $g$ by
$U_{\bar{Q}}\left(F\right)$. Moreover, for any $\lambda\in{\mathcal{A}}_{\tilde{M}}$
satisfying $\sigma_{\tilde{M}}^{\tilde{Q}}\left(\lambda,{\mathcal{Z}}\left(g\right)\right)\tau_{\tilde{Q}}\left(\lambda-Z_{\tilde{Q}}\left(g\right)\right)=1$,
there exists $\tilde{P}\in{\mathcal{P}}\left(\tilde{M}\right)$ such that
$\tilde{P}\subset\tilde{Q}$ and $\lambda\in{\mathcal{A}}_{\tilde{P}}^{+}$.
Therefore, the restriction of $\sigma_{\tilde{M}}^{\tilde{G}}\left(\cdot,{\mathcal{Y}}\left(g\right)\right)$
to ${\mathcal{A}}_{\tilde{P}}^{+}$ only depends on $Y_{\tilde{P}}\left(g\right)$,
which is to say it is invariant under left translation of $g$ by
$U_{\bar{Q}}\left(F\right)$. This gives us $v\left(\bar{u}g,Y,\tilde{Q}\right)=v\left(g,Y,\tilde{Q}\right)$.

We now prove $\kappa_{n}\left(\bar{u}g,\tilde{Q}\right)=\kappa_{n}\left(g,\tilde{Q}\right)$.
The above argument shows that the function 
$$
g\mapsto\sigma_{\tilde{M}}^{\tilde{Q}}\left(\cdot,{\mathcal{Z}}\left(g\right)\right)\tau_{\tilde{Q}}\left(\cdot-Z\left(g\right)_{\tilde{Q}}\right)
$$
defined on $\left\{ g\in G\left(F\right):\,\sigma\left(g\right)\leq c^{\prime}\log n\right\} $
is $U_{\bar{Q}}\left(F\right)$-left invariant. It suffices to prove
that for any $t\in T\left(F\right)$ such that $\sigma_{\tilde{M}}^{\tilde{Q}}\left(H_{\tilde{M}}\left(t\right),{\mathcal{Z}}\left(g\right)\right)\tau_{\tilde{Q}}\left(H_{\tilde{M}}\left(t\right)-Z\left(g\right)_{\tilde{Q}}\right)=1$, one has 
$$
\mathbf{1}_{G_{n}}\left(tg\right)=\mathbf{1}_{G_{n}}\left(t\bar{u}g\right),
$$
for $n$ sufficiently large and $\bar{u}\in U_{\bar{Q}}\left(F\right)$
and $g\in G\left(F\right)$ satisfying $\sigma\left(\bar{u}\right),\sigma\left(g\right),\sigma\left(\bar{u}g\right)\leq c^{\prime}\log n$.
Let $R\left(A_{T},U_{Q}\right)$ be the set of roots of $A_{T}$ in
$U_{Q}$. By the assumption on $g$ and $Z$, there exists $c_{5}>0$
depending on $c_{4}$ such that $\left\langle \beta,H_{\tilde{M}}\left(t\right)\right\rangle \geq c_{5}\log n$, for all $\beta\in R\left(A_{T},U_{Q}\right)$. Let $e$ be a positive
number which is sufficiently large. By the above inequality and the assumption on $\bar{u}$, if $c_{4}$ is sufficiently large, one has
$$
t\bar{u}t^{-1}\in\exp\left(B\left(0,n^{-e}\right)\right).
$$
Let $\tilde{P_{t}}\in{\mathcal{P}}\left(\tilde{M}\right)$ be a twisted
parabolic subgroup such that $H_{\tilde{M}}\left(t\right)\in{\mathcal{A}}_{\tilde{P}_{t}}$.
Let $\mathfrak{p}_{t}$ be the Lie algebra of $P_{t}$. Then there
exists $X_{P}\in B\left(0,n^{-e}\right)\cap\mathfrak{p}_{t}\left(F\right)$
such that $t\bar{u}t^{-1}=e^{X_{P}}$. One has 
$$
\mathbf{1}_{G_{n}}\left(t\bar{u}g\right)=\mathbf{1}_{G_{n}}\left(e^{X_{P}}tg\right).
$$
By choosing $e$ sufficiently large, it follows that $\mathbf{1}_{G_{n}}\left(e^{X_{P}}tg\right)=\mathbf{1}_{G_{n}}\left(tg\right)$, which is to say $\mathbf{1}_{G_{n}}\left(tg\right)=\mathbf{1}_{G_{n}}\left(t\bar{u}g\right)$. We finish our proof for Claim \ref{3.11}.
\end{proof}
By Claim \ref{3.10} and Claim \ref{3.11}, if $c_{3}$ and $n$ are large enough
and $Z$ satisfies conditions in these two claims, then $I_{n}^{\tilde{Q}}\left(\tilde{f}\right)=I_{Y}^{\tilde{Q}}\left(\tilde{f}\right)$, for any $\tilde{Q}\in{\mathcal{F}}\left(\tilde{M}\right)$. This gives
us $I_{n}\left(\tilde{f}\right)=I_{Y}\left(\tilde{f}\right)$ as desired.
\end{proof}
We return to the proof of Proposition \ref{3.8}. By Proposition \ref{3.9}, one has 
$$
I_{n}\left(\tilde{f}\right)=\int_{T\left(F\right)\backslash G\left(F\right)}\tilde{f}\left(g^{-1}\tilde{x}g\right)v\left(g,Y\right)dg,
$$
for $n$ sufficiently large and $Y\in{\mathcal{A}}_{M_{\min}}$ satisfying $\underset{\alpha\in\Delta}{\inf}\alpha\left(Y\right)\gg\log n$. Let ${\mathcal{R}}$ be a lattice in ${\mathcal{A}}_{M_{\min},F}\otimes\mathbb{Q}$. Let ${\mathcal{R}}^{\vee}$ be the co-lattice containing elements $\lambda\in{\mathcal{A}}_{M_{\min}}^{*}$ satisfying $\lambda\left(\zeta\right)\in2\pi\mathbb{Z}$ for any $\zeta\in{\mathcal{R}}$. By \cite[Lemma 4.7]{Wal12b}, there exists a finite subset $\Xi\subset i{\mathcal{A}}_{M_{\min}}^{*}/i{\mathcal{R}}^{\vee}$
and an integer $N\geq1$ as well as a family of functions $g\mapsto c_{{\mathcal{R}}}\left(\Lambda,k,g\right)$
for $\Lambda\in\Xi$ and $0\leq k\leq N$ such that, for any $g\in G\left(F\right)$
and $Y\in{\mathcal{R}}$ satisfying $\alpha\left(Y\right)\gg\sigma\left(g\right)$
for all $\alpha\in\Delta$, one has 
$$
v\left(g,Y\right)=\underset{\Lambda\in\Xi,\,0\leq k\leq N}{\sum}c_{{\mathcal{R}}}\left(\Lambda,k,g\right)e^{\Lambda\left(Y\right)}Y^{k}.
$$
The functions $Y\mapsto e^{\Lambda\left(Y\right)}Y^{k}$ are linearly
independent on the cone
$\left\{ Y\in\mathcal{R}:\ \sigma\left(g\right)\leq c_{3}\alpha\left(Y\right),\,\forall\alpha\in\Delta\right\}$. Since $I_{n}\left(\tilde{f}\right)$ does not depend on $Y$, we can
replace $v\left(g,Y\right)$ by the constant term $c_{{\mathcal{R}}}\left(0,0,g\right)$.
We can also apply the above result for the lattice $\frac{1}{k}{\mathcal{R}}$,
for any $k\in\mathbb{N}^{*}$. By \cite[Lemma 4.7]{Wal12b}, one has
$$
\left|c_{\frac{1}{k}{\mathcal{R}}}\left(0,0,g\right)-\left(-1\right)^{a_{\tilde{M}}}v_{\tilde{M}}\left(g\right)\right|\ll k^{-1}\sigma_{T\backslash G}\left(g\right)^{a_{\tilde{M}}},
$$
for any $g\in G\left(F\right)$ and $k\in\mathbb{N}^{*}$. Since $\int_{T\left(F\right)\backslash G\left(F\right)}\tilde{f}\left(g^{-1}\tilde{x}g\right)\sigma_{T\backslash G}\left(g\right)^{a_{\tilde{M}}}dg$ is absolutely convergent, by taking the limit, it follows that  
$$
\underset{n}{\lim}\,I_{n}\left(\tilde{f}\right)=\left(-1\right)^{a_{\tilde{M}}}\int_{T\left(F\right)\backslash G\left(F\right)}\tilde{f}\left(g^{-1}\tilde{x}g\right)v_{\tilde{M}}\left(g\right)dg=\theta_{\tilde{f}}\left(\tilde{x}\right)
$$
as desired.
\end{proof}

\subsection{Weighted characters of strongly cuspidal functions}\label{sec3.6}

Let $M$ be a Levi subgroup of $G$ and $K$ be the special maximal
compact subgroup of $G\left(F\right)$. Let $\sigma$ be a tempered
representation of $M\left(F\right)$. We fix $P\in{\mathcal{P}}\left(M\right)$.
Following Arthur (cf. \cite{Art94}), for any $f\in{\mathcal{C}}\left(G\left(F\right)\right)$, $L\in{\mathcal{L}}\left(M\right)$
and $Q\in{\mathcal{F}}\left(L\right)$, we can define a weighted character $J_{L}^{Q}\left(\sigma,f\right)$. In particular, when $L=Q=G$, this reduces to the usual character, i.e. 
$
J_{G}^{G}\left(\sigma,f\right)=\text{Trace}_{i_{M}^{G}\left(\sigma\right)}\left(f\right)
$
for any $f\in{\mathcal{C}}\left(G\left(F\right)\right)$.

Let $\tilde{P}=\tilde{M}U$ be a (semi-standard) twisted parabolic subgroup of $\tilde{G}$
and $\tilde{\tau}$ be a tempered representation of $\tilde{M}\left(F\right)$.
The definition of weighted characters can be extended to twisted groups. We denote a weighted character of $\tilde{\tau}$ by $J_{\tilde{M}}^{\tilde{G}}\left(\tilde{\tau},\tilde{f}\right)$, for any $\tilde{f}\in C_{c}^{\infty}\left(\tilde{G}\left(F\right)\right)$.
When $\tilde{M}=\tilde{G}$, it is the character of $\tilde{\tau}$,
which is denoted by $\theta_{\tilde{\tau}}$. By \cite[Theorem 2]{Clo87},
we have $\theta_{\tilde{\tau}}$ is a locally integrable distribution.

In Section \ref{sec2.3}, we have defined a set $\underline{{\mathcal{X}}}\left(G\right)$
of virtual tempered representations of $G\left(F\right)$. Let $\pi\in\underline{{\mathcal{X}}}\left(G\right)$.
There exists a pair $\left(M,\sigma\right)$, where $M$ is a Levi
subgroup of $G$ and $\sigma\in\underline{{\mathcal{X}}}_{\text{ell}}\left(M\right)$,
such that $\pi=i_{M}^{G}\left(\sigma\right)$. For any $f\in{\mathcal{C}}_{\text{scusp}}\left(G\left(F\right)\right)$,
we set 
$$
\hat{\theta}_{f}\left(\pi\right)=\left(-1\right)^{a_{G}-a_{M}}J_{M}^{G}\left(\sigma,f\right).
$$
By \cite[Lemma 5.4.1]{BP20}, this definition is well-defined since the pair $\left(M,\sigma\right)$
is well-defined up to conjugacy. We can also adapt the definition of $\hat{\theta}_f$ to twisted groups.

\subsection{Quasi-characters attached to strongly cuspidal functions}\label{sec3.8}
We extend \cite[Proposition 5.7.1]{BP20} to twisted groups.
\begin{proposition}\label{3.18}
Let $\tilde{x}\in\tilde{G}\left(F\right)_{\text{ell}}$ be an elliptic
element and let $\Omega_{\tilde{x}}\subset G_{\tilde{x}}\left(F\right)$
be a $G$-good open neighborhood of $1$ which is relatively compact
modulo $G_{\tilde{x}}$-conjugation. Set $\Omega=\left(\Omega_{\tilde{x}}\tilde{x}\right)^{G}$.
Then there exists a linear map 
$$
\begin{array}{ccc}
{\mathcal{C}}_{\text{scusp}}\left(\Omega_{\tilde{x}}\right) & \rightarrow & {\mathcal{C}}_{\text{scusp}}\left(\Omega\right)\\
f & \mapsto & \tilde{f}
\end{array}
$$
such that the following properties hold.
\begin{enumerate}
\item For any $f\in{\mathcal{C}}_{\text{scusp}}\left(\Omega_{\tilde{x}}\right)$, we have 
$$
\left(\theta_{\tilde{f}}\right)_{\tilde{x},\Omega_{\tilde{x}}}=\underset{z\in Z_{G}\left(\tilde{x}\right)\left(F\right)/G_{\tilde{x}}\left(F\right)}{\sum}{}^{z}\theta_{f}.
$$
\item There exists a function $\alpha\in C_{c}^{\infty}\left(Z_{G}\left(\tilde{x}\right)\left(F\right)\backslash G\left(F\right)\right)$
satisfying 
$$
\int_{Z_{G}\left(\tilde{x}\right)\left(F\right)\backslash G\left(F\right)}\alpha\left(g\right)dg=1
$$
and such that for any $f\in{\mathcal{C}}_{\text{scusp}}\left(\Omega_{\tilde{x}}\right)$
and $g\in G\left(F\right)$, there exists $z\in Z_{G}\left(\tilde{x}\right)\left(F\right)$
with 
$$
\left(^{zg}\tilde{f}\right)_{\tilde{x},\Omega_{\tilde{x}}}=\alpha\left(g\right)f.
$$
\end{enumerate}
\end{proposition}

\begin{proof}
We follow the approach in \cite[Proposition 5.7.1]{BP20}. Let $\pi:G\left(F\right)\rightarrow Z_{G}\left(\tilde{x}\right)\left(F\right)\backslash G\left(F\right)$
be the natural projection. It is an $F$-analytic locally trivial
fibration. We fix an open subset ${\mathcal{U}}\subseteq Z_{G}\left(\tilde{x}\right)\left(F\right)\backslash G\left(F\right)$
and an $F$-analytic section $s:{\mathcal{U}}\rightarrow G\left(F\right)$. Since $\Omega_{\tilde{x}}$ is a $G$-good open subset, the map 
$$
\begin{array}{cccc}
\beta: & \Omega_{\tilde{x}}\times{\mathcal{U}} & \rightarrow & \tilde{G}\left(F\right)\\
 & \left(y,g\right) & \mapsto & s\left(g\right)^{-1}y\tilde{x}s\left(g\right)
\end{array}
$$
is an open embedding of $F$-analytic spaces. For all $f\in{\mathcal{S}}\left(\Omega_{\tilde{x}}\right)$
and $\varphi\in C_{c}^{\infty}\left({\mathcal{U}}\right)$, we define a
function $f_{\varphi}$ on $\tilde{G}\left(F\right)$ by 
$$
f_{\varphi}\left(\gamma\right)=\begin{cases}
f\left(y\right)\varphi\left(g\right) & \text{if \ensuremath{\gamma=s\left(g\right)^{-1}y\tilde{x}s\left(g\right)} for some \ensuremath{g\in{\mathcal{U}}} and some \ensuremath{y\in\Omega_{\tilde{x}}}}\\
0 & \text{otherwise.}
\end{cases}
$$
It is clear that $f_{\varphi}\in{\mathcal{C}}\left(\Omega\right)$. We
now construct a linear map satisfying the two conditions in the proposition.
We choose a function $\alpha\in C_{c}^{\infty}\left({\mathcal{U}}\right)$
such that $\int_{Z_{G}\left(\tilde{x}\right)\left(F\right)\backslash G\left(F\right)}\alpha\left(g\right)dg=1$ and set $\tilde{f}=\left(\left(\eta_{\tilde{x}}^{\tilde{G}}\right)^{-1/2}f\right)_{\alpha}$
for all $f\in{\mathcal{C}}_{\text{scusp}}\left(\Omega_{\tilde{x}}\right)$.
The second condition follows. We now prove $\tilde{f}$ is strongly
cuspidal. Let $\tilde{P}=\tilde{M}U$ be a proper twisted parabolic
subgroup of $\tilde{G}$. Let $\tilde{m}\in\tilde{M}\left(F\right)\cap\tilde{G}_{\text{reg}}\left(F\right)$.
We need to show 
$$
\int_{U\left(F\right)}\tilde{f}\left(u^{-1}\tilde{m}u\right)du=0.
$$
If the conjugacy class of $\tilde{m}$ does not meet $\Omega_{\tilde{x}}\tilde{x}$,
then it is easy to see that the above integration is equal to $0$.
Assume $\tilde{m}=g^{-1}m_{\tilde{x}}\tilde{x}g$ for some $m_{\tilde{x}}\in\Omega_{\tilde{x}}$.
Set $\tilde{M}^{\prime}=g\tilde{M}g^{-1}$, $U^{\prime}=gUg^{-1}$
and $\tilde{P}^{\prime}=g\tilde{P}g^{-1}=\tilde{M}^{\prime}U^{\prime}$.
Then $\tilde{x}$ lies in $\tilde{M}^{\prime}\left(F\right)$ and
$P_{\tilde{x}}^{\prime}=M_{\tilde{x}}^{\prime}U_{\tilde{x}}^{\prime}$
is a parabolic subgroup of $G_{\tilde{x}}$ which is proper (the assumption
$\tilde{x}\in\tilde{G}\left(F\right)_{\text{ell}}$ is required here).
One has 
$$
\int_{U\left(F\right)}\tilde{f}\left(u^{-1}\tilde{m}u\right)du=\int_{U^{\prime}\left(F\right)}{}^{g}\tilde{f}\left(u^{\prime^{-1}}m_{\tilde{x}}\tilde{x}u^{\prime}\right)du^{\prime}
=\int_{U_{\tilde{x}}^{\prime}\left(F\right)\backslash U^{\prime}\left(F\right)}\int_{U_{\tilde{x}}^{\prime}\left(F\right)}\left(^{u^{\prime}g}\tilde{f}\right)_{\tilde{x},\Omega_{\tilde{x}}}\left(u_{\tilde{x}}^{\prime^{-1}}m_{\tilde{x}}u_{\tilde{x}}^{\prime}\right)du_{\tilde{x}}^{\prime}du^{\prime}.
$$
For all $u^{\prime}\in U^{\prime}\left(F\right)$, there exists $z\in Z_{G}\left(\tilde{x}\right)\left(F\right)$
such that $\left(^{u^{\prime}g}\tilde{f}\right)_{\tilde{x},\Omega_{\tilde{x}}}$
is a scalar multiple of $^{z}f$. Since $f$ is strongly cuspidal,
it follows that the above integral is zero, which is to say $\tilde{f}$
is strongly cuspidal. It remains to show 
$$
\left(\theta_{\tilde{f}}\right)_{\tilde{x},\Omega_{\tilde{x}}}=\underset{z\in Z_{G}\left(\tilde{x}\right)\left(F\right)/G_{\tilde{x}}\left(F\right)}{\sum}{}^{z}\theta_{f}.
$$
Let $\left(G_{n}\right)_{n}$ be an increasing sequence of compact
subsets modulo $A_{\tilde{G}}\left(F\right)$ of $G\left(F\right)$
which is of moderate growth. Assume further that $\underset{n}{\cup}G_{n}=G\left(F\right)$ and $KG_{n}K=G_{n}$.
Let $G_{\tilde{x},n}=G_{n}\cap G_{\tilde{x}}$. Then $\left(G_{\tilde{x},n}\right)_{n}$
be an increasing sequence of compact subsets modulo $A_{G_{\tilde{x}}}\left(F\right)$
of $G_{\tilde{x}}\left(F\right)$ such that $\underset{n}{\cup}G_{\tilde{x},n}=G\left(F\right)$
and $K_{\tilde{x}}G_{\tilde{x},n}K_{\tilde{x}}=G_{\tilde{x},n}$,
where $K_{\tilde{x}}=K\cap G_{\tilde{x}}$. Moreover, we can easily see that $\left(G_{\tilde{x},n}\right)_{n}$ is of moderate growth. Let $y\in\Omega_{\tilde{x},\text{reg}}$.
By Proposition \ref{3.8}, one has 
$$
\theta_{\tilde{f}}\left(y\tilde{x}\right)=\underset{n}{\lim}\int_{A_{\tilde{G}}\left(F\right)\backslash\tilde{G}_{n}}\tilde{f}\left(g^{-1}y\tilde{x}g\right)dg
$$
$$
=\underset{n}{\lim}\int_{Z_{G}\left(\tilde{x}\right)\left(F\right)\backslash G\left(F\right)}\underset{z\in G_{\tilde{x}}\left(F\right)\backslash Z_{G}\left(\tilde{x}\right)\left(F\right)}{\sum}\int_{A_{G_{\tilde{x}}}\left(F\right)\backslash G_{\tilde{x},n}\left(F\right)}\left(^{g}\tilde{f}\right)_{\tilde{x},\Omega_{\tilde{x}}}\left(z^{-1}g_{\tilde{x}}^{-1}yg_{\tilde{x}}z\right)dg_{\tilde{x}}dg
$$
$$
=\underset{n}{\lim}\int_{Z_{G}\left(\tilde{x}\right)\left(F\right)\backslash G\left(F\right)}\alpha\left(g\right)\underset{z\in G_{\tilde{x}}\left(F\right)\backslash Z_{G}\left(\tilde{x}\right)\left(F\right)}{\sum}\int_{A_{G_{\tilde{x}}}\left(F\right)\backslash G_{\tilde{x},n}\left(F\right)}f\left(z^{-1}g_{\tilde{x}}^{-1}yg_{\tilde{x}}z\right)dg_{\tilde{x}}dg
$$
$$
=\underset{n}{\lim}\underset{z\in G_{\tilde{x}}\left(F\right)\backslash Z_{G}\left(\tilde{x}\right)\left(F\right)}{\sum}\int_{A_{G_{\tilde{x}}}\left(F\right)\backslash G_{\tilde{x},n}\left(F\right)}f\left(z^{-1}g_{\tilde{x}}^{-1}yg_{\tilde{x}}z\right)dg_{\tilde{x}}=\underset{z\in G_{\tilde{x}}\left(F\right)\backslash Z_{G}\left(\tilde{x}\right)\left(F\right)}{\sum}{}^{z}\theta_{f}\left(y\right),
$$
noting that on the last line we use an analog of Proposition \ref{3.8} for
connected groups whose proof is similar to the twisted one. This ends
the proof of Proposition \ref{3.18}.
\end{proof}
We state a twisted analog of \cite[Corollary 5.7.2]{BP20}.

\begin{corollary}\label{3.19}
Let $\chi$ be a character of $A_{\tilde{G}}\left(F\right)$. Then 
\begin{enumerate}
\item There exists $\Omega\subseteq A_{\tilde{G}}\left(F\right)\backslash\tilde{G}\left(F\right)$
a completely $G\left(F\right)$-invariant open subset which is relatively
compact modulo conjugation and contains $A_{\tilde{G}}\left(F\right)\backslash\tilde{G}\left(F\right)_{\text{ell}}$
such that the linear map 
$$
f\in{\mathcal{C}}_{\text{scusp}}\left(\Omega,\chi\right)\mapsto\theta_{f}\in QC_{c}\left(\Omega,\chi\right)
$$
is surjective.
\item For all $\theta\in QC\left(\tilde{G}\left(F\right)\right)$, there
exists a compact subset $\Omega_{\theta}\subseteq{\mathcal{X}}_{\text{ell}}\left(\tilde{G}\right)$
such that 
$$
\int_{\Gamma_{\text{ell}}\left(\tilde{G}\right)}D^{\tilde{G}}\left(x\right)\theta\left(x\right)\theta_{\tilde{\pi}}\left(x\right)dx=0
$$
for all $\tilde{\pi}\in{\mathcal{X}}_{\text{ell}}\left(\tilde{G}\right)-\Omega_{\theta}$,
the integral above being absolutely convergent.
\item For all $\tilde{\pi}\in{\mathcal{X}}_{\text{ell}}\left(\tilde{G}\right)$,
there exists $f\in{\mathcal{C}}_{\text{scusp}}\left(\tilde{G}\left(F\right)\right)$
such that for all $\tilde{\pi}^{\prime}\in{\mathcal{X}}_{\text{ell}}\left(\tilde{G}\right)$,
we have $\hat{\theta}_{f}\left(\tilde{\pi}^{\prime}\right)\neq0$ if and only if $\tilde{\pi}^{\prime}=\tilde{\pi}$.
\end{enumerate}
\end{corollary}

\begin{proof}
The proof is similar to \cite[Corollary 5.7.2]{BP20}, noting that we
can replace the orthogonality relations of Arthur in \cite{Art93}
with ones in the twisted case in \cite[Theorem 7.3]{Wal12b}.
\end{proof}

\section{The local twisted Gan-Gross-Prasad conjecture}\label{sec4}
\subsection{Weil representation of unitary groups}\label{sec4.1}

Let $F$ be a nonarchimedean local field of characteristic $0$ and
$\left(W,\left\langle ,\right\rangle \right)$ be a nondegenerate
symplectic vector space of dimension $2n$ over $F$. We denote by $\text{Sp}\left(W\right)$
the isometry group of $W$. Let $H\left(W\right)=W\oplus F$ be the Heisenberg group with operation 
$$
\left(w_{1},t_{1}\right)\left(w_{2},t_{2}\right)=\left(w_{1}+w_{2},t_{1}+t_{2}+\frac{1}{2}\left\langle w_{1},w_{2}\right\rangle \right).
$$
The center of $H\left(W\right)$ is $Z=\left\{ \left(0,t\right)\mid\ t\in F\right\} $.
We define an action of $\text{Sp}\left(W\right)$ to $H\left(W\right)$
$$
g\cdot\left(w,t\right)=\left(gw,t\right),\ \text{for }g\in\text{Sp}\left(W\right)\text{ and }\left(w,t\right)\in H\left(W\right).
$$
Let $\psi$ be a nontrivial additive character of $F$. By Stone-von
Neumann theorem, there exists a unique (up to isomorphism) smooth
irreducible representation $\left(\rho_{\psi},S\right)$ of $H\left(W\right)$
with central character $\psi$. Since any $g\in\text{Sp}\left(W\right)$
acts trivially on $Z$, the representation $\left(\rho_{\psi}^{g},S\right)$
given by $\rho_{\psi}^{g}\left(h\right)=\rho_{\psi}\left(h^{g}\right)$
has the central character $\psi$, hence is isomorphic to $\left(\rho_{\psi},S\right)$.
Thus, for each $g\in\text{Sp}\left(W\right)$, there exists an automorphism
$A\left(g\right):S\rightarrow S$ such that $A\left(g^{-1}\right)\rho_{\psi}\left(h\right)A\left(g\right)=\rho_{\psi}\left(h^{g}\right)$. The above automorphism $A\left(g\right)$ is unique up to a scalar in $\mathbb{C}^{\times}$. One can define a central extension $\text{Mp}\left(W\right)$, which is called the metaplectic cover of $\text{Sp}\left(W\right)$, 
$$
1\longrightarrow\mathbb{C}^{\times}\longrightarrow\text{Mp}\left(W\right)\longrightarrow\text{Sp}\left(W\right)\longrightarrow1
$$
such that $A$ can be lifted to a representation $\omega_{\psi}$
of $\text{Mp}\left(W\right)$ by $\omega_{\psi}\left(g,A\left(g\right)\right)=A\left(g\right)$.

Let $E$ be a quadratic extension of $F$ and $\left(V,\left\langle \cdot,\cdot\right\rangle _{V}\right)$ be a skew-Hermitian space over $E$ of dimension $n$. Then $\left(\text{Res}_{E/F}V,\text{Tr}_{E/F}\left\langle \cdot,\cdot\right\rangle _{V}\right)$
is a $2n$-dimensional symplectic space over $F$. We denote by $\text{Sp}\left(\text{Res}_{E/F}V\right)$
the symplectic group associated to the above symplectic space. As
mentioned above, one can define the Weil representation $\omega_{\psi}$
of $\text{Mp}\left(\text{Res}_{E/F}V\right)$. Let $\mu$ be a conjugate-symplectic character
of $E^{\times}$, i.e. $\mu\mid_{F^{\times}}=\omega_{E/F}$, where
$\omega_{E/F}$ is the quadratic character factoring through $F^{\times}/N_{E/F}E^{\times}\overset{\sim}{\longrightarrow}\left\{ \pm1\right\} $.
By \cite{Kud94} and \cite[Section 1]{HKS96}, the character $\mu$
determines an inclusion $\mu:U\left(V\right)\rightarrow\text{Mp}\left(\text{Res}_{E/F}V\right)$
splitting $\text{Mp}\left(\text{Res}_{E/F}V\right)\rightarrow\text{Sp}\left(\text{Res}_{E/F}V\right)$.
This gives us the Weil representation $\omega_{V,\psi,\mu}$ of $U\left(V\right)$.

Let $K$ be a field extension of $F$ not containing $E$ and $L=K\otimes_{F}E$.
Let $\left(M,\left\langle ,\right\rangle \right)$ be a skew-Hermitian
space relative to $L/K$. Then $\left(\text{Res}_{L/E}M,\text{Tr}_{L/E}\left\langle \right\rangle \right)$ is a skew-Hermitian space over $E$ and one has an inclusion $i:\text{Res}_{K/F}\text{U}\left(M\right)\hookrightarrow\text{U}\left(\text{Res}_{L/E}M\right)$ defined over $F$, where $\text{Res}_{K/F}\text{U}\left(M\right)$ is the usual Weil restriction. We prove a functorial property of Weil representations.
\begin{proposition}\label{4.1}
Denote $\psi_{K}=\psi\circ\text{Tr}_{K/F}$ and $\mu_{L}=\mu\circ\text{Nm}_{L/E}$ so that $\mu_L$ is conjugate-symplectic relative to $L/K$.
Let $\omega_{M,\psi_{K},\mu_{L}}$ be the Weil representation of $\text{U}\left(M\right)$
and $\omega_{\text{Res}_{L/E}M,\psi,\mu}$ be the Weil representation
of $\text{U}\left(\text{Res}_{L/E}M\right)$. Then 
$$
\omega_{M,\psi_{K},\mu_{L}}\cong i^{*}\omega_{\text{Res}_{L/E}M,\psi,\mu},
$$
where $i^{*}\omega_{\text{Res}_{L/E}M,\psi,\mu}$ is the pullpack
representation of $\omega_{\text{Res}_{L/E}M,\psi,\mu}$ via the inclusion
$i:\text{U}\left(M\right)\hookrightarrow\text{U}\left(\text{Res}_{L/E}M\right)$.
\end{proposition}

\begin{proof}
Let $W_{1}=\text{Res}_{L/K}M$ and $W_{2}=\text{Res}_{L/F}M=\text{Res}_{K/F}W_{1}$
be two symplectic spaces over $K$ and $F$ respectively. Let $H\left(W_{1}\right)$
and $H\left(W_{2}\right)$ be corresponding Heisenberg groups. We
have the following surjective map 
$$
\begin{array}{ccccc}
j & : & H\left(W_{1}\right) & \longrightarrow & H\left(W_{2}\right)\\
 &  & \left(w,t\right) & \mapsto & \left(\text{Res}_{K/F}w,\text{Tr}_{K/F}t\right)
\end{array}.
$$
Let $\left(\rho_{\psi},S\right)$ be the irreducible representation
of $H\left(W_{2}\right)$ whose central character is $\psi$. Then
$j$ induces an irreducible representation $\left(j^{*}\rho_{\psi},S\right)$
of $H\left(W_{1}\right)$ with central character $\psi_{K}$. Hence,
one can see that $\omega_{M,\psi_{K},\mu_{L}}$ and $\omega_{\text{Res}_{L/E}M,\psi,\mu}$
share the same underlying vector space and the following diagram 
$$
\begin{tikzcd}
Mp(W_1) \arrow[r, hook] \arrow[d] & Mp(W_2) \arrow[d] \\
Sp(W_1) \arrow[r, hook]           & Sp(W_2)          
\end{tikzcd}
$$
commutes. There remains to show the two splittings 
$$
\begin{tikzcd}
                                                    & Sp(W_2) &                                    \\
U(\text{Res}_{L/E}M) \arrow[ru] \arrow[rr, "\mu_L"] &         & Mp(W_2) \arrow[lu]                 \\
U(M) \arrow[u, hook] \arrow[rr, "\mu"] \arrow[rd]   &         & Mp(W_1) \arrow[u, hook] \arrow[ld] \\
                                                    & Sp(W_1) &                                   
\end{tikzcd}
$$
commute. This follows from a direct computation by using the construction
of the splitting in \cite{Kud94}, noting that for any $g\in\text{U}\left(M\right)$,
one has $\mu_{L}\left(\det g\right)=\mu\left(\det (i\left(g\right))\right)$.
\end{proof}

\subsection{The twisted Gan-Gross-Prasad triples}\label{sec4.2}

Let $F$ be a nonarchimedean local field of characteristic $0$ and $E$ be a quadratic field extension over $F$. Let $V$ be a non-degenerate $n$-dimensional skew-Hermitian space over $E$. We denote $U\left(V\right)$ by $H_{V}$ and $\text{Res}_{E/F}GL\left(V\right)$
by $G$. Note that $G$ is isomorphic to $\text{Res}_{E/F} GL_n(E)$ and thus independent of choices of skew-Hermitian spaces $V$.

Let $\psi$ be a nontrivial additive character of $F$ and $\mu$
be a conjugate-symplectic character of $E^{\times}$. Let $\omega_{V,\psi,\mu}$
be the Weil representation associated to $\left(V,\psi,\mu\right)$.
Let $\pi$ be an irreducible generic representation of $G\left(F\right)$.
We consider the problem of determining 
$$
m_{V}\left(\pi\right)=\dim\text{Hom}_{H_{V}}\left(\pi,\omega_{V,\psi,\mu}\right).
$$
We have the following decomposition $\omega_{V,\psi,\mu}=\underset{\eta}{\bigoplus}\,\omega_{V,\psi,\mu,\eta}$, where $\eta$ runs through the set of characters of the center $U\left(1\right)$ and $\omega_{V,\psi,\mu,\eta}$ is the $\eta$-isotypic part of $\omega_{V,\psi,\mu}$. By the Howe duality proved in \cite{GT16}, $\omega_{V,\psi,\mu,\eta}$
is either $0$ or irreducible and admissible. In particular, the later holds when $\dim V \geq 2$. Let $\chi_\pi$ be the central character
of $\pi$. For simplicity, we also denote by $\chi_\pi$ the restriction
of $\chi_\pi$ to $Z_{G}\left(F\right)\cap H\left(F\right)=U\left(1\right)$.
Then 
$$
\text{Hom}_{H_{V}}\left(\pi,\omega_{V,\psi,\mu}\right)=\text{Hom}_{H_{V}}\left(\pi,\omega_{V,\psi,\mu,\chi_\pi}\right).
$$
 We call a triple $\left(G,H_{V},\omega_{V,\psi,\mu}\right)$ a twisted
Gan-Gross-Prasad triple.

\subsection{Transfer of character distributions of dual pairs}\label{sec4.3}

Let $F$ be a local field of characteristic $0$. In this section, we recall some results in \cite{Prz93}. Although these results are proved when $F=\mathbb{R}$, they can be adapted easily in the case of our interest, i.e. the dual pair $(\text{U}_1,\text{U}(V))$ and $(\text{GL}_1,\text{GL}(V))$ over a nonarchimedean field $F$.

Let $(W,\langle,\rangle)$ be a symplectic space of dimension $2n$ over $F$ and $(G,G^\prime)$ be an irreducible dual pair in $\text{Sp}(W)$. Let $\psi$ be a nontrivial additive character of $F$ and $\omega_\psi$ be the Weil representation of $\text{Mp}_{2n}(F)$. Let $\pi$ be an irreducible admissible representation of $G(F)$. We set $\Theta(\pi)=\left(\omega_\psi\otimes\pi^\vee\right)_{G(F)}$, which is a representation of $G^\prime(F)$. We say that $\Theta(\pi)$ is the theta lift of $\pi$ for the dual pair $(G,G^\prime)$.

We recall some estimates of character distributions. By \cite[Theorem 3]{Mil77}, the character distribution $\Theta_\pi$ has the rate of growth $\gamma\geq 0$ if there exists a constant $d\geq 0$ such that
$$
\left|\Theta_\pi(g)\right| \ll \left|D^G(g)\right|^{-1/2}\Xi^G(g)^{-\gamma}(g)\sigma_G(g)^d, \text{ for any }g\in G(F).
$$
When $\pi$ is unitary, we have $\gamma \leq 1$ (c.f. \cite[4.5.1]{BoWa00}). Let $\mathbb{D}$ be a finite dimensional division algebra over $F$ with an involution. Let $(V,\langle,\rangle_V)$ and $(V^\prime,\langle,\rangle_{V^\prime})$ be two nondegenerate forms over $\mathbb{D}$, one hermitian and one skew-hermitian, such that $G$ and $G^\prime$ are isometry groups of $V$ and $V^\prime$ respectively. We set
$$
r=2\dim_{F}\mathfrak{g}/\dim_F V \text{ and } r^\prime=2\dim_{F}\mathfrak{g^\prime}/\dim_F V^\prime.
$$
We recall \cite[Proposition 4.11]{Prz93}.
\begin{proposition}\label{4.2}
    We fix a matrix coefficient $\Omega(g)=\left|\langle \omega_\psi(g)v,v \rangle \right|$ for some $v\in \omega_\psi$.
    \begin{enumerate}
        \item If $\Omega_\pi$ has the rate of growth $\gamma < \dim_\mathbb{D}V^\prime/(r-1)-1$, then 
        $$
        \int_{G(F)}\left|\Theta_\pi(g)\right|\Omega(g)dg < \infty.
        $$
        \item If $\gamma < \dim_\mathbb{D}V^\prime/(r-1)-1$ and the form $(v_1,v_2)=\int_{G(F)}\overline{\Theta_\pi(g)}\langle \omega_\psi(g)v_1,v_2\rangle dg$
        is positive semidefinite and nontrivial, then $\Theta_{\Theta(\pi)}$ has the rate of growth $\gamma^\prime=1-\lambda^\prime$, where
        $$
        \lambda^\prime = \left( 1- \frac{(1+\gamma)(r-1)}{\dim_{\mathbb{D}}V^\prime}\right)\cdot\frac{\dim_{\mathbb{D}}V}{r^\prime-1}.
        $$
    \end{enumerate}
\end{proposition}
By the above proposition, we can deduce the following corollary.
\begin{corollary}\label{4.3}
    Let $\chi$ be a character of the center $U_1$ (resp. $GL_1$) of $U(V)$ (resp. $GL(V)$) and $\omega_{V,\psi,\mu,\chi}$ (resp. $\omega_{\mu,\chi}$) be the $\chi$-isotypic subspace of $\omega_{V,\psi,\mu}$ (resp. $\omega_{\mu}$). Then the character distribution $\Theta_{\omega_{V,\psi,\mu,\chi}}$ (resp. $\Theta_{\omega_{\mu,\chi}}$) has the rate of growth $\frac{n-2}{n-1}$.
\end{corollary}

We recall the notion of wavefront set of character distributions. Since $\Theta_\pi$ is a quasi-character, there exists a good neighborhood $\omega$ of $0$ in $\mathfrak{g}(F)$ such that $\Theta_\pi(\exp(X))=\underset{\mathcal{O}\in \text{Nil}(g(F))}{\sum}c_{\Theta_\pi,\mathcal{O}}(1)\hat{j}(\mathcal{O},X)$, for any $X\in \omega_{\text{rss}}$. The wavefront set $\text{WF}(\pi)$ of $\pi$ is the closure of the union of nilpotent orbits $\mathcal{O}$ in $\mathfrak{g}(F)$ such that $c_{\Theta_\pi,\mathcal{O}}$ is nonzero.

We identify $W$ with $\text{Hom}_\mathbb{D}(V,V^\prime)$. We define two canonical moment maps
$$
\tau_G:w\in W \mapsto w^*w \in \mathfrak{g}\ \ \  \text{ and } \ \ \ \tau_{G^\prime}:w\in W \mapsto ww^* \in \mathfrak{g}^\prime,
$$
where $w^*\in \text{Hom}_\mathbb{D}(V^\prime,V)$ is the adjoint of $w$ with respect to $\langle,\rangle$. We recall \cite[Corollary 7.10]{Prz93}.
\begin{proposition}\label{4.4}
    Suppose the pair $(G,G^\prime)$ is in the stable range with $G$ the smaller member. For any finite-dimensional irreducible unitary representation $\pi$ of $G(F)$, we have 
    $$
    WF(\Theta(\pi))=\tau_{G^\prime}\left(\tau_{G}^{-1}(WF(\pi))\right)=\tau_{G^\prime}\left(\tau_{G}^{-1}(0)\right).
    $$
\end{proposition}
We deduce the following corollary, whose proof follows easily from the above proposition.
\begin{corollary}\label{4.5}
    Assume that $V$ is not the two-dimensional nonsplit skew-hermitian space over $E$. Let $\chi$ be a character of the center $U_1$ (resp. $GL_1$) of $U(V)$ (resp. $GL(V)$) and $\omega_{V,\psi,\mu,\chi}$ (resp. $\omega_{\mu,\chi}$) be the $\chi$-isotypic subspace of $\omega_{V,\psi,\mu}$ (resp. $\omega_{\mu}$). Then the wavefront set of $\omega_{V,\psi,\mu,\chi}$ (resp. $\omega_{\mu,\chi}$) is the closure of a minimal nilpotent orbit $\mathcal{O}_m$ (which is independent of $\chi$) in $\mathfrak{u}(V)$ (resp. $\mathfrak{gl}(V)$). Moreover, we have the leading coefficient $c_{\theta_{\omega_{V,\psi,\mu,\chi}},\mathcal{O}_m}(1)=1$.
\end{corollary}
\begin{proof}
    We consider the case $(G,G^\prime)=(U_1,U(V))$, as the case $(\text{GL}_1,\text{GL}(V))$ is similar. Let $\langle,\rangle$ be the skew-hermitian form of $V$. We also fix the hermitian form $(x,y)=\bar{x}y$ for $U_1$. In this case, we have $\tau_{G}(v)=\langle v,v \rangle$, for any $v\in V$. Therefore, we obtain $\tau_G^{-1}(0)=V^0=\{v\in V\mid\,\langle v,v \rangle=0\}$. By Witt theorem, there are two $U(V)$-orbits in $V^0$, which are $\{0\}$ and $V^0\backslash\{0\}$. We set
    $$
    \mathcal{O}_0=\tau_{G^\prime}(\{0\})=\{0\}\ \ \ \text{and}\ \ \ \mathcal{O}_m=\tau_{G^\prime}(V^0\backslash\{0\}).
    $$
    There remains to show $\mathcal{O}_m$ is a minimal nilpotent orbit in $\mathfrak{u}(V)$. We may fix an $E$-basis $\{e_1,\ldots,e_n\}$ for $V$ such that $\langle v,v^\prime \rangle = \bar{v}^TXv^\prime$ with an anti-diagonal matrix $X$. In this case, we have $e_1\in V_0\backslash \{0\}$ and
    $$
    \tau_{G^\prime}(e_1)=\left(\begin{array}{cccc}
    0 & \cdots & 0 & *\\
    \vdots & \ddots & \vdots & 0\\
    0 & \ldots & 0 & 0
    \end{array}\right).
    $$
    This gives us $\mathcal{O}_m$ is a minimal nilpotent orbit in $\mathfrak{u}(V)$. By Proposition \ref{4.4}, it follows that $WF(\omega_{V,\psi,\mu,\chi})=\overline{\mathcal{O}_m}$. By \cite{MW87} and \cite[Theorem 1.1]{GZ14}, we have 
    $$
    c_{\theta_{\omega_{V,\psi,\mu,\chi},\mathcal{O}_m}}(1)
    =\dim \text{Wh}_{\mathcal{O}_m}(\omega_{V,\psi,\mu,\chi})=\dim \text{Wh}_{0}(\chi)=1.
    $$
\end{proof}

\subsection{Spherical variety structure and some estimates}\label{sec4.4}

Let $\theta$ be an involution of $G$ defined over $F$ such that $H_{V}=G^{\theta}$.
We choose a norm on $Z_G\backslash G$ and denote it by $\bar{\sigma}:=\sigma_{Z_{G}\backslash G}$. Let $T\subset G$
be a (not necessarily maximal) torus. We say $T$ is $\theta$-split if $\theta\left(t\right)=t^{-1}$ for all $t\in T$, and we say it is $\left(\theta,F\right)$-split if it is $\theta$-split and split as a torus over $F$. 

Let $A_{j}$, where $j\in J$, be representatives of the $H_{V}\left(F\right)$-conjugacy classes of maximal $\left(\theta,F\right)$-split tori of $G$. There are finitely many of them and by \cite{BO07,DS11}, one has a weak Cartan decomposition, i.e. there exists a compact subset ${\mathcal{K}}_{G}\subset G\left(F\right)$ such that 
$$
G\left(F\right)=\underset{j\in J}{\bigcup}H_{V}\left(F\right)A_{j}\left(F\right){\mathcal{K}}_{G}.
$$
Let $X=HZ_{G}\backslash G$. Let $C\subset G\left(F\right)$ be a
compact subset with nonempty interior and set $\Xi_{C}^{X}\left(x\right)=\text{vol}_{X}\left(xC\right)^{-1/2}$, for all $x\in X(F)$, where $\text{vol}_{X}$ refers to a $G\left(F\right)$-invariant
measure on $X(F)$. If $C^{\prime}\subset G\left(F\right)$ is an another
compact subset with nonempty interior, then $\Xi_{C}^{X}\sim\Xi_{C^{\prime}}^{X}$.
Hence, we can fix a choice of $C$ and denote $\Xi_{C}^{X}$ by $\Xi^{X}$. We prove some estimates for later uses.
\begin{proposition}\label{4.6}
$ $
\begin{enumerate}
\item For every compact subset ${\mathcal{K}}\subset G\left(F\right)$, we have
the following equivalences 
$$
\Xi^{X}\left(xk\right)\sim\Xi^{X}\left(x\right)\text{ and }\sigma_{X}\left(xk\right)\sim\sigma_{X}\left(x\right)
$$
for all $x\in Z_{G}\left(F\right)H\left(F\right)\backslash G\left(F\right)$
and $k\in{\mathcal{K}}$.
\item Let $A_{0}$ be a $\left(\theta,F\right)$-split subtorus of $G$.
Then there exists $d>0$ such that 
$$
\Xi^{G}\left(a\right)\bar{\sigma}\left(a\right)^{-d}\ll\Xi^{X}\left(a\right)\ll\Xi^{G}\left(a\right)
\text{ and }
\sigma_{X}\left(a\right)\sim\bar{\sigma}\left(a\right),\ \forall a\in A_{0}\left(F\right).
$$
\item There exists $d>0$ such that the integral 
$$
\int_{Z_{G}\left(F\right)H\left(F\right)\backslash G\left(F\right)}\Xi^{X}\left(x\right)^{2}\sigma_{X}\left(x\right)^{-d}dx
$$
is absolutely convergent.
\item For all $d>0$, there exists $d^{\prime}>0$ such that 
$$
\int_{Z_{H}\left(F\right)\backslash H\left(F\right)}\Xi^{G}\left(hx\right)\bar{\sigma}\left(hx\right)^{-d^{\prime}}dh\ll\Xi^{X}\left(x\right)\sigma_{X}\left(x\right)^{-d}
$$
for all $x\in X$.
\item For all $d>0$, there exists $d^{\prime}>0$ such that 
$$
\int_{Z_{G}\left(F\right)H\left(F\right)\backslash G\left(F\right)}{\bf 1}_{\sigma_{X}\leq c}\left(x\right)\Xi^{X}\left(x\right)^{2}\sigma_{X}\left(x\right)^{d}dx\ll c^{d^{\prime}}
$$
for all $c\geq1$.
\item For any $\phi,\phi^\prime \in \omega_{V,\psi,\mu}$, there exists $\epsilon>0$ such that 
$$
\int_{Z_H(F)\backslash H(F)}\Xi^G(h)e^{\epsilon \sigma(h)}\langle \phi,\omega_{V,\psi,\mu}(h)\phi^\prime\rangle dh
$$
is absolutely convergent.
\item Let $P_{\min}=M_{\min} U_{\min}$ be a good minimal parabolic subgroup of $G$. For any $\phi,\phi^\prime \in \omega_{V,\psi,\mu}$, there exists $\epsilon >0$ such that the following integral 
$$
\int_{H(F)}\Xi^G(hm)e^{\epsilon \sigma(h)}\langle \phi,\omega_{V,\psi,\mu}(h)\phi^\prime\rangle dh
$$
is absolutely convergent for all $m\in M_{\min}(F)$ and there exists $d>0$ such that the above integral is essentially bounded above by $\delta_{P_{\min}}(m)^{-1/2}\sigma(m)^d$, for all $m\in M_{\min}(F)$.
\item Let $\phi,\phi^{\prime}\in\omega_{V,\psi,\mu}$ and $d>0$. Then the
integral 
$$
I_{\phi,\phi^{\prime},d}\left(c,x\right)=\underset{Z_{H}\left(F\right)\backslash H\left(F\right)}{\iint}{\bf 1}_{\sigma\geq c}\left(h^{\prime}\right)\Xi^{G}\left(hx\right)\Xi^{G}\left(h^{\prime}hx\right)\sigma\left(hx\right)^{d}\sigma\left(h^{\prime}hx\right)^{d}
\left\langle \phi,\omega_{V,\psi,\mu}\left(h^\prime\right)\phi^{\prime}\right\rangle dh^{\prime}dh$$
is absolutely convergent for all $x\in Z_{G}\left(F\right)H\left(F\right)\backslash G\left(F\right)$
and $c\geq1$. Moreover, there exists $\epsilon>0$ and $d^{\prime}>0$
such that 
$$
I_{\phi,\phi^{\prime},d}\left(c,x\right)\ll\Xi^{X}\left(x\right)^{2}\sigma_{X}\left(x\right)^{d^{\prime}}e^{-\epsilon c}
$$
for all $x\in Z_{G}\left(F\right)H\left(F\right)\backslash G\left(F\right)$
and $c\geq1$.
\end{enumerate}
\end{proposition}

\begin{proof}
Part 1, 2, 3, 4 are established in \cite[Proposition 2.3.1]{BP18}.
The proof of part 5 is exactly the same as for \cite[Proposition 6.7.1 (iv)]{BP20}.

We prove part 6. We mimic the proof of Proposition 1.1.1(1) in \cite{Xue16}. Let $A_{0}$
be a maximal split subtorus of $H$ and $M_{0}$ be its centralizer
in $H$. We fix the standard minimal parabolic subgroup $P_{0}$ of
$H$ with the Levi decomposition $P_{0}=M_{0}N_{0}$. Let $\Delta$
be the set of simple roots with respect to $\left(P_{0},A_{0}\right)$.
Let $\delta_{P_{0}}$ be the modulus character of $P_{0}$. Let 
$$
A_{0}^{+}=\left\{ a\in A_{0}\mid\left|\alpha\left(a\right)\right|\leq1\text{ for all }\alpha\in\Delta\right\}
$$
$$
=\left\{ \text{diag}\left[a_{1},\ldots,a_{r},1,\ldots,1,a_{r}^{\tau,-1},\ldots,a_{1}^{\tau,-1}\right]\mid\ \left|a_{1}\right|\leq\ldots\leq\left|a_{r}\right|\leq1\right\} ,
$$
where $r$ is the split rank of $H\left(F\right)$. We fix a special
maximal compact subgroup $K$ of $H$. We have the Cartan decomposition
$H\left(F\right)=KA_{0}^{+}\left(F\right)K$. We recall some basic
estimates in Appendix B in \cite{Xue16}. One has 
$$
\int_{H\left(F\right)}f\left(g\right)dg=\int_{A_{0}^{+}\left(F\right)}\nu\left(m\right)\underset{K\times K}{\iint}f\left(k_{1}mk_{2}\right)dk_{1}dk_{2}dm,
$$
where $\nu\left(m\right)$ is a positive-valued function on $A_{0}^{+}$
such that there is a constant $A>0$ with 
$$
A^{-1}\delta_{P_{0}}\left(m\right)^{-1}\leq\nu\left(m\right)\leq A\delta_{P_{0}}\left(m\right)^{-1}.
$$
We also fix a minimal parabolic subgroup $P_{1}=M_{1}N_{1}$ for $G$
such that $P_{0}\subset P_{1}$ and $M_{0}\subset M_{1}$. There exists
a constant $d$ such that $\delta_{P_{1}}\left(m\right)^{\frac{1}{2}}\ll\Xi^{G}\left(m\right)\ll\delta_{P_{1}}\left(m\right)^{\frac{1}{2}}\sigma\left(m\right)^{d}$. Since $\delta_{P_{1}}\left(a\right)\ll\delta_{P_{0}}\left(a\right)^{2}$, for any $a\in A_{0}^{+}\left(F\right)$, it suffices to show that there exists $\epsilon >0$ such that 
$$
\int_{A_{0}^{+}\left(F\right)} \sigma\left(a\right)^{d}e^{\epsilon \sigma(a)}\left\langle \phi,\omega_{V,\psi,\mu}\left(a\right)\phi^\prime\right\rangle da\text{ is absolutely convergent.}
$$
Moreover, we have $\sigma\left(a\right)^{d}e^{\epsilon \sigma(a)} \ll e^{\epsilon^\prime \sigma(a)}$, for any $\epsilon^\prime > \epsilon$. Hence, we only need to show that there exists $\epsilon >0$ such that 
$$
\int_{A_{0}^{+}\left(F\right)} e^{\epsilon \sigma(a)}\left\langle \phi,\omega_{V,\psi,\mu}\left(a\right)\phi^\prime\right\rangle da\text{ is absolutely convergent.}
$$
Let $a=\text{diag}\left[a_{1},\ldots,a_{r},1,\ldots,1,a_{r}^{\tau,-1},\ldots,a_{1}^{\tau,-1}\right]$,
where $\left|a_{1}\right|\leq\ldots\leq\left|a_{r}\right|\leq1$.
As in \cite[Appendix D]{Xue16}, we have $\left|\left\langle \phi,\omega_{V,\psi,\mu}\left(a\right)\phi^\prime\right\rangle \right|\ll_{\phi,\phi^\prime}\left|a_{1}\ldots a_{r}\right|^{\frac{1}{2}}$. Therefore
$$
\int_{A_{0}^{+}\left(F\right)} e^{\epsilon \sigma(a)}\left|\left\langle \phi,\omega_{V,\psi,\mu}\left(a\right)\phi^\prime\right\rangle\right| da \ll
\int_{\left|x_{1}\right|\leq\ldots\leq\left|x_{r}\right|\leq1}\left|x_{1}\ldots x_{r}\right|^{\frac{1}{2}-\epsilon}dx,
$$
and the integral on the right hand side converges when $\epsilon$ is sufficiently small. This gives us the absolute convergence in part 6.

We now prove part 7. Since $\Xi^G(g) \sim \Xi^G(g^{-1})$ and $\sigma(h) \sim \sigma(h^{-1})$, it is equivalent to show the following

For sufficiently small $\epsilon >0$, the integral
$$
\int_{H(F)}\Xi^G(mh)e^{\epsilon \sigma(h)}\langle \phi,\omega_{V,\psi,\mu}(h)\phi^\prime\rangle dh
$$
is absolutely convergent for all $m\in M_{\min}(F)$ and there exists $d>0$ such that the above integral is essentially bounded above by
$\delta_{P_{\min}}(m)^{1/2}\sigma(m)^d$ for all $m\in M_{\min}(F)$.
Let $K$ be the special maximal compact subgroup of $G(F)$. We fix a map $m_{P_{\min}}:G(F)\rightarrow M_{\min}(F)$ such that $g\in m_{P_{\min}}(g)U_{\min}(F)K$ for all $g\in G(F)$. There exists $d>0$ such that $\Xi^G(g)\ll \delta_{P_{\min}}(m_{P_{\min}}(g))^{1/2}\sigma(m_{P_{\min}}(g))^d$. This gives us
$$
\int_{H(F)}\Xi^G(mh)e^{\epsilon \sigma(h)}\left|\langle \phi,\omega_{V,\psi,\mu}(h)\phi^\prime\rangle\right| dh
$$
$$
\ll \delta_{P_{\min}}(m)^{1/2}\sigma(m)^d \int_{H(F)}\delta_{P_{\min}}(m_{P_{\min}}(h))^{1/2}\sigma(m_{P_{\min}}(h))^de^{\epsilon \sigma(h)}\left|\langle \phi,\omega_{V,\psi,\mu}(h)\phi^\prime\rangle\right| dh.
$$
Since $\sigma(m_{P_{\min}}(h))^de^{\epsilon \sigma(h)} \ll e^{\epsilon^\prime \sigma(h)}$, for any $\epsilon^\prime > \epsilon$, it suffices to show that when $\epsilon>0$ is sufficiently small, the integral
$$
\int_{H(F)}\delta_{P_{\min}}(m_{P_{\min}}(h))^{1/2}e^{\epsilon \sigma(h)}\left|\langle \phi,\omega_{V,\psi,\mu}(h)\phi^\prime\rangle\right| dh
$$
is absolutely convergent. This essentially follows from the proof of part 6.

There remains to show part 8. Observe ${\bf 1}_{\sigma\geq c}\left(h^{\prime}\right)\sigma\left(h^{\prime}\right)^{d}\ll e^{\epsilon\sigma\left(h^{\prime}\right)}e^{-\epsilon c/2}$.
By part 1 and 2 and the weak Cartan decomposition mentioned at the beginning of this subsection, it suffices to show there exists $d^\prime>0$ and $\epsilon>0$ such that 
$$
\underset{Z_{H}\left(F\right)\backslash H\left(F\right)}{\iint}\Xi^{G}\left(ha\right)\Xi^{G}\left(h^{\prime}ha\right) e^{\epsilon\sigma\left(h\right)}e^{\epsilon\sigma\left(h^{\prime}\right)}\left|\left\langle \phi,\omega_{V,\psi,\mu}\left(h^\prime\right)\phi^{\prime}\right\rangle\right| dh^{\prime}dh\ll\Xi^{G}\left(a\right)^{2}\sigma(a)^{d^\prime},
$$
for any $a\in A_{j}^{+}$. By changing of variable $h^{\prime}\mapsto h^{\prime}h^{-1}$,
we need to show 
$$
\underset{Z_{H}\left(F\right)\backslash H\left(F\right)}{\iint}\Xi^{G}\left(ha\right)\Xi^{G}\left(h^{\prime}a\right)e^{\epsilon\sigma\left(h\right)}e^{\epsilon\sigma\left(h^{\prime}\right)}\left|\left\langle \phi,\omega_{V,\psi,\mu}\left(h^\prime h^{-1}\right)\phi^{\prime}\right\rangle\right| dh^{\prime}dh\ll\Xi^{G}\left(a\right)^{2}\sigma(a)^{d^\prime},
$$
for any $a\in A_{j}^{+}$. By part 7, there exists $d^\prime$ such that 
$$
\int_{Z_{H}\left(F\right)\backslash H\left(F\right)}\Xi^{G}\left(h^{\prime}a\right)e^{\epsilon\sigma\left(h^{\prime}\right)}\left|\left\langle \phi,\omega_{V,\psi,\mu}\left(h^\prime h^{-1}\right)\phi^{\prime}\right\rangle\right| dh^{\prime} \ll \Xi^G(a)\sigma(a)^{d^\prime},
$$
noting that the above inequality only depends on $||\phi||$ and $||\phi^\prime||$, and hence holds for all $h$. There remains to show 
$$
\int_{Z_{H}\left(F\right)\backslash H\left(F\right)}\Xi^{G}\left(ha\right)e^{\epsilon\sigma\left(h\right)}dh \ll \Xi^G(a).
$$
Let $K$ be a maximal compact subgroup
of $G\left(F\right)$ by which $\Xi^{G}$ is right invariant. By the
doubling principle in \cite[Proposition 1.5.1(vi)]{BP20}, one has 
$$
\int_{Z_{H}\left(F\right)\backslash H\left(F\right)}\Xi^{G}\left(ha\right)e^{\epsilon\sigma\left(h\right)}dh\ll\int_{Z_{H}\left(F\right)\backslash H\left(F\right)}\int_K\Xi^{G}\left(hka\right)dke^{\epsilon\sigma\left(h\right)}dh^{\prime}
$$
$$
=\Xi^{G}\left(a\right)\int_{Z_{H}\left(F\right)\backslash H\left(F\right)}\Xi^{G}\left(h\right)e^{\epsilon\sigma\left(h\right)}dh\ll\Xi^{G}\left(a\right),
$$
noting that $\int_{Z_{H}\left(F\right)\backslash H\left(F\right)}\Xi^{G}\left(h\right)e^{\epsilon\sigma\left(h\right)}dh$
is absolutely convergent if $\epsilon$ is sufficiently small.
\end{proof}

\section{Explicit tempered intertwinings}\label{sec5}

From now to section \ref{sec7}, we prove part (i) of Theorem \ref{thm1.2} by induction. In \cite{GGP23},
they have shown the conjecture for $n=1,2$. We now consider $n\geq3$.
\textbf{The proof under the induction hypothesis only finishes at the end
of section \ref{sec7}.}

Let $\pi\in\text{Temp}\left(G,\chi^{-1}\right)$. Let $T\in\text{End}\left(\pi\right)^\infty$
and $\phi_{1},\phi_{2}\in\omega_{V,\psi,\mu,\chi}$ be Schwartz functions.
We define a linear form 
$$
{\mathcal{L}}_{\pi}^{\phi_{1},\phi_{2}}\left(T\right)=\int_{Z_{H}\left(F\right)\backslash H\left(F\right)}\text{Tr}\left(\pi\left(h^{-1}\right)T\right)\left\langle \phi_{1},\omega_{V,\psi,\mu,\chi}\left(h\right)\phi_{2}\right\rangle dh.
$$
In this section, we prove the following theorem.
\begin{theorem}\label{5.1}
For any $\pi\in\text{Temp}\left(G,\chi^{-1}\right)$, 
$$
m_{V}\left(\bar{\pi}\right)\neq0\text{ if and only if }\text{there exists }\phi_{1},\phi_{2}\in\omega_{V,\psi,\mu,\chi}\text{ such that }{\mathcal{L}}_{\pi}^{\phi_{1},\phi_{2}}\not\equiv0.
$$
\end{theorem}

\subsection{Absolute convergence of the linear form ${\mathcal{L}}_{\pi}^{\phi_{1},\phi_{2}}$}\label{sec5.1}

Let $\pi$ be a tempered representation of $G\left(F\right)$ whose
central character is $\chi^{-1}$. For all $T\in\text{End}\left(\pi\right)$,
the function $g\in G\left(F\right)\mapsto\text{Trace}\left(\pi\left(g^{-1}\right)T\right)$ is in the weak Harish-Chandra Schwartz space ${\mathcal{C}}^{w}\left(Z_{G}\left(F\right)\backslash G\left(F\right),\chi\right)$ (defined in Section \ref{sec2.1}).
In order to prove ${\mathcal{L}}_{\pi}^{\phi_{1},\phi_{2}}$ is absolutely
convergent, we consider the following lemma.
\begin{lemma}\label{5.2}
One has 
$$
\int_{Z_{H}\left(F\right)\backslash H\left(F\right)}f\left(h\right)\left\langle \phi_{1},\omega_{V,\psi,\mu,\chi}\left(h\right)\phi_{2}\right\rangle dh\text{ is absolutely convergent,}
$$
for all $f\in{\mathcal{C}}^{w}\left(Z_{G}\left(F\right)\backslash G\left(F\right),\chi\right)$.
Moreover, given $\phi_{1}$ and $\phi_{2}$, the above integral defines
a continuous linear form on $\mathcal{C}^{w}\left(Z_{G}\left(F\right)\backslash G\left(F\right),\chi\right)$. 
\end{lemma}

\begin{proof}
The absolute convergence of this integral essentially follows from the proof of Proposition \ref{4.6}(vi), noting that here we have
$$
\int_{Z_{H}\left(F\right)\backslash H\left(F\right)}|f\left(h\right)||\left\langle \phi_{1},\omega_{V,\psi,\mu,\chi}\left(h\right)\phi_{2}\right\rangle| dh
\ll \int_{|x_1|\leq \ldots \leq |x_r|\leq 1}|x_1\ldots x_r|^{1/2}\sum_{i=1}^{r}(1-\log |x_i|)^d dx.
$$
The continuity argument is straightforward from the above estimate.
\end{proof}

\subsection{Some preparation}\label{sec5.2}

Let $\pi$ be a tempered irreducible representation of $G\left(F\right)$ whose
central character is $\chi^{-1}$. Recall that $\pi\otimes\bar{\pi}\cong\text{End}\left(\pi\right)^\infty$
as $G\left(F\right)\times G\left(F\right)$-representations via the map $e\otimes e^{\prime}\mapsto T_{e,e^{\prime}}:=\left(\cdot,e^{\prime}\right)e$.
We identify ${\mathcal{L}}_{\pi}^{\phi_{1},\phi_{2}}$ with the continuous
sesquilinear form on $\pi$ given by 
$$
{\mathcal{L}}_{\pi}^{\phi_{1},\phi_{2}}\left(e,e^{\prime}\right):={\mathcal{L}}_{\pi}^{\phi_{1},\phi_{2}}\left(T_{e,e^{\prime}}\right),\text{ for any }e,e^{\prime}\in\pi.
$$
If ${\mathcal{L}}_{\pi}^{\phi_{1},\phi_{2}}\not\equiv0$ for some $\phi_{1},\phi_{2}\in\omega_{V,\psi,\mu,\chi}$,
one has $m_{V}\left(\bar{\pi}\right)\neq0$. Thus, in order to prove
Theorem \ref{5.1}, it suffices to show the opposite direction.

We define a continuous linear map 
$$
\begin{array}{ccccc}
L_{\pi}^{\phi_{1},\phi_{2}} & : & \pi & \rightarrow & \bar{\pi}^*\\
 &  & e & \mapsto & \left(e^{\prime}\mapsto{\mathcal{L}}_{\pi}^{\phi_{1},\phi_{2}}\left(e,e^{\prime}\right)\right)
\end{array},
$$
Let $T\in\text{End}\left(\pi\right)^\infty$. By Theorem \ref{A.1}, we have $L_{\pi}^{\phi_{1},\phi_{2}}$
is of finite rank. We can define its traces and hence $\text{Trace}\left(TL_{\pi}^{\phi_{1},\phi_{2}}\right)=\text{Trace}\left(L_{\pi}^{\phi_{1},\phi_{2}}T\right)={\mathcal{L}}_{\pi}^{\phi_{1},\phi_{2}}\left(T\right)$. We prove some properties of the above linear forms.
\begin{lemma}\label{5.3}
We have the following
\begin{enumerate}
\item Let $\phi_{1},\phi_{2}\in\omega_{V,\psi,\mu,\chi}$. The maps 
$$
\pi\in{\mathcal{X}}_{\text{temp}}\left(G,\chi^{-1}\right)\mapsto L_{\pi}^{\phi_{1},\phi_{2}}\in\text{Hom}\left(\pi,\bar{\pi}\right)
$$
$$
\pi\in{\mathcal{X}}_{\text{temp}}\left(G,\chi^{-1}\right)\mapsto{\mathcal{L}}_{\pi}^{\phi_{1},\phi_{2}}\in\text{End}\left(\pi\right)^{*}
$$
are smooth in the following sense: Let $K$ be the special maximal
compact subgroup of $G\left(F\right)$. For every parabolic subgroup
$Q=LU_{Q}$ of $G$ and $\sigma\in\Pi_{2}\left(Z_{G}\backslash L,\chi^{-1}\right)$,
the maps 
$$
\lambda\in i{\mathcal{A}}_{L}^{*}\mapsto{\mathcal{L}}_{\pi_{\lambda}}^{\phi_{1},\phi_{2}}\in\text{End}\left(\pi_{\lambda}\right)^{*}\simeq\text{End}\left(\pi_{K}\right)^{*}
$$
$$
\lambda\in i{\mathcal{A}}_{L}^{*}\mapsto L_{\pi_{\lambda}}^{\phi_{1},\phi_{2}}\in\text{End}\left(\pi_{\lambda},\bar{\pi}_{\lambda}\right)\simeq\text{End}\left(\pi_{K},\bar{\pi}_{K}\right)
$$
are smooth, where we set $\pi_{\lambda}=i_{Q}^{G}\left(\sigma_{\lambda}\right)$
and $\pi_{K}=i_{Q\cap K}^{K}\left(\sigma\right)$.
\item Suppose $\pi\in\text{Temp}\left(G,\chi^{-1}\right)$.
If $S,T\in\text{End}\left(\pi\right)^\infty$, then $SL_{\pi}^{\phi_{1},\phi_{2}}T\in\text{End}\left(\pi\right)^\infty$.
If $m_{V}\left(\bar{\pi}\right)\leq1$, one has 
$$
{\mathcal{L}}_{\pi}^{\phi_{1},\phi_{2}}\left(SL_{\pi}^{\phi_{3},\phi_{4}}T\right)={\mathcal{L}}_{\pi}^{\phi_{1},\phi_{4}}\left(S\right){\mathcal{L}}_{\pi}^{\phi_{3},\phi_{2}}\left(T\right).
$$
\item Let $f\in{\mathcal{C}}\left(Z_{G}\left(F\right)\backslash G\left(F\right),\chi\right)$
and $\phi_{1},\phi_{2}\in\omega_{V,\psi,\mu,\chi}$. Then 
$$
\int_{Z_{H}\left(F\right)\backslash H\left(F\right)}f\left(h\right)\left\langle \phi_{1},\omega_{V,\psi,\mu,\chi}\left(h\right)\phi_{2}\right\rangle dh=\int_{{\mathcal{X}}_{\text{temp}}\left(G,\chi^{-1}\right)}{\mathcal{L}}_{\pi}^{\phi_{1},\phi_{2}}\left(\pi\left(f\right)\right)\mu\left(\pi\right)d\pi
$$
and both integrals are absolutely convergent.
\item Let $f\in{\mathcal{C}}\left(Z_{G}\left(F\right)\backslash G\left(F\right),\chi\right)$
and $f^{\prime}\in{\mathcal{C}}\left(Z_{G}\left(F\right)\backslash G\left(F\right),\chi^{-1}\right)$
and $\phi_{1},\phi_{2},\phi_{3},\phi_{4}\in\omega_{V,\psi,\mu,\chi}$.
Then we have the following equality 
$$
\int_{{\mathcal{X}}_{\text{temp}}\left(G,\chi^{-1}\right)}{\mathcal{L}}_{\pi}^{\phi_{1},\phi_{2}}\left(\pi\left(f\right)\right){\mathcal{L}}_{\pi}^{\phi_{3},\phi_{4}}\left(\pi\left(f^{\prime,\vee}\right)\right)\mu\left(\pi\right)d\pi
$$
$$
=\underset{Z_{H}\left(F\right)\backslash H\left(F\right)}{\iint}\int_{Z_{G}\left(F\right)\backslash G\left(F\right)}f\left(hgh^{\prime}\right)f^{\prime}\left(g\right)\left\langle \phi_{3},\omega_{V,\psi,\mu,\chi}\left(h^{\prime}\right)\phi_{2}\right\rangle \left\langle \phi_{1},\omega_{V,\psi,\mu,\chi}\left(h\right)\phi_{4}\right\rangle dgdh^{\prime}dh,
$$
where the first integral is absolutely convergent and the second integral
is convergent in that order but not necessarily as a triple integral.
\end{enumerate}
\end{lemma}

\begin{proof}
Part 1 and 3 are analogs of \cite[Lemma 7.2.2]{BP20} and \cite[Lemma 8.2.1]{BP20b}.
We prove the second and the forth part.

2. For $e,e^{\vee}\in\pi$ and $\phi,\phi^{\vee}\in\omega_{V,\psi,\mu.\chi}$,
we define 
$$
\alpha\left(e,e^{\vee},\phi,\phi^{\vee}\right)=\int_{Z_{H}\left(F\right)\backslash H\left(F\right)}\left\langle e,\pi\left(h\right)e^{\vee}\right\rangle \left\langle \phi,\omega_{V,\psi,\mu.\chi}\left(h\right)\phi^{\vee}\right\rangle dh.
$$
Let $e_{1},e_{1}^{\vee},e_{2},e_{2}^{\vee}\in\pi$. By continuity,
we only need to prove the statement for $S=e_{1}\otimes e_{1}^{\vee}$
and $T=e_{2}\otimes e_{2}^{\vee}$. A direct computation shows that
$$
SL_{\pi}^{\phi_{3},\phi_{4}}T\left(e\right)=\left\langle e,e_{2}^{\vee}\right\rangle \left(\int_{Z_{H}\left(F\right)\backslash H\left(F\right)}\left\langle e_{2},\pi\left(h\right)e_{1}^{\vee}\right\rangle \left\langle \phi_{3},\omega_{V,\psi,\mu,\chi}\left(h\right)\phi_{4}\right\rangle dh\right)e_{1}.
$$
Hence, we have 
$$
\begin{array}{ccc}
{\mathcal{L}}_{\pi}^{\phi_{1},\phi_{2}}\left(SL_{\pi}^{\phi_{3},\phi_{4}}T\right) & = & \left(\int_{Z_{H}\left(F\right)\backslash H\left(F\right)}\left\langle e_{2},\pi\left(h\right)e_{1}^{\vee}\right\rangle \left\langle \phi_{3},\omega_{V,\psi,\mu,\chi}\left(h\right)\phi_{4}\right\rangle dh\right)\\
 &  & \times\left(\int_{Z_{H}\left(F\right)\backslash H\left(F\right)}\left\langle e_{1},\pi\left(h\right)e_{2}^{\vee}\right\rangle \left\langle \phi_{1},\omega_{V,\psi,\mu,\chi}\left(h\right)\phi_{2}\right\rangle dh\right)\\
 & = & \alpha\left(e_{2},e_{1}^{\vee},\phi_{3},\phi_{4}\right)\alpha\left(e_{1},e_{2}^{\vee},\phi_{1},\phi_{2}\right).
\end{array}
$$
On the other hand, we can see that 
$$
{\mathcal{L}}_{\pi}^{\phi_{1},\phi_{4}}\left(S\right){\mathcal{L}}_{\pi}^{\phi_{3},\phi_{2}}\left(T\right)=\alpha\left(e_{1},e_{1}^{\vee},\phi_{1},\phi_{4}\right)\alpha\left(e_{2},e_{2}^{\vee},\phi_{3},\phi_{2}\right).
$$
The following maps 
$$
\left(e^{\vee},\phi\right)\mapsto\alpha\left(e_{1},e^\vee,\phi_{1},\phi\right)\ \ \ \text{and}\ \ \ \left(e^{\vee},\phi\right)\mapsto\alpha\left(e_{2},e^\vee,\phi_{3},\phi\right)
$$
are propotional, as they belong to $\text{Hom}_{H\left(F\right)}\left(\bar{\pi}\otimes\bar{\omega}_{V,\psi,\mu,\chi},\mathbb{C}\right)$
and $m_{V}\left(\bar{\pi}\right)\leq1$. Thus, one has 
$$
{\mathcal{L}}_{\pi}^{\phi_{1},\phi_{2}}\left(SL_{\pi}^{\phi_{3},\phi_{4}}T\right)={\mathcal{L}}_{\pi}^{\phi_{1},\phi_{4}}\left(S\right){\mathcal{L}}_{\pi}^{\phi_{3},\phi_{2}}\left(T\right).
$$

4. The right hand side of the equality in part 4 is equal to
$$
\iint_{Z_{H}\left(F\right)\backslash H\left(F\right)}\left(f^{\prime,\vee}*\left(L\left(h^{-1}\right)f\right)\right)\left(h^{\prime}\right)\left\langle \phi_{3},\omega_{V,\psi,\mu,\chi}\left(h^{\prime}\right)\phi_{2}\right\rangle \left\langle \phi_{1},\omega_{V,\psi,\mu,\chi}\left(h\right)\phi_{4}\right\rangle dh^{\prime}dh.
$$
The Fourier transform of $f^{\prime,\vee}*\left(L\left(h^{-1}\right)f\right)$
is given by $\pi\in{\mathcal{X}}_{\text{temp}}\left(G,\chi^{-1}\right)\mapsto\pi\left(f^{\prime,\vee}\right)\pi\left(h^{-1}\right)\pi\left(f\right)$, which is compactly supported. Part 3 gives us 
$$
\int_{Z_{H}\left(F\right)\backslash H\left(F\right)}\left(f^{\prime,\vee}*\left(L\left(h^{-1}\right)f\right)\right)\left(h^{\prime}\right)\left\langle \phi_{3},\omega_{V,\psi,\mu,\chi}\left(h^{\prime}\right)\phi_{2}\right\rangle dh^{\prime}
$$
is absolutely convergent and equal to 
$$
\int_{{\mathcal{X}}_{\text{temp}}\left(G,\chi^{-1}\right)}{\mathcal{L}}_{\pi}^{\phi_{3},\phi_{2}}\left(\pi\left(f^{\prime,\vee}\right)\pi\left(h^{-1}\right)\pi\left(f\right)\right)\mu\left(\pi\right)d\pi
$$
$$
=\int_{{\mathcal{X}}_{\text{temp}}\left(G,\chi^{-1}\right)}\text{Trace}\left(\pi\left(h^{-1}\right)\pi\left(f\right)L_{\pi}^{\phi_{3},\phi_{2}}\pi\left(f^{\prime,\vee}\right)\right)\mu\left(\pi\right)d\pi.
$$
The section $\pi\in{\mathcal{X}}_{\text{temp}}\left(G,\chi^{-1}\right)\mapsto\pi\left(h^{-1}\right)\pi\left(f\right)L_{\pi}^{\phi_{3},\phi_{2}}\pi\left(f^{\prime,\vee}\right)$
is smooth and compactly supported, hence by the matricial Paley-Wiener
theorem, it is the Fourier transform of a Harish-Chandra Schwartz
function. Applying part 3 to this function, one has 
$$
\underset{Z_{H}\left(F\right)\backslash H\left(F\right)}{\iint}\left(f^{\prime,\vee}*\left(L\left(h^{-1}\right)f\right)\right)\left(h^{\prime}\right)\left\langle \phi_{3},\omega_{V,\psi,\mu,\chi}\left(h^{\prime}\right)\phi_{2}\right\rangle \left\langle \phi_{1},\omega_{V,\psi,\mu,\chi}\left(h\right)\phi_{4}\right\rangle dh^{\prime}dh
$$
is absolutely convergent and equal to the following absolutely convergent
integral 
$$
\int_{{\mathcal{X}}_{\text{temp}}\left(G,\chi^{-1}\right)}{\mathcal{L}}_{\pi}^{\phi_{1},\phi_{4}}\left(\pi\left(f\right)L_{\pi}^{\phi_{3},\phi_{2}}\pi\left(f^{\prime,\vee}\right)\right)\mu\left(\pi\right)d\pi
=\int_{{\mathcal{X}}_{\text{temp}}\left(G,\chi^{-1}\right)}{\mathcal{L}}_{\pi}^{\phi_{1},\phi_{2}}\left(\pi\left(f\right)\right){\mathcal{L}}_{\pi}^{\phi_{3},\phi_{4}}\left(\pi\left(f^{\prime,\vee}\right)\right)\mu\left(\pi\right)d\pi.
$$
\end{proof}

\subsection{Explicit intertwinings and parabolic induction}\label{sec5.3}

Let $\pi=i_{P}^{G}\left(\sigma\right)$ with central character
$\chi$, where $P=MN$ and 
$$
M=\text{Res}_{E/F}\text{GL}\left(W_{1}\right)\times\ldots\times\text{Res}_{E/F}\text{GL}\left(W_{a}\right)
$$
and $\sigma$ is an irreducible tempered representation of $M\left(F\right)$.
In this section, we prove the following proposition.
\begin{proposition}\label{5.4}
Let $\phi_{1},\phi_{2}\in\omega_{V,\psi,\mu,\chi}$. If ${\mathcal{L}}_{\pi}^{\phi_{1},\phi_{2}}\neq0$,
then for any $\lambda\in i{\mathcal{A}}_{M}^{*}$, one has 
$$
{\mathcal{L}}_{\pi_{\lambda}}^{\phi_{1},\phi_{2}}\neq0,
$$
where $\pi_{\lambda}:= i_{P}^{G}\left(\sigma_{\lambda}\right)$.
\end{proposition}

\begin{proof}
We consider the following simple claim.
\begin{claim}\label{5.5}
For any $h_{1},h_{2}\in H\left(F\right)$, one has 
$$
{\mathcal{L}}_{\pi}^{\phi_{1},\phi_{2}}\neq0\Leftrightarrow{\mathcal{L}}_{\pi}^{\omega_{V,\psi,\mu,\chi}\left(h_{1}\right)\phi_{1},\omega_{V,\psi,\mu,\chi}\left(h_{2}\right)\phi_{2}}\neq0.
$$
\end{claim}

\begin{proof}
Let $T\in\text{End}\left(\pi\right)^\infty$ such that ${\mathcal{L}}_{\pi}^{\phi_{1},\phi_{2}}\left(T\right)\neq0$.
Then 
$$
{\mathcal{L}}_{\pi}^{\omega_{\psi,\mu}\left(h\right)\phi_{1},\omega_{\psi,\mu}\left(h^{\prime}\right)\phi_{2}}\left(\pi\left(h_{1}\right)T\pi\left(h_{2}^{-1}\right)\right)={\mathcal{L}}_{\pi}^{\phi_{1},\phi_{2}}\left(T\right)\neq0.
$$
\end{proof}
Assume $P$ is a good parabolic subgroup of $G$, i.e. $HP$ is open
in $G$. The invariant scalar product of $\pi$ is given by 
$$
\left(e,e^{\prime}\right)=\int_{P\left(F\right)\backslash G\left(F\right)}\left(e\left(g\right),e^{\prime}\left(g\right)\right)dg,\text{ for any }e,e^{\prime}\in\pi,
$$
where the scalar product inside the above integral is the scalar product
of $\sigma$. We have 
$$
{\mathcal{L}}_{\pi}^{\phi_{1},\phi_{2}}\left(e,e^{\prime}\right)=\int_{Z_{H}\left(F\right)\backslash H\left(F\right)}\int_{P\left(F\right)\backslash G\left(F\right)}\left(e\left(g\right),e^{\prime}\left(gh\right)\right)dg\left\langle \phi_{1},\omega_{V,\psi,\mu,\chi}\left(h\right)\phi_{2}\right\rangle dh,
$$
for $e,e^{\prime}\in\pi$. We prove the above expression is absolutely
convergent. Let $K$ be the special maximal compact subgroup of $G\left(F\right)$.
By choosing a suitable Haar measure on $K$, one has 
$$
\int_{P\left(F\right)\backslash G\left(F\right)}\left|\left(e\left(g\right),e^{\prime}\left(gh\right)\right)\right|dg=\int_{K}\left|\left(e\left(k\right),e^{\prime}\left(k_P(kh)\right)\right)\right|dk
$$
$$
=\int_{K}\delta_{P}\left(l_{P}\left(kh\right)\right)^{1/2}\left|\left(e\left(k\right),\sigma\left(l_{P}\left(k_P(kh)\right)\right)e^{\prime}\left(k_P(kh)\right)\right)\right|dk,
$$
for all $h\in H\left(F\right)$, where $l_{P}:G\left(F\right)\rightarrow M\left(F\right)$
and $k_{P}:G\left(F\right)\rightarrow K$ are maps such that $l_{P}\left(g\right)^{-1}gk_{P}\left(g\right)^{-1}\in N\left(F\right)$
for all $g\in G\left(F\right)$. Since $\sigma$ is tempered, one
has 
$$
\int_{P\left(F\right)\backslash G\left(F\right)}\left|\left(e\left(g\right),e^{\prime}\left(gh\right)\right)\right|dg\ll\int_{K}\delta_{P}\left(l_{P}\left(kh\right)\right)^{1/2}\Xi^{M}\left(l_{P}\left(kh\right)\right)dk=\Xi^{G}\left(h\right).
$$
Then the absolute convergence comes from Lemma \ref{5.2}.

Let $P_{1},\ldots,P_{t}$ be representatives of good $H\left(F\right)$-conjugacy
classes of $P$, i.e. some $P^\prime$ such that $H(F)P^\prime(F)$ is open in $G(F)$, inside the $G\left(F\right)$-conjugacy class of $P$.
We can choose $P_{1},\ldots,P_{t}$ to be good parabolic subgroups
of $G$. These conjugacy classes parametrizes open $H\left(F\right)$-orbits
in $G\left(F\right)/P\left(F\right)$. Let $H_{P_{i}}=H\cap P_{i}$,
for $i=\overline{1,t}$. By a statement about
open orbits in Lemma 4.2 in \cite{GGP23},
it follows that $H_{P_i}=U\left(W_{i,1}\right)\times\ldots\times U\left(W_{i,a}\right)$, where $W_{i,j}$ is a nondegenerate skew-hermitian space, for any $(i,j)\in [1,t]\times [1,a]$. Note that $\dim W_{i,j}=\dim W_{i^\prime,j}$ for any $i,i^\prime$. We can choose Haar measures such that 
$$
\int_{P\left(F\right)\backslash G\left(F\right)}\phi\left(g\right)dg=\sum_{i=1}^t\int_{H_{P_{i}}\left(F\right)\backslash H\left(F\right)}\phi\left(g_{i}h\right)dh,\text{ for any }\phi\in L^{1}\left(P\left(F\right)\backslash G\left(F\right),\delta_{P}\right),
$$
where $g_{i}\in G\left(F\right)$ satisfying $P_{i}=g_i^{-1}Pg_i$. Then
$$
{\mathcal{L}}_{\pi}^{\phi_{1},\phi_{2}}\left(e,e^{\prime}\right)=\sum_{i=1}^t\int_{Z_{H}\left(F\right)\backslash H\left(F\right)}\int_{H_{P_{i}}\left(F\right)\backslash H\left(F\right)}\left(e\left(g_{i}h^{\prime}\right),e^{\prime}\left(g_{i}h^{\prime}h\right)\right)dh^{\prime}\left\langle \phi_{1},\omega_{V,\psi,\mu,\chi}\left(h\right)\phi_{2}\right\rangle dh,
$$
which is absolutely convergent. Observe 
$$
{\mathcal{L}}_{\pi}^{\phi_{1},\phi_{2}}\left(e,e^{\prime}\right)=\sum_{i=1}^t\int_{\left(H_{P_{i}}\left(F\right)\backslash H\left(F\right)\right)^{2}}{\mathcal{L}}_{\sigma,P_{i}}^{\omega_{V,\psi,\mu,\chi}\left(h^{\prime}\right)\phi_{1},\omega_{V,\psi,\mu,\chi}\left(h\right)\phi_{2}}\left(e\left(g_{i}h^{\prime}\right),e^{\prime}\left(g_{i}h\right)\right)dhdh^{\prime},
$$
where 
$$
{\mathcal{L}}_{\sigma,P_{i}}^{\phi_{1},\phi_{2}}\left(v,v^{\prime}\right)=\int_{Z_{H}\left(F\right)\backslash H_{P_{i}}\left(F\right)}\left(v,\sigma\left(h_{P_{i}}\right)v^{\prime}\right)\left\langle \phi_{1},\omega_{V,\psi,\mu,\chi}\left(h_{P_{i}}\right)\phi_{2}\right\rangle d_{M}h_{P_{i}},
$$
for all $v,v^{\prime}\in\sigma$. By the induction hypothesis stated at the beginning of Section \ref{sec5}, there
is a unique $P_{i}$ such that $\dim\text{Hom}_{H_{P_{i}}}\left(\sigma,\omega_{V,\psi,\mu,\chi}\right)\neq0$. For simplicity, we can assume such $P_{i}$ is $P$. Then for any $P_{i}\neq P$, one has ${\mathcal{L}}_{\sigma,P_{i}}^{\phi_{1},\phi_{2}}\equiv0$.
This gives us 
$$
{\mathcal{L}}_{\pi}^{\phi_{1},\phi_{2}}\left(e,e^{\prime}\right)=\int_{\left(H_{P}\left(F\right)\backslash H\left(F\right)\right)^{2}}{\mathcal{L}}_{\sigma,P}^{\omega_{V,\psi,\mu,\chi}\left(h^{\prime}\right)\phi_{1},\omega_{V,\psi,\mu,\chi}\left(h\right)\phi_{2}}\left(e\left(h^{\prime}\right),e^{\prime}\left(h\right)\right)dhdh^{\prime}.\ \ \left(*\right)
$$
We set 
$$
{\mathcal{L}}_{\sigma,P}^{\phi_{1},\phi_{2}}\left(T_{\sigma}\right)=\int_{Z_{H}\left(F\right)\backslash H_{P}\left(F\right)}\text{Trace}\left(\sigma\left(h_{P}^{-1}\right)T_{\sigma}\right)\left\langle \phi_{1},\omega_{V,\psi,\mu,\chi}\left(h_{P}\right)\phi_{2}\right\rangle d_{M}h_{P},
$$
for $T_{\sigma}\in\text{End}\left(\sigma\right)^\infty$. This integral is
absolutely convergent since $\sigma$ is tempered (we only need to
follow the approach in Lemma \ref{5.2}). The following claim reduces the
study of ${\mathcal{L}}_{\pi}^{\phi_{1},\phi_{2}}$ to ${\mathcal{L}}_{\sigma,P}^{\phi_{1},\phi_{2}}$.
\begin{claim}\label{5.6}
Let $\phi_{1},\phi_{2}\in\omega_{V,\psi,\mu,\chi}$. Then ${\mathcal{L}}_{\pi}^{\phi_{1},\phi_{2}}\neq0$
if and only if there exists $h,h^{\prime}\in H\left(F\right)$ such
that ${\mathcal{L}}_{\sigma,P}^{\omega_{V,\psi,\mu,\chi}\left(h^{\prime}\right)\phi_{1},\omega_{V,\psi,\mu,\chi}\left(h\right)\phi_{2}}\neq0$.
\end{claim}

\begin{proof}
By $\left(*\right)$, if ${\mathcal{L}}_{\pi}^{\phi_{1},\phi_{2}}\neq0$,
there exists $h,h^{\prime}\in H\left(F\right)$ such that ${\mathcal{L}}_{\sigma,P}^{\omega_{V,\psi,\mu,\chi}\left(h^{\prime}\right)\phi_{1},\omega_{V,\psi,\mu,\chi}\left(h\right)\phi_{2}}$ is nonzero.
There remains to prove the reverse direction. Assume ${\mathcal{L}}_{\sigma}^{\phi_{1},\phi_{2}}\neq0$.
Let $s:{\mathcal{U}}\rightarrow H\left(F\right)$ be an analytic section
over an open subset ${\mathcal{U}}$ of $H_{P}\left(F\right)\backslash H\left(F\right)$
whose image contained in an open compact subgroup $K$ of $H\left(F\right)$
fixing $\phi_{1}$ and $\phi_{2}$. For a smooth compactly supported
function $\phi\in C_{c}^{\infty}\left({\mathcal{U}},\sigma\right)$, we
set 
$$
e_{\phi}\left(g\right)=\begin{cases}
\delta_{P}\left(m\right)^{1/2}\sigma\left(m\right)\phi\left(h\right) & \text{if }g=mns\left(h\right)\text{ with }m\in M\left(F\right),\ n\in N\left(F\right),\ h\in{\mathcal{U}}\\
0 & \text{otherwise.}
\end{cases}
$$
This defines an element of $\pi$. One has 
$$
{\mathcal{L}}_{\pi}^{\phi_{1},\phi_{2}}\left(e_{\phi},e_{\phi^{\prime}}\right)=\int_{\left(H_{P}\left(F\right)\backslash H\left(F\right)\right)^{2}}{\mathcal{L}}_{\sigma,P}^{\omega_{V,\psi,\mu,\chi}\left(h^{\prime}\right)\phi_{1},\omega_{V,\psi,\mu,\chi}\left(h\right)\phi_{2}}\left(\phi\left(h^{\prime}\right),\phi^{\prime}\left(h\right)\right)dhdh^{\prime},
$$
for any $\phi,\phi^{\prime}\in C_{c}^{\infty}\left({\mathcal{U}},\sigma\right)$.
Assume ${\mathcal{L}}_{\sigma}^{\phi_{1},\phi_{2}}\left(v_{0},v_{0}^{\prime}\right)\neq0$.
We set $\phi\left(h\right)=f\left(h\right)v_{0}$ and $\phi^{\prime}\left(h\right)=\overline{f\left(h\right)}v_{0}^{\prime}$,
where $f\in C_{c}^{\infty}\left({\mathcal{U}}\right)$ is nonzero. Then
${\mathcal{L}}_{\pi}^{\phi_{1},\phi_{2}}\left(e_{\phi},e_{\phi^{\prime}}\right)\neq0$.
We finish our proof for this claim.
\end{proof}
By Claim \ref{5.5} and Claim \ref{5.6}, it suffices to show that ${\mathcal{L}}_{\sigma}^{\phi_{1},\phi_{2}}$
is nonzero if and only if ${\mathcal{L}}_{\sigma_{\lambda}}^{\phi_{1},\phi_{2}}$
is nonzero, where $\lambda$ is an unramified character of $M\left(F\right)$.
Since $P$ is a good parabolic subgroup of $G$, we can set $H_{P}=U\left(W_{1}\right)\times\ldots\times U\left(W_{a}\right)$. Therefore, $\sigma$ and $\sigma_{\lambda}$ coincide when restricting
to $H_{P}\left(F\right)$. This gives us ${\mathcal{L}}_{\sigma}^{\phi_{1},\phi_{2}}\left(v,v^{\prime}\right)={\mathcal{L}}_{\sigma_{\lambda}}^{\phi_{1},\phi_{2}}\left(v,v^{\prime}\right)$, for any $v,v^{\prime}\in\sigma$, noting that $\sigma$ and $\sigma_{\lambda}$ share the same underlying space. We finish our proof for Proposition \ref{5.4}.
\end{proof}

\subsection{Proof of Theorem \ref{5.1}}\label{sec5.4}

We prove the following lemma, which is a consequence of the Cartan
decomposition.
\begin{lemma}\label{5.7}
Let $\pi$ be a tempered representation of $G\left(F\right)$ with central character $\chi$ and $l\in\text{Hom}_{H}\left(\pi,\omega_{V,\psi,\mu,\chi}\right)$.Then for any $f\in \mathcal{C}(Z_G(F)\backslash G(F),\chi)$ and $e\in \pi$ and $\phi \in \omega_{V,\psi,\mu,\chi}$, we have  
$$
\int_{Z_G(F)\backslash G(F)} f(g) \left\langle \phi,l\left(\pi\left(g\right)e\right)\right\rangle dg
$$
is absolutely convergent. 
\end{lemma}

\begin{proof}
Let $A_{0}$ be a maximal $\left(\theta,F\right)$-split torus of
$G$. Then $M_{0}=\text{Cent}_{G}\left(A_{0}\right)$ is a minimal
$\theta$-split Levi subgroup of $G$. We fix $P_{0}=M_{0}N_{0}$
to be a $\theta$-split parabolic subgroup of $G$. We set 
$$
A_{P_{0}}^{+}=\left\{ a\in A_{0}\left(F\right)\mid\ \left|\alpha\left(a\right)\right|\leq 1,\ \forall\alpha\in\Delta_{P_{0}}\right\} .
$$
By the Cartan decomposition and since there exists $d>0$ such that $\int_{Z_G(F)\backslash G(F)} \Xi^G(g)^2 \sigma(g)^{-d} dg$ converges, it suffices to show $\left|\langle \phi, l(\pi(a)e \rangle \right| \ll \Xi^G(a)$, for any $a\in A_{P_0}^+(F)$. Let $K_H$ be a compact open subgroup of $H(F)$ such that $\phi \in \omega_{V,\psi,\mu,\chi}^{K_H}$. Let $J\subset G\left(F\right)$ be a compact open subgroup which fixes
$e$. There exists a compact open subgroup $K_{P_0} \subset P_0\left(F\right)$ such that 
$
K_{P_0}\subset aJa^{-1}$, for any $a\in A_{P_{0}}^{+}(F)$. Since $P_0$ is $\theta$-split, we have $H(F)P_0(F)$ is open in $G(F)$. Therefore, we can choose a compact open subgroup $J^\prime$ of $G(F)$ such that $J^\prime$ is contained in $K_HK_{P_0}$. Let $T_{J^{\prime},\phi}$ be a linear form of $\pi$ defined by $T_{J^{\prime},\phi}(v)=\frac{1}{\text{vol}(J^\prime)}\int_{J^\prime}\langle\phi,l(\pi(k)v)\rangle dk$, for $v\in \pi$. We can see that $T_{J^{\prime},\phi}\in\pi^{\vee}$. Moreover, for any $a\in A_{P_{0}}^{+}$,
we have 
$$
\left\langle T_{J^{\prime},\phi},\pi\left(a\right)e\right\rangle=\frac{1}{\text{vol}(J^\prime)}\int_{J^\prime}\langle\phi,l(\pi(ka)e)\rangle dk=\langle \phi, l(\pi(a)e)\rangle.
$$
Since $\pi$ is tempered, we have $\left\langle T_{J^{\prime},\phi},\pi\left(a\right)e\right\rangle\ll\Xi^{G}\left(a\right)$, which is to say
$$
\left|\langle \phi, l(\pi(a)e \rangle \right|\ll\Xi^{G}\left(a\right),\,\forall a\in A_{P_{0}}^{+}(F).
$$
\end{proof}
We now prove Theorem \ref{5.1}.
\begin{proof}
Let $\sigma\in\Pi_{2}\left(Z_{G}\backslash M,\chi\right)$
such that $\pi_0=i_{P}^{G}\left(\sigma\right)$,
where $P=MN$ is a parabolic subgroup of $G$. Assume $m_{V}\left(\pi_{0}\right)\neq0$
and let $l$ be a nonzero element in $\text{Hom}_{H}\left(\pi_{0},\omega_{V,\psi,\mu,\chi}\right)$.
Let $e\in\pi_{0}$ and $\phi\in\omega_{V,\psi,\mu,\chi}$ and $f\in{\mathcal{C}}\left(Z_{G}\left(F\right)\backslash G\left(F\right),\chi\right)$.
By Lemma \ref{5.7}, we have  
$$
\int_{Z_{G}\left(F\right)\backslash G\left(F\right)}\left\langle \phi,l\left(\pi_{0}\left(g\right)e\right)\right\rangle f\left(g\right)dg\text{ is absolutely convergent.}
$$
Next, we compute $\int_{Z_{G}\left(F\right)\backslash G\left(F\right)}\left\langle \phi,l\left(\pi_{0}\left(g\right)e\right)\right\rangle f\left(g\right)dg$
in two ways. We choose a compact open subgroup $J$ of $G\left(F\right)$
such that $f$ is $J$-left invariant. We have  
$$
\int_{Z_{G}\left(F\right)\backslash G\left(F\right)}\left\langle \phi,l\left(\pi_{0}\left(g\right)e\right)\right\rangle f\left(g\right)dg
=\frac{1}{\text{vol}\left(J\right)}\int_{Z_{G}\left(F\right)\backslash G\left(F\right)}f\left(g\right)\int_{J}\left\langle \phi,l\left(\pi_{0}\left(kg\right)e\right)\right\rangle dkdg
$$
$$
=\int_{Z_{G}\left(F\right)\backslash G\left(F\right)}f\left(g\right)\left\langle T_{J,\phi},\pi_{0}\left(g\right)e\right\rangle dg=\left\langle T_{J,\phi},\pi_{0}\left(\bar{f}\right)e\right\rangle =\left\langle \phi,l\left(\pi_{0}\left(\bar{f}\right)e\right)\right\rangle .
$$
On the other hand,
$$
\int_{Z_{G}\left(F\right)\backslash G\left(F\right)}\left\langle \phi,l\left(\pi_{0}\left(g\right)e\right)\right\rangle f\left(g\right)dg
=\int_{Z_{G}\left(F\right)H\left(F\right)\backslash G\left(F\right)}\int_{Z_{H}\left(F\right)\backslash H\left(F\right)}\left\langle \phi,\omega_{V,\psi,\mu,\chi}\left(h\right)l\left(\pi_{0}\left(x\right)e\right)\right\rangle f\left(hx\right)dhdx.
$$
By Lemma \ref{5.3}(iii), it follows that 
$$
\int_{Z_{G}\left(F\right)\backslash G\left(F\right)}\left\langle \phi,l\left(\pi_{0}\left(g\right)e\right)\right\rangle f\left(g\right)dg
=\int_{Z_{G}\left(F\right)H\left(F\right)\backslash G\left(F\right)}\int_{{\mathcal{X}}_{\text{temp}}\left(G\left(F\right),\chi^{-1}\right)}{\mathcal{L}}_{\pi}^{\phi,l\left(\pi_{0}\left(x\right)e\right)}\left(\pi\left(f\right)\pi\left(x^{-1}\right)\right)\mu\left(\pi\right)d\pi dx.
$$
Let $T\in C_{c}^{\infty}\left({\mathcal{X}}_{\text{temp}}\left(G,\chi^{-1}\right),{\mathcal{E}}\left(G,\chi^{-1}\right)\right)$.
By substituting $f=f_{T}$ (see in Theorem \ref{2.2}) to these above equations, one has 
\begin{equation}\label{eqn4.1}
    \left\langle \phi,l\left(T_{\bar{\pi}_{0}}e\right)\right\rangle =\int_{Z_{G}\left(F\right)H\left(F\right)\backslash G\left(F\right)}\int_{{\mathcal{X}}_{\text{temp}}\left(G\left(F\right),\chi^{-1}\right)}{\mathcal{L}}_{\pi}^{\phi,l\left(\pi_{0}\left(x\right)e\right)}\left(T_{\pi}\pi\left(x^{-1}\right)\right)\mu\left(\pi\right)d\pi dx.
\end{equation}
Since $l\neq0$, we can choose $e\in\pi$ such that $l\left(e\right)\neq0$.
Let $T_{0}\in\text{End}\left(\pi\right)^\infty$ satisfying $T_{0}e=e$.
We choose $T^{0}\in C_{c}^{\infty}\left({\mathcal{X}}_{\text{temp}}\left(G,\chi^{-1}\right),{\mathcal{E}}\left(G,\chi^{-1}\right)\right)$
such that $T_{\bar{\pi}_{0}}^{0}=T_{0}$ and $\text{Supp}\left(T^{0}\right)\subset{\mathcal{O}}$, where ${\mathcal{O}}$ is the connected component containing $\bar{\pi}_0$.
Since $l\left(e\right)\neq0$, we can choose $\phi_{2}\in\omega_{V,\psi,\mu,\chi}$
satisfying $\left\langle \phi_{2},l\left(e\right)\right\rangle \neq0$.
Applying the equation ($\ref{eqn4.1}$) to $T^{0}$ and $\phi_{2}$, one has 
$$
\left\langle \phi_{2},l\left(e\right)\right\rangle =\int_{Z_{G}\left(F\right)H\left(F\right)\backslash G\left(F\right)}\int_{{\mathcal{X}}_{\text{temp}}\left(G\left(F\right),\chi^{-1}\right)}{\mathcal{L}}_{\pi}^{\phi_{2},l\left(\pi_{0}\left(x\right)e\right)}\left(T_{\pi}\pi\left(x^{-1}\right)\right)\mu\left(\pi\right)d\pi dx.
$$
Since the LHS is not zero, there exists $x_{0}\in Z_{G}\left(F\right)H\left(F\right)\backslash G\left(F\right)$
and $\rho\in i{\mathcal{A}}_{M}^{*}$ such that ${\mathcal{L}}_{\bar{\pi}_{0,\rho}}^{\phi_{2},l\left(\pi_{0}\left(x_{0}\right)e\right)}\neq0$.
By Proposition \ref{5.4}, it follows that ${\mathcal{L}}_{\bar{\pi}_0}^{\phi_{2},l\left(\pi_{0}\left(x_{0}\right)e\right)}\neq0$.
We finish our proof for Theorem \ref{5.1}.
\end{proof}

\subsection{Tempered intertwinings and the multiplicity}\label{sec5.5}

In this section, we prove the following result, which is used to obtain
a spectral expansion for our local trace formula.
\begin{corollary}\label{5.8}
Let ${\mathcal{K}}\subset{\mathcal{X}}_{\text{temp}}\left(G,\chi^{-1}\right)$
be a compact subset. Assume $m_{V}\left(\bar{\pi}\right)\leq1$, for
any irreducible representation $\pi\in{\mathcal{K}}$. There exists a section
$T\in{\mathcal{C}}\left({\mathcal{X}}_{\text{temp}}\left(G,\chi^{-1}\right),{\mathcal{E}}\left(G,\chi^{-1}\right)\right)$
and $\phi_{1},\phi_{2}\in\omega_{V,\psi,\mu,\chi}$ such that 
$$
{\mathcal{L}}_{\pi}^{\phi_{1},\phi_{2}}\left(T_{\pi}\right)=m_{V}\left(\bar{\pi}\right)
$$
for all $\pi\in{\mathcal{K}}$.
\end{corollary}

\begin{proof}
Let ${\mathcal{O}}$ be a connected component in ${\mathcal{X}}_{\text{temp}}\left(G\right)$
and ${\mathcal{K}}_{{\mathcal{O}}}={\mathcal{K}}\cap{\mathcal{O}}$. We may assume ${\mathcal{K}}_{{\mathcal{O}}}\neq\emptyset$.
Let $\pi\in{\mathcal{K}}_{{\mathcal{O}}}$. We first show that there
exists $\phi_{\pi},\phi_{\pi}^{\prime}\in\omega_{V,\psi,\mu,\chi}$
and a section $T$ in ${\mathcal{C}}\left({\mathcal{X}}_{\text{temp}}\left(G,\chi^{-1}\right),{\mathcal{E}}\left(G,\chi^{-1}\right)\right)$
such that ${\mathcal{L}}_{\pi^\prime}^{\phi_{\pi},\phi_{\pi}^{\prime}}\left(T_{\pi^\prime}\right)=m_{V}\left(\bar{\pi}^\prime\right)$, for $\pi^\prime$ in a neighborhood of $\pi$ in ${\mathcal{O}}$. By Theorem \ref{5.1} and Proposition \ref{5.4}, together with $m_{V}\left(\bar{\pi}\right)\leq1$
for any irreducible representation $\pi^\prime\in{\mathcal{K}}$, it follows
that the function $\pi^\prime\in{\mathcal{K}}\mapsto m_{V}\left(\bar{\pi}^\prime\right)$
is locally constant. If $m_{V}\left(\bar{\pi}\right)=0$,
we only need to take $T=0$. If $m_{V}\left(\bar{\pi}\right)\neq0$, by Theorem \ref{5.1}, there exists $\phi_{\pi},\phi_{\pi}^{\prime}\in\omega_{V,\psi,\mu,\chi}$
and $T_{0}\in\text{End}\left(\pi\right)$ such that ${\mathcal{L}}_{\pi}^{\phi_{\pi},\phi_{\pi}^{\prime}}\left(T_{0}\right)\neq0$.
We choose $T^{0}\in{\mathcal{C}}\left({\mathcal{X}}_{\text{temp}}\left(G,\chi^{-1}\right),{\mathcal{E}}\left(G,\chi^{-1}\right)\right)$
such that $T_{\pi}^{0}=T_{0}$. Since the function $\pi^\prime\in{\mathcal{X}}_{\text{temp}}\left(G,\chi^{-1}\right)\mapsto{\mathcal{L}}_{\pi^\prime}^{\phi_{\pi},\phi_{\pi}^{\prime}}\left(T_{\pi^\prime}^{0}\right)$
is smooth, we can find a smooth and compactly supported function $\lambda$
on ${\mathcal{X}}_{\text{temp}}\left(G,\chi^{-1}\right)$ such that $\lambda\left(\pi^\prime\right){\mathcal{L}}_{\pi^\prime}^{\phi_{\pi},\phi_{\pi}^{\prime}}\left(T_{\pi^\prime}^{0}\right)=1=m_V(\bar{\pi}^\prime)$ for $\pi^\prime$ in a neighborhood of $\pi$. We set $T=\lambda T^{0}$.
This gives us 
$$
{\mathcal{L}}_{\pi^\prime}^{\phi_{\pi},\phi_{\pi}^{\prime}}\left(T_{\pi^\prime}\right)=m_V(\bar{\pi}^\prime),
$$
for any $\pi^\prime$ in some neighborhood of $\pi$ in ${\mathcal{O}}$.

Next, by Proposition \ref{5.4} and a partition of unity, we can choose a section $T^{{\mathcal{O}}}$
such that ${\mathcal{L}}_{\pi^\prime}^{\phi_{\pi},\phi_{\pi}^{\prime}}\left(T_{\pi^\prime}^{{\mathcal{O}}}\right)=m_{V}\left(\bar{\pi}^\prime\right)$, for any $\pi^\prime$ in ${\mathcal{K}_{O}}$, noting that Proposition \ref{5.4} gives us a common pair $(\phi_{\pi},\phi_{\pi}^{\prime})$ in the whole component $\mathcal{O}$. For simplicity of notations, we set $(\phi_\mathcal{O},\phi_\mathcal{O}^\prime)=(\phi_{\pi},\phi_{\pi}^{\prime})$. 

Since ${\mathcal{K}}$ is compact, there are finitely many connected component ${\mathcal{O}}_{1},\ldots,{\mathcal{O}}_{n}$
whose intersections with ${\mathcal{K}}$ are nonempty. We prove the final statement inductively over the number of nonzero connected components. The case $n=1$ has been shown in the previous paragraph. For simplicity, we consider $n=2$, as larger cases are similar. We can choose $(\phi_{{\mathcal{O}}_{1}},\phi_{{\mathcal{O}}_{1}}^{\prime})$ and $(\phi_{{\mathcal{O}}_{2}},\phi_{{\mathcal{O}}_{2}}^{\prime})$ in $\omega_{V,\psi,\mu,\chi}$ and $T^{{\mathcal{O}}_{1}},T^{{\mathcal{O}}_{2}}\in{\mathcal{C}}\left({\mathcal{X}}_{\text{temp}}\left(G,\chi^{-1}\right),{\mathcal{E}}\left(G,\chi^{-1}\right)\right)$
such that 
$$
{\mathcal{L}}_{\pi}^{\phi_{{\mathcal{O}}_{1}},\phi_{{\mathcal{O}}_{1}}^{\prime}}\left(T_{\pi}^{{\mathcal{O}}_{1}}\right)=m_{V}\left(\bar{\pi}\right),\, \text{ for any } \pi\in\mathcal{K}_{\mathcal{O}_1}
$$
and
$$
{\mathcal{L}}_{\pi}^{\phi_{{\mathcal{O}}_{2}},\phi_{{\mathcal{O}}_{2}}^{\prime}}\left(T_{\pi}^{{\mathcal{O}}_{2}}\right)=m_{V}\left(\bar{\pi}\right),\,\text{for any } \pi \in {\mathcal{K}}_{{\mathcal{O}}_{2}}.
$$
If either $m_V(\bar{\pi})=0$ for all $\pi\in\mathcal{K}_{\mathcal{O}_1}$ or $\pi\in\mathcal{K}_{\mathcal{O}_2}$, we can extend $T^{\mathcal{O}_2}$ or $T^{\mathcal{O}_1}$ by zero to the other component respectively and choose it as our new $T$. Assume this is not the case. By our assumption on the multiplicity and since the function $\pi\in{\mathcal{K}}\mapsto m_{V}\left(\bar{\pi}\right)$
is locally constant, we have $m_V(\bar{\pi})=1$, for any $\pi\in \mathcal{K}_{\mathcal{O}_1}\sqcup\mathcal{K}_{\mathcal{O}_2}$. Since we can always normalize our section $T$ by multiplying with a smooth function $\lambda$ on 
$\mathcal{X}_\text{temp}(G,\chi^{-1})$, it suffices to show that there exists $\phi,\phi^\prime\in \omega_{V,\psi,\mu,\chi}$ such that
$
{\mathcal{L}}_{\pi}^{\phi,\phi^{\prime}}\neq0$ for any $\pi \in\mathcal{K}_{\mathcal{O}_1}\sqcup\mathcal{K}_{\mathcal{O}_2}$. If  ${\mathcal{L}}_{\pi}^{\phi_{{\mathcal{O}}_{1}},\phi_{{\mathcal{O}}_{2}}^{\prime}}={\mathcal{L}}_{\pi}^{\phi_{{\mathcal{O}}_{2}},\phi_{{\mathcal{O}}_{1}}^{\prime}}={\mathcal{L}}_{\pi}^{\phi_{{\mathcal{O}}_{2}},\phi_{{\mathcal{O}}_{2}}^{\prime}}=0$, for any $\pi\in{\mathcal{K}}_{{\mathcal{O}}_{1}}$, and ${\mathcal{L}}_{\pi}^{\phi_{{\mathcal{O}}_{1}},\phi_{{\mathcal{O}}_{2}}^{\prime}}={\mathcal{L}}_{\pi}^{\phi_{{\mathcal{O}}_{2}},\phi_{{\mathcal{O}}_{1}}^{\prime}}={\mathcal{L}}_{\pi}^{\phi_{{\mathcal{O}}_{1}},\phi_{{\mathcal{O}}_{1}}^{\prime}}=0
$, for any $\pi\in{\mathcal{K}}_{{\mathcal{O}}_{2}}$, we set $\phi=\phi_{{\mathcal{O}}_{1}}+\phi_{{\mathcal{O}}_{2}}$
and $\phi^{\prime}=\phi_{{\mathcal{O}}_{1}}^{\prime}+\phi_{{\mathcal{O}}_{2}}^{\prime}$. A direct computation shows
that 
$$
{\mathcal{L}}_{\pi}^{\phi,\phi^{\prime}}={\mathcal{L}}_{\pi}^{\phi_{{\mathcal{O}}_{1}},\phi_{{\mathcal{O}}_{1}}^{\prime}}\neq 0,\text{ for all }\pi\in{\mathcal{K}}_{{\mathcal{O}}_{1}}
$$
and
$$
{\mathcal{L}}_{\pi}^{\phi,\phi^{\prime}}={\mathcal{L}}_{\pi}^{\phi_{{\mathcal{O}}_{2}},\phi_{{\mathcal{O}}_{2}}^{\prime}}\neq 0,\text{ for all }\pi\in{\mathcal{K}}_{{\mathcal{O}}_{2}}.
$$
If ${\mathcal{L}}_{\pi}^{\phi_{{\mathcal{O}}_{2}},\phi_{{\mathcal{O}}_{2}}^{\prime}}\neq0$
for some $\pi\in{\mathcal{K}}_{{\mathcal{O}}_{1}}$, then by Proposition \ref{5.4}, we have 
${\mathcal{L}}_{\pi}^{\phi_{{\mathcal{O}}_{2}},\phi_{{\mathcal{O}}_{2}}^{\prime}}\neq0$
for any $\pi\in{\mathcal{K}}_{{\mathcal{O}}_{1}}$. In this case, we take $(\phi,\phi^\prime)=(\phi_{{\mathcal{O}}_{2}},\phi_{{\mathcal{O}}_{2}}^{\prime})$. A similar argument works when ${\mathcal{L}}_{\pi}^{\phi_{{\mathcal{O}}_{1}},\phi_{{\mathcal{O}}_{1}}^{\prime}}\neq0$
for some $\pi$ in ${\mathcal{K}}_{{\mathcal{O}}_{2}}$. 

We now consider the case when ${\mathcal{L}}_{\pi}^{\phi_{{\mathcal{O}}_{1}},\phi_{{\mathcal{O}}_{2}}^{\prime}}\neq0$
for some $\pi\in{\mathcal{K}}_{{\mathcal{O}}_{2}}$. The remaining cases are similar to this case. By Proposition \ref{5.4}, we have 
${\mathcal{L}}_{\pi}^{\phi_{{\mathcal{O}}_{1}},\phi_{{\mathcal{O}}_{2}}^{\prime}}\neq0$
for any $\pi\in{\mathcal{K}}_{{\mathcal{O}}_{2}}$. If ${\mathcal{L}}_{\pi}^{\phi_{{\mathcal{O}}_{1}},\phi_{{\mathcal{O}}_{2}}^{\prime}}\neq 0$ for some $\pi\in \mathcal{K}_{\mathcal{O}_1}$, we choose $(\phi,\phi^\prime)=(\phi_{{\mathcal{O}}_{1}},\phi_{{\mathcal{O}}_{2}}^{\prime})$. Again, by Proposition \ref{5.4}, we have ${\mathcal{L}}_{\pi}^{\phi,\phi^{\prime}}={\mathcal{L}}_{\pi}^{\phi_{{\mathcal{O}}_{1}},\phi_{{\mathcal{O}}_{2}}^{\prime}}\neq 0$, for all $\pi \in \mathcal{K}_{\mathcal{O}_1}$. Else, we choose $(\phi,\phi^\prime)=(\phi_{{\mathcal{O}}_{1}},\phi_{{\mathcal{O}}_{1}}^\prime+\phi_{{\mathcal{O}}_{2}}^{\prime})$.
In this case, we have
$$
{\mathcal{L}}_{\pi}^{\phi,\phi^{\prime}}={\mathcal{L}}_{\pi}^{\phi_{{\mathcal{O}}_{1}},\phi_{{\mathcal{O}}_{2}}^{\prime}}\neq 0,
\,
\text{ for any } \pi \in \mathcal{K}_{\mathcal{O}_2}
$$
and
$$
{\mathcal{L}}_{\pi}^{\phi,\phi^{\prime}}={\mathcal{L}}_{\pi}^{\phi_{{\mathcal{O}}_{1}},\phi_{{\mathcal{O}}_{1}}^{\prime}}\neq 0,
\,
\text{ for any } \pi \in \mathcal{K}_{\mathcal{O}_1}.
$$
Therefore, we can always choose $\phi$ and $\phi^\prime$ such that ${\mathcal{L}}_{\pi}^{\phi,\phi^{\prime}}\neq0$, for any $\pi\in {\mathcal{K}}_{{\mathcal{O}}_{1}}\sqcup{\mathcal{K}}_{{\mathcal{O}}_{2}}$. We have finished our proof.
\end{proof}

\section{The distribution $J_V$ and a spectral expansion}\label{sec6}

\subsection{The distributions $J_{V}$ and $J_{V,\text{spec}}$}\label{sec6.1}

For all $f\in{\mathcal{C}}\left(Z_{G}\left(F\right)\backslash G\left(F\right),\chi\right)$,
let us define a kernel function $K\left(f,\cdot\right)$ by 
$$
K\left(f,x\right)=\underset{i}{\sum}\int_{Z_{H_{V}}\left(F\right)\backslash H_{V}\left(F\right)}f\left(x^{-1}hx\right)\left\langle \phi_{i},\omega_{V,\psi,\mu,\chi}\left(h\right)\phi_{i}\right\rangle dh,
$$
where $x\in Z_{G}\left(F\right)H\left(F\right)\backslash G\left(F\right)$
and $\left\{ \phi_{i}\right\} _{i\in I}$ is an orthonormal basis
of $\omega_{V,\psi,\mu,\chi}$. Note that the above definition is independent of the choice of the orthonormal basis. Let 
$$
J_{V}\left(f\right)=\int_{Z_{G}\left(F\right)H_{V}\left(F\right)\backslash G\left(F\right)}K\left(f,x\right)dx.
$$
We set 
$$
J_{V,\text{spec}}\left(f\right)=\int_{{\mathcal{X}}\left(G,\chi^{-1}\right)}\hat{\theta}_{f}\left(\pi\right)m_{V}\left(\bar{\pi}\right)d\pi,
$$
where $f\in{\mathcal{C}}_{\text{scusp}}\left(Z_{G}\left(F\right)\backslash G\left(F\right),\chi\right)$.
By Lemma 3.4, this integral is absolutely convergent. This section
is devoted to prove the following theorem.
\begin{theorem}\label{6.1}
For any $f\in{\mathcal{C}}_{\text{scusp}}\left(Z_{G}\left(F\right)\backslash G\left(F\right),\chi\right)$, we have the following spectral expansion 
$$
J_{V}\left(f\right)=J_{V,\text{spec}}\left(f\right),
$$
here ${\mathcal{X}}\left(G\right)$ consists of virtual representations
of the form $i_{M}^{G}\left(\sigma\right)$, where $M$ is a Levi subgroup
of $G$ and $\sigma$ is an elliptic representation of $M(F)$.
\end{theorem}

We first prove that the integral defining $J_V(f)$ is absolutely convergent.
\begin{theorem}\label{6.2}
Let $f\in{\mathcal{C}}\left(Z_{G}\left(F\right)\backslash G\left(F\right),\chi\right)$. Then 
    \begin{enumerate}
        \item The linear form $K(f,x)$ converges. Moreover, if $f\in{\mathcal{C}}_{scusp}\left(Z_{G}\left(F\right)\backslash G\left(F\right),\chi\right)$, then for all $d>0$, we have
        $$
        |K(f,x)| \ll \Xi^X(x)^2 \sigma_X(x)^{-d},
        $$
        for any $x\in Z_{G}\left(F\right)H(F)\backslash G\left(F\right)$.
        \item If $f\in{\mathcal{C}}_{scusp}\left(Z_{G}\left(F\right)\backslash G\left(F\right),\chi\right)$, the integral defining $J_V(f)$ is absolutely convergent.
    \end{enumerate}
\end{theorem}
\begin{proof}
    It suffices to prove the first part, as the second part follows from it and Proposition \ref{4.6}(iii). Let $K$ be a compact open subgroup of $G(F)$ such that $f$ is $K$-biinvariant. We set $K_H=x^{-1}Kx\cap H(F)$. For simplicity of notations, we also denote by ${\phi_i}_{i\in I}$ an orthonormal basis for $\omega_{V,\psi,\mu,\chi}^{K_H}$. Since $\omega_{V,\psi,\mu,\chi}$ is admissible, we have $I$ is of finite cardinality. By Lemma \ref{5.2}, we have
    $$
    \int_{Z_H(F)\backslash H(F)}f(x^{-1}hx)|\langle \phi_i, \omega_{V,\psi,\mu,\chi}(h)\phi_i \rangle|dh <\infty,
    $$
    for any $i\in I$. Since $I$ is of finite cardinality, we can see that $K(f,x)$ converges. There remains to show that for all $d>0$, we have
    $$
    |K(f,x)| \ll \Xi^X(x)^2 \sigma_X(x)^{-d},
    $$
    for any $x\in Z_{G}\left(F\right)\backslash G\left(F\right)$. We follow closely the proof of \cite[Theorem 2.1.1]{BP18}. Let $A_0$ be a maximal $(\theta,F)$-split subtorus of $G$ and $M_0=\text{Cent}_G(A_0)$ and $P_0=M_0U_0\in \mathcal{P}^\theta(M_0)$. Let $\bar{P}_0=\theta(P_0)$ be the opposite parabolic subgroup and set
    $$
    A^+_0= \{a\in A_0(F):\ |\alpha(a)|\geq 1,\,\forall \alpha \in \Delta(A_0,\bar{P}_0)\}.
    $$
    By the weak Cartan decomposition mentioned in Section \ref{sec4.4} and Proposition \ref{4.6}(i), it suffices to prove the statement for $x=a\in A^+_0$. For all $Q\in \mathcal{F}^\theta(M_0)$ and $\delta>0$, we set
    $$
    A^+_{\bar{Q}}(\delta)=\{a\in A_0(F):\ |\alpha(a)|\geq e^{\delta \bar{\sigma}(a)},\,\forall \alpha \in \Delta(A_0,U_{\bar{Q}})\},
    $$
    where $\bar{Q}=\theta(Q)$ and $U_{\bar{Q}}$ is its unipotent radical. If $\delta$ is sufficiently small, we have
    $$
    A^+_0=\bigcup_{Q\in\mathcal{F}^\theta(M_0)-{G},\,P_0\subset Q} A^+_{\bar{Q}}(\delta)\cap A^+_0.
    $$
    We now fix $Q\in\mathcal{F}^\theta(M_0)-{G}$ such that $P_0\subset Q$ and $\delta>0$ sufficiently small. It suffices to prove the statement for $x=a\in A^+_{\bar{Q}}(\delta)\cap A^+_0$. Let $U_Q$ be the unipotent radical of $Q$ and $L=Q\cap \bar{Q}$ and $H_L=H\cap L$ and $H^Q= H_L \ltimes U_Q$. We define a representation $\omega_{V,\psi,\mu,\chi}^Q$ of $H^Q(F)$ by setting $\omega_{V,\psi,\mu,\chi}^Q(h_Lu_Q)=\omega_{V,\psi,\mu,\chi}(h_L)$, for all $h_L\in H_L(F)$ and $u_Q\in U_Q(F)$. By \cite[Proposition 2.3.1(vi)]{BP18}, $H^Q$ is unimodular. We fix a Haar measure $dh^Q$ on $H^Q(F)$ and define
    $$
    K^Q(f,x)=\sum_{i\in I}\int_{Z_H(F)\backslash H^Q(F)}f(x^{-1}h^Qx)\langle \phi_i,\omega_{V,\psi,\mu,\chi}^Q(h^Q)\phi_i\rangle dh^Q,
    $$
    for $f\in \mathcal{C}(Z_G(F)\backslash G(F),\chi)$ and $x\in G(F)$. By mimicing the proof of Lemma \ref{5.2} and noting that the restriction of the Weil representation $\omega_{V,\psi,\mu}$ to $H_L(F)$ is also a Weil representation, we can show that $K^Q(f,x)$ converges. Since $U_Q\subset H^Q \subset Q$, when $f$ is strongly cuspidal, we have $K^Q(f,x)=0$.
    
    Hence, it suffices to show there exists $c>0$ such that for any $f\in \mathcal{C}(Z_G(F)\backslash G(F),\chi)$ and $d>0$, we have
    $$
    \left| K(f,a)-cK^Q(f,a)\right| \ll \Xi^X(a)^2\sigma_X(a)^{-d},
    $$
    for all $a\in A^+_{\bar{Q}}(\delta)\cap A^+_0$. We recall the setting in \cite[Theorem 2.1.1]{BP18}. We fix $f\in \mathcal{C}\left(Z_G(F)\backslash G(F),\chi \right)$ and $d>0$. Since the $F$-analytic map $h=\bar{p}u\mapsto u:H(F)\cap \bar{P}_0(F)U_0(F)\rightarrow U_0(F)$ is submersive at the origin, we can find a compact open neighborhood $\mathcal{U}_0$ of $1$ in $U_0(F)$ and an $F$-analytic map $h:\mathcal{U}_0\rightarrow H(F)$ such that $h(u)\in \bar{P}_0(F)u$ for all $u\in \mathcal{U}_0$ and $h(1)=1$. We define $\mathcal{U}_Q=\mathcal{U}_0\cap U_Q(F)$ and $\mathcal{H}=H_L(F)h(\mathcal{U}_Q)$ and fix Haar measures $dh_L$ and $du_Q$ on $H_L(F)$ and $U_Q(F)$ such that their product and the fixed Haar measure $dh^Q$ on $H^Q(F)$ coincide. By \cite[(2.4.4)]{BP18}, the map
    $$
    H_L(F)\times \mathcal{U}_Q\rightarrow H(F):(h_L,u_Q)\mapsto h_Lh(u_Q)
    $$
    is an $F$-analytic open embeding with image $\mathcal{H}$ and there exists a smooth function $j\in C^\infty(\mathcal{U}_Q)$ such that
    $$
    \int_{\mathcal{H}}\varphi(h)dh=\int_{H_L(F)}\int_{\mathcal{U}_Q}\varphi(h_Lh(u_Q))j(u_Q)du_Qdh_L
    $$
    for all $\varphi\in L^1(\mathcal{H})$. We fix sufficiently small $\epsilon>0$ such that
    $aU_Q[<\epsilon\bar{\sigma}(a)]a^{-1}\subseteq \mathcal{U}_Q$, for all $a\in A^+_{\bar{Q}}(\delta)$. We set
    $$
    H^{<\epsilon,a}=H_L(F)h(aU_Q[<\epsilon\bar{\sigma}(a)]a^{-1})\ \ \text{ and }\ \ 
    H^{Q,<\epsilon,a}=H_L(F)aU_Q[<\epsilon\bar{\sigma}(a)]a^{-1}
    $$
    and 
    $$
    K^{<\epsilon}(f,a)=\sum_{i\in I}\int_{Z_H(F)\backslash H^{<\epsilon,a}}f(a^{-1}ha)\langle\phi_i,\omega_{V,\psi,\mu,\chi}(h)\phi_i\rangle dh,
    $$
    $$
    K^{Q,<\epsilon}(f,a)=\sum_{i\in I}\int_{Z_H(F)\backslash H^{Q,<\epsilon,a}}f(a^{-1}h^Qa)\langle\phi_i,\omega^Q_{V,\psi,\mu,\chi}(h^Q)\phi_i\rangle dh^Q,
    $$
    for $a\in A^+_{\bar{Q}}(\delta)$. We take $c=j(1)$. We want to show that for $\epsilon$ sufficiently small, we have
    \begin{equation}\label{eqn5.1}
        |K(f,a)-K^{<\epsilon}(f,a)| \ll \Xi^X(a)^2\sigma_X(a)^{-d}
    \end{equation}
    and
    \begin{equation}\label{eqn5.2}
        |K^Q(f,a)-K^{Q,<\epsilon}(f,a)| \ll \Xi^X(a)^2\sigma_X(a)^{-d}
    \end{equation}
    and
    \begin{equation}\label{eqn5.3}
        K^{<\epsilon}(f,a)=cK^{Q,<\epsilon}(f,a)
    \end{equation}
    for all $a\in A^+_{\bar{Q}}(\delta) \cap A^+_0$. We first prove (\ref{eqn5.1}). By \cite[Lemma 2.2.1(i)-(iii)]{BP18} and since $H^{<\epsilon,a}=H(F)\cap \bar{Q}(F)aU_Q[<\epsilon\bar{\sigma}(a)]a^{-1}$, we have $\bar{\sigma}(a)\ll \bar{\sigma}(a^{-1}ha$ for all $a\in A^+_{Q}(\delta)$ and $h\in H(F) \backslash H^{<\epsilon,a}$. Hence, for any $d_1>0$ and $d_2>0$, the left hand side of $(6.1)$ is essentially bounded by
    $$
    \bar{\sigma}(a)^{-d_1}\int_{Z_H(F)\backslash H(F)}\Xi^G(a^{-1}ha)\bar{\sigma}(a^{-1}ha)^{-d_2}|\Theta_{\omega_{V,\psi,\mu,\chi}}(h)|dh,
    $$
    for all $a\in A^+_{\bar{Q}}(\delta)$. We want to show there exists $d_3>0$ such that
    $$
     \bar{\sigma}(a)^{-d_1}\int_{Z_H(F)\backslash H(F)}\Xi^G(a^{-1}ha)\bar{\sigma}(a^{-1}ha)^{-d_2}|\Theta_{\omega_{V,\psi,\mu,\chi}}(h)|dh \ll \Xi^X(a)^2\sigma_X(a)^{d_3}\bar{\sigma}(a)^{-d_1},
    $$
    for all $a\in A^+_{\bar{Q}}(\delta)$. By part (i) and (ii) in Proposition \ref{4.6}, it suffices to give the existence of $d_4, d_5>0$ such that
    $$
    \int_{Z_H(F)\backslash H(F)}\Xi^G(a^{-1}ha)\bar{\sigma}(h)^{-d_4}|\Theta_{\omega_{V,\psi,\mu,\chi}}(h)|dh \ll \Xi^G(a)^2\bar{\sigma}(a)^{d_5},
    $$
    for all $a\in A^+_0$. Let $K_G$ be a maximal compact subgroup of $G(F)$. We have
    $$
    \int_{Z_H(F)\backslash H(F)}\Xi^G(a^{-1}ha)\bar{\sigma}(h)^{-d_4}|\Theta_{\omega_{V,\psi,\mu,\chi}}(h)|dh
    $$
    $$
    \ll \int_{Z_H(F)\backslash H(F)}\int_{K\times K}\Xi^G(a^{-1}k_1hk_2a)dk_1dk_2\bar{\sigma}(h)^{-d_4}|\Theta_{\omega_{V,\psi,\mu,\chi}}(h)|dh
    $$
    $$
    =\Xi^G(a)^2\int_{Z_H(F)\backslash H(F)}\Xi^G(h)\bar{\sigma}(h)^{-d_4}|\Theta_{\omega_{V,\psi,\mu,\chi}}(h)|dh,
    $$
    noting that the last equality follows from the doubling principle in \cite[Proposition 1.5.1(vi)]{BP20}. It suffices to show there exists $d_4>0$ such that
    $$
    \int_{Z_H(F)\backslash H(F)}\Xi^G(h)\bar{\sigma}(h)^{-d_4}|\Theta_{\omega_{V,\psi,\mu,\chi}}(h)|dh <\infty.
    $$
    By Corollary \ref{4.3}, there exists $d_5>0$ such that we have 
    $$
    \int_{Z_H(F)\backslash H(F)}\Xi^G(h)\bar{\sigma}(h)^{-d_4}|\Theta_{\omega_{V,\psi,\mu,\chi}}(h)|dh
    \ll \int_{Z_H(F)\backslash H(F)}|D^H(h)|^{-1/2}\Xi^H(h)^{\frac{n}{n-1}}\bar{\sigma}(h)^{-d_4+d_5}dh.
    $$
    Since the above integral converges for some suitable choice of $d_4$, this ends the proof for (\ref{eqn5.1}). The proof of (\ref{eqn5.2}) follows the same way to one of (\ref{eqn5.1}). The proof of (\ref{eqn5.3}) follows closely to the proof of \cite[(2.4.7)]{BP18}, noting that here we only need to replace the character $\chi(h)$ in \cite[(2.4.7)]{BP18} by $\langle \phi_i,\omega_{V,\psi,\mu,\chi}\phi_i \rangle$ and prove the statement for every $i\in I$.
\end{proof}

\subsection{An auxiliary distribution and some estimates}\label{sec6.2}

In this subsection, we mimic some estimates in Section 9.2 in \cite{BP20}.
We fix a strongly cuspidal function $f\in{\mathcal{C}}_{\text{scusp}}\left(Z_{G}\left(F\right)\backslash G\left(F\right),\chi\right)$.
For any $f^{\prime}\in{\mathcal{C}}\left(Z_{G}\left(F\right)\backslash G\left(F\right),\chi^{-1}\right)$,
we define the following integrals 
$$
K_{f,f^{\prime}}^{A}\left(g_{1},g_{2}\right)=\int_{Z_{G}\left(F\right)\backslash G\left(F\right)}f\left(g_{1}^{-1}gg_{2}\right)f^{\prime}\left(g\right)dg,\ \ g_{1},g_{2}\in G\left(F\right)
$$
$$
K_{f,f^{\prime},\phi,\phi^{\prime}}^{1}\left(g,x\right)=\int_{Z_{H}\left(F\right)\backslash H\left(F\right)}K_{f,f^{\prime}}^{A}\left(g,hx\right)\left\langle \phi,\omega_{V,\psi,\mu,\chi}\left(h\right)\phi^{\prime}\right\rangle dh,\ \ g,x\in G\left(F\right)
$$
$$
K_{f,f^{\prime},\phi,\phi^{\prime}}^{2}\left(x,y\right)=\int_{Z_{H}\left(F\right)\backslash H\left(F\right)}K_{f,f^{\prime},\phi,\phi^{\prime}}^{1}\left(h^{-1}x,h^{-1}y\right)dh,\ \ x,y\in G\left(F\right)
$$
$$
J_{\text{aux}}\left(f,f^{\prime},\phi,\phi^{\prime}\right)=\int_{Z_{G}\left(F\right)H\left(F\right)\backslash G\left(F\right)}K_{f,f^{\prime},\phi,\phi^{\prime}}^{2}\left(x,x\right)dx.
$$

\begin{proposition}\label{6.3}
We fix a strongly cuspidal function $f\in{\mathcal{C}}_{\text{scusp}}\left(Z_{G}\left(F\right)\backslash G\left(F\right),\chi\right)$.
\begin{enumerate}
\item \label{5.3.1}The integral defining $K_{f,f^{\prime}}^{A}\left(g_{1},g_{2}\right)$
is absolutely convergent. For all $g_{1}\in G\left(F\right)$, the
map 
$$
g_{2}\in G\left(F\right)\mapsto K_{f,f^{\prime}}^{A}\left(g_{1},g_{2}\right)
$$
belongs to ${\mathcal{C}}\left(Z_{G}\left(F\right)\backslash G\left(F\right),\chi\right)$.
Moreover, for any $d>0$, there exists $d^{\prime}>0$ such that for
every continuous semi-norm $\nu$ on ${\mathcal{C}}_{d^{\prime}}^{w}\left(Z_{G}\left(F\right)\backslash G\left(F\right),\chi\right)$,
there exists a continuous semi-norm $\mu$ on ${\mathcal{C}}\left(G\left(F\right)\right)$
satisfying $\nu\left(K_{f,f^{\prime}}^{A}\left(g,\cdot\right)\right)\leq\mu\left(f^{\prime}\right)\Xi^{G}\left(g\right)\sigma\left(g\right)^{-d}$, for all $f^{\prime}\in{\mathcal{C}}\left(Z_{G}\left(F\right)\backslash G\left(F\right),\chi^{-1}\right)$ and all $g\in G\left(F\right)$.
\item \label{5.3.2}The integral defining $K_{f,f^{\prime},\phi,\phi^{\prime}}^{1}$ is
absolutely convergent. Moreover, for all $d>0$ there exists $d^{\prime}>0$
and a continuous semi-norm $\nu_{d,d^{\prime}}$ on ${\mathcal{C}}\left(Z_G(F)\backslash G\left(F\right),\chi^{-1}\right)$
such that 
$$
\left|K_{f,f^{\prime},\phi,\phi^{\prime}}^{1}\left(g,x\right)\right|\leq\nu_{d,d^{\prime}}\left(f^{\prime}\right)\Xi^{G}\left(g\right)\sigma\left(g\right)^{-d}\Xi^{X}\left(x\right)\sigma_{X}\left(x\right)^{d^{\prime}}
$$
 for all $f^{\prime}\in{\mathcal{C}}\left(Z_{G}\left(F\right)\backslash G\left(F\right),\chi^{-1}\right)$
and all $g,x\in G\left(F\right)$.
\item \label{5.3.3}The integral defining $K_{f,f^{\prime},\phi,\phi^{\prime}}^{2}\left(x,y\right)$
is absolutely convergent. More precisely, for every $d>0$ there exists
a continuous semi-norm $\nu_{d}$ on ${\mathcal{C}}\left(Z_G(F)\backslash G\left(F\right),\chi^{-1}\right)$
such that 
$$
\left|K_{f,f^{\prime},\phi,\phi^{\prime}}^{2}\left(x,x\right)\right|\leq\nu_{d}\left(f^{\prime}\right)\Xi^{X}\left(x\right)^{2}\sigma_{X}\left(x\right)^{-d}
$$
for all $f^{\prime}\in{\mathcal{C}}\left(Z_{G}\left(F\right)\backslash G\left(F\right),\chi^{-1}\right)$
and all $x\in Z_{G}\left(F\right)H\left(F\right)\backslash G\left(F\right)$.
\item \label{5.3.4}The integral defining $J_{\text{aux}}\left(f,f^{\prime}\right)$ is
absolutely convergent. Moreover, the linear form 
$$
f^{\prime}\in{\mathcal{C}}\left(Z_{G}\left(F\right)\backslash G\left(F\right),\chi^{-1}\right)\mapsto J_{\text{aux}}\left(f,f^{\prime}\right)
$$
is continuous.
\end{enumerate}
\end{proposition}

\begin{proof}
Part (\ref{5.3.1}) follows from \cite[Theorem 5.5.1(i)]{BP20}. We prove (\ref{5.3.2}). The first statement follows from (\ref{5.3.1}) and Lemma \ref{5.2}. By Proposition \ref{4.6} (part (i) and (ii)) and the weak Cartan decomposition, it suffices to prove the following statement: for any $d>0$, we have 
$$
\int_{Z_{H}\left(F\right)\backslash H\left(F\right)}\Xi^{G}\left(ha\right)\sigma\left(h\right)^{d}\left\langle \phi,\omega_{V,\psi,\mu,\chi}\left(h\right)\phi^{\prime}\right\rangle dh\ll\Xi^{G}\left(a\right),\ \forall a\in A_{P_{0}}^{+}.
$$
We follow the method in \cite[Proposition 2.3.1(iv)]{BP18}. If $K$
is a maximal compact subgroup of $G\left(F\right)$ by which $\Xi^{G}$
is right invariant, there exists compact open subgroups $J\subset G\left(F\right)$
and $J_{H}\subset H\left(F\right)$ such that $J\subset J_{H}aKa^{-1}$,
for all $a\in A_{P_{0}}^{+}$. Hence, for all $d>0$ and $k\in J$
and $a\in A_{P_{0}}^{+}$, by writing $k=k_{H}ak_{G}a^{-1}$ with
$k_{H}\in J_{H}$ and $k_{G}\in K$, we have 
$$
\int_{Z_{H}\left(F\right)\backslash H\left(F\right)}\Xi^{G}\left(hka\right)\sigma\left(h\right)^{d}\left\langle \phi,\omega_{V,\psi,\mu,\chi}\left(h\right)\phi^{\prime}\right\rangle dh
$$
$$
=\int_{Z_{H}\left(F\right)\backslash H\left(F\right)}\Xi^{G}\left(ha\right)\sigma\left(hk_{H}^{-1}\right)^{d}\left\langle \phi,\omega_{V,\psi,\mu,\chi}\left(hk_{H}^{-1}\right)\phi^{\prime}\right\rangle dh.
$$
Hence, for all $d>0$, we have 
$$
\int_{Z_{H}\left(F\right)\backslash H\left(F\right)}\Xi^{G}\left(ha\right)\sigma\left(h\right)^{d}\left\langle \phi,\omega_{V,\psi,\mu,\chi}\left(h\right)\phi^{\prime}\right\rangle dh
$$
$$
\ll\int_{Z_{H}\left(F\right)\backslash H\left(F\right)}\int_{K}\Xi^{G}\left(hka\right)dk\sigma\left(h\right)^{d}\left\langle \phi,\omega_{V,\psi,\mu,\chi}\left(h\right)\phi^{\prime}\right\rangle dh
$$
$$
=\Xi^{G}\left(a\right)\int_{Z_{H}\left(F\right)\backslash H\left(F\right)}\Xi^{G}\left(h\right)\sigma\left(h\right)^{d}\left\langle \phi,\omega_{V,\psi,\mu,\chi}\left(h\right)\phi^{\prime}\right\rangle dh\ll\Xi^{G}\left(a\right),
$$
noting that the last inequality follows from Lemma \ref{5.2}. This finishes the proof of (\ref{5.3.2}).

By part (\ref{5.3.2}) and Proposition \ref{4.6}(iv), the integral defining $K_{f,f^{\prime},\phi,\phi^{\prime}}^{2}\left(x,y\right)$
is absolutely convergent. Let $\left\{ \phi_{i}\right\} _{i\in I}$ be an orthonormal basis for $\omega_{V,\psi,\mu,\chi}$.
We have  
$$
K_{f,f^{\prime},\phi,\phi^{\prime}}^{2}\left(x,x\right)=\iint_{Z_{H}\left(F\right)\backslash H\left(F\right)}\int_{Z_{G}\left(F\right)\backslash G\left(F\right)}f\left(x^{-1}h_{1}gh_{2}h_{1}^{-1}x\right)f^{\prime}\left(g\right)dg\left\langle \phi,\omega_{V,\psi,\mu,\chi}\left(h_{2}\right)\phi^{\prime}\right\rangle dh_{2}dh_{1}
$$
$$
=\iint_{Z_{H}\left(F\right)\backslash H\left(F\right)}\int_{Z_{G}\left(F\right)\backslash G\left(F\right)}f\left(x^{-1}h_{1}gh_{2}x\right)f^{\prime}\left(g\right)dg\left\langle \phi,\omega_{V,\psi,\mu,\chi}\left(h_{2}h_{1}\right)\phi^{\prime}\right\rangle dh_{2}dh_{1}
$$
$$
=\underset{i\in I}{\sum}\iint_{Z_{H}\left(F\right)\backslash H\left(F\right)}\int_{Z_{G}\left(F\right)\backslash G\left(F\right)}f\left(x^{-1}h_{1}gh_{2}x\right)f^{\prime}\left(g\right)dg
\left\langle \phi,\omega_{V,\psi,\mu,\chi}\left(h_{2}\right)\phi_{i}\right\rangle \left\langle \phi_{i},\omega_{V,\psi,\mu,\chi}\left(h_{1}\right)\phi^{\prime}\right\rangle dh_{2}dh_{1}
$$
$$
=\underset{i\in I}{\sum}\int_{{\mathcal{X}}_{\text{temp}}\left(G,\chi^{-1}\right)}{\mathcal{L}}_{\pi}^{\phi_{i},\phi_{i}}\left(\pi\left(x\right)\pi\left(f\right)\pi\left(x\right)^{-1}\right){\mathcal{L}}_{\pi}^{\phi,\phi^{\prime}}\left(\pi\left(f^{\prime,\vee}\right)\right)\mu\left(\pi\right)d\pi,
$$
noting that the last line follows from Lemma \ref{5.3}(iv). For $f^{\prime}\in{\mathcal{C}}\left(G\left(F\right)\right)$ and $\phi,\phi^\prime\in\omega_{V,\psi,\mu,\chi}$,
the section 
$$
T\left(f^{\prime},\phi,\phi^\prime\right):\pi\in{\mathcal{X}}_{\text{temp}}\left(G\right)\mapsto{\mathcal{L}}_{\pi}^{\phi,\phi^{\prime}}\left(\pi\left(f^{\prime,\vee}\right)\right)\pi\left(f\right)\in\text{End}\left(\pi\right)^\infty
$$
is smooth and compactly supported, i.e. it belongs to ${\mathcal{C}}\left({\mathcal{X}}_{\text{temp}}\left(G\right),{\mathcal{E}}\left(G\right)\right)$.
By Theorem \ref{2.2} and \cite[Lemma 5.3.1(i)]{BP20}, there exists a unique strongly cuspidal function $\varphi_{f^{\prime}}\in{\mathcal{C}}_{\text{scusp}}\left(Z_G(F)\backslash G\left(F\right),\chi^{-1}\right)$
such that $\pi\left(\varphi_{f^{\prime}}\right)={\mathcal{L}}_{\pi}^{\phi,\phi^{\prime}}\left(\pi\left(f^{\prime,\vee}\right)\right)\pi\left(f\right),\ \forall\pi\in{\mathcal{X}}_{\text{temp}}\left(G\right)$. By Lemma \ref{5.3}(iii), it follows that 
$$
K_{f,f^{\prime},\phi,\phi^{\prime}}^{2}\left(x,x\right)=\underset{i\in I}{\sum}\int_{Z_{H}\left(F\right)\backslash H\left(F\right)}\varphi_{f^{\prime}}\left(x^{-1}hx\right)\left\langle \phi_{i},\omega_{V,\psi,\mu,\chi}\left(h\right)\phi_{i}\right\rangle dh=K(\varphi_{f^\prime},x).
$$
By Theorem \ref{6.2}, for all $d>0$, there exists a continuous seminorm $\nu_d$ on $\mathcal{C}(Z_G(F)\backslash G(F),\chi^{-1})$ such that $|K^2_{f,f^\prime,\phi,\phi^\prime}(x,x)|\leq \nu_d(f^\prime)\Xi^X(x)^2\sigma_X(x)^{-d}$, for all $f^\prime \in \mathcal{C}(Z_G(F)\backslash G(F),\chi^{-1})$ and $x\in Z_G(F)H(F)\backslash G(F)$. We need to show that the linear map $f^\prime \mapsto \varphi_{f^\prime}$ is continuous on $\mathcal{C}(Z_G(F)\backslash G(F),\chi^{-1})$, which follows from Theorem \ref{2.2} and the fact that
$$
f^\prime\in \mathcal{C}(Z_G(F)\backslash G(F),\chi^{-1}) \mapsto 
\left(\pi\in \mathcal{X}_\text{temp}(G,\chi) \mapsto \overline{\mathcal{L}_\pi^{\phi,\phi^\prime}(\pi(\bar{f}^\prime))}\right)\in C^\infty(\mathcal{X}_\text{temp}(G,\chi))
$$
is continuous. The last part follows from part (\ref{5.3.3}) and Proposition \ref{4.6}(iii).
\end{proof}
\begin{proposition}\label{6.4}
For $f\in{\mathcal{C}}_{\text{scusp}}\left(Z_{G}\left(F\right)\backslash G\left(F\right),\chi\right)$, we have  
$$
J_{\text{aux}}\left(f,f^{\prime},\phi,\phi^{\prime}\right)=\int_{{\mathcal{X}}\left(G,\chi^{-1}\right)}\hat{\theta}_{f}\left(\pi\right){\mathcal{L}}_{\pi}^{\phi,\phi^{\prime}}\left(\pi\left(f^{\prime,\vee}\right)\right)d\pi
$$
for any $f^{\prime}\in{\mathcal{C}}\left(Z_{G}\left(F\right)\backslash G\left(F\right),\chi^{-1}\right)$.
\end{proposition}

\begin{proof}
For $N>0$ and $M>0$, let us denote by $\alpha_{N}$ and $\beta_{M}$
to be characteristic functions of $$\left\{ x\in Z_{G}\left(F\right)H\left(F\right)\backslash G\left(F\right):\ \sigma_{Z_{G}H\backslash G}\left(x\right)\leq N\right\}\text{ and }\left\{ g\in Z_{G}\left(F\right)\backslash G\left(F\right):\ \sigma_{Z_{G}\backslash G}\left(g\right)\leq M\right\},$$
respectively. For all $N\geq1$ and $C>0$, we set 
$$
J_{\text{aux},N}\left(f,f^{\prime},\phi,\phi^{\prime}\right)=\int_{Z_{G}\left(F\right)H\left(F\right)\backslash G\left(F\right)}\alpha_{N}\left(x\right)\iint_{Z_{H}\left(F\right)\backslash H\left(F\right)}K_{f,f^{\prime}}^{A}\left(hx,h^{\prime}hx\right)
\left\langle \phi,\omega_{V,\psi,\mu,\chi}\left(h^{\prime}\right)\phi^{\prime}\right\rangle dh^{\prime}dhdx
$$
and 
$$
J_{\text{aux},N,C}\left(f,f^{\prime},\phi,\phi^{\prime}\right)=\int_{Z_{G}\left(F\right)H\left(F\right)\backslash G\left(F\right)}\alpha_{N}\left(x\right)\underset{Z_{H}\left(F\right)\backslash H\left(F\right)}{\iint}\beta_{C\log\left(N\right)}\left(h^{\prime}\right)K_{f,f^{\prime}}^{A}\left(hx,h^{\prime}hx\right)
$$
$$
\left\langle \phi,\omega_{V,\psi,\mu,\chi}\left(h^{\prime}\right)\phi^{\prime}\right\rangle dh^{\prime}dhdx.
$$
Observe $J_{\text{aux}}\left(f,f^{\prime},\phi,\phi^{\prime}\right)=\underset{N\rightarrow\infty}{\lim}J_{\text{aux},N}\left(f,f^{\prime},\phi,\phi^{\prime}\right)$. We prove the following claim.
\begin{claim}\label{6.5}
The triple integrals defining $J_{\text{aux},N}\left(f,f^{\prime},\phi,\phi^{\prime}\right)$
and $J_{\text{aux},N,C}\left(f,f^{\prime},\phi,\phi^{\prime}\right)$
are absolutely convergent and there exists $C>0$ such that 
$$
\left|J_{\text{aux},N}\left(f,f^{\prime},\phi,\phi^{\prime}\right)-J_{\text{aux},N,C}\left(f,f^{\prime},\phi,\phi^{\prime}\right)\right|\ll N^{-1}
$$
for all $N\geq1$.
\end{claim}

\begin{proof}
By \cite[Theorem 5.5.1(i)]{BP20}, there exists $d>0$ such that 
$$
K_{f,f^{\prime}}^{A}\left(hx,h^{\prime}hx\right)\ll\Xi^{G}\left(hx\right)\Xi^{G}\left(h^{\prime}hx\right)\sigma\left(hx\right)^{d}\sigma\left(h^{\prime}hx\right)^{d}.
$$
This gives us 
$$
\left|J_{\text{aux},N}\left(f,f^{\prime},\phi,\phi^{\prime}\right)\right|\ll\int_{Z_{G}\left(F\right)H\left(F\right)\backslash G\left(F\right)}\alpha_{N}\left(x\right)\underset{Z_{H}\left(F\right)\backslash H\left(F\right)}{\iint}\Xi^{G}\left(hx\right)\Xi^{G}\left(h^{\prime}hx\right)
$$
$$
\sigma\left(hx\right)^{d}\sigma\left(h^{\prime}hx\right)^{d}\left\langle \phi,\omega_{V,\psi,\mu,\chi}\left(h^{\prime}\right)\phi^{\prime}\right\rangle dh^{\prime}dhdx
$$
and 
$$
\left|J_{\text{aux},N,C}\left(f,f^{\prime},\phi,\phi^{\prime}\right)\right|\ll\int_{Z_{G}\left(F\right)H\left(F\right)\backslash G\left(F\right)}\alpha_{N}\left(x\right)\underset{Z_{H}\left(F\right)\backslash H\left(F\right)}{\iint}\beta_{C\log\left(N\right)}\left(h^{\prime}\right)\Xi^{G}\left(hx\right)
$$
$$
\Xi^{G}\left(h^{\prime}hx\right)\sigma\left(hx\right)^{d}\sigma\left(h^{\prime}hx\right)^{d}\left\langle \phi,\omega_{V,\psi,\mu,\chi}\left(h^{\prime}\right)\phi^{\prime}\right\rangle dh^{\prime}dhdx
$$
and 
$$
\left|J_{\text{aux},N}\left(f,f^{\prime},\phi,\phi^{\prime}\right)-J_{\text{aux},N,C}\left(f,f^{\prime},\phi,\phi^{\prime}\right)\right|\ll\int_{Z_{G}\left(F\right)H\left(F\right)\backslash G\left(F\right)}\alpha_{N}\left(x\right)\underset{Z_{H}\left(F\right)\backslash H\left(F\right)}{\iint}{\bf 1}_{\sigma\geq C\log\left(N\right)}\left(h^{\prime}\right)
$$
$$
\Xi^{G}\left(hx\right)\Xi^{G}\left(h^{\prime}hx\right)\sigma\left(hx\right)^{d}\sigma\left(h^{\prime}hx\right)^{d}\left\langle \phi,\omega_{V,\psi,\mu,\chi}\left(h^{\prime}\right)\phi^{\prime}\right\rangle dh^{\prime}dhdx,
$$
for any $N\geq1$ and $C\geq1$. By Proposition \ref{4.6}(viii), there exists
$d^{\prime}$ such that 
$$
\underset{Z_{H}\left(F\right)\backslash H\left(F\right)}{\iint}\Xi^{G}\left(hx\right)\Xi^{G}\left(h^{\prime}hx\right)\sigma\left(hx\right)^{d}\sigma\left(h^{\prime}hx\right)^{d}\left\langle \phi,\omega_{V,\psi,\mu,\chi}\left(h^{\prime}\right)\phi^{\prime}\right\rangle dh^{\prime}dh
\ll\Xi^{X}\left(x\right)^{2}\sigma_{X}\left(x\right)^{d^{\prime}}dx,
$$
which is to say $J_{\text{aux},N}\left(f,f^{\prime},\phi,\phi^{\prime}\right)$
and $J_{\text{aux},N,C}\left(f,f^{\prime},\phi,\phi^{\prime}\right)$
are absolutely convergent. Moreover, 
$$
\underset{Z_{H}\left(F\right)\backslash H\left(F\right)}{\iint}{\bf 1}_{\sigma\geq C\log\left(N\right)}\left(h^{\prime}\right)\Xi^{G}\left(hx\right)\Xi^{G}\left(h^{\prime}hx\right)\sigma\left(hx\right)^{d}\sigma\left(h^{\prime}hx\right)^{d}\left\langle \phi,\omega_{V,\psi,\mu,\chi}\left(h^{\prime}\right)\phi^{\prime}\right\rangle dh^{\prime}dh
$$
$$
\ll e^{-\epsilon C\log\left(N\right)}\Xi^{X}\left(x\right)^{2}\sigma_{X}\left(x\right)^{d^{\prime}}
$$
for all $N\geq1$ and $C>0$. By Proposition \ref{4.6}(v), there exists $d^{\prime\prime}>0$
such that 
$$
\int_{Z_{G}\left(F\right)H\left(F\right)\backslash G\left(F\right)}\alpha_{N}\left(x\right)\Xi^{X}\left(x\right)^{2}\sigma_{X}\left(x\right)^{d^{\prime}}dx\ll e^{-\epsilon C\log\left(N\right)}N^{d^{\prime\prime}},
$$
for all $N\geq1$ and $C>0$. Choosing $C$ big enough, one can prove
Claim \ref{6.5}.
\end{proof}
Let us fix $C>0$ which satisfies the above claim. Then $J_{\text{aux}}\left(f,f^{\prime},\phi,\phi^{\prime}\right)=\underset{N\rightarrow\infty}{\lim}J_{\text{aux},N,C}\left(f,f^{\prime},\phi,\phi^{\prime}\right)$. Since the triple integral defining $J_{\text{aux},N,C}\left(f,f^{\prime},\phi,\phi^{\prime}\right)$
is absolutely convergent, one can write 
$$
J_{\text{aux},N,C}\left(f,f^{\prime},\phi,\phi^{\prime}\right)=\int_{Z_{H}\left(F\right)\backslash H\left(F\right)}\beta_{C\log\left(N\right)}\left(h^{\prime}\right)\left\langle \phi,\omega_{V,\psi,\mu,\chi}\left(h^{\prime}\right)\phi^{\prime}\right\rangle 
\int_{Z_{G}\left(F\right)\backslash G\left(F\right)}\alpha_{N}\left(g\right)K_{f,f^{\prime}}^{A}\left(g,h^{\prime}g\right)dgdh^{\prime}.$$
We prove the following estimate.
\begin{claim}\label{6.6}
$ $
$$
\left|J_{\text{aux},N,C}\left(f,f^{\prime},\phi,\phi^{\prime}\right)-\begin{array}{c}
\int_{Z_{H}\left(F\right)\backslash H\left(F\right)}\beta_{C\log\left(N\right)}\left(h^{\prime}\right)\left\langle \phi,\omega_{V,\psi,\mu,\chi}\left(h^{\prime}\right)\phi^{\prime}\right\rangle \\
\int_{Z_{G}\left(F\right)\backslash G\left(F\right)}K_{f,f^{\prime}}^{A}\left(g,h^{\prime}g\right)dgdh^{\prime}
\end{array}\right|\ll N^{-1}
$$
for all $N\geq1$.
\end{claim}

\begin{proof}
By \cite[Theorem 5.5.1(iii)]{BP20}, there exists $c_{1}>0$ such that for all $d>0$, there exists $d^{\prime}>0$ such that $\left|K_{f,f^{\prime}}^{A}\left(g,h^{\prime}g\right)\right|\ll\Xi^{G}\left(g\right)^{2}\sigma\left(g\right)^{-d}e^{c_{1}\sigma\left(h\right)}\sigma\left(h^{\prime}\right)^{d^{\prime}}$, for all $g\in G\left(F\right)$ and $h\in H\left(F\right)$. By \cite[Proposition 1.5.1(v)]{BP20},
we can choose $d_{0}>0$ such that the function $g\mapsto\Xi^{G}\left(g\right)^{2}\sigma\left(g\right)^{-d_{0}}$ is integrable over $G\left(F\right)$. Hence, for all $d>d_{0}$, there exists $d^{\prime}>0$ such that 
$$
\left|J_{\text{aux},N,C}\left(f,f^{\prime},\phi,\phi^{\prime}\right)-\int_{Z_{H}\left(F\right)\backslash H\left(F\right)}\beta_{C\log\left(N\right)}\left(h^{\prime}\right)\left\langle \phi,\omega_{V,\psi,\mu,\chi}\left(h^{\prime}\right)\phi^{\prime}\right\rangle \int_{Z_{G}\left(F\right)\backslash G\left(F\right)}K_{f,f^{\prime}}^{A}\left(g,h^{\prime}g\right)dgdh^{\prime}\right|
$$
$$
\ll N^{c_{1}C-d+d_{0}}\log\left(N\right)^{d^{\prime}}\int_{Z_{H}\left(F\right)\backslash H\left(F\right)}\beta_{C\log\left(N\right)}\left(h^{\prime}\right)\left\langle \phi,\omega_{V,\psi,\mu,\chi}\left(h^{\prime}\right)\phi^{\prime}\right\rangle dh^{\prime}
$$
for all $N\geq1$. Moreover, there exists $c_{2}>0$ such that 
$$
\int_{Z_{H}\left(F\right)\backslash H\left(F\right)}\beta_{C\log\left(N\right)}\left(h^{\prime}\right)\left\langle \phi,\omega_{V,\psi,\mu,\chi}\left(h^{\prime}\right)\phi^{\prime}\right\rangle dh^{\prime}\ll N^{c_{2}},
$$
for all $N\geq1$. Therefore, if one let $d$ be sufficiently large,
one can prove Claim \ref{6.6}.
\end{proof}
By Claim \ref{6.6}, it follows that  
$$
J_{\text{aux}}\left(f,f^{\prime},\phi,\phi^{\prime}\right)=\underset{N\rightarrow\infty}{\lim}\int_{Z_{H}\left(F\right)\backslash H\left(F\right)}\beta_{C\log\left(N\right)}\left(h\right)\left\langle \phi,\omega_{V,\psi,\mu,\chi}\left(h^{\prime}\right)\phi^{\prime}\right\rangle\int_{Z_{G}\left(F\right)\backslash G\left(F\right)}K_{f,f^{\prime}}^{A}\left(g,hg\right)dgdh.
$$
Since $f$ is strongly cuspidal, by \cite[Theorem 5.5.1]{BP20}, we have
$$
\int_{Z_{G}\left(F\right)\backslash G\left(F\right)}K_{f,f^{\prime}}^{A}\left(g,hg\right)dg=\int_{{\mathcal{X}}\left(G,\chi^{-1}\right)}\hat{\theta}_{f}\left(\pi\right)\theta_{\bar{\pi}}\left(R\left(h^{-1}\right)f^{\prime}\right)d\pi.
$$
Observe $\left|\theta_{\bar{\pi}}\left(R\left(h^{-1}\right)f^{\prime}\right)\right|\ll\Xi^{G}\left(h\right)$, for all $h\in H\left(F\right)$. Thus, the double integral 
$$
\int_{Z_{H}\left(F\right)\backslash H\left(F\right)}\left\langle \phi,\omega_{V,\psi,\mu,\chi}\left(h^{\prime}\right)\phi^{\prime}\right\rangle \int_{{\mathcal{X}}\left(G,\chi^{-1}\right)}\hat{\theta}_{f}\left(\pi\right)\theta_{\bar{\pi}}\left(R\left(h^{-1}\right)f^{\prime}\right)d\pi dh
$$
is absolutely convergent and equal to $J_{\text{aux}}\left(f,f^{\prime},\phi,\phi^{\prime}\right)$.
By switching the two integrals and Lemma \ref{5.3}(iii), it follows that 
$$
J_{\text{aux}}\left(f,f^{\prime},\phi,\phi^{\prime}\right)=\int_{{\mathcal{X}}\left(G,\chi^{-1}\right)}\hat{\theta}_{f}\left(\pi\right){\mathcal{L}}_{\pi}^{\phi,\phi^{\prime}}\left(\pi\left(f^{\prime,\vee}\right)\right)d\pi.
$$
\end{proof}

\subsection{Proof of Theorem \ref{6.1}}\label{sec6.3}

By the induction hypothesis stated at the beginning of Section \ref{sec5} and Theorem 4.8 in \cite{GGP23},
we have $m_{V}\left(\pi\right)\leq1$, for any irreducible tempered generalized 
principal series $\pi=i_{P}^{G}\,\sigma$. In order to treat the discrete
series part, we need the following proposition.
\begin{proposition}\label{6.7}
Let $\pi$ be a discrete series representation of $G\left(F\right)$
whose central character is $\chi$. Let 
$$
\left(\pi\otimes\overline{\omega_{V,\psi,\mu,\chi}}\right)_{H_{V}}=\left(\pi\otimes\overline{\omega_{V,\psi,\mu,\chi}}\right)/\left\langle \begin{array}{c}
\pi\left(h\right)v\otimes\omega_{V,\psi,\mu,\chi}\left(h\right)\phi-v\otimes\phi,\\
v\in\pi,\,\phi\in\omega_{V,\psi,\mu,\chi},\,h\in H_{V}\left(F\right)
\end{array}\right\rangle .
$$
Assume $\left(\pi\otimes\overline{\omega_{V,\psi,\mu,\chi}}\right)_{H_{V}}$
is finite dimensional. The following scalar product 
$$
\begin{array}{cccc}
{\mathcal{B}}_{\pi}: & \left(\pi\otimes\overline{\omega_{V,\psi,\mu,\chi}}\right)\times\left(\bar{\pi}\otimes\omega_{V,\psi,\mu,\chi}\right) & \longrightarrow & \mathbb{C}\\
 & \left(v\otimes\phi_{2},v^{\vee}\otimes\phi_{1}\right) & \mapsto & \underset{Z_{H_{V}}\left(F\right)\backslash H_{V}\left(F\right)}{\int}\left\langle \pi\left(h\right)v,v^{\vee}\right\rangle \left\langle \phi_{1},\omega_{V,\psi,\mu,\chi}\left(h\right)\phi_{2}\right\rangle dh
\end{array}
$$
induces a nondegenerate scalar product of $\left(\pi\otimes\overline{\omega_{V,\psi,\mu,\chi}}\right)_{H_{V}}\times\left(\bar{\pi}\otimes\omega_{V,\psi,\mu,\chi}\right)_{H_{V}}$.
\end{proposition}

\begin{proof}
Since 
$$
{\mathcal{B}}_{\pi}\left(\pi\left(h_{1}\right)v\otimes\omega_{V,\psi,\mu,\chi}\left(h_{1}\right)\phi_{2},\pi\left(h_{2}\right)v^{\vee}\otimes\omega_{V,\psi,\mu,\chi}\left(h_{2}\right)\phi_{1}\right)={\mathcal{B}}_{\pi}\left(v\otimes\phi_{2},v^{\vee}\otimes\phi_{1}\right),
$$
for any $v\otimes\phi_{2}\in\pi\otimes\overline{\omega_{V,\psi,\mu,\chi}}$
and $v^{\vee}\otimes\phi_{1}\in\bar{\pi}\otimes\omega_{V,\psi,\mu,\chi}$,
it follows that ${\mathcal{B}}_{\pi}$ factors through $\left(\pi\otimes\overline{\omega_{V,\psi,\mu,\chi}}\right)_{H_{V}}\times\left(\bar{\pi}\otimes\omega_{V,\psi,\mu,\chi}\right)_{H_{V}}$.
It suffices to show this pairing is nondegenerate. By Theorem \ref{A.1},
$\left(\pi\otimes\overline{\omega_{V,\psi,\mu,\chi}}\right)_{H_{V}}$
is finite dimensional. Hence, we need to show 
$$
v^{\vee}\otimes\phi\mapsto{\mathcal{B}}_{\pi}\left(\cdot,v^{\vee}\otimes\phi\right)\in\text{Hom}_{H_{V}}\left(\pi\otimes\overline{\omega_{V,\psi,\mu,\chi}},\mathbb{C}\right)
$$
is surjective. One has the following isomorphism 
$$
\begin{array}{ccc}
\text{Hom}_{H_{V}}\left(\pi,\omega_{V,\psi,\mu,\chi}\right) & \overset{\sim}{\longrightarrow} & \text{Hom}_{H_{V}}\left(\pi\otimes\overline{\omega_{V,\psi,\mu,\chi}},\mathbb{C}\right)\\
l & \mapsto & v\otimes\phi\mapsto\left\langle l\left(v\right),\phi\right\rangle 
\end{array}
$$
We construct the inverse of the above map as follows. Let $\ell\in\text{Hom}_{H_{V}}\left(\pi\otimes\overline{\omega_{V,\psi,\mu,\chi}},\mathbb{C}\right)$.
Let $l\in\text{Hom}_{H_{V}}\left(\pi,\omega_{V,\psi,\mu,\chi}\right)$
such that for any $v\otimes\phi\in\pi\otimes\overline{\omega_{V,\psi,\mu,\chi}}$,
one has $\left\langle l\left(v\right),\phi\right\rangle =\ell\left(v\otimes\phi\right)$.
This gives us the inverse of the above map. Moreover, by using a similar argument as in the proof of Lemma \ref{5.7}, for any $d>0$, we have $|l(\pi(a)v)| \ll \Xi^G(a)\sigma(a)^{-d},$ for $a\in A_{P_0}^+(F)$, noting that here $\pi$ is a discrete series. By Proposition \ref{4.6}(iii), we have
\begin{equation}\label{eqn5.4}
    \int_{Z_G(F)H(F)\backslash G(F)}|l(\pi(x)v)|^2dx < \infty.
\end{equation}
One can regard ${\mathcal{B}}_{\pi}\left(\cdot,v^{\vee}\otimes\phi\right)$
as an element in $\text{Hom}_{H_{V}}\left(\pi,\omega_{V,\psi,\mu,\chi}\right)$
by this way. This gives us a scalar product on $\text{Hom}_{H_{V}}\left(\pi,\omega_{V,\psi,\mu,\chi}\right)$
characterized by 
$$
\int_{Z_{G}\left(F\right)H_{V}\left(F\right)\backslash G\left(F\right)}\left\langle \ell\left(\pi\left(x\right)v\right),\ell^{\prime}\left(\pi\left(x\right)v^{\prime}\right)\right\rangle dx=\left(\ell,\ell^{\prime}\right)\left(v,v^{\prime}\right),
$$
where $\ell,\ell^{\prime}\in\text{Hom}_{H_{V}}\left(\pi,\omega_{V,\psi,\mu,\chi}\right)$
and $v,v^{\prime}\in\pi$, noting that $\left\langle ,\right\rangle $
is the scalar product of the space of Schwartz functions realising the Weil representation. Moreover, the abosolute convergence of the LHS follows from Cauchy-Schwarz inequality and (\ref{eqn5.4}).

Let $l\in\text{Hom}_{H_{V}}\left(\pi,\omega_{V,\psi,\mu,\chi}\right)$
which is orthogonal to all the forms ${\mathcal{B}}_{\pi}\left(\cdot,v\otimes\phi\right)$,
where $v\otimes\phi\in\pi\otimes\overline{\omega_{V,\psi,\mu,\chi}}$.
We need to show $l=0$. One has 
$$
0=\int_{Z_{G}\left(F\right)H\left(F\right)\backslash G\left(F\right)}\left\langle l\left(\pi\left(x\right)v^{\prime}\right),{\mathcal{B}}_{\pi}\left(\cdot,v\otimes\phi\right)\left(\pi\left(x\right)v^{\prime}\right)\right\rangle dx
$$
$$
=\int_{Z_{G}\left(F\right)H\left(F\right)\backslash G\left(F\right)}\int_{Z_{H}\left(F\right)\backslash H\left(F\right)}\left\langle v,\pi\left(hx\right)v^{\prime}\right\rangle \left\langle \omega_{V,\psi,\mu,\chi}\left(h\right)l\left(\pi\left(x\right)v^{\prime}\right),\phi\right\rangle dhdx
$$
$$
=\int_{Z_{G}\left(F\right)\backslash G\left(F\right)}\left\langle v,\pi\left(g\right)v^{\prime}\right\rangle \left\langle l\left(\pi\left(g\right)v^{\prime}\right),\phi\right\rangle dg=\frac{1}{d\left(\pi\right)}\left\langle v^{\prime},v^{\prime}\right\rangle  \left\langle l\left(v\right),\phi\right\rangle.
$$
Since we can choose $v^{\prime}$ and $\phi$ arbitrarily, it follows
that $l=0$ as desired.
\end{proof}
We now prove Theorem \ref{6.1}.
\begin{proof}
Let $f\in{\mathcal{C}}_{\text{scusp}}\left(Z_{G}\left(F\right)\backslash G\left(F\right),\chi\right)$.
By Theorem \ref{2.2}(i), there exists a cusp form $f_{2}$ in ${\mathcal{C}}_{\text{scusp}}\left(Z_{G}\left(F\right)\backslash G\left(F\right),\chi\right)$ such that $\pi\left(f_{2}\right)=\pi\left(f\right)$, for any discrete
series $\pi$. Let $f_{1}=f-f_{2}$. Then $\pi\left(f_{1}\right)=0$,
for any discrete series $\pi$. Thus, one has $m_{V}\left(\pi\right)\leq1$
whenever $\pi\left(f_{1}\right)\neq0$. Let $\left\{ \phi_{i}\right\} _{i}$
be an orthonormal basis for $\omega_{V,\psi,\mu,\chi}$. We consider
$$
J_{V}\left(f\right)=\int_{Z_{G}\left(F\right)H\left(F\right)\backslash G\left(F\right)}\underset{i}{\sum}\left(\int_{Z_{H_{V}}\left(F\right)\backslash H_{V}\left(F\right)}f\left(x^{-1}hx\right)\left\langle \phi_{i},\omega_{V,\psi,\mu,\chi}\left(h\right)\phi_{i}\right\rangle dh\right)dx.
$$
Let $x\in G\left(F\right)$. Assume both $^{x}f_{1}$ and $^{x}f_{2}$
are fixed by a compact open subgroup $K$. Let $K_{H}=K\cap H_{V}\left(F\right)$
and $\left\{ \phi_{i}\right\} _{i\in I}$ be an orthonormal basis for $\omega_{V,\psi,\mu,\chi}^{K_{H}}$. Then
$$
K\left(f,x\right)=\underset{i\in I}{\sum}\int_{Z_{H_{V}}\left(F\right)\backslash H_{V}\left(F\right)}f\left(x^{-1}hx\right)\left\langle \phi_{i},\omega_{V,\psi,\mu,\chi}\left(h\right)\phi_{i}\right\rangle dh.
$$
By Lemma \ref{5.2}, we have 
$$
\int_{Z_{H_{V}}\left(F\right)\backslash H_{V}\left(F\right)}f\left(x^{-1}hx\right)\left\langle \phi_{i},\omega_{V,\psi,\mu,\chi}\left(h\right)\phi_{i}\right\rangle dh\text{ is absolutely convergent.}
$$
This gives us $K\left(f,x\right)$ is absolutely convergent. We now
consider $K\left(f_{1},x\right)$ and $J_{V}\left(f_{1}\right)$.
Since $\pi\mapsto\pi\left(f_{1}\right)$ lies in $C_{c}^{\infty}\left({\mathcal{X}}_{\text{temp}}\left(G,\chi^{-1}\right),{\mathcal{E}}\left(G,\chi^{-1}\right)\right)$,
by Corollary \ref{5.8}(ii), there exists a function $f^{\prime}\in{\mathcal{C}}\left(Z_{G}\left(F\right)\backslash G\left(F\right),\chi^{-1}\right)$
and $\phi,\phi^{\prime}\in\omega_{V,\psi,\mu,\chi}$ such that ${\mathcal{L}}_{\pi}^{\phi,\phi^{\prime}}\left(\pi\left(f^{\prime,\vee}\right)\right)=m_{V}\left(\bar{\pi}\right)$, for any $\pi\in{\mathcal{X}}_{\text{temp}}\left(G,\chi^{-1}\right)$ such
that $\pi\left(f\right)\neq0$. By Theorem \ref{5.1} and Corollary \ref{5.8}(i) and since we have $m_{V}\left(\pi\right)\leq1$ whenever $\pi\left(f_{1}\right)\neq0$,
one can see that 
$$
\int_{Z_{H_{V}}\left(F\right)\backslash H_{V}\left(F\right)}f_{1}\left(x^{-1}hx\right)\left\langle \phi_{i},\omega_{V,\psi,\mu,\chi}\left(h\right)\phi_{i}\right\rangle dh
=\int_{{\mathcal{X}}_{\text{temp}}\left(G\left(F\right),\chi^{-1}\right)}{\mathcal{L}}_{\pi}^{\phi_{i},\phi_{i}}\left(\pi\left(x\right)\pi\left(f_{1}\right)\pi\left(x^{-1}\right)\right)\mu\left(\pi\right)d\pi
$$
$$
=\int_{{\mathcal{X}}_{\text{temp}}\left(G\left(F\right),\chi^{-1}\right)}{\mathcal{L}}_{\pi}^{\phi_{i},\phi_{i}}\left(\pi\left(x\right)\pi\left(f_{1}\right)\pi\left(x^{-1}\right)\right){\mathcal{L}}_{\pi}^{\phi,\phi^{\prime}}\left(\pi\left(f^{\prime,\vee}\right)\right)\mu\left(\pi\right)d\pi.
$$
By Lemma \ref{5.3}(iv), the above integral is absolutely convergent and equal
to
$$
\underset{Z_{H_{V}}\left(F\right)\backslash H_{V}\left(F\right)}{\iint}\int_{Z_{G}\left(F\right)\backslash G\left(F\right)}f^{\prime}\left(g\right)f_{1}\left(x^{-1}h_{1}gh_{2}x\right)\left\langle \phi,\omega_{V,\psi,\mu,\chi}\left(h_{2}\right)\phi_{i}\right\rangle \left\langle \phi_{i},\omega_{V,\psi,\mu,\chi}\left(h_{1}\right)\phi^{\prime}\right\rangle dgdh_{2}dh_{1}.
$$
By the orthogonality of $\left\{ \phi_{i}\right\} _{i\in I}$, it
follows that 
$$
K\left(f_{1},x\right)=\underset{i\in I}{\sum}\underset{Z_{H_{V}}\left(F\right)\backslash H_{V}\left(F\right)}{\iint}\underset{Z_{G}\left(F\right)\backslash G\left(F\right)}{\int}f^{\prime}\left(g\right)f_{1}\left(x^{-1}h_{1}gh_{2}x\right)
\left\langle \phi,\omega_{V,\psi,\mu,\chi}\left(h_{2}\right)\phi_{i}\right\rangle \left\langle \phi_{i},\omega_{V,\psi,\mu,\chi}\left(h_{1}\right)\phi^{\prime}\right\rangle dgdh_{2}dh_{1}
$$
$$
=\underset{Z_{H_{V}}\left(F\right)\backslash H_{V}\left(F\right)}{\iint}\underset{Z_{G}\left(F\right)\backslash G\left(F\right)}{\int}f^{\prime}\left(g\right)f_{1}\left(x^{-1}h_{1}gh_{2}x\right)dg\left\langle \phi,\omega_{V,\psi,\mu,\chi}\left(h_{2}h_{1}\right)\phi^{\prime}\right\rangle dh_{2}dh_{1}.
$$
This gives us $K\left(f_{1},x\right)=K_{f_{1},f^{\prime},\phi,\phi^{\prime}}^{2}\left(x,x\right)$. We now consider the outer integral. We have 
$$
J_{V}\left(f_{1}\right)=\int_{Z_{G}\left(F\right)H_{V}\left(F\right)\backslash G\left(F\right)}K_{f_{1},f^{\prime},\phi,\phi^{\prime}}^{2}\left(x,x\right)dx=J_{\text{aux}}\left(f_{1},f^{\prime},\phi,\phi^{\prime}\right).
$$
By Proposition \ref{6.4}, it follows that 
$$
J_{V}\left(f_{1}\right)=\int_{{\mathcal{X}}_{\text{temp}}\left(G\left(F\right),\chi^{-1}\right)}\hat{\theta}_{f_{1}}\left(\pi\right){\mathcal{L}}_{\pi}^{\phi,\phi^{\prime}}\left(\pi\left(f^{\prime,\vee}\right)\right)d\pi
=\int_{{\mathcal{X}}_{\text{temp}}\left(G\left(F\right),\chi^{-1}\right)}\hat{\theta}_{f_{1}}\left(\pi\right)m_{V}\left(\bar{\pi}\right)d\pi.
$$
There remains to consider $J_{V}\left(f_{2}\right)$. For any $\pi\in\text{Irr}_{\text{sqr}}\left(G,\chi\right)$, we define $
f_{2,\pi}\left(g\right)=\text{Trace}\left(\bar{\pi}\left(g^{-1}\right)\bar{\pi}\left(f_{2}\right)\right)$. By the Harish-Chandra Plancherel formula, we have $J_{V}\left(f_{2}\right)=\underset{\pi\in\text{Irr}_{\text{sqr}}\left(G,\chi\right)}{\sum}d\left(\pi\right)J_{V}\left(f_{2,\pi}\right)$, where $d\left(\pi\right)$ is the formal degree of $\pi$. We fix
$\pi\in\text{Irr}_{\text{sqr}}\left(G,\chi\right)$. It suffices to
show 
$$
J_{V}\left(f_{2,\pi}\right)=J_{V,spec}\left(f_{2,\pi}\right)=d\left(\pi\right)^{-1}m_{V}\left(\pi\right)\text{Trace}\left(\bar{\pi}\left(f_{2,\pi}\right)\right).
$$
Since $f_{2,\pi}$ is a sum of matrix coefficients of $\pi$, we only
need to prove 
$$
J_{V}\left(f_{v,v^{\vee}}\right)=d\left(\pi\right)^{-1}m_{V}\left(\pi\right)f_{v,v^{\vee}},
$$
for all $\left(v,v^{\vee}\right)\in\pi\times\bar{\pi}$, where $f_{v,v^{\vee}}\left(g\right)=\left\langle \pi\left(g\right)v,v^{\vee}\right\rangle $.
We fix $\left(v,v^{\vee}\right)\in\pi\times\bar{\pi}$. Then  
$$
K\left(f_{v,v^{\vee}},x\right)=\underset{i}{\sum}\ {\mathcal{B}}_{\pi}\left(\pi\left(x\right)v\otimes\phi_{i},\pi\left(x\right)v^{\vee}\otimes\phi_{i}\right).
$$
To simplify our computation, suppose that we can choose a basis $\left(v_{1}\otimes w_{1}^{\vee},\ldots,v_{N}\otimes w_{N}^{\vee}\right)$
of $\left(\pi\otimes\overline{\omega_{V,\psi,\mu,\chi}}\right)_{H_{V}}$
and let $\left(v_{1}^{\vee}\otimes w_{1},\ldots,v_{N}^{\vee}\otimes w_{N}\right)$
be its dual in $\left(\bar{\pi}\otimes\omega_{V,\psi,\mu,\chi}\right)_{H_{V}}$.
We lift the two bases to $\pi\otimes\overline{\omega_{V,\psi,\mu,\chi}}$
and $\bar{\pi}\otimes\omega_{V,\psi,\mu,\chi}$. For any $x\in Z_{G}\left(F\right)H\left(F\right)\backslash G\left(F\right)$,
we have  
$$
{\mathcal{B}}_{\pi}\left(\pi\left(x\right)v\otimes\phi_{i},\pi\left(x\right)v^{\vee}\otimes\phi_{i}\right)=\sum_{j=1}^N{\mathcal{B}}_{\pi}\left(\pi\left(x\right)v\otimes\phi_{i},v_{j}^{\vee}\otimes w_{j}\right){\mathcal{B}}_{\pi}\left(v_{j}\otimes w_{j}^{\vee},\pi\left(x\right)v^{\vee}\otimes\phi_{i}\right)
$$
$$
=\sum_{j=1}^N\iint_{Z_{H_{V}}\left(F\right)\backslash H_{V}\left(F\right)}\left\langle \pi\left(h_{1}x\right)v,v_{j}^{\vee}\right\rangle \left\langle w_{j},\omega_{V,\psi,\mu,\chi}\left(h_{1}\right)\phi_{i}\right\rangle 
\left\langle \pi\left(h_{2}\right)v_{j},\pi\left(x\right)v^{\vee}\right\rangle \left\langle \phi_{i},\omega_{V,\psi,\mu,\chi}\left(h_{2}\right)w_{j}^{\vee}\right\rangle dh_{1}dh_{2}.
$$
This gives us 
$$
{\mathcal{B}}_{\pi}\left(\pi\left(x\right)v\otimes\phi_{i},\pi\left(x\right)v^{\vee}\otimes\phi_{i}\right)=\sum_{j=1}^N\underset{i\in I}{\sum}\iint_{Z_{H_{V}}\left(F\right)\backslash H_{V}\left(F\right)}\left\langle \pi\left(h_{1}x\right)v,v_{j}^{\vee}\right\rangle\left\langle w_{j},\omega_{V,\psi,\mu,\chi}\left(h_{1}\right)\phi_{i}\right\rangle
$$
$$
 \left\langle \pi\left(h_{2}\right)v_{j},\pi\left(x\right)v^{\vee}\right\rangle \left\langle \phi_{i},\omega_{V,\psi,\mu,\chi}\left(h_{2}\right)w_{j}^{\vee}\right\rangle dh_{1}dh_{2}
$$
$$
=\sum_{j=1}^N\iint_{Z_{H_{V}}\left(F\right)\backslash H_{V}\left(F\right)}\left\langle \pi\left(h_{1}x\right)v,v_{j}^{\vee}\right\rangle \left\langle \pi\left(h_{2}\right)v_{j},\pi\left(x\right)v^{\vee}\right\rangle \left\langle w_{j},\omega_{V,\psi,\mu,\chi}\left(h_{1}h_{2}\right)w_{j}^{\vee}\right\rangle dh_{1}dh_{2}.
$$
Therefore 
$$
J_{V}\left(f_{v,v^{\vee}}\right)=\sum_{j=1}^N\int_{Z_{G}\left(F\right)H_{V}\left(F\right)\backslash G\left(F\right)}\iint_{Z_{H_{V}}\left(F\right)\backslash H_{V}\left(F\right)}\left\langle \pi\left(h_{1}x\right)v,v_{j}^{\vee}\right\rangle \left\langle \pi\left(h_{2}\right)v_{j},\pi\left(x\right)v^{\vee}\right\rangle
$$
$$
\left\langle w_{j},\omega_{V,\psi,\mu,\chi}\left(h_{1}h_{2}\right)w_{j}^{\vee}\right\rangle dh_{1}dh_{2}dx
$$
$$
=\sum_{j=1}^N\int_{Z_{G}\left(F\right)H_{V}\left(F\right)\backslash G\left(F\right)}\iint_{Z_{H_{V}}\left(F\right)\backslash H_{V}\left(F\right)}\left\langle \pi\left(h_{1}x\right)v,v_{j}^{\vee}\right\rangle \left\langle \pi\left(h_{2}\right)v_{j},\pi\left(x\right)v^{\vee}\right\rangle
$$
$$
\left\langle w_{j},\omega_{V,\psi,\mu,\chi}\left(h_{1}h_{2}\right)w_{j}^{\vee}\right\rangle dh_{1}dh_{2}dx
$$
$$
=\sum_{j=1}^N\int_{Z_{G}\left(F\right)\backslash G\left(F\right)}\left\langle v_{j},\pi\left(x\right)v^{\vee}\right\rangle {\mathcal{B}}_{\pi}\left(\pi\left(x\right)v\otimes w_{j}^{\vee},v_{j}^{\vee}\otimes w_{j}^{\vee}\right)dg
$$
$$
=\frac{1}{d\left(\pi\right)}\sum_{j=1}^N\left\langle v,v^{\vee}\right\rangle {\mathcal{B}}_{\pi}\left(v_{j}\otimes w_{j}^{\vee},v_{j}^{\vee}\otimes w_{j}\right)=\frac{1}{d\left(\pi\right)}m_{V}\left(\pi\right)f_{v,v^{\vee}}\left(1\right).
$$
\end{proof}

\section{Geometric side and the strong multiplicity one property}\label{sec7}

\subsection{Linearization of trace formulas}\label{sec7.1}

Let $V$ and $V^{\prime}$ be the two non-isomorphic $n$-dimensional skew-Hermitian spaces over $E$. Let $H_{V}=U\left(V\right)$ and $H_{V^{\prime}}=U\left(V^{\prime}\right)$. We set $J\left(f\right)=J_{V}\left(f\right)+J_{V^{\prime}}\left(f\right)$, for $f\in{\mathcal{C}}_{\text{scusp}}\left(Z_{G}\left(F\right)\backslash G\left(F\right),\chi\right)$.
Let $m\left(\pi\right)=m_{V}\left(\pi\right)+m_{V^{\prime}}\left(\pi\right)$.
By Theorem \ref{6.1}, it follows that 
$$
J\left(f\right)=\int_{{\mathcal{X}}\left(G,\chi^{-1}\right)}\hat{\theta}_{f}\left(\pi\right)m\left(\bar{\pi}\right)d\pi.
$$
We define the following linear form
$$
J_{\text{qc}}\left(\theta\right)=c_\theta(1)+\underset{\pi\in\Pi_2\left(G,\chi^{-1}\right)}{\sum}(m\left(\bar{\pi}\right)-1)\int_{\Gamma_{\text{ell}}\left(Z_{G}\backslash G\right)}D^{G}\left(x\right)\theta\left(x\right)\theta_{\pi}\left(x\right)dx,
$$
for $\theta\in QC\left(Z_{G}\left(F\right)\backslash G\left(F\right),\chi\right)$.
The remaining of this section is devoted to prove the following theorem
and then prove Theorem \ref{thm1.2}(i).
\begin{theorem}\label{7.1}
For any $\theta\in QC\left(Z_{G}\left(F\right)\backslash G\left(F\right),\chi\right)$, we have 
$$
J_{\text{qc}}\left(\theta\right)=c_{\theta}\left(1\right).
$$
\end{theorem}

We prove Theorem \ref{7.1} by induction on $n$. When $n=1$, it is clear that Theorem \ref{7.1} holds. From now, assume Theorem \ref{7.1} holds for lower rank cases.

\subsection{Semisimple descent and the support of $J_{\text{qc}}$}\label{sec7.2}

For simplicity, we denote $H_{V}$ and $H_{V^{\prime}}$ by $H$,
depending on whether we are considering $J_{V}$ or $J_{V^{\prime}}$.
Let $\Gamma\left(H\right)$ and $\Gamma\left(G\right)$ be the
sets of semisimple conjugacy classes in $H\left(F\right)$ and $G\left(F\right)$
respectively. Moreover, we denote by $\Gamma_{\text{stab}}\left(H\right)$ the set of semisimple stable conjugacy classes in $H\left(F\right)$. The following lemma gives us a way to estimate the linear form $J_{\text{qc}}$ using the linear form $J(f)$.
\begin{lemma}\label{7.2}
For any $f\in \mathcal{C}(Z_G(F)\backslash G(F),\chi)$, we have 
$$
J_{qc}(\theta_f)=J(f).
$$
\end{lemma}

\begin{proof}
By Theorem \ref{6.1}, we have 
$$
J\left(f\right)=\int_{{\mathcal{X}}\left(G,\chi^{-1}\right)}\hat{\theta}_{f}\left(\pi\right)m\left(\bar{\pi}\right)d\pi
$$
$$
=\int_{{\mathcal{X}}_{\text{ind}}\left(G,\chi^{-1}\right)}\hat{\theta}_{f}\left(\pi\right)m\left(\bar{\pi}\right)d\pi
+\underset{\pi\in\Pi_2\left(G,\chi^{-1}\right)}{\sum}m\left(\bar{\pi}\right)\int_{\Gamma\left(Z_{G}\backslash G\right)}D^{G}\left(x\right)\theta_{f}\left(x\right)\theta_{\pi}\left(x\right)dx.
$$
By the induction hypothesis, we have $m(\pi) = 1 = c_{\theta_\pi}(1)$, for any $\pi \in\mathcal{X}_\text{ind}(G,\chi)$. Hence, by Fourier inverse formula, we obtain
$$
J(f)=\int_{{\mathcal{X}}\left(G,\chi^{-1}\right)}\hat{\theta}_{f}\left(\pi\right)c_{\theta_{\bar{\pi}}}(1)d\pi
+\underset{\pi\in\Pi_2\left(G,\chi^{-1}\right)}{\sum}(m\left(\bar{\pi}\right)-1)\int_{\Gamma\left(Z_{G}\backslash G\right)}D^{G}\left(x\right)\theta_{f}\left(x\right)\theta_{\pi}\left(x\right)dx.
$$
$$
=c_{\theta_f}(1)+\underset{\pi\in\Pi_2\left(G,\chi^{-1}\right)}{\sum}(m\left(\bar{\pi}\right)-1)\int_{\Gamma\left(Z_{G}\backslash G\right)}D^{G}\left(x\right)\theta_{f}\left(x\right)\theta_{\pi}\left(x\right)dx=J_{qc}(\theta_f).
$$
\end{proof}
The following proposition shows that $\text{Supp}\left(J_{\text{qc}}\right)\subseteq Z_{G}\left(F\right)$.
\begin{proposition}\label{7.3}
Let $\theta\in QC\left(Z_{G}\left(F\right)\backslash G\left(F\right),\chi\right)$
and assume that $Z_{G}\left(F\right)\,\cap\,\text{Supp}\left(\theta\right)$ is empty. Then 
$J_{\text{qc}}\left(\theta\right)=0$.
\end{proposition}

\begin{proof}
Since $J_{\text{qc}}$ is supported in $\Gamma_{\text{ell}}\left(G\right)$,
by a process of partition of unity, we only need to prove the equality
for $\theta\in QC\left(Z_{G}\left(F\right)\backslash\Omega,\chi^{-1}\right)$,
where $\Omega$ is a completely $G\left(F\right)$-invariant
open subset of $G\left(F\right)$ of the form $\Omega_{x}^{G}$ for some noncentral element $x\in G\left(F\right)_{\text{ell}}$ and $G$-good open neighborhood $\Omega_{x}\subseteq G_{x}\left(F\right)$ of $x$. Moreover, we are able to shrink $\Omega_x$ as small as we want. Hence, we assume $\Omega_{x}$ is relatively compact modulo conjugation and $Z_{G}\left(F\right)$.

Assume that $x_{i}$ is not $G\left(F\right)$-conjugate to any element
of $H_{ss}\left(F\right)$. Since $\Gamma\left(H\right)_{\text{stab}}$
is closed in $\Gamma\left(G\right)$, if $\Omega_{x}$
is sufficiently small, then $\Omega\cap\Gamma\left(H\right)=\emptyset$.
Let $f\in \mathcal{C}_{\text{scusp}}(Z_G(F)\backslash\Omega,\chi)$
such that $\theta_{f_n}=\theta$. By
Lemma \ref{7.2}, $J_{qc}\left(\theta\right)=J_{qc}(\theta_f)=J(f)=0$. Hence, $\Omega$ has no contribution in $J_{qc}$.

Assume that $x$ is $G\left(F\right)$-conjugate to some elements
of $H_{ss}\left(F\right)$. We may assume $x\in H_{ss}\left(F\right)$.
Since $x$ is elliptic in $G(F)$, we can see that $(G_x,H_x,\omega_{V,\psi,\mu}\mid_{H_x})$ forms a twisted GGP triple. We choose a function $f_x$ in
${\mathcal{C}}_{\text{scusp}}\left(Z_G(F)\backslash\Omega_x,\chi\right)$ such that $\theta_{f_x}=\theta_{x,\Omega_x}$. Let $f=\widetilde{f_{x}}\in{\mathcal{C}}_{\text{scusp}}\left(Z_G(F)\backslash\Omega,\chi\right)$
be a lift defined in Proposition \ref{3.18}. Then we have $\theta_{f}=\theta$. By Lemma \ref{7.2}, it follows that 
$$
J_{\text{qc}}\left(\theta\right)=J\left(f\right)=J_V(f)+J_{V^\prime}(f).
$$  
By Section \ref{sec3.1}, we have 
$$
J_{V}\left(f\right)=\int_{Z_{G}\left(F\right)H\left(F\right)\backslash G\left(F\right)}\underset{i}{\sum}\int_{Z_{H_{V}}\left(F\right)\backslash H\left(F\right)}f\left(g^{-1}hg\right)\left\langle \phi_{i},\omega_{V,\psi,\mu,\chi}\left(h\right)\phi_{i}\right\rangle dhdg
$$
$$
=\int_{Z_{G}\left(F\right)H\left(F\right)\backslash G\left(F\right)}\underset{i}{\sum}\int_{H_{x}\left(F\right)\backslash H\left(F\right)}\int_{Z_{H}\left(F\right)\backslash H_{x}\left(F\right)}\left(^{hg}f\right)_{x,\Omega_{x}}\left(h_{x}\right)\left\langle \phi_{i},\omega_{V,\psi,\mu,\chi}\left(h_{x}\right)\phi_{i}\right\rangle dh_{x}dhdg.
$$
Assume one moment that the exterior double integral and the sum above
are absolutely convergent. Then $J_{V}\left(f\right)$ is equal to
\begin{equation}\label{eqn6.1}
    \int_{G_{x}\left(F\right)\backslash G\left(F\right)}\int_{Z_{G}\left(F\right)H_{x}\left(F\right)\backslash G_{x}\left(F\right)}\underset{i}{\sum}\int_{Z_{H}\left(F\right)\backslash H_{x}\left(F\right)}\left(^{g}f\right)_{x,\Omega_{x}}\left(g_{x}^{-1}h_{x}g_{x}\right)\left\langle \phi_{i},\omega_{V,\psi,\mu,\chi}\left(h_{x}\right)\phi_{i}\right\rangle dh_{x}dg_{x}dg
\end{equation}
We introduce a function $\alpha$ on $G_{x}\left(F\right)\backslash G\left(F\right)$
as in \cite[Proposition 5.7.1]{BP20}. Up to translating $g$ by an
element in $G_{x}\left(F\right)$, we may assume that $\left(^{g}f\right){}_{x,\Omega_{x}}=\alpha\left(g\right)f_x$.
We have  
$$
\int_{Z_{G}\left(F\right)H_{x}\left(F\right)\backslash G_{x}\left(F\right)}\underset{i}{\sum}\int_{Z_{H}\left(F\right)\backslash H_{x}\left(F\right)}\left(^{g}f\right)_{x,\Omega_{x}}\left(g_{x}^{-1}h_{x}g_{x}\right)\left\langle \phi_{i},\omega_{V,\psi,\mu,\chi}\left(h_{x}\right)\phi_{i}\right\rangle dh_{x}dg_{x}
$$
$$
=\alpha\left(g\right)\int_{Z_{G}\left(F\right)H_{x}\left(F\right)\backslash G_{x}\left(F\right)}\underset{i}{\sum}\int_{Z_{H}\left(F\right)\backslash H_{x}\left(F\right)}f_x\left(g_{x}^{-1}h_{x}g_{x}\right)\left\langle \phi_{i},\omega_{V,\psi,\mu,\chi}\left(h_{x}\right)\phi_{i}\right\rangle dh_{x}dg_{x}
$$
$$
=\alpha\left(g\right)\text{vol}\left(Z_{G_{x}}\left(F\right)/Z_{G}\left(F\right)\right)\text{vol}\left(Z_{H_{x}}\left(F\right)/Z_{H}\left(F\right)\right)^{-1}J^{H_x}\left(f_x\right),
$$
where $J^{H_{x}}$ is the linear form in style of $J_{V}$ for the twisted GGP triple $\left(G_x,H_x,\omega_{H_x,\psi,\mu}\right)$.
By Section \ref{sec6.3} and since the function $\alpha$ is compactly supported, the exterior double integral of (\ref{eqn6.1}) is absolutely convergent.
Moreover, as $\int_{G_{x}\left(F\right)\backslash G\left(F\right)}\alpha\left(g\right)dg=1$, we have 
$$
J_{V}\left(f\right)=\text{vol}\left(Z_{G_{x}}\left(F\right)/Z_{G}\left(F\right)\right)\text{vol}\left(Z_{H_{x}}\left(F\right)/Z_{H}\left(F\right)\right)^{-1}J^{H_x}\left(f_x\right).
$$
Since $x\in H\left(F\right)_{\text{ell}}$, it follows that $x$ is
stably conjugate to an element $x^{\prime}\in H_{V^{\prime}}\left(F\right)_{\text{ell}}$.
By doing similarly to above, we obtain 
$$
J_{V^{\prime}}\left(f\right)=\text{vol}\left(Z_{G_{x}}\left(F\right)/Z_{G}\left(F\right)\right)\text{vol}\left(Z_{\left(H_{V^{\prime}}\right)_{x^{\prime}}}\left(F\right)/Z_{H_{V^{\prime}}}\left(F\right)\right)^{-1}J^{(H_{V^\prime})_{x^\prime}}\left(f_x\right).
$$
By the induction hypothesis and Lemma \ref{7.2}, we have 
$$
J^{H_x}(f_x)+J^{(H_{V^\prime})_{x^\prime}}\left(f_x\right)=J_{\text{qc}}^{G_x}\left(\theta_{x,\Omega_x}\right)=c_{\theta_{x,\Omega_x}}\left(1\right)=0,
$$
where $J_{\text{qc}}^{G_{x}}$ is an analog of the linear form $J_{\text{qc}}$
for $G_{x}$. This gives us $J_{\text{qc}}\left(\theta\right)=0$, for any quasi-character $\theta$ supported outside $Z_{G}\left(F\right)$.
\end{proof}

\subsection{Descent to the Lie algebra and homogeneity}\label{sec7.3}

We identify $\mathfrak{g}_{0}$ and $\mathfrak{h}_{V,0}$ (i.e. the sets containing trace-$0$ elements in $\mathfrak{g}$ and $\mathfrak{h}_{V}$) with the
Lie algebras of $G/Z_{G}$ and $H_{V}/Z_{H_{V}}$, respectively. For
a function $f\in C_{c}^{\infty}\left(\mathfrak{g}_{0}\left(F\right)\right)$,
let us define a function $K^{\text{Lie}}\left(f,\cdot\right)$ on
$Z_{G}\left(F\right)H_{V}\left(F\right)\backslash G\left(F\right)$
by 
$$
K_{V}^{\text{Lie}}\left(f,g\right)=\underset{i}{\sum}\int_{\mathfrak{h}_{V,0}\left(F\right)}f\left(g^{-1}Xg\right)\left\langle \phi_{i},\omega_{V,\psi,\mu,\chi}\left(\exp X\right)\phi_{i}\right\rangle dX,
$$
where $g\in Z_{G}\left(F\right)H_{V}\left(F\right)\backslash G\left(F\right)$
and $\left\{ \phi_{i}\right\} _{i\in I}$ is an orthonormal basis
of $\omega_{V,\psi,\mu,\chi}$. Let 
$$
J_{V}^{\text{Lie}}\left(f\right)=\int_{Z_{G}\left(F\right)H_{V}\left(F\right)\backslash G\left(F\right)}K_{V}^{\text{Lie}}\left(f,g\right)dg.
$$
Similar to the group case, we consider the linear form $J^{\text{Lie}}\left(f\right)=J_{V}^{\text{Lie}}\left(f\right)+J_{V^{\prime}}^{\text{Lie}}
\left(f\right)$, for strongly cuspidal function $f$ on $\mathfrak{g}_{0}\left(F\right)$. By Theorem \ref{6.2}, there exists a neighborhood $\omega_{0}\subseteq\mathfrak{g}_{0}\left(F\right)$
of $0$ such that $J^{\text{Lie}}\left(f\right)$ is absolutely convergent
when $f\in{\mathcal{C}}_{\text{scusp}}\left(\omega_{0}\right)$.
\begin{proposition}\label{7.4}
\begin{enumerate}
\item Let $\omega\subseteq\omega_{0}$ be a $G$-excellent open neighborhood
of $0$ and we set $\Omega=Z_{G}\left(F\right)\exp\left(\omega\right)$.
For any $f\in{\mathcal{C}}_{\text{scusp}}\left(Z_{G}\left(F\right)\backslash\Omega,\chi\right)$,
we have  
$$
J_{V}\left(f\right)=J_{V}^{\text{Lie}}\left(f_{\omega}\right),
$$
where $f_{\omega}$ is defined in Section \ref{sec3.1}.
\item There exists $c_{\chi}\in\mathbb{C}$ such that for any $\theta\in QC\left(Z_{G}\left(F\right)\backslash G\left(F\right),\chi\right)$,
we have 
$$
J_{\text{qc}}\left(\theta\right)=c_{\chi}\cdot c_{\theta}\left(1\right).
$$
\end{enumerate}
\end{proposition}

\begin{proof}
Let $\omega_{\mathfrak{h}}=\omega\cap\mathfrak{h}_{0}\left(F\right)\subseteq\mathfrak{h}_{0}\left(F\right)$.
Then $\omega_{\mathfrak{h}}$ is an $H$-excellent open neighborhood
of $0$. We have  
$$
\int_{Z_{H}\left(F\right)\backslash H\left(F\right)}f\left(g^{-1}hg\right)\left\langle \phi_{i},\omega_{V,\psi,\mu,\chi}\left(h\right)\phi_{i}\right\rangle dh
=\int_{\omega_{\mathfrak{h}}}j^{H}\left(X\right)f\left(g^{-1}e^{X}g\right)\left\langle \phi_{i},\omega_{V,\psi,\mu,\chi}\left(\exp X\right)\phi_{i}\right\rangle dh
$$
$$
=\int_{\mathfrak{h}_{0}\left(F\right)}f_{\omega}\left(g^{-1}e^{X}g\right)\left\langle \phi_{i},\omega_{V,\psi,\mu,\chi}\left(\exp X\right)\phi_{i}\right\rangle dh,
$$
for any $f\in{\mathcal{C}}_{\text{scusp}}\left(Z_{G}\left(F\right)\backslash\Omega,\chi\right)$
and $g\in G\left(F\right)$. Therefore, we obtain $J_{V}\left(f\right)=J_{V}^{\text{Lie}}\left(f_{\omega}\right)$.

There remains to prove part (ii). Let $\omega\subseteq\omega_{0}$ be a $G$-excellent open neighborhood
of $0$ and set $\Omega=Z_{G}\left(F\right)\exp\left(\omega\right)$.
By Proposition \ref{7.3}, we only need to consider $\theta\in QC\left(Z_{G}\left(F\right)\backslash\Omega,\chi\right)$.
Observe $J_{\text{qc}}\left(\theta\right)=\underset{{\mathcal{O}}\in\text{Nil}\left(g_{0}\right)}{\sum}c_{{\mathcal{O}}}\cdot c_{\theta,{\mathcal{O}}}\left(1\right)$. Let $f\in{\mathcal{C}}_{\text{scusp}}\left(Z_{G}\left(F\right)\backslash\Omega,\chi\right)$
such that $\theta_f=\theta$. By Lemma \ref{7.2} and part (i),
we have $J_{\text{qc}}\left(\theta\right)=J(f)=J^{\text{Lie}}\left(f_{\omega}\right)$. By Corollary \ref{4.5}, there is a minimal nilpotent orbit ${\mathcal{O}}_{m}$ in $\mathfrak{h}_{V,0}(F)$ such that $\text{WF}(\omega_{V,\psi,\mu,\chi})=\overline{\mathcal{O}_m}$. Let $S_{V}$ be an element in ${\mathcal{O}}_{m}$. We can shrink $\omega$ so that for any
$f\in{\mathcal{C}}_{\text{scusp}}\left(\omega\right)$, we have the following
local character expansion 
$$
\underset{i}{\sum}\int_{\mathfrak{h}_{V,0}\left(F\right)}f\left(g^{-1}Xg\right)\left\langle \phi_{i},\omega_{V,\psi,\mu,\chi}\left(\exp X\right)\phi_{i}\right\rangle dX
=c_{0,V}\int_{\mathfrak{h}_{V,0}\left(F\right)}f\left(g^{-1}Xg\right)dX+\int_{\mathcal{O}_{m}\left(F\right)}\widehat{^{g}f\mid_{\mathfrak{h}_{0}\left(F\right)}}\left(X\right)dX
$$
$$
=c_{0,V}\int_{\mathfrak{h}_{V,0}\left(F\right)}f\left(g^{-1}Xg\right)dX+D^{H_{V}}\left(S_{V}\right)^{1/2}\int_{\left(H_{V}\right)_{S_{V}}\left(F\right)\backslash H\left(F\right)}\widehat{^{g}f\mid_{\mathfrak{h}_{0}\left(F\right)}}\left(x^{-1}S_{V}x\right)dx,
$$
where the above Fourier transform is over $\mathfrak{h}_{0}\left(F\right)$.
Let $B\left(\cdot,\cdot\right)$ be the canonical nondegenerate bilinear
form on $\mathfrak{g}_{0}\left(F\right)$. Let $\Sigma_{V}\left(F\right)=\mathfrak{h}_{V,0}\left(F\right)^{\perp}$
in $\mathfrak{g}_{0}\left(F\right)$. For any $\phi\in C_{c}^{\infty}\left(\mathfrak{g}_{0}\left(F\right)\right)$,
we have  
$$
\int_{\mathfrak{h}_{V,0}\left(F\right)}\phi\left(X\right)\psi(B\left(X,Y\right))dX=\int_{\Sigma_{V}\left(F\right)}\widehat{\phi}\left(X+Y\right)d\mu_{\Sigma_{V}}\left(X\right).
$$
This gives us 
$$
J_{V}^{\text{Lie}}\left(f\right)=c_{0,V}\int_{Z_{G}\left(F\right)H_{V}\left(F\right)\backslash G\left(F\right)}\int_{\mathfrak{h}_{V,0}\left(F\right)}f\left(g^{-1}Xg\right)dXdg+D^{H_{V}}\left(S_{V}\right)^{1/2}
$$
$$
\int_{Z_{G}\left(F\right)H_{V}\left(F\right)\backslash G\left(F\right)}\int_{\left(H_{V}\right)_{S_{V}}\left(F\right)\backslash H\left(F\right)}\int_{\Sigma_{V}\left(F\right)}\widehat{^{g}f}\left(X+x^{-1}S_{V}x\right)d\mu_{\Sigma_{V}}Xdxdg
$$
$$
=c_{0,V}\int_{Z_{G}\left(F\right)H_{V}\left(F\right)\backslash G\left(F\right)}\int_{\mathfrak{h}_{V,0}\left(F\right)}f\left(g^{-1}Xg\right)dXdg
$$
$$
+D^{H_{V}}\left(S_{V}\right)^{1/2}\int_{Z_{G}\left(F\right)\left(H_{V}\right)_{S_{V}}\left(F\right)\backslash G\left(F\right)}\int_{\Sigma_{V}\left(F\right)+S_{V}}\widehat{f}\left(g^{-1}Xg\right)d\mu_{\Sigma_{V}}Xdg.
$$
For $\lambda\in F^{\times}$, we denote $f_{\lambda}\left(X\right)=f\left(\lambda^{-1}X\right)$. If we take $\lambda\in{\mathcal{O}}_{F}^{\times}$, then $\text{Supp}\left(f_{\lambda}\right)$ is still contained in $\omega$. We have  
$$
J_{V}^{\text{Lie}}\left(f_{\lambda}\right)=\left|\lambda\right|^{n^{2}-1}c_{0,V}\int_{Z_{G}\left(F\right)H_{V}\left(F\right)\backslash G\left(F\right)}\int_{\mathfrak{h}_{V,0}\left(F\right)}f\left(g^{-1}Xg\right)dXdg
$$
$$
+\left|\lambda\right|^{n^{2}-n}D^{H_{V}}\left(S_{V}\right)^{1/2}\int_{Z_{G}\left(F\right)\left(H_{V}\right)_{S_{V}}\left(F\right)\backslash G\left(F\right)}\int_{\Sigma_{V}\left(F\right)+S_{V}}\widehat{f}\left(g^{-1}Xg\right)d\mu_{\Sigma_{V}}Xdg.
$$
Similarly, for $J_{V^{\prime}}^{\text{Lie}}$, we can also obtain
$$
J_{V^\prime}^{\text{Lie}}\left(f_{\lambda}\right)=\left|\lambda\right|^{n^{2}-1}c_{0,V^{\prime}}\int_{Z_{G}\left(F\right)H_{V^{\prime}}\left(F\right)\backslash G\left(F\right)}\int_{\mathfrak{h}_{V^{\prime},0}\left(F\right)}f\left(g^{-1}Xg\right)dXdg
$$
$$
+\left|\lambda\right|^{n^{2}-n}D^{H_{V^{\prime}}}\left(S_{V^{\prime}}\right)^{1/2}\int_{Z_{G}\left(F\right)\left(H_{V^{\prime}}\right)_{S_{V^{\prime}}}\left(F\right)\backslash G\left(F\right)}\int_{\Sigma_{V^{\prime}}\left(F\right)+S_{V^{\prime}}}\widehat{f}\left(g^{-1}Xg\right)d\mu_{\Sigma_{V^{\prime}}}Xdg.
$$
Let $\theta_{0}\in QC_{c}\left(\omega\right)$ such that $\theta=\theta_{0}\circ\exp^{-1}$
in $QC\left(Z_{G}\left(F\right)\backslash \Omega,\chi\right)$. It follows that $J_{\text{qc}}\left(\theta\right)=\underset{{\mathcal{O}}\in\text{Nil}\left(g_{0}\right)}{\sum}c_{{\mathcal{O}}}\cdot c_{\theta_{0},{\mathcal{O}}}\left(0\right)$. For $\lambda\in{\mathcal{O}}_{F}^{\times}$, we denote $\theta_{0,\lambda}\left(X\right)=\theta\left(\lambda^{-1}X\right)$
and $\theta_{\lambda}=\theta_{0,\lambda}\circ\exp^{-1}$. Then 
$$
J_{\text{qc}}\left(\theta_{\lambda}\right)=\underset{{\mathcal{O}}\in\text{Nil}\left(g_{0}\right)}{\sum}c_{{\mathcal{O}}}\cdot c_{\theta_{0,\lambda},{\mathcal{O}}}\left(0\right)=\underset{{\mathcal{O}}\in\text{Nil}\left(g_{0}\right)}{\sum}\left|\lambda\right|^{\frac{\dim\left({\mathcal{O}}\right)}{2}}c_{{\mathcal{O}}}\cdot c_{\theta_{0},{\mathcal{O}}}\left(0\right).
$$
On the other hand, since $\theta_{f_{\omega,\lambda}}\circ\exp^{-1}=\theta_{\lambda}$,
we have  
$$
J_{\text{qc}}\left(\theta_{\lambda}\right)= J^{\text{Lie}}\left(f_{\omega,\lambda}\right)
=\left|\lambda\right|^{n^{2}-1}c_{0,V}\int_{Z_{G}\left(F\right)H_{V}\left(F\right)\backslash G\left(F\right)}\int_{\mathfrak{h}_{V,0}\left(F\right)}f_\omega\left(g^{-1}Xg\right)dXdg
$$
$$
+\left|\lambda\right|^{n^{2}-n}D^{H_{V}}\left(S_{V}\right)^{1/2}\int_{Z_{G}\left(F\right)\left(H_{V}\right)_{S_{V}}\left(F\right)\backslash G\left(F\right)}\int_{\Sigma_{V}\left(F\right)+S_{V}}\widehat{f_\omega}\left(g^{-1}Xg\right)d\mu_{\Sigma_{V}}Xdg
$$
$$
+\left|\lambda\right|^{n^{2}-1}c_{0,V^{\prime}}\int_{Z_{G}\left(F\right)H_{V^{\prime}}\left(F\right)\backslash G\left(F\right)}\int_{\mathfrak{h}_{V^{\prime},0}\left(F\right)}f_\omega\left(g^{-1}Xg\right)dXdg
$$
$$
+\left|\lambda\right|^{n^{2}-n}D^{H_{V^{\prime}}}\left(S_{V^{\prime}}\right)^{1/2}\int_{Z_{G}\left(F\right)\left(H_{V^{\prime}}\right)_{S_{V^{\prime}}}\left(F\right)\backslash G\left(F\right)}\int_{\Sigma_{V^{\prime}}\left(F\right)+S_{V^{\prime}}}\widehat{f_\omega}\left(g^{-1}Xg\right)d\mu_{\Sigma_{V^{\prime}}}Xdg.
$$
By comparing the above two expression, since $\dim(\mathcal{O})\leq \dim(\mathcal{O}_{reg})=2(n^2-n)$, it follows that $c_{{\mathcal{O}}}=0$,
for any non-regular nilpotent orbit ${\mathcal{O}}$. Moreover
\begin{equation}\label{eqn6.2}
J_{\text{qc}}\left(\theta\right)=D^{H_{V}}\left(S_{V}\right)^{1/2}\int_{Z_{G}\left(F\right)\left(H_{V}\right)_{S_{V}}\left(F\right)\backslash G\left(F\right)}\int_{\Sigma_{V}\left(F\right)+S_{V}}\widehat{f_\omega}\left(g^{-1}Xg\right)d\mu_{\Sigma_{V}}Xdg
\end{equation}
$$
+D^{H_{V^{\prime}}}\left(S_{V^{\prime}}\right)^{1/2}\int_{Z_{G}\left(F\right)\left(H_{V^{\prime}}\right)_{S_{V^{\prime}}}\left(F\right)\backslash G\left(F\right)}\int_{\Sigma_{V^{\prime}}\left(F\right)+S_{V^{\prime}}}\widehat{f_\omega}\left(g^{-1}Xg\right)d\mu_{\Sigma_{V^{\prime}}}Xdg.
$$
Since $G=\text{Res}_{E/F}\text{GL}\left(V\right)$ has only one regular nilpotent orbit, there exists $c_\chi\in \mathbb{C}$ such that $J_{\text{qc}}\left(\theta\right)= c_\chi \cdot c_{\theta}(1)$.
\end{proof}

\subsection{End of the proof of Theorem \ref{7.1} and the strong multiplicity one of the local twisted GGP conjecture}\label{sec7.4}
We prove Theorem \ref{7.1}.
\begin{proof}
    According to Proposition \ref{7.4}(ii), there remains to show the coefficient
$c_{\chi}$ is 1. By Proposition \ref{7.4}(ii), it follows that 
$$
\underset{\pi\in \Pi_2\left(G,\chi^{-1}\right)}{\sum}(m\left(\bar{\pi}\right)-1)\int_{\Gamma_{\text{ell}}\left(Z_{G}\backslash G\right)}D^{G}\left(x\right)\theta\left(x\right)\theta_{\pi}\left(x\right)dx=(c_{\chi}-1)\cdot c_{\theta}\left(1\right),
$$
for any $\theta\in QC\left(Z_{G}\left(F\right)\backslash G\left(F\right),\chi\right)$.
By \cite[Corollary 5.8]{GGP23}, it follows that $m\left(\text{St}\right)=1$,
where $\text{St}$ is the Steinberg representation of $G\left(F\right)$.
By substituting $\theta=\theta_{\text{St}}$, one has $c_{1}=1$.

We now consider the case $\chi\neq1$. We take motivations from the
proof of Proposition 5.2.3 in \cite{Wan17}. Let $\omega \subseteq \mathfrak{g}_0(F)$ be a $G(F)$-excellent neighborhood of $0$ (see Section \ref{sec3.1}) and $\mathcal{U}=\exp(\omega)\subseteq Z_G(F)\backslash G(F)$. Assume further that $\mathcal{U}$ admits a section $s:{\mathcal{U}}\rightarrow G\left(F\right)$ preserving $G(F)$-conjugation such that it maps semisimple elements in $Z_G(F)\backslash G(F)$ to semisimple elements in $G(F)$ and the map 
$$
Z_G(F)\times \text{Im}(s)\longrightarrow Z_G(F)\text{Im}\left(s\right):(s,x)\mapsto sx
$$
is bijective. We set $X=Z_{G}\left(F\right)\text{Im}\left(s\right)$. For any $f\in C_{c}^{\infty}\left(Z_{G}\left(F\right)\backslash X,\chi\right)$,
we can construct a function $f^{\prime}\in C_{c}^{\infty}\left(Z_{G}\left(F\right)\backslash X\right)$
such that $f\left(zg\right)=\chi\left(z\right)f^{\prime}\left(g\right)$,
for any $z\in Z_{G}\left(F\right)$ and $g\in \text{Im}(s)$. 

We can assume $f$ to be strongly cuspidal. We want to show that $f^{\prime}$ is strongly cuspidal. Let $P=MU$ be a proper parabolic subgroup. We consider 2 elements $mu$ and $mu^\prime$ in $P(F)\cap X$. Then there exists unique $z,z^\prime\in Z_G(F)$ and $x,x^\prime \in \text{Im}(s)$ such that $mu=zx$ and $mu^\prime=z^\prime x^\prime$. Observe that $x,x^\prime \in P(F)$. By Levi decomposition, we can factorize $x=x_Mx_U$ and $x^\prime=x_M^\prime x^\prime_U$, where $x_M,x_M^\prime\in M(F)$ and $x_U,x_U^\prime \in U(F)$. Hence, we have $m=zx_M=z^\prime x_M^\prime$. Since $\omega$ is $G(F)$-excellent (in particular it is completely $G(F)$-invariant) and $s$ is semisimple, it follows that $x_M,x_M^\prime\in \text{Im}(s)$. Since $zx_M=z^\prime x_M^\prime$ and $Z_G(F)\times \text{Im}(s)\cong Z_G(F)\text{Im}(s)$, we have $z=z^\prime$ and $x_M=x_M^\prime$. Thus, whenever $mu=zx\in P(F)\cap X$, the element $z\in Z_G(F)$ only depends on $m$. Therefore, for any $m\in M(F)$, we have 
$$
\int_{U(F)}f^\prime(mu)du = \chi(z(m))^{-1}\int_{U(F)}f(mu)du=0,
$$
where $z(m)\in Z_G(F)$ only depends on $m$. This gives us $f^\prime$ is strongly cuspidal. From the definition of $f$ and $f^\prime$, it follows that $c_{\theta_{f}}\left(1\right)=c_{\theta_{f^{\prime}}}\left(1\right)$. To distinguish the change of central characters, we add indexes $\chi$ and $1$ to all the linear forms involving central characters $\chi$ and $1$. Since we have shown $J_{\text{qc},1}\left(\theta_{f^{\prime}}\right)=c_{\theta_{f^{\prime}}}\left(1\right)$,
there remains to prove $J_{\text{qc},\chi}\left(\theta_{f}\right)=J_{\text{qc},1}\left(\theta_{f^{\prime}}\right)$. 
It suffices to show $J_\chi(f)=J_1(f^\prime)$. Observe 
$$
J_{V,\chi}(f)=\int_{Z_G(F)H_V(F)\backslash G(F)}\int_{Z_H(F)\backslash H(F)}f(x^{-1}hx)\Theta_{\omega_{V,\psi,\mu,\chi}}(h)dhdx
$$
$$
=\int_{Z_G(F)H_V(F)\backslash G(F)}\int_{ H(F)\cap X}f(x^{-1}hx)\Theta_{\omega_{V,\psi,\mu,\chi}}(h)dhdx
$$
$$
=\int_{Z_G(F)H_V(F)\backslash G(F)}\int_{Z_H(F)\backslash H(F)\cap \mathcal{U}}f^\prime(x^{-1}s(h)x)\Theta_{\omega_{V,\psi,\mu,\chi}}(s(h))dhdx
$$
$$
=\int_{Z_G(F)H_V(F)\backslash G(F)}\int_{\omega \cap \mathfrak{h}_{V,0}}(f^\prime\circ s)_\omega(x^{-1}Xx)\Theta_{\omega_{V,\psi,\mu,\chi}}(\exp X)dhdx,
$$
noting that the last line follows from Proposition \ref{7.4}. By (\ref{eqn6.2}) in the proof of Proposition \ref{7.4}, we have
$$
J_\chi(f)=\int_{Z_G(F)H_V(F)\backslash G(F)}\int_{\omega \cap \mathfrak{h}_{V,0}}(f^\prime\circ s)_\omega(x^{-1}Xx)\Theta_{\omega_{V,\psi,\mu,\chi}}(\exp X)dhdx
$$
$$
+\int_{Z_G(F)H_V(F)\backslash G(F)}\int_{\omega \cap \mathfrak{h}_{V,0}}(f^\prime\circ s)_\omega(x^{-1}Xx)\Theta_{\omega_{V,\psi,\mu,\chi}}(\exp X)dhdx
$$
$$
=D^{H_{V}}\left(S_{V}\right)^{1/2} \int_{Z_{G}\left(F\right)\left(H_{V}\right)_{S_{V}}\left(F\right)\backslash G\left(F\right)}\int_{\Sigma_{V}\left(F\right)+S_{V}}\widehat{(f^\prime\circ s)_\omega}\left(g^{-1}Xg\right)d\mu_{\Sigma_{V}}Xdg
$$
$$
+D^{H_{V^{\prime}}}\left(S_{V^{\prime}}\right)^{1/2}\int_{Z_{G}\left(F\right)\left(H_{V^{\prime}}\right)_{S_{V^{\prime}}}\left(F\right)\backslash G\left(F\right)}\int_{\Sigma_{V^{\prime}}\left(F\right)+S_{V^{\prime}}}\widehat{(f^\prime\circ s)_\omega}\left(g^{-1}Xg\right)d\mu_{\Sigma_{V^{\prime}}}Xdg.
$$
A similar argument shows that 
$$
J_1(f^\prime)=D^{H_{V}}\left(S_{V}\right)^{1/2} \int_{Z_{G}\left(F\right)\left(H_{V}\right)_{S_{V}}\left(F\right)\backslash G\left(F\right)}\int_{\Sigma_{V}\left(F\right)+S_{V}}\widehat{(f^\prime\circ s)_\omega}\left(g^{-1}Xg\right)d\mu_{\Sigma_{V}}Xdg
$$
$$
+D^{H_{V^{\prime}}}\left(S_{V^{\prime}}\right)^{1/2}\int_{Z_{G}\left(F\right)\left(H_{V^{\prime}}\right)_{S_{V^{\prime}}}\left(F\right)\backslash G\left(F\right)}\int_{\Sigma_{V^{\prime}}\left(F\right)+S_{V^{\prime}}}\widehat{(f^\prime\circ s)_\omega}\left(g^{-1}Xg\right)d\mu_{\Sigma_{V^{\prime}}}Xdg,
$$
which is to say $J_\chi(f)=J_1(f^\prime)$. Therefore $J_{\text{qc},\chi}(\theta_f)=J_{\text{qc},1}(\theta_{f^\prime})=c_{\theta_{f^\prime}}(1)=c_{\theta_f}(1)$, which implies $c_\chi=1$. This ends our proof for Theorem \ref{7.1}.
\end{proof}
We now prove Theorem \ref{thm1.2}(i), which finish the induction hypothesis stated at the beginning of section \ref{sec5}.
\begin{proof}
If $\Pi$ is a tempered principal series of $G\left(F\right)$,
by Theorem 4.8 in \cite{GGP23} and the induction hypothesis,
we have $m\left(\Pi\right)=1.$
Suppose $\Pi$ is a discrete series of $G\left(F\right)$ with
central character $\chi$. By Theorem \ref{7.1}, we have  
$$
\underset{\pi\in\Pi_2\left(G,\chi^{-1}\right)}{\sum}(m\left(\bar{\pi}\right)-1)\int_{\Gamma_{\text{ell}}\left(Z_{G}\backslash G\right)}D^{G}\left(x\right)\theta\left(x\right)\theta_{\pi}\left(x\right)dx=0.
$$
By taking $\theta=\theta_{\bar{\Pi}}$, it follows that $m\left(\Pi\right)=1$. We have finished our proof for Theorem \ref{thm1.2}(i).
\end{proof}

\section{A twisted local trace formula and tempered intertwinings}\label{sec8}
\subsection{Twisted endoscopy and matching orbits}\label{sec8.1}

Let $\left(M,\tilde{M}\right)$ be a twisted group. Assume $M$ is split and fix an element $\tilde{\theta}\in\tilde{M}\left(F\right)$.
We assume further there exists a pinning of $M$ defined over $F$ which is invariant
under $\theta_{\tilde{\theta}}$. We fix a regular nilpotent orbit
of $\mathfrak{m}_{\tilde{\theta}}\left(F\right)$. Let $\left(G,s,^{L}\xi\right)$
be an endoscopic triple of $\left(M,\tilde{M}\right)$ in sense of
section 2.1 in \cite{KS99}. Let $G_{\text{reg}}\left(F\right)/\text{stconj}$
and $\tilde{M}_{\text{reg}}\left(F\right)/\text{stconj}$ be sets
of stable conjugacy classes of regular semisimple elements in $G\left(F\right)$ and $M\left(F\right)$,
respectively. As described in \cite{KS99}, we have a twisted endoscopic
correspondence from a subset of $G_{\text{reg}}\left(F\right)/\text{stconj}$,
which is called $\tilde{M}$-strongly regular, to $\tilde{M}_{\text{reg}}\left(F\right)/\text{stconj}$.
We then have the transfer factor $\Delta_{G,\tilde{M}}\left(y,\tilde{x}\right)$
defined on $G_{\text{reg}}\left(F\right)/\text{stconj}\times\tilde{M}_{\text{reg}}\left(F\right)/\text{conj}$,
which is zero if $\tilde{x}$ and $y$ are not corresponding. Let
$\tilde{\Theta}$ be a conjugate-invariant locally integrable distribution
of $\tilde{M}\left(F\right)$ and $\Theta^{G}$ be a stably conjugate-invariant
locally integrable distribution of $G\left(F\right)$. Then $\tilde{\Theta}$
is called a transfer of $\Theta^{G}$ if $\tilde{\Theta}\left(\tilde{x}\right)D^{\tilde{M}}\left(\tilde{x}\right)^{1/2}=\underset{y}{\sum}\Theta^{G}\left(y\right)D^{G}\left(y\right)^{1/2}\Delta_{G,\tilde{M}}\left(y,\tilde{x}\right)$, for any $\tilde{x}\in\tilde{M}_{\text{reg}}\left(F\right)/\text{conj}$, where the sum is over $y\in G_{\text{reg}}\left(F\right)/\text{stconj}$ corresponding to $\tilde{x}$.

We now describe twisted endoscopy in our setting. Let $E/F$ be a quadratic extension of $p$-adic fields. Let $\sigma$
be the nontrivial element in $\text{Gal}\left(E/F\right)$ and we
denote $\bar{x}=\sigma\left(x\right)$ for any $x\in E$. We consider $G=\text{Res}_{E/F}\text{GL}_{n}\overset{\Delta}{\longrightarrow}M=\text{Res}_{E/F}\text{GL}_{n}\times\text{GL}_{n}$. Let 
$\theta_{n}:\left(g,h\right)\mapsto\left(J_{n}{}^{t}\bar{h}^{-1}J_{n}^{-1},J_{n}{}^{t}\bar{g}^{-1}J_{n}^{-1}\right)$ be an involution on $M$, where 
$J_{n}=\left(\begin{array}{cccc}
0 & \cdots & 0 & -1\\
\vdots & 0 & 1 & 0\\
0 & \cdots & 0 & \vdots\\
\left(-1\right)^{n} & 0 & \cdots & 0
\end{array}\right)$. We denote $\tilde{M}=M\theta_{n}$. Let $G_{\text{ell}}\left(F\right)$
and $\tilde{M}_{\text{ell}}\left(F\right)$ be the subsets of
regular elliptic elements of $G\left(F\right)$ and $\tilde{M}\left(F\right)$.
We have the following bijection from $G_{\text{ell}}\left(F\right)/\text{conj}$
to $\tilde{M}_{\text{ell}}\left(F\right)/\text{conj}$
$$
\begin{array}{ccc}
G_{\text{ell}}\left(F\right)/\text{conj} & \longleftrightarrow & \tilde{M}_{\text{ell}}\left(F\right)/\text{conj}\\
x & \mapsto & \left(x,1\right)\theta_{n}
\end{array},
$$
whose inverse is $\left(x,y\right)\theta_{n} \mapsto xJ_n\,^{t}\bar{y}^{-1}J_n^{-1}$. The transfer factor $\Delta_{G,\tilde{M}}\left(y,\tilde{x}\right)$
is equal to $1$ whenever $y$ and $\tilde{x}$ are corresponding.
Moreover, in this case, the transfer is well-defined at the level of elliptic germs defined in Definition \ref{3.3}.

\subsection{Twisted endoscopic character identity}\label{sec8.2}

In this subsection, we prove a twisted endoscopic character identity
between irreducible elliptic representations of $G\left(F\right)$
and $\tilde{M}\left(F\right)$. Let $\pi$ be an irreducible elliptic
representation of $G\left(F\right)$. Then $\pi\times{}^{\sigma}\pi^{\vee}$
is an irreducible elliptic representation of $M\left(F\right)$ which
can be extended to a representation $\widetilde{\pi\times{}^{\sigma}\pi^{\vee}}$
of $\tilde{M}\left(F\right)$. From now, we fix an extension $\widetilde{\pi\times{}^{\sigma}\pi^{\vee}}\left(\theta_{n}\right)\left(v\otimes w\right)=\pi\left(J_{n}\right)w\otimes{}^\sigma\pi^\vee\left(J_{n}\right)v$. In this case, we can see that $w(\widetilde{\pi\times{}^{\sigma}\pi^{\vee}},\psi)=1$. We prove the following character identity.
\begin{theorem}\label{8.1}
Let $\pi$ be an irreducible elliptic representation of $G\left(F\right)$.
By taking the normalization at the beginning of this subsection, we have 
$$
\Theta_{\pi}\left(y\right)D^{G}\left(y\right)^{1/2}=\Theta_{\widetilde{\pi\times{}^{\sigma}\pi^{\vee}}}\left(\tilde{x}\right)D^{\tilde{M}}\left(\tilde{x}\right)^{1/2},
$$
for any $y$ and $\tilde{x}$ are corresponding.
\end{theorem}

\begin{proof}
Let $f_{1}\otimes f_{2}\in C_{c}^{\infty}\left(M\left(F\right)\right)\simeq C_{c}^{\infty}\left(\tilde{M}\left(F\right)\right)$
be a strongly cuspidal function. We define a function $f\in C_{c}^{\infty}\left(G\left(F\right)\right)$
by $f\left(x\right)=\int_{G\left(F\right)}f_{1}\left(xy\right)f_{2}\left(J_{n}{}^{t}\bar{y}J_{n}^{-1}\right)dy$, for $x\in G\left(F\right)$. Let $\gamma$ be an element in $G_{\text{ell}}\left(F\right)_{\text{reg}}$ and $\tilde{\gamma}=\left(\gamma,1\right)\theta_{n}\in \tilde{M}(F)$. Define 
$$
O_{\gamma}\left(f\right)=\int_{G_{\gamma}\left(F\right)\backslash G\left(F\right)}f\left(g^{-1}\gamma g\right)dg\text{ and }O_{\tilde{\gamma}}\left(f_{1}\otimes f_{2}\right)=\int_{M_{\tilde{\gamma}}\left(F\right)\backslash M\left(F\right)}\left(f_{1}\otimes f_{2}\right)\left(m^{-1}\tilde{\gamma}m\right)dm.
$$
For any $\gamma\in G_{\text{ell}}\left(F\right)_\text{reg}$, we have 
$$
O_{\tilde{\gamma}}\left(f_{1}\otimes f_{2}\right)=\iint_{G_{\gamma}\left(F\right)\backslash G\left(F\right)\times G\left(F\right)}f_{1}\left(g_{1}^{-1}\gamma J_{n}{}^{t}\bar{g}_{2}^{-1}J_{n}^{-1}\right)f_{2}\left(g_{2}^{-1}J_{n}{}^{t}\bar{g}_{1}^{-1}J_{n}^{-1}\right)dg_{2}dg_{1}
$$
$$
=\int_{G_{\gamma}\left(F\right)\backslash G\left(F\right)}\int_{G\left(F\right)}f_{1}\left(g_{1}^{-1}\gamma g_{1}J_{n}{}^{t}\bar{g}_{2}^{-1}J_{n}^{-1}\right)f_{2}\left(g_{2}^{-1}\right)dg_{2}dg_{1},
$$
noting that $J_{n}^{2}=\left(-I_{n}\right)^{n}\in Z_{G}\left(F\right)$.
Then by definition 
$$
O_{\tilde{\gamma}}\left(f_{1}\otimes f_{2}\right)=\int_{G_{\gamma}\left(F\right)\backslash G\left(F\right)}f\left(g_{1}^{-1}\gamma g_{1}\right)dg_{1}=O_{\gamma}\left(f\right).
$$
By the corollary in \cite[Section 4.4]{MW16}, it suffices to show 
$\text{Trace}\left(\pi\left(f\right)\right)=\text{Trace}\left(\widetilde{\pi\times{}^{\sigma}\pi^{\vee}}\right)\left(f_{1}\otimes f_{2}\right)$. Let $v\otimes w\in\pi\times{}^{\sigma}\pi^{\vee}$. We have 
$$
\left(\widetilde{\pi\times{}^{\sigma}\pi^{\vee}}\right)\left(f_{1}\otimes f_{2}\right)\left(v\otimes w\right)=\left(\int_{G\left(F\right)}f_{1}\left(x\right)\pi\left(x\right)wdx\right)\otimes\left(\int_{G\left(F\right)}f_{2}\left(y\right)\pi\left(^{t}\bar{y}^{-1}\right)vdx\right).
$$
Assume $f_{1}$ and $f_{2}$ are bi-invariant by an open compact subgroup
$K$ of $G\left(F\right)$. We can assume $K=J_{n}\bar{K}^{t,-1}J_{n}^{-1}$.
Let $\left\{ v_{i}\right\} _{i=\overline{1,N}}$ be a basis for $\pi^{K}$.
We have 
$$
\text{Trace}\left(\left(\widetilde{\pi\times{}^{\sigma}\pi^{\vee}}\right)\left(f_{1}\otimes f_{2}\right)\right)
=\underset{i,j}{\sum}\left\langle \int_{G\left(F\right)}f_{1}\left(x\right)\pi\left(x\right)\pi\left(J_{n}\right)v_{j}dx,\,v_{i}\right\rangle \left\langle v_{i}\,,\int_{G\left(F\right)}f_{2}\left(y\right)\pi\left(^{t}\bar{y}^{-1}\right)\pi\left(J_{n}\right)v_{j}dx\right\rangle 
$$
$$
=\underset{i,j}{\sum}\left\langle \int_{G\left(F\right)}f_{1}\left(xJ_{n}^{-1}\right)\pi\left(x\right)v_{j}dx,\,v_{i}\right\rangle \left\langle v_{i}\,,\int_{G\left(F\right)}f_{2}\left(^{t}\bar{y}^{-1}J_{n}^{-1}\right)\pi\left(y\right)v_{j}dx\right\rangle 
$$
$$
=\text{Tr}\left(\pi\left(f_{1}\left(\cdot J_{n}^{-1}\right)\right)\circ\pi\left(f_{2}\left(^{t}\left(\bar{\cdot}\right)^{-1}J_{n}^{-1}\right)\right)\right)
=\text{Tr}\left(\pi\left(f_{1}\left(\cdot J_{n}^{-1}\right)*f_{2}\left(^{t}\left(\bar{\cdot}\right)^{-1}J_{n}^{-1}\right)\right)\right)=\text{Tr}\left(\pi\left(f\right)\right),
$$
where $\left(f_{1}*f_{2}\right)\left(x\right)=\int_{G\left(F\right)}f_{1}\left(xy\right)f_{2}\left(y^{-1}\right)dy$ is the usual convolution.
\end{proof}

\subsection{A local trace formula for $\epsilon$-factors}\label{sec8.3}

Let $V$ be an $n$-dimensional vector space over $E$. We fix a basis for $V$. Then $G=\text{Res}_{E/F}\text{GL}(V)$ can be identified with $\text{Res}_{E/F}\text{GL}_n$. The restriction of $\theta_{n}$ to $G$ is $g\mapsto J_{n}{}^{t}\bar{g}^{-1}J_{n}^{-1}$,
for which we still use the same notation. Let $\tilde{G}=G\theta_{n}$.
Then $\left(G,\tilde{G}\right)$ forms a twisted group. Let ${\mathcal{S}}\left(V\right)$
be the space of Schwarz functions on $V$. We define the Weil representation
$\omega_{\mu}$ of $G\left(F\right)$ realized on ${\mathcal{S}}\left(V\right)$
by $\left(\omega_{\mu}\left(g\right)\phi\right)\left(v\right)=\left|\det g\right|^{\frac{1}{2}}\mu\left(\det g\right)\phi\left(vg\right)$, for any $g\in G\left(F\right)$ and $\phi\in{\mathcal{S}}\left(V\right)$. Since $\omega_\mu$ is $\theta_n$-invariant, it can be extended to a representation of $\tilde{G}\left(F\right)$.
We fix an extension $\tilde{\omega}_{\psi,\mu}$ by letting
$\tilde{\omega}_{\psi,\mu}\left(\theta_{n}\right)\phi=\widehat{\overline{\phi}\left(\cdot J_{n}\right)}$,
where $\bar{\phi}\left(v\right)=\phi\left(\bar{v}\right)$ induced by the action of $\text{Gal}(E/F)$ on $\text{Res}_{E/F}V$ (depending on our choice of $E$-basis) and $\hat{\phi}$
is the Fourier transform of $\phi$ with respect to $\psi_{E}=\psi\circ\text{Tr}_{E/F}$.
We fix a central character $\chi\times {}^\sigma\chi^\vee$ of $M\left(F\right)$. The restriction of $\chi\times {}^\sigma\chi^\vee$ to $Z_G\left(F\right)$
is $\chi \cdot {}^\sigma\chi^\vee$. Let $\omega_{\mu,\chi}$ be the $(\chi \cdot {}^\sigma\chi^\vee)$-isotypic summand of $\omega_{\mu}$. We can see that $\omega_{\mu,\chi}$ is stable under the action of $\tilde{\omega}_{\psi,\mu}(\theta_n)$. We denote by $\tilde{\omega}_{\psi,\mu,\chi}$ the corresponding extension of $\omega_{\mu,\chi}$ to a representation of $\tilde{G}(F)$. Let $\left\{ \phi_{i}\right\} _{i\in I}$
be an orthonormal basis for $\tilde{\omega}_{\psi,\mu,\chi}$ . For
any $x\in M\left(F\right)$, we set 
$$
K_{\chi}\left(\tilde{f},x\right)=\underset{i}{\sum}\int_{Z_G\left(F\right)\backslash\tilde{G}\left(F\right)}\tilde{f}\left(x^{-1}\tilde{g}x\right)\left\langle \phi_i,\tilde{\omega}_{\psi,\mu,\chi}\left(\tilde{g}\right)\phi_{i}\right\rangle d\tilde{g},
$$
where $\tilde{f}\in \mathcal{C}_\text{scusp}(Z_M(F)\backslash \tilde{M}(F),\chi \times {}^\sigma \chi^\vee)$. The above kernel function is locally constant and invariant under
$G\left(F\right)Z_M\left(F\right)$. We define the following
linear form 
$$
J_{\chi}\left(\tilde{f}\right)=\int_{G\left(F\right)Z_M\left(F\right)\backslash M\left(F\right)}K_\chi\left(\tilde{f},x\right)dx,
$$
for $\tilde{f}\in \mathcal{C}_\text{scusp}(Z_M(F)\backslash \tilde{M}(F),\chi \times {}^\sigma \chi^\vee)$. Similar to Theorem \ref{6.2}, the integrals defining $K_\chi$ and $J_\chi$ are absolutely convergent.
\subsection{Tempered intertwinings and $\epsilon$-factors}\label{sec8.4}

Let $\left(\pi,E_{\pi}\right)\in\text{Temp}\left(G\right)$. We identify $^{\sigma}\pi^\vee$ with the representation $g\mapsto \pi(^t\bar{g}^{-1})$. For $e_{1},e_{2},e_{1}^{\prime},e_{2}^{\prime}\in E_{\pi}$
and $\phi,\phi^{\prime}\in\omega_{\mu}$, we define 
$$
{\mathcal{L}}_{\pi}\left(e_{1}\otimes e_{2}\otimes\phi,e_{1}^{\prime}\otimes e_{2}^{\prime}\otimes\phi^{\prime}\right)=\int_{G\left(F\right)}\left\langle \pi\left(g\right)e_{1},e_{1}^{\prime}\right\rangle \left\langle \pi\left(^{t}\bar{g}^{-1}\right)e_{2},e_{2}^{\prime}\right\rangle \left\langle \phi,\omega_{\mu}\left(g\right)\phi^{\prime}\right\rangle dg.
$$
A similar argument to \cite[Theorem 1.1.1(1)]{Xue16} shows that the above expression is absolutely convergent. Observe 
$$
{\mathcal{L}}_{\pi}\left(\pi\left(g\right)e_{1}\otimes\pi\left(^{t}\bar{g}^{-1}\right)e_{2}\otimes\phi,e_{1}^{\prime}\otimes e_{2}^{\prime}\otimes\omega_{\mu}\left(g\right)\phi^{\prime}\right)={\mathcal{L}}_{\pi}\left(e_{1}\otimes e_{2}\otimes\phi,e_{1}^{\prime}\otimes e_{2}^{\prime}\otimes\phi^{\prime}\right),
$$
for any $g\in G\left(F\right)$, $e_{1},e_{2},e_{1}^{\prime},e_{2}^{\prime}\in E_{\pi}$
and $\phi,\phi^{\prime}\in\omega_{\mu}$. By \cite[Theorem 15.1]{GGP12a}, we have  $\dim\text{Hom}_{G}\left(\pi\otimes{}^{\sigma}\pi^{\vee},\omega_{\mu}\right)=1$.
We fix a nonzero Whittaker functional $\xi$ of $\pi$ with respect
to $\left(U_{n},\psi_{E}\right)$. Then $\xi\circ\pi\left(w_{n}\right)$
is a Whittaker functional of $^{\sigma}\pi^{\vee}$ with respect to
$\left(U_{n},\psi_{E}^{-1}\right)$, where $w_n$ is the anti-diagonal identity matrix.
For $e\in E_{\pi}$, we define
$$
\begin{array}{ccccc}
W_{e} & : & G\left(F\right) & \rightarrow & \mathbb{C}\\
 &  & g & \mapsto & \xi\left(\pi\left(g\right)e\right)
\end{array}
$$
and 
$$
\begin{array}{ccccc}
W_{e}^{-} & : & G\left(F\right) & \rightarrow & \mathbb{C}\\
 &  & g & \mapsto & \xi\left(\pi\left(w_{n}\right)\pi\left(^{t}\bar{g}^{-1}\right)e\right).
\end{array}
$$
We set  
$$
\epsilon_{\psi}\left(\widetilde{\pi\times{}^{\sigma}\pi^{\vee}}\right)=\omega_{\pi}\left(-1\right)^{n}\omega_{E/F}\left(-1\right)^{n\left(n-1\right)/2}\epsilon\left(\frac{1}{2},\pi\times{}^{\sigma}\pi^{\vee}\times\mu^{-1},\psi_{E}\right).
$$

\begin{proposition}\label{8.2}
For any $e_{1},e_{2},e_{1}^{\prime},e_{2}^{\prime}\in E_{\pi}$ and
$\phi,\phi^{\prime}\in\omega_{\mu}$ and $\tilde{y}\in\tilde{G}\left(F\right)$,
we have  
$$
{\mathcal{L}}_{\pi}\left(\widetilde{\pi\times{}^{\sigma}\pi^{\vee}}\left(\tilde{y}\right)\left(e_{1}\otimes e_{2}\right)\otimes\phi,e_{1}^{\prime}\otimes e_{2}^{\prime}\otimes\tilde{\omega}_{\psi,\mu}\left(\tilde{y}\right)\phi^{\prime}\right)=\epsilon_{\psi}\left(\widetilde{\pi\times{}^{\sigma}\pi^{\vee}}\right){\mathcal{L}}_{\pi}\left(e_{1}\otimes e_{2}\otimes\phi,e_{1}^{\prime}\otimes e_{2}^{\prime}\otimes\phi^{\prime}\right).
$$
\end{proposition}

\begin{proof}
This equality is a restatement of the local functional equation for the local Rankin-Selberg integral in \cite{JPSS83}. Indeed, let 
$$
L_{\pi}\left(e\otimes e^{\prime},\Phi,s\right)=\int_{U\left(F\right)\backslash G\left(F\right)}W_{e}\left(g\right)W_{e^{\prime}}^{-}\left(g\right)\Phi\left(e_{n}g\right)\mu\left(\det g\right)^{-1}\left|\det g\right|^{s}dg,
$$
where $s\in\mathbb{C}$ such that $\text{Re}\left(s\right)>0$ and
$e,e^{\prime}\in E_{\pi}$ and $\Phi\in\omega_{\mu}$. By \cite{JPSS83},
the above expression is absolutely convergent. The linear form $e\otimes e^{\prime}\otimes\Phi\mapsto L_{\pi}\left(e\otimes e^{\prime},\Phi,\frac{1}{2}\right)$
is $G\left(F\right)$-invariant, i.e. it belongs to $\text{Hom}_{G}\left(\pi\otimes{}^{\sigma}\pi^{\vee}\otimes\omega_{\mu}^{\vee},1\right)$.
Since $e\otimes e^{\prime}\otimes\Phi\mapsto{\mathcal{L}}_{\pi}\left(e\otimes e^{\prime}\otimes\Phi,e_{1}^{\prime}\otimes e_{2}^{\prime}\otimes\phi^{\prime}\right)$
also lies in $\text{Hom}_{G}\left(\pi\otimes{}^{\sigma}\pi^{\vee}\otimes\omega_{\mu}^{\vee},1\right)$,
there exists a constant $C_{0}$ depending on $e_{1}^{\prime}\otimes e_{2}^{\prime}\otimes\phi^{\prime}$
such that 
$$
{\mathcal{L}}_{\pi}\left(e\otimes e^{\prime}\otimes\Phi,e_{1}^{\prime}\otimes e_{2}^{\prime}\otimes\phi^{\prime}\right)=C_{0}L_{\pi}\left(e\otimes e^{\prime},\Phi,\frac{1}{2}\right),
$$
for any $e,e^{\prime}\in E_{\pi}$ and $\Phi\in\omega_{\mu}$. It
suffices to prove 
$$
L_{\pi}\left(\widetilde{\pi\times{}^{\sigma}\pi^{\vee}}\left(\tilde{y}\right)\left(e\otimes e^{\prime}\right),\tilde{\omega}_{\psi,\mu}^{\vee}\left(\tilde{y}\right)\Phi,\frac{1}{2}\right)=\epsilon_{\psi}\left(\widetilde{\pi\times{}^{\sigma}\pi^{\vee}}\right)L_{\pi}\left(e\otimes e^{\prime},\Phi,\frac{1}{2}\right).
$$
Without loss of generality, assume $\tilde{y}=\theta_{n}$. We have
$$
L_{\pi}\left(\widetilde{\pi\times{}^{\sigma}\pi^{\vee}}\left(\theta_{n}\right)\left(e\otimes e^{\prime}\right),\tilde{\omega}_{\psi,\mu}^{\vee}\left(\theta_{n}\right)\Phi,\frac{1}{2}\right)
=\int_{U\left(F\right)\backslash G\left(F\right)}W_{e^{\prime}}\left(gJ_{n}\right)W_{e}^{-}\left(gJ_{n}\right)\widehat{\overline{\Phi}\left(\cdot J_{n}\right)}\left(e_{n}g\right)\mu\left(\det g\right)^{-1}\left|\det g\right|^{\frac{1}{2}}dg.
$$
A direct computation shows that  
$W_{e^{\prime}}\left(gJ_{n}\right)W_{e}^{-}\left(gJ_{n}\right)=\xi\left(\pi\left(gJ_{n}\right)e^{\prime}\right)\xi\left(\pi\left(w_{n}{}^{t}\bar{g}^{-1}J_{n}\right)e\right)$. Let $D_{n}=\left(\begin{array}{ccc}
\left(-1\right)^{n}\\
 & \ddots\\
 &  & -1
\end{array}\right)$. Then $J_{n}=w_{n}D_{n}$. By substituting $g\mapsto gJ_{n}^{-1}$, we have 
$$
L_{\pi}\left(\widetilde{\pi\times{}^{\sigma}\pi^{\vee}}\left(\theta_{n}\right)\left(e\otimes e^{\prime}\right),\tilde{\omega}_{\psi,\mu}^{\vee}\left(\theta_{n}\right)\Phi,\frac{1}{2}\right)
$$
$$
=\int_{U\left(F\right)\backslash G\left(F\right)}W_{e}\left(w_{n}{}^{t}\bar{g}^{-1}\right)W_{e^{\prime}}^{-}\left(w_{n}{}^{t}\bar{g}^{-1}\right)\widehat{\overline{\Phi}\left(\cdot J_{n}\right)}\left(e_{n}gJ_{n}^{-1}\right)\mu\left(\det (gJ^{-1}_n)\right)^{-1}\left|\det g\right|^{\frac{1}{2}}dg
$$
$$
=\omega_{E/F}\left(-1\right)^{\frac{n(n+1)}{2}}\int_{U\left(F\right)\backslash G\left(F\right)}W_{e}\left(w_{n}{}^{t}g^{-1}\right)W_{e^{\prime}}^{-}\left(w_{n}{}^{t}g^{-1}\right)\widehat{\overline{\Phi}\left(\cdot J_{n}\right)}\left(e_{n}\bar{g}J_{n}^{-1}\right)\mu\left(\det (w_n{}^t\bar{g}^{-1})\right)^{-1}\left|\det g\right|^{\frac{1}{2}}dg.
$$
Observe 
$$
\widehat{\overline{\Phi}\left(\cdot J_{n}\right)}\left(e_{n}\bar{g}J_{n}^{-1}\right)=\int_{\text{Mat}_{1,n}}\Phi\left(\bar{x}J_{n}\right)\psi_{E}\left(x^{t}\left(e_{n}\bar{g}J_{n}^{-1}\right)\right)dx
=\int_{\text{Mat}_{1,n}}\Phi\left(xJ_{n}\right)\psi_{E}\left(\bar{x}J_{n}{}^{t}\bar{g}{}^{t}e_{n}\right)dx 
$$
$$
= \int_{\text{Mat}_{1,n}}\Phi\left(x\right)\psi_{E}\left(x\,{}^{t}g{}^{t}e_{n}\right)dx
=\int_{\text{Mat}_{1,n}}\Phi\left(x\right)\psi_{E}\left(x{}^{t}\left(e_{n}g\right)\right)dx=\hat{\Phi}\left(e_{n}g\right).
$$
Therefore
$$
L_{\pi}\left(\widetilde{\pi\times{}^{\sigma}\pi^{\vee}}\left(\theta_{n}\right)\left(e\otimes e^{\prime}\right),\tilde{\omega}_{\psi,\mu}^{\vee}\left(\theta_{n}\right)\Phi,\frac{1}{2}\right)
$$
$$
=\omega_{E/F}\left(-1\right)^{\frac{n(n+1)}{2}}\int_{U\left(F\right)\backslash G\left(F\right)}W_{e}\left(w_{n}{}^{t}g^{-1}\right)W_{e^{\prime}}^{-}\left(w_{n}{}^{t}g^{-1}\right)\hat{\Phi}\left(e_{n}g\right)\mu\left(\det (w_n{}^t\bar{g}^{-1})\right)^{-1}\left|\det g\right|^{\frac{1}{2}}dg.
$$
By \cite[Theorem 2.7]{JPSS83}, the right hand side is equal to 
$$
\omega_{\pi}\left(-1\right)^{n}\omega_{E/F}\left(-1\right)^{n\left(n-1\right)/2}\epsilon\left(\frac{1}{2},\pi\times{}^{\sigma}\pi^{\vee}\times\mu^{-1},\psi_{E}\right)L_{\pi}\left(e\otimes e^{\prime},\Phi,\frac{1}{2}\right),
$$
which is to say 
$$
L_{\pi}\left(\widetilde{\pi\times{}^{\sigma}\pi^{\vee}}\left(\theta_{n}\right)\left(e\otimes e^{\prime}\right),\tilde{\omega}_{\psi,\mu}^{\vee}\left(\theta_{n}\right)\Phi,\frac{1}{2}\right)=\epsilon_{\psi}\left(\widetilde{\pi\times{}^{\sigma}\pi^{\vee}}\right)L_{\pi}\left(e\otimes e^{\prime},\Phi,\frac{1}{2}\right).
$$
\end{proof}

\section{A spectral expansion of $J_{\chi}\left(\tilde{f}\right)$}\label{sec9}

Let $\tilde{f}\in \mathcal{C}_\text{scusp}(Z_M(F)\backslash \tilde{M}(F),\chi \times {}^\sigma \chi^\vee)$.
Define 
$$
J_{\chi,\text{ spec}}\left(\tilde{f}\right)=\underset{\tilde{L}\in\mathcal{L}(\tilde{L}_{\min})}{\sum}|\widetilde{W}^L||\widetilde{W}^G|^{-1}(-1)^{a_{\tilde{L}}-a_{\tilde{G}}}\int_{E_{\text{ell}}(Z_M(F)\backslash\tilde{L}(F),\chi^{-1}\times{}^\sigma\chi)}\hat{\theta}_{\tilde{f}}\left(\tilde{\pi}\right)\epsilon_{\psi}\left(\tilde{\pi}^{\vee}\right)d\tilde{\pi}.
$$
In this section, we prove the following theorem.
\begin{theorem}\label{9.1}
We have the following spectral
expansion 
$$
J_{\chi}\left(\tilde{f}\right)=J_{\chi,\text{spec}}\left(\tilde{f}\right).
$$
\end{theorem}
\subsection{Linear form ${\mathcal{L}}_{\pi}^{\phi_{1},\phi_{2}}$}\label{sec9.1}

Let $\pi$ be a tempered representation of $M\left(F\right)$ with central character $\chi^{-1}\times {}^\sigma\chi$.
For all $T\in\text{End}\left(\pi\right)^\infty$, the function $m\in M\left(F\right)\mapsto\text{Trace}\left(\pi\left(m^{-1}\right)T\right)$ is in the weak Harish-Chandra Schwartz space ${\mathcal{C}}^{w}\left(Z_M\left(F\right)\backslash M\left(F\right),\chi\times {}^\sigma\chi^\vee\right)$.
Let $\phi_{1},\phi_{2}\in\omega_{\mu,\chi}$. We define the following linear form
$$
\begin{array}{ccccc}
{\mathcal{L}}_{\pi}^{\phi_{1},\phi_{2}} & : & \text{End}\left(\pi\right)^\infty & \rightarrow & \mathbb{C}\\
 &  & T & \mapsto & \int_{Z_G\left(F\right)\backslash G\left(F\right)}\text{Trace}\left(\pi\left(g^{-1}\right)T\right)\left\langle \phi_{1},\omega_{\mu,\chi}\left(g\right)\phi_{2}\right\rangle dg
\end{array}.
$$
By adapting \cite[Proposition 1.1.1]{Xue16}, the above integral is
absolutely convergent. Moreover, we have $\pi\otimes\bar{\pi}\cong\text{End}\left(\pi\right)^\infty$
as $M\left(F\right)\times M\left(F\right)$-representations via the
map $e\otimes e^{\prime}\mapsto T_{e,e^{\prime}}:=\left(\cdot,e^{\prime}\right)e$. We identify ${\mathcal{L}}_{\pi}^{\phi_{1},\phi_{2}}$ with the continuous sesquilinear form on $\pi$ given by ${\mathcal{L}}_{\pi}^{\phi_{1},\phi_{2}}\left(e,e^{\prime}\right):={\mathcal{L}}_{\pi}^{\phi_{1},\phi_{2}}\left(T_{e,e^{\prime}}\right)$, for $e,e^{\prime}\in\pi$. We define a continuous linear map 
$$
\begin{array}{ccccc}
L_{\pi}^{\phi_{1},\phi_{2}} & : & \pi & \rightarrow & \overline{\pi}^{\vee}\\
 &  & e & \mapsto & \left(e^{\prime}\mapsto{\mathcal{L}}_{\pi}^{\phi_{1},\phi_{2}}\left(e,e^{\prime}\right)\right)
\end{array}.
$$
Since $\pi$ is tempered, we have $\bar{\pi}^{\vee}\cong\pi$. Therefore,
the following linear map can be viewed as an element in $\text{End}\left(\pi\right)^\infty$.
Moreover, $L_{\pi}^{\phi_{1},\phi_{2}}$ is of finite rank. Hence,
we have $\text{Trace}\left(TL_{\pi}^{\phi_{1},\phi_{2}}\right)=\text{Trace}\left(L_{\pi}^{\phi_{1},\phi_{2}}T\right)={\mathcal{L}}_{\pi}^{\phi_{1},\phi_{2}}\left(T\right)$. We give some properties of the above linear forms.
\begin{lemma}\label{9.2}
${}$
\begin{enumerate}
\item Suppose $\pi\in\text{Temp}\left(M,\chi^{-1}\times {}^\sigma\chi\right)$ or ${\mathcal{X}}_{\text{temp}}\left(M,\chi^{-1}\times {}^\sigma\chi\right)$.
If $S,T\in\text{End}\left(\pi\right)^\infty$, then 
$$
{\mathcal{L}}_{\pi}^{\phi_{1},\phi_{2}}\left(SL_{\pi}^{\phi_{3},\phi_{4}}T\right)={\mathcal{L}}_{\pi}^{\phi_{1},\phi_{4}}\left(S\right){\mathcal{L}}_{\pi}^{\phi_{3},\phi_{2}}\left(T\right).
$$
\item Let $f\in{\mathcal{C}}\left(Z_M\left(F\right)\backslash M\left(F\right),\chi\times {}^\sigma\chi^{-1}\right)$
and $\phi_{1},\phi_{2}\in\omega_{\mu,\chi}$. Then 
$$
\int_{Z_G\left(F\right)\backslash G\left(F\right)}f\left(g\right)\left\langle \phi_{1},\omega_{\mu,\chi}\left(g\right)\phi_{2}\right\rangle dg=\int_{{\mathcal{X}}_{\text{temp}}\left(M,\chi^{-1}\times {}^\sigma\chi\right)}{\mathcal{L}}_{\pi}^{\phi_{1},\phi_{2}}\left(\pi\left(f\right)\right)\mu\left(\pi\right)d\pi
$$
both integrals being absolutely convergent.
\item Let $f\in{\mathcal{C}}\left(Z_M\left(F\right)\backslash M\left(F\right),\chi\times {}^\sigma\chi^{-1}\right)$
and $f^{\prime}\in{\mathcal{C}}\left(Z_M\left(F\right)\backslash M\left(F\right),\chi^{-1}\times {}^\sigma\chi\right)$
and $\phi_{1},\phi_{2},\phi_{3},\phi_{4}\in\omega_{\mu,\chi}$. Then
we have the equality 
$$
\int_{{\mathcal{X}}_{\text{temp}}\left(M,\chi^{-1}\times {}^\sigma\chi\right)}{\mathcal{L}}_{\pi}^{\phi_{1},\phi_{2}}\left(\pi\left(f\right)\right){\mathcal{L}}_{\pi}^{\phi_{3},\phi_{4}}\left(\pi\left(f^{\prime,\vee}\right)\right)\mu\left(\pi\right)d\pi
$$
$$
=\underset{Z_G\left(F\right)\backslash G\left(F\right)}{\iint}\int_{Z_M\left(F\right)\backslash M\left(F\right)}f\left(gmg^{\prime}\right)f^{\prime}\left(m\right)\left\langle \phi_{3},\omega_{\mu,\chi}\left(g^{\prime}\right)\phi_{2}\right\rangle \left\langle \phi_{1},\omega_{\mu,\chi}\left(g\right)\phi_{4}\right\rangle dmdg^{\prime}dg,
$$
where the first integral is absolutely convergent and the second one converges in that order but not necessarily as a triple integral.
\end{enumerate}
\end{lemma}

\begin{proof}
The proof follows that of Lemma \ref{5.3} verbatim, noting that here we have the multiplicity one property $\dim\text{Hom}_{G}\left(\pi,\omega_{\mu}\right)=1$
for any irreducible generic representation $\pi$ of $M\left(F\right)$.
\end{proof}

\subsection{Some estimates and an auxiliary distribution}\label{sec9.2}

We fix a function $f$ in the space ${\mathcal{C}}_{\text{scusp}}\left(Z_M\left(F\right)\backslash M\left(F\right),\chi\times {}^\sigma\chi^{-1}\right)$.
For any $f^{\prime}$ in ${\mathcal{C}}\left(Z_M\left(F\right)\backslash M\left(F\right),\chi^{-1}\times {}^\sigma\chi\right)$
and $\phi,\phi^{\prime}\in\omega_{\mu,\chi}$, we define the following integrals 
$$
K_{f,f^{\prime}}^{A}\left(m_{1},m_{2}\right)=\int_{Z_M\left(F\right)\backslash M\left(F\right)}f\left(m_{1}^{-1}mm_{2}\right)f^{\prime}\left(m\right)dm,\ \ m_{1},m_{2}\in M\left(F\right),
$$
$$
K_{f,f^{\prime},\phi,\phi^{\prime}}^{1}\left(m,x\right)=\int_{Z_G\left(F\right)\backslash G\left(F\right)}K_{f,f^{\prime}}^{A}\left(m,gx\right)\left\langle \phi,\omega_{\mu,\chi}\left(g\right)\phi^{\prime}\right\rangle dg,\ \ m,x\in M\left(F\right),
$$
$$
K_{f,f^{\prime},\phi,\phi^{\prime}}^{2}\left(x,y\right)=\int_{Z_G\left(F\right)\backslash G\left(F\right)}K_{f,f^{\prime},\phi,\phi^{\prime}}^{1}\left(gx,\theta_{n}\left(gy\right)\right)dg,\ \ x,y\in M\left(F\right),
$$
$$
J_{\text{aux}}\left(f,f^{\prime},\phi,\phi^{\prime}\right)=\int_{Z_M\left(F\right)G\left(F\right)\backslash M\left(F\right)}K_{f,f^{\prime},\phi,\phi^{\prime}}^{2}\left(x,x\right)dx.
$$
\begin{proposition}\label{9.3}
Let $f$ in ${\mathcal{C}}_{\text{scusp}}\left(Z_M\left(F\right)\backslash M\left(F\right),\chi\times {}^\sigma\chi^{-1}\right)$ and $f^{\prime}\in{\mathcal{C}}\left(Z_M\left(F\right)\backslash M\left(F\right),\chi^{-1}\times{}^\sigma\chi\right)$. The integrals defining $K_{f,f^{\prime}}^{A}\left(m_{1},m_{2}\right)$, $K_{f,f^{\prime},\phi,\phi^{\prime}}^{1}(m,x)$, $K_{f,f^{\prime},\phi,\phi^{\prime}}^{2}\left(x,y\right)$ and $J_{\text{aux}}\left(f,f^{\prime}\right)$ are absolutely convergent.
\end{proposition}

\begin{proof}
The proof follows from Proposition \ref{6.3}.
\end{proof}
\begin{proposition}\label{9.4}
We have the following equality 
$$
J_{\text{aux}}\left(f,f^{\prime},\phi,\phi^{\prime}\right)=\underset{\tilde{L}\in\mathcal{L}(\tilde{L}_{\min})}{\sum}\frac{(-1)^{a_{\tilde{L}}-a_{\tilde{G}}}|\widetilde{W}^L|}{|\widetilde{W}^G|}\int_{E_{\text{ell}}(Z_M(F)\backslash\tilde{L}(F),\chi^{-1}\times{}^\sigma\chi)}\hat{\theta}_{\tilde{f}}\left(\tilde{\pi}\right){\mathcal{L}}_{\pi}^{\phi,\phi^{\prime}}\left(\pi\left(f^{\prime,\vee}\right)\tilde{\pi}\left(\theta_{n}\right)\right)d\tilde{\pi},
$$
for all $f^{\prime}\in{\mathcal{C}}\left(Z_M\left(F\right)\backslash M\left(F\right),\chi^{-1}\times{}^\sigma\chi\right)$,
where $\tilde{f}\in{\mathcal{C}}_{\text{scusp}}\left(Z_M\left(F\right)\backslash\tilde{M}\left(F\right),\chi\times{}^\sigma\chi^\vee\right)$
is given by $\tilde{f}\left(x\theta_{n}\right)=f\left(x\right)$, for $x\in M\left(F\right)$.
\end{proposition}

\begin{proof}
For $N>0$, let us denote by $\alpha_{N}:Z_M\left(F\right)G\left(F\right)\backslash M\left(F\right)\rightarrow\left\{ 0,1\right\}$ the characteristic function of the set $\left\{ x\in Z_M\left(F\right)G\left(F\right)\backslash M\left(F\right):\ \sigma_{Z_MG\backslash M}\left(x\right)\leq N\right\}$ and $\beta_{N}$ the characteristic function of the set $\left\{ m\in Z_M\left(F\right)\backslash M\left(F\right):\ \sigma_{Z_M\backslash M}\left(m\right)\leq N\right\} $.
For all $N\geq1$ and $C>0$, define 
$$
J_{\text{aux},N}\left(f,f^{\prime},\phi,\phi^{\prime}\right)=\int_{Z_M\left(F\right)G\left(F\right)\backslash M\left(F\right)}\alpha_{N}\left(x\right)\iint_{Z_G\left(F\right)\backslash G\left(F\right)}K_{f,f^{\prime}}^{A}\left(g_{1}x,g_{2}\theta_{n}\left(g_{1}x\right)\right)
\left\langle \phi,\omega_{\mu,\chi}\left(g_{2}\right)\phi^{\prime}\right\rangle dg_{2}dg_{1}dx
$$
and 
$$
J_{\text{aux},N,C}\left(f,f^{\prime},\phi,\phi^{\prime}\right)
=\int_{Z_M\left(F\right)G\left(F\right)\backslash M\left(F\right)}\alpha_{N}\left(x\right)\iint_{Z_G\left(F\right)\backslash G\left(F\right)}\beta_{C\log\left(N\right)}\left(g_{2}\right)
$$
$$
K_{f,f^{\prime}}^{A}\left(g_{1}x,g_{2}\theta_{n}\left(g_{1}x\right)\right)\left\langle \phi,\omega_{\mu,\chi}\left(g_{2}\right)\phi^{\prime}\right\rangle dg_{2}dg_{1}dx.
$$
We can easily see that $J_{\text{aux}}\left(f,f^{\prime},\phi,\phi^{\prime}\right)=\underset{N\rightarrow\infty}{\lim}J_{\text{aux},N}\left(f,f^{\prime},\phi,\phi^{\prime}\right)$. We consider the following claim, whose proof is similar to claim \ref{6.5}.
\begin{claim}\label{9.5}
The triple integrals defining $J_{\text{aux},N}\left(f,f^{\prime},\phi,\phi^{\prime}\right)$
and $J_{\text{aux},N,C}\left(f,f^{\prime},\phi,\phi^{\prime}\right)$
are absolutely convergent and there exists $C>0$ such that 
$$
\left|J_{\text{aux},N}\left(f,f^{\prime},\phi,\phi^{\prime}\right)-J_{\text{aux},N,C}\left(f,f^{\prime},\phi,\phi^{\prime}\right)\right|\ll N^{-1}
$$
for all $N\geq1$.
\end{claim}
Let us fix such $C>0$. It follows that $J_{\text{aux}}\left(f,f^{\prime},\phi,\phi^{\prime}\right)=\underset{N\rightarrow\infty}{\lim}J_{\text{aux},N,C}\left(f,f^{\prime},\phi,\phi^{\prime}\right)$. Since the triple integral defining $J_{\text{aux},N,C}\left(f,f^{\prime},\phi,\phi^{\prime}\right)$
is absolutely convergent, one can write 
$$
J_{\text{aux},N,C}\left(f,f^{\prime},\phi,\phi^{\prime}\right)=\int_{Z_G\left(F\right)\backslash G\left(F\right)}\beta_{C\log\left(N\right)}\left(g_{2}\right)\left\langle \phi,\omega_{\mu,\chi}\left(g_{2}\right)\phi^{\prime}\right\rangle
$$
$$
\int_{Z_M\left(F\right)\backslash M\left(F\right)}\alpha_{N}\left(g_{1}\right)K_{f,f^{\prime}}^{A}\left(g_{1},g_{2}\theta_{n}\left(g_{1}\right)\right)dg_{1}dg_{2}.
$$
We have the following estimate, whose proof is similar to claim \ref{6.6}.
\begin{claim}\label{9.6}
$ $
$$
\left|J_{\text{aux},N,C}\left(f,f^{\prime},\phi,\phi^{\prime}\right)-\begin{array}{c}
\int_{Z_G\left(F\right)\backslash G\left(F\right)}\beta_{C\log\left(N\right)}\left(g_{2}\right)\left\langle \phi,\omega_{\mu,\chi}\left(g_{2}\right)\phi^{\prime}\right\rangle \\
\int_{Z_M\left(F\right)\backslash M\left(F\right)}K_{f,f^{\prime}}^{A}\left(g_{1},g_{2}\theta_{n}\left(g_{1}\right)\right)dg_{1}dg_{2}
\end{array}\right|\ll N^{-1}
$$
for all $N\geq1$.
\end{claim}
From the above claim, we deduce  
$$
J_{\text{aux}}\left(f,f^{\prime},\phi,\phi^{\prime}\right)=\underset{N\rightarrow\infty}{\lim}\int_{Z_G\left(F\right)\backslash G\left(F\right)}\beta_{C\log\left(N\right)}\left(g_{2}\right)\left\langle \phi,\omega_{\mu,\chi}\left(g_{2}\right)\phi^{\prime}\right\rangle
\int_{Z_M\left(F\right)\backslash M\left(F\right)}K_{f,f^{\prime}}^{A}\left(g_{1},g_{2}\theta_{n}\left(g_{1}\right)\right)dg_{1}dg_{2},
$$
for $C$ sufficiently large. Since $f$ is strongly cuspidal, by \cite[Theorem 5.1]{MW18} and \cite[Proposition 2.9]{BW23}, it follows that 
$$
\int_{Z_M\left(F\right)\backslash M\left(F\right)}K_{f,f^{\prime}}^{A}\left(g_{1},g_{2}\theta_{n}\left(g_{1}\right)\right)dg_{1}
=\underset{\tilde{L}\in\mathcal{L}(\tilde{L}_{\min})}{\sum}\frac{(-1)^{a_{\tilde{L}}-a_{\tilde{G}}}|\widetilde{W}^L|}{|\widetilde{W}^G|}
$$
$$
\int_{E_{\text{ell}}(Z_M(F)\backslash\tilde{L}(F),\chi^{-1}\times{}^\sigma\chi)}\hat{\theta}_{\tilde{f}}\left(\tilde{\pi}\right)\theta_{\tilde{\pi}^{\vee}}\left(R\left(g_{2}^{-1}\right)\tilde{f^{\prime}}\right)d\tilde{\pi},
$$
where $\tilde{f}^{\prime}$ is given by $\tilde{f}^{\prime}\left(x\theta_{n}\right)=f^{\prime}\left(x\right)$, for any $x\in M\left(F\right)$. Observe $\left|\theta_{\tilde{\pi}^{\vee}}\left(R\left(g^{-1}\right)\tilde{f^{\prime}}\right)\right|\ll\Xi^{M}\left(g\right)$, for all tempered representation $\tilde{\pi}$ of $\tilde{M}(F)$ and $g\in G\left(F\right)$. Thus, the discussion in Section \ref{sec9.1} shows that the double integral 
$$
\underset{\tilde{L}\in\mathcal{L}(\tilde{L}_{\min})}{\sum}\frac{(-1)^{a_{\tilde{L}}-a_{\tilde{G}}}|\widetilde{W}^L|}{|\widetilde{W}^G|}\int_{Z_G\left(F\right)\backslash G\left(F\right)}\left\langle \phi,\omega_{\mu,\chi}\left(g_{2}\right)\phi^{\prime}\right\rangle \int_{E_{\text{ell}}(Z_M(F)\backslash\tilde{L}(F),\chi^{-1}\times{}^\sigma\chi)}
\hat{\theta}_{\tilde{f}}\left(\tilde{\pi}\right)\theta_{\tilde{\pi}^{\vee}}\left(R\left(g_{2}^{-1}\right)\tilde{f^{\prime}}\right)d\tilde{\pi}dg_{2}
$$
is absolutely convergent and equal to $J_{\text{aux}}\left(f,f^{\prime},\phi,\phi^{\prime}\right)$.
By switching the two integrals and Lemma \ref{9.2}(ii), we have 
$$
J_{\text{aux}}\left(f,f^{\prime},\phi,\phi^{\prime}\right)=\underset{\tilde{L}\in\mathcal{L}(\tilde{L}_{\min})}{\sum}\frac{(-1)^{a_{\tilde{L}}-a_{\tilde{G}}}|\widetilde{W}^L|}{|\widetilde{W}^G|}\int_{E_{\text{ell}}(Z_M(F)\backslash\tilde{L}(F),\chi^{-1}\times{}^\sigma\chi)}
\hat{\theta}_{\tilde{f}}\left(\tilde{\pi}\right){\mathcal{L}}_{\pi}^{\phi,\phi^{\prime}}\left(\pi\left(f^{\prime,\vee}\right)\tilde{\pi}\left(\theta_{n}\right)\right)d\tilde{\pi}.
$$
\end{proof}

\subsection{Proof of Theorem \ref{9.1}}\label{sec9.3}

We fix $\tilde{y}\in\tilde{G}\left(F\right)$. For simplicity, we can choose $\tilde{y}$ to be $\theta_n$. As in Theorem \ref{9.1}, for $\tilde{f} \in {\mathcal{C}}_{\text{scusp}}\left(Z_M\left(F\right)\backslash \tilde{M}\left(F\right),\chi\times{}^\sigma\chi^\vee\right)$, we define $f\in{\mathcal{C}}_{\text{scusp}}\left(Z_M\left(F\right)\backslash M\left(F\right),\chi\times{}^\sigma\chi^\vee\right)$
such that $f\left(m\right)=\tilde{f}\left(m\tilde{y}\right)$. Let
$x\in M\left(F\right)$. We fix a distinguished compact open subgroup
$K_{f}$ of $K$ such that $f$ is $K_{f}$-bi-invariant and a compact
open subgroup $K_{x}^{\prime}$ of $G\left(F\right)$ such that $xK_{x}^{\prime}x^{-1}\subset K_{f}$
and $\theta_{\tilde{y}}\left(x\right)K_{x}^{\prime}\theta_{\tilde{y}}\left(x\right)^{-1}\subset K_{f}$.
Let $\left\{ \phi_{i}\right\} _{i}$ be an orthonormal basis for $\omega_{\mu,\chi}^{K_{g}^{\prime}}$,
which is of finite dimensional. We consider 
$$
K_{\chi}\left(\tilde{f},x\right)=\underset{i}{\sum}\int_{Z_G\left(F\right)\backslash\tilde{G}\left(F\right)}\tilde{f}\left(x^{-1}\tilde{g}x\right)\left\langle \phi_{i},\tilde{\omega}_{\psi,\mu,\chi}\left(\tilde{g}\right)\phi_{i}\right\rangle d\tilde{g}
$$
$$
=\underset{i}{\sum}\int_{Z_G\left(F\right)\backslash G\left(F\right)}f\left(x^{-1}g\theta_{\tilde{y}}\left(x\right)\right)\left\langle \phi_{i},\tilde{\omega}_{\psi,\mu,\chi}\left(g\tilde{y}\right)\phi_{i}\right\rangle dg.
$$
From the discussion in Section \ref{sec9.1} or \cite[Appendix D.1]{Xue16},
one can show that 
$$
\int_{Z_G\left(F\right)\backslash G\left(F\right)}f\left(x^{-1}g\theta_{\tilde{y}}\left(x\right)\right)\left\langle \phi_{i},\tilde{\omega}_{\psi,\mu,\chi}\left(g\tilde{y}\right)\phi_{i}\right\rangle dg\text{ is absolutely convergent.} 
$$
This gives us the linear form $K_{\chi}\left(\tilde{f},x\right)$
is absolutely convergent. By applying the Plancherel formula for $f$,
as in Lemma \ref{9.2}(ii), we have 
$$
K_{\chi}\left(\tilde{f},x\right)=\underset{i}{\sum}\int_{{\mathcal{X}}_{\text{temp}}\left(M,\chi^{-1}\times{}^\sigma\chi\right)}{\mathcal{L}}_{\pi}^{\phi_{i},\tilde{\omega}_{\psi,\mu,\chi}\left(\tilde{y}\right)\phi_{i}}\left(\pi\left(\theta_{\tilde{y}}\left(x\right)\right)\pi\left(f\right)\pi\left(x\right)^{-1}\right)\mu\left(\pi\right)d\pi.
$$
Since $\dim\text{Hom}_{G}\left(\pi,\omega_{\mu}\right)=1$ for any
tempered irreducible representation $\pi$, by mimicing \cite[Corollary 7.6.1(ii)]{BP20} as in Corollary \ref{5.8}, one can construct a function $f^{\prime}\in{\mathcal{C}}\left(Z_M\left(F\right)\backslash M\left(F\right),\chi^{-1}\times{}^\sigma\chi\right)$
and $\phi,\phi^{\prime}\in\omega_{\mu,\chi}$ such that ${\mathcal{L}}_{\pi}^{\phi,\phi^{\prime}}\left(\pi\left(f^{\prime,\vee}\right)\right)=1$
whenever $\pi\left(f\right)\neq0$. Observe 
$$
K_{\chi}\left(\tilde{f},x\right)=\underset{i}{\sum}\int_{{\mathcal{X}}_{\text{temp}}\left(M,\chi^{-1}\times{}^\sigma\chi\right)}{\mathcal{L}}_{\pi}^{\phi_{i},\tilde{\omega}_{\psi,\mu,\chi}\left(\tilde{y}\right)\phi_{i}}\left(\pi\left(\theta_{\tilde{y}}\left(x\right)\right)\pi\left(f\right)\pi\left(x\right)^{-1}\right){\mathcal{L}}_{\pi}^{\phi,\phi^{\prime}}\left(\pi\left(f^{\prime,\vee}\right)\right)\mu\left(\pi\right)d\pi
$$
$$
=\underset{i}{\sum}\underset{Z_G\left(F\right)\backslash G\left(F\right)}{\iint}\int_{Z_M\left(F\right)\backslash M\left(F\right)}f\left(x^{-1}g_{1}mg_{2}\theta_{\tilde{y}}\left(x\right)\right)f^{\prime}\left(m\right)
\left\langle \phi,\tilde{\omega}_{\psi,\mu,\chi}\left(g_{2}\tilde{y}\right)\phi_{i}\right\rangle \left\langle \phi_{i},\omega_{\mu,\chi}\left(g_{1}\right)\phi^{\prime}\right\rangle dmdg_{2}dg_{1}
$$
$$
=\underset{Z_G\left(F\right)\backslash G\left(F\right)}{\iint}\int_{Z_M\left(F\right)\backslash M\left(F\right)}f\left(x^{-1}g_{1}mg_{2}\theta_{\tilde{y}}\left(x\right)\right)f^{\prime}\left(m\right)\left\langle \phi,\tilde{\omega}_{\psi,\mu,\chi}\left(g_{2}\theta_{\tilde{y}}\left(g_{1}\right)\tilde{y}\right)\phi^{\prime}\right\rangle dmdg_{2}dg_{1},
$$
noting that the second line comes from Lemma \ref{9.2}(iii). We now consider
the outer integral. By substituting $\tilde{y}=\theta_n$, we have 
$$
\tilde{J}_{\chi}\left(\tilde{f}\right)=\int_{Z_M\left(F\right)G\left(F\right)\backslash M\left(F\right)}K_{\chi}\left(\tilde{f},x\right)dx=J_{\text{aux}}\left(f,f^{\prime},\phi,\tilde{\omega}_{\psi,\mu,\chi}\left(\theta_{n}\right)\left(\phi^{\prime}\right)\right).
$$
By Proposition \ref{9.3}, the linear form $\tilde{J}_{\chi}\left(\tilde{f}\right)$
is absolutely convergent. By Proposition \ref{9.4}, we have 
$$
\tilde{J}_{\chi}\left(\tilde{f}\right)=\underset{\tilde{L}\in\mathcal{L}(\tilde{L}_{\min})}{\sum}\frac{(-1)^{a_{\tilde{L}}-a_{\tilde{G}}}|\widetilde{W}^L|}{|\widetilde{W}^G|}
\int_{E_{\text{ell}}(Z_M(F)\backslash\tilde{L}(F),\chi^{-1}\times{}^\sigma\chi)}\hat{\theta}_{\tilde{f}}\left(\tilde{\pi}\right){\mathcal{L}}_{\pi}^{\phi,\tilde{\omega}_{\psi,\mu,\chi}\left(\theta_{n}\right)\phi^{\prime}}\left(\pi\left(f^{\prime,\vee}\right)\tilde{\pi}\left(\theta_{n}\right)\right)d\tilde{\pi}
$$
$$
=\underset{\tilde{L}\in\mathcal{L}(\tilde{L}_{\min})}{\sum}\frac{(-1)^{a_{\tilde{L}}-a_{\tilde{G}}}|\widetilde{W}^L|}{|\widetilde{W}^G|}\int_{E_{\text{ell}}(Z_M(F)\backslash\tilde{L}(F),\chi^{-1}\times{}^\sigma\chi)}\hat{\theta}_{\tilde{f}}\left(\tilde{\pi}\right)\epsilon_\psi\left(\tilde{\pi}^{\vee}\right){\mathcal{L}}_{\pi}^{\phi,\phi^{\prime}}\left(\pi\left(f^{\prime,\vee}\right)\right)d\tilde{\pi}
$$
$$
=\underset{\tilde{L}\in\mathcal{L}(\tilde{L}_{\min})}{\sum}\frac{(-1)^{a_{\tilde{L}}-a_{\tilde{G}}}|\widetilde{W}^L|}{|\widetilde{W}^G|}\int_{E_{\text{ell}}(Z_M(F)\backslash\tilde{L}(F),\chi^{-1}\times{}^\sigma\chi)}\hat{\theta}_{\tilde{f}}\left(\tilde{\pi}\right)\epsilon_\psi\left(\tilde{\pi}^{\vee}\right)d\tilde{\pi},
$$
noting that the second line follows from Proposition \ref{8.2}.

\section{Comparison of two trace formulae and the epsilon dichotomy}\label{sec10}
\subsection{Linearization of trace formulas}\label{sec10.1}

Let $V$ and $V^{\prime}$ be the two equivalent classes of non-degenerate $n$-dimensional skew-hermitian spaces. We recall some linear forms in Section \ref{sec7}.
Let $\chi$ be a character of $Z_{G}\left(F\right)$.
For strongly cuspidal function $f\in{\mathcal{C}}_{\text{scusp}}\left(Z_G(F)\backslash G(F),\chi\right)$, we set 
$$
J_{V,\chi}\left(f\right)=\int_{Z_{G}\left(F\right)H_{V}\left(F\right)\backslash G\left(F\right)}\underset{i}{\sum}\int_{Z_{H_{V}}\left(F\right)\backslash H_{V}\left(F\right)}f\left(g^{-1}hg\right)\left\langle \phi_{i},\omega_{V,\psi,\mu,\chi}\left(h\right)\phi_{i}\right\rangle dhdg
$$
and 
$$
J_{V^{\prime},\chi}\left(f\right)=\int_{Z_{G}\left(F\right)H_{V^\prime}\left(F\right)\backslash G\left(F\right)}\underset{i}{\sum}\int_{Z_{H_{V^\prime}}\left(F\right)\backslash H_{V^\prime}\left(F\right)}f\left(g^{-1}hg\right)\left\langle \phi_{i},\omega_{V^{\prime},\psi,\mu,\chi}\left(h\right)\phi_{i}\right\rangle dhdg
$$
and 
$$
J_{\chi}\left(f\right)=\mu\left(\det V\right)J_{V,\chi}\left(f\right)+\mu\left(\det V^{\prime}\right)J_{V^{\prime},\chi}\left(f\right).
$$
For any $\theta\in{\text{QC}}\left(Z_G(F)\backslash G(F),\chi\right)$, let 
$$
J_{\text{qc},V,\chi}\left(\theta\right)=\underset{\pi\in\Pi_{2}\left(G,\chi^{-1}\right)}{\sum}m_{V}\left(\bar{\pi}\right)\int_{Z_{G}\left(F\right)\backslash\Gamma\left(G\right)_{\text{ell}}}D^{G}\left(x\right)\theta\left(x\right)\theta_{\pi}\left(x\right)dx
$$
and 
$$
J_{\text{qc},V^{\prime},\chi}\left(\theta\right)=\underset{\pi\in\Pi_{2}\left(G,\chi^{-1}\right)}{\sum}m_{V^{\prime}}\left(\bar{\pi}\right)\int_{Z_{G}\left(F\right)\backslash\Gamma\left(G\right)_{\text{ell}}}D^{G}\left(x\right)\theta\left(x\right)\theta_{\pi}\left(x\right)dx
$$
and 
$$
J_{\text{qc},\chi}\left(\theta\right)=\mu\left(\det V\right)J_{\text{qc},V,\chi}\left(\theta\right)+\mu\left(\det V^{\prime}\right)J_{\text{qc},V^{\prime},\chi}\left(\theta\right).
$$
Similarly, for any $\tilde{\theta}\in{\text{QC}}\left(Z_M(F)\backslash\tilde{M}(F),\chi\times{}^\sigma\chi^\vee\right)$, we define
$$
\tilde{J}_{\text{qc},\chi}\left(\tilde{\theta}\right)=\underset{\tilde{\pi}\in\Pi_{2}\left(\tilde{M},\chi^{-1}\times{}^\sigma\chi\right)}{\sum}\epsilon_{\psi}\left(\tilde{\pi}^{\vee}\right)\int_{Z_M\left(F\right)\backslash\Gamma\left(\tilde{M}\right)_{\text{ell}}}D^{\tilde{M}}\left(\tilde{x}\right)\tilde{\theta}\left(\tilde{x}\right)\theta_{\tilde{\pi}}\left(\tilde{x}\right)d\tilde{x}.
$$
\begin{remark}\label{10.1}
        Since the above linear forms are supported on elliptic locus of $G(F)$ and $\tilde{M}(F)$, they are well-defined as linear forms of the sets of elliptic germs of quasi-characters (as in Definition \ref{3.3}) $\mathcal{E}(Z_G(F)\backslash G(F),\chi)$ and $\mathcal{E}(Z_M(F)\backslash \tilde{M}(F),\chi\times {}^\sigma \chi^\vee)$.
\end{remark}
In the remaining part of this section, we will prove the following theorem by induction.
\begin{theorem}\label{10.2}
For any $\theta\in{\mathcal{E}}\left(Z_G(F)\backslash G(F),\chi\right)$, we have $J_{\text{qc},\chi}\left(\theta\right)=\tilde{J}_{\text{qc},\chi}\left(\tilde{\theta}\right)$, where $\tilde{\theta}$ is a transfer of $\theta$ defined in section \ref{sec8.1}.
\end{theorem}
For $n=1$, the above theorem follows from \cite[Theorem 3.1]{GGP23}. From the next subsection, we prove Theorem \ref{sec10.2} for $n \geq 2$. We show that Theorem \ref{sec10.2} implies part (ii) of Theorem \ref{thm1.2}.
\begin{proof}
By \cite[Theorem 4.8]{GGP23}, it suffices to prove the statement for
discrete series. Let $\pi$ be an irreducible discrete series representation
of $G\left(F\right)$. Part (i) of Theorem \ref{thm1.2} gives us $m_{V}\left(\pi\right)+m_{V^{\prime}}\left(\pi\right)=1$. By substituting $\theta=\theta_{\pi}$ to Theorem \ref{sec10.2}, we have 
$$
\mu\left(\det V\right)m_{V}\left(\pi\right)+\mu\left(\det V^{\prime}\right)m_{V^{\prime}}\left(\pi\right)=\omega_{\pi}\left(-1\right)^{n}\omega_{E/F}\left(-1\right)^{n\left(n-1\right)/2}\epsilon\left(\frac{1}{2},\pi\times{}^{\sigma}\pi^{\vee}\times\mu^{-1},\psi_{E}\right).
$$
Therefore 
$$
m_{V}\left(\pi\right)=1\text{ if and only if }\mu\left(\det V\right)=\omega_{\pi}\left(-1\right)^{n}\omega_{E/F}\left(-1\right)^{n\left(n-1\right)/2}\epsilon\left(\frac{1}{2},\pi\times{}^{\sigma}\pi^{\vee}\times\mu^{-1},\psi_{E}\right).
$$
\end{proof}
We end this subsection by proving the following lemma, which relates the linear forms of quasi-characters $J_{qc,\chi}$ and $\tilde{J}_{qc,\chi}$ to $J_\chi$ and $\tilde{J}_\chi$.
\begin{lemma}\label{10.3}
Let $f\in \mathcal{C}_{scusp}(Z_G(F)\backslash G(F),\chi)$ and $\tilde{f}\in \mathcal{C}_{scusp}\left(Z_M(F)\backslash\tilde{M}(F),\chi\times {}^\sigma\chi^\vee\right)$ such that $\tilde{\theta}_{\tilde{f}}$ is a transfer of $\theta_f$. Then
$$
J_{qc,\chi}(\theta_f)-\tilde{J}_{qc,\chi}(\theta_{\tilde f})=J_\chi(f)-\tilde{J}_\chi(\tilde{f}).
$$
\end{lemma}

\begin{proof}
For any $\pi\in \mathcal{X}(G,\chi^{-1})$, we have  
$$
\hat{\theta}_{f}\left(\pi\right)=\int_{\Gamma\left(Z_G(F)\backslash G(F),\chi\right)}D^{G}\left(x\right)\theta_{f}\left(x\right)\theta_{\pi}\left(x\right)dx
$$
and 
$$
\hat{\theta}_{\tilde{f}}\left(\widetilde{\pi\times {}^\sigma\pi^\vee}\right)=\int_{\Gamma\left(Z_M(F)\backslash\tilde{M}(F),\chi^{-1} \times {}^\sigma\chi\right)}D^{\tilde{M}}\left(\tilde{x}\right)\theta_{\tilde{f}}\left(\tilde{x}\right)\theta_{\tilde{\pi}}\left(\tilde{x}\right)d\tilde{x}.
$$
Since $\theta_{\tilde{f}}$ is a transfer of $\theta_f$, it follows that $\hat{\theta}_{f}\left(\pi\right)=\hat{\theta}_{\tilde{f}}\left(\widetilde{\pi\times {}^\sigma\pi^\vee}\right)$. The induction hypothesis and \cite[Theorem 4.8]{GGP23} gives us
$$\mu(\det V)m_V(\pi)+\mu(\det V^\prime)m_{V^\prime}(\pi)=\epsilon_\psi(\widetilde{\pi \times {}^\sigma\pi^\vee}),$$
for any $\pi \in \mathcal{X}_{ind}(G)$.
Using the spectral expansions of $J_\chi$ and $\tilde{J}_\chi$ (i.e. Theorem \ref{6.1} and Theorem \ref{9.1}) and the above discussion, we have 
$$
J_\chi(f)-\tilde{J}_\chi(\tilde{f})= \int_{\mathcal{X}(G,\chi^{-1})}\hat{\theta}_f(\pi)(\mu(\det V)m_V(\bar{\pi})+\mu(\det V^\prime)m_{V^\prime}(\bar{\pi}))d\pi-\int_{\mathcal{X}(Z_M\backslash \tilde{M},\chi^{-1}\times {}^\sigma \chi)}\hat{\theta}_{\tilde{f}}(\tilde{\pi})\epsilon_\psi(\tilde{\pi}^\vee) d\tilde{\pi}$$
$$
=\sum_{\pi\in \Pi_2(G,\chi^{-1})}(\mu(\det V)m_V(\pi)+\mu(\det V^\prime)m_{V^\prime}(\pi))\int_{\Gamma_\text{ell}(Z_G(F)\backslash G(F))}D^G(x)\theta_f(x)\theta_\pi(x)dx
$$
$$-\sum_{\tilde{\pi}\in\Pi_2(\tilde M,\chi^{-1}\times{}^\sigma\chi)}\int_{\Gamma_\text{ell}(Z_M(F)\backslash \tilde{M}(F))}D^{\tilde{M}}(\tilde x)\theta_{\tilde f}(\tilde x)\theta_{\tilde \pi}(\tilde x)dx=J_{qc,\chi}(\theta_f)-\tilde{J}_{qc,\chi}(\theta_{\tilde{f}}).
$$
\end{proof}

\subsection{Comparision at non-identity locus via semisimple descent}\label{sec10.2}

In this section, we prove the following proposition.
\begin{proposition}\label{10.4}
For any $\theta\in{\mathcal{E}}\left(Z_G(F)\backslash G(F),\chi\right)$ such that $\text{Supp}\left(\theta\right)\cap Z_{G}\left(F\right)=\emptyset$,
we have 
$$
J_{\text{qc},\chi}\left(\theta\right)=\tilde{J}_{\text{qc},\chi}\left(\tilde{\theta}\right),
$$
where $\tilde{\theta}$ is a transfer of $\theta$.
\end{proposition}

\begin{proof}
Since $J_{\text{qc},\chi}$ is supported in $\Gamma_{\text{ell}}\left(G\right)$,
by a partition of unity process, we only need to prove the equality
for $\theta\in{\mathcal{E}}\left(\Omega,\chi\right)$,
where $\Omega$ is a completely $G\left(F\right)$-invariant open
subset of $G\left(F\right)$ of the form $\Omega_{x}^{G}$ for some
$x\in G\left(F\right)_{\text{ell}}\backslash Z_{G}\left(F\right)$
and some $G$-good open neighborhood $\Omega_{x}\subseteq G_{x}\left(F\right)$
of $x$. We can assume $\Omega_{x}$ is relatively compact modulo
conjugation and $Z_{G}\left(F\right)$. Let $f_x\in {\mathcal{C}}_{\text{scusp}}\left(\Omega_{x},\chi\right)$ such that $\theta_{f_x}=\theta_{x,\Omega_{x}}$. Let $f\in {\mathcal{C}}_{\text{scusp}}\left(\Omega,\chi\right)$ be a lift
of $f_x$ via Proposition \ref{3.18}. 

Let $\tilde{x}$ be an element in $\tilde{M}(F)_\text{ell}$ which matches $x$ and $\tilde{\Omega}=(\Omega_x\tilde{x})^M$. By shrinking $\Omega_{x}$ if necessary, we can assume that the transfer
of $\theta$, which is $\tilde{\theta}$, is supported in $\tilde{\Omega}_{\tilde{x}}^{M}$,
where $\tilde{\Omega}_{\tilde{x}}$ is an $M$-good neighborhood $\tilde{\Omega}_{\tilde{x}}\subseteq M_{\tilde{x}}\tilde{x}$
of $\tilde{x}$. Let $\tilde{f}_{\tilde{x}}\in{\mathcal{C}}_{\text{scusp}}\left(\tilde{\Omega}_{\tilde{x}},\chi\right)$
such that $\theta_{\tilde{f}_{\tilde{x}}}=\tilde{\theta}_{\tilde{x},\tilde{\Omega}_{\tilde{x}}}$.
By Proposition \ref{3.18}, we denote a lift of $\tilde{f}_{\tilde{x}}$
to ${\mathcal{C}}_{\text{scusp}}\left(Z_M(F)\backslash\tilde{M}(F),\chi\times{}^\sigma\chi^\vee\right)$ by $\tilde{f}$.
By Lemma \ref{10.3}, we have $J_{qc,\chi}(\theta)-\tilde{J}_{qc,\chi}(\tilde{\theta})=J_\chi(f)-\tilde{J}_\chi(\tilde{f})$. It suffices to show $J_\chi(f)=\tilde{J}_\chi(\tilde{f})$. If $x$ is not $G\left(F\right)$-conjugate to $H_{V}\left(F\right)$
(and hence neither is $H_{V^{\prime}}\left(F\right)$), we can shrink
$\Omega_{x}$ so that $\Omega\cap H_{V}\left(F\right)=\Omega\cap H_{V^{\prime}}\left(F\right)=\emptyset$.
In this case, we have $J_{\chi}\left(f\right)=0$. We can also see that ${}^g\tilde{f}$ is zero when restricting to $\tilde{G}(F)$, for any $g\in M(F)$. This gives us $\tilde{J}_\chi(\tilde{f})=0$, which is to say $J_{qc,\chi}(\theta)=\tilde{J}_{qc,\chi}(\tilde{\theta})$. 

Assume $x\in H_{V}\left(F\right)$. By using Harish-Chandra
descent as in Proposition \ref{7.3}, we can easily show $J_{\chi}\left(f\right)=J_{\chi}^{G_{x}}\left(f_{x}\right)$, where $G_{x}=R_{L/K}\text{GL}_{q}\left(L\right)$ for some $n/q$-dimensional field extension $K$ of $F$ and $L=K\otimes_{F}E$ and $J_{\chi}^{G_{x}}$ is the analogous linear form of $J_{\chi}$ on ${\mathcal{C}}_{\text{scusp}}\left(G_{x}\right)$. Let $\widetilde{G_{x}\times G_{x}}=\left(G_{x}\times G_{x}\right)\tilde{x}$
and $\widetilde{\theta_{x,\Omega_{x}}}\in QC\left(Z_M(F)\backslash \widetilde{G_{x}\times G_{x}}(F),\chi\times {}^\sigma\chi^\vee\right)$ be a transfer of $\theta_{x,\Omega_{x}}$. 
By the induction hypothesis, we have 
\begin{equation}\label{eqn9.1}
J_{\text{qc},\chi}^{G_{x}}\left(\theta_{x,\Omega_{x}}\right)=J_{\text{qc},\chi}^{\widetilde{G_{x}\times G_{x}}}\left(\widetilde{\theta_{x,\Omega_{x}}}\right).
\end{equation}
We can see that $M_{\tilde{x}}\tilde{x}=\left(G_{x}\times G_{x}\right)_{\tilde{x}}\tilde{x}$
as (twisted) subgroups in $\tilde{M}$. By the definition of semisimple descent and transfer, we have $\left(\widetilde{\theta_{x,\Omega_{x}}}\right)_{\tilde{x},\tilde{\Omega}_{\tilde{x}}}=\tilde{\theta}_{\tilde{x},\tilde{\Omega}_{\tilde{x}}}$. Let $\tilde{f}^{\widetilde{G_{x}\times G_{x}}}$ be a lift of
$\tilde{f}_{\tilde{x}}$ to ${\mathcal{C}}_{\text{scusp}}\left(\widetilde{G_{x}\times G_{x}}\right)$ given by Proposition \ref{3.18}. By Lemma \ref{10.3} and (\ref{eqn9.1}). Observe
$$
J_\chi^{G_x}(f_x)-\tilde{J}_\chi^{\widetilde{G_x\times G_x}}(\tilde{f}^{\widetilde{G_{x}\times G_{x}}})=J_{qc,\chi}^{G_x}(\theta_{x,\Omega_x})-\tilde{J}_\chi^{\widetilde{G_x\times G_x}}(\widetilde{\theta_{x,\Omega_x}})=0.
$$
It suffices to show $\tilde{J}_\chi(\tilde{f})=\tilde{J}_\chi^{\widetilde{G_x\times G_x}}(\tilde{f}^{\widetilde{G_{x}\times G_{x}}})$. We have 
$$
\tilde{J}_{\chi}\left(\tilde{f}\right)=\int_{G\left(F\right)Z_M\left(F\right)\backslash M\left(F\right)}\underset{i}{\sum}\underset{Z_G\left(F\right)\backslash\tilde{G}\left(F\right)}{\int}\tilde{f}\left(m^{-1}\tilde{g}m\right)\left\langle \phi_{i},\tilde{\omega}_{\psi,\mu}\left(\tilde{g}\right)\phi_{i}\right\rangle d\tilde{g}dm
$$
$$
=\int_{G\left(F\right)Z_M\left(F\right)\backslash M\left(F\right)}\underset{i}{\sum}\underset{G_{\tilde{x}}\left(F\right)\backslash G\left(F\right)}{\int}\ \underset{Z_G\left(F\right)\backslash G_{\tilde{x}}\left(F\right)\tilde{x}}{\int}\left(^{m}\tilde{f}\right)_{\tilde{x},\tilde{\Omega}_{\tilde{x}}}\left(g_{\tilde{x}}^{-1}\tilde{g}g_{\tilde{x}}\right)\left\langle \phi_{i},\tilde{\omega}_{\psi,\mu}\left(\tilde{g}\right)\phi_{i}\right\rangle d\tilde{g}dg_{\tilde{x}}dm
$$
$$
=\underset{M_{\tilde{x}}\left(F\right)\backslash M\left(F\right)}{\int}\ \underset{Z_M\left(F\right)G_{\tilde{x}}\left(F\right)\backslash M_{\tilde{x}}\left(F\right)}{\int}\underset{i}{\sum}\underset{Z_G\left(F\right)\backslash G_{\tilde{x}}\left(F\right)\tilde{x}}{\int}\left(^{m}\tilde{f}\right)_{\tilde{x},\tilde{\Omega}_{\tilde{x}}}\left(m_{\tilde{x}}^{-1}\tilde{g}m_{\tilde{x}}\right)\left\langle \phi_{i},\tilde{\omega}_{\psi,\mu}\left(\tilde{g}\right)\phi_{i}\right\rangle d\tilde{g}dm_{\tilde{x}}dm
$$
$$
=\underset{M_{\tilde{x}}\left(F\right)\backslash M\left(F\right)}{\int}\alpha\left(m\right)\underset{Z_M\left(F\right)G_{\tilde{x}}\left(F\right)\backslash M_{\tilde{x}}\left(F\right)}{\int}\underset{i}{\sum}\underset{Z_G\left(F\right)\backslash G_{\tilde{x}}\left(F\right)\tilde{x}}{\int}\tilde{f}_{\tilde{x}}\left(m_{\tilde{x}}^{-1}\tilde{g}m_{\tilde{x}}\right)\left\langle \phi_{i},\tilde{\omega}_{\psi,\mu}\left(\tilde{g}\right)\phi_{i}\right\rangle d\tilde{g}dm_{\tilde{x}}dm
$$
$$
=\underset{Z_M\left(F\right)G_{\tilde{x}}\left(F\right)\backslash M_{\tilde{x}}\left(F\right)}{\int}\underset{i}{\sum}\underset{Z_G\left(F\right)\backslash G_{\tilde{x}}\left(F\right)\tilde{x}}{\int}\tilde{f}_{\tilde{x}}\left(m_{\tilde{x}}^{-1}\tilde{g}m_{\tilde{x}}\right)\left\langle \phi_{i},\tilde{\omega}_{\psi,\mu}\left(\tilde{g}\right)\phi_{i}\right\rangle d\tilde{g}dm_{\tilde{x}}.
$$
A similar computation
as above gives us 
$$
\tilde{J}_{\chi}^{\widetilde{G_{x}\times G_{x}}}\left(\tilde{f}^{\widetilde{G_{x}\times G_{x}}}\right)
=\underset{Z_M\left(F\right)G_{\tilde{x}}\left(F\right)\backslash M_{\tilde{x}}\left(F\right)}{\int}\underset{i}{\sum}\underset{Z_G\left(F\right)\backslash G_{\tilde{x}}\left(F\right)\tilde{x}}{\int}\tilde{f}_{\tilde{x}}\left(m_{\tilde{x}}^{-1}\tilde{g}m_{\tilde{x}}\right)\left\langle \phi_{i},\tilde{\omega}_{\psi\circ\text{Tr}_{K/F},\mu\circ\text{N}_{L/E}}\left(\tilde{g}\right)\phi_{i}\right\rangle d\tilde{g}dm_{\tilde{x}}.
$$
Since the restriction of the Weil representation $\omega_{\mu}$ to
$G_{x}(F)$ is the Weil representation $\omega_{\mu\circ\text{N}_{L/E}}$,
we have $\tilde{J}_\chi(\tilde{f})=\tilde{J}_\chi^{\widetilde{G_x\times G_x}}(\tilde{f}^{\widetilde{G_{x}\times G_{x}}})$ as desired.
\end{proof}

\subsection{Descent to Lie algebra and homogeneity}\label{sec10.3}

We recall a Lie algebra analog of the linear form $J_{\chi}\left(f\right)$
defined in Section \ref{sec7}. We identity the Lie algebra of $Z_G\backslash G$ with the subalgebra $\mathfrak{g}_0$ of trace-$0$ elements in the Lie algebra $\mathfrak{g}$ of $G$. For all $f\in\mathcal{C}_{\text{scusp}}$$\left(\mathfrak{g}_{0}\left(F\right)\right)$,
we set  
$$
J_{V,\chi}^{\text{Lie}}\left(f\right)=\int_{Z_{G}\left(F\right)H_{V}\left(F\right)\backslash G\left(F\right)}\underset{i}{\sum}\int_{\mathfrak{h}_{V,0}\left(F\right)}f\left(g^{-1}Xg\right)\left\langle \phi_{i},\omega_{V,\psi,\mu,\chi}\left(\exp X\right)\phi_{i}\right\rangle dXdg
$$
and 
$$
J_{V^{\prime},\chi}^{\text{Lie}}\left(f\right)=\int_{Z_{G}\left(F\right)H_{V^\prime}\left(F\right)\backslash G\left(F\right)}\underset{i}{\sum}\int_{\mathfrak{h}_{V^\prime,0}\left(F\right)}f\left(g^{-1}Xg\right)\left\langle \phi_{i},\omega_{V^{\prime},\psi,\mu,\chi}\left(\exp X\right)\phi_{i}\right\rangle dXdg
$$
and 
$$
J_{\chi}^{\text{Lie}}\left(f\right)=\mu\left(\text{det}V\right)J_{V,\chi}^{\text{Lie}}\left(f\right)+\mu\left(\text{det}V^{\prime}\right)J_{V^{\prime},\chi}^{\text{Lie}}\left(f\right).
$$
By Theorem \ref{6.1} and using a descent to Lie algebra, $J^{\text{Lie}}_\chi\left(f\right)$ is absolutely convergent when
$f$ is supported near $0$. In this section, we prove the following
proposition.
\begin{proposition}\label{10.5}
There exists a constant $c_{\chi}\in\mathbb{C}$ such that for any
$\theta\in{\mathcal{E}}\left(Z_G(F)\backslash G(F),\chi\right)$, we have 
$$
J_{\text{qc},\chi}\left(\theta\right)=c_{\chi}\cdot c_{\theta}\left(1\right)+\tilde{J}_{\text{qc},\chi}\left(\tilde{\theta}\right),
$$
where $\tilde{\theta}$ is a transfer of $\theta$.
\end{proposition}

\begin{proof}
Let $\omega\subseteq\mathfrak{g}_{0}\left(F\right)$ is a $G$-excellent
open neighborhood of $0$ and set $\Omega=Z_{G}\left(F\right)\exp\left(\omega\right)$.
By Proposition \ref{10.4}, it suffices to consider $\theta\in{\mathcal{E}}\left(Z_{G}\left(F\right)\backslash\Omega,\chi\right)$.
Let $\tilde{\theta}$ be a transfer of $\theta$ in ${\mathcal{E}}_{\text{}}\left(Z_M\left(F\right)\backslash\tilde{M}\left(F\right),\chi\times{}^\sigma\chi^\vee\right)$.
We can view $J_{\text{qc},\chi}(\theta)-\tilde{J}_{\text{qc},\chi}(\tilde{\theta})$ as a linear form on $\mathcal{E}(Z_G(F)\backslash G(F),\chi)$. By Proposition \ref{10.4}, we can see that the support of $J_{\text{qc},\chi}(\theta)-\tilde{J}_{\text{qc},\chi}(\tilde{\theta})$ is contained in $\{1\}$. This gives us  
\begin{equation}\label{eqn9.2}   J_{\text{qc},\chi}\left(\theta\right)=\underset{\mathcal{O}\in\text{Nil}\left(\mathfrak{g}_{0}\right)}{\sum}c_{{\mathcal{O}}}\cdot c_{\theta,{\mathcal{O}}}\left(1\right)+\underset{{\mathcal{O}}\in\text{Nil}\left(\mathfrak{g}_{0}\right)}{\sum}\tilde{c}_{{\mathcal{O}}}\cdot c_{\tilde{\theta},{\mathcal{O}}}\left(\theta_{n}\right)+\tilde{J}_{\text{qc}}\left(\tilde{\theta},\chi\right),  
\end{equation}
noting that here we identify $G\left(F\right)\cong M_{\theta_{n}}\left(F\right)$. Let
$f\in{\mathcal{C}}_{\text{scusp}}\left(Z_{G}\left(F\right)\backslash\Omega,\chi\right)$
such that $\theta_f=\theta$. By Proposition \ref{7.4}(i), we have $J_\chi(f)=J^{\text{Lie}}_\chi(f_\omega)$. Let ${\mathcal{O}}_{V,m}$ be the minimal nilpotent orbit in $\mathfrak{h}_{V,0}$ such that $WF(\omega_{V,\psi,\mu,\chi})=\overline{\mathcal{O}_{V,m}}$ (c.f. Corollary \ref{4.5}) and $S_{V}$ be an element in ${\mathcal{O}}_{V,m}$. We can shrink $\omega$ so that
for any $f\in{\mathcal{C}}_{\text{scusp}}\left(\omega\right)$, we have
the following local character expansion 
$$
\underset{i}{\sum}\int_{\mathfrak{h}_{V,0}\left(F\right)}f\left(g^{-1}Xg\right)\left\langle \phi_{i},\omega_{V,\psi,\mu,\chi}\left(\exp X\right)\phi_{i}\right\rangle dX
$$
$$
=c_{0,V}\int_{\mathfrak{h}_{V,0}\left(F\right)}f\left(g^{-1}Xg\right)dX+c_{1,V}D^{H_{V}}\left(S_{V}\right)^{1/2}\int_{\left(H_{V}\right)_{S_{V}}\left(F\right)\backslash H_{V}\left(F\right)}\widehat{f^{g}\mid_{\mathfrak{h}_{V,0}\left(F\right)}}\left(x^{-1}S_{V}x\right)dx.
$$
Let $B\left(\cdot,\cdot\right)$ be the canonical nondegenerate bilinear
form on $\mathfrak{g}_{0}$. Let $\Sigma_{V}=\mathfrak{h}_{V,0}^{\perp}$
in $\mathfrak{g}_{0}$. Using the proof of Proposition \ref{7.4},
it follows that 
$$
J_{V,\chi}^{\text{Lie}}\left(f\right)=c_{0,V}\int_{Z_{G}\left(F\right)H_{V}\left(F\right)\backslash G\left(F\right)}\int_{\mathfrak{h}_{V,0}\left(F\right)}f\left(g^{-1}Xg\right)dXdg
$$
$$
+c_{1,V}D^{H_{V}}\left(S_{V}\right)^{1/2}\int_{Z_{G}\left(F\right)H_{S_{V}}(F)\backslash G(F)}\int_{\Sigma_{V}\left(F\right)+S_{V}}\widehat{f}\left(g^{-1}Xg\right)d\mu_{\Sigma_{V}}Xdg.
$$
For $\lambda\in\left(F^{\times}\right)^{2}$, let $f_{\lambda}\left(X\right)=f\left(\lambda^{-1}X\right)$.
If we take $\lambda\in\left({\mathcal{O}}_{F}^{\times}\right)^{2}$, then
$\text{supp}\left(f_{\lambda}\right)$ is still contained in $\omega$.
We have 
$$
J_{V,\chi}^{\text{Lie}}\left(f_{\lambda}\right)=\left|\lambda\right|^{n^{2}-1}c_{0,V}\int_{Z_{G}\left(F\right)H_{V}\left(F\right)\backslash G\left(F\right)}\int_{\mathfrak{h}_{V,0}\left(F\right)}f\left(g^{-1}Xg\right)dXdg
$$
$$
+\left|\lambda\right|^{n^{2}-n}c_{1,V}D^{H_{V}}\left(S_{V}\right)^{1/2}\int_{Z_{G}\left(F\right)H_{S_{V}}(F)\backslash G(F)}\int_{\Sigma_{V}\left(F\right)+S_{V}}\widehat{f}\left(g^{-1}Xg\right)d\mu_{\Sigma_{V}}Xdg.
$$
Similarly, for the skew-hermitian space $V^{\prime}$, we also
obtain
$$
J_{V^{\prime},\chi}^{\text{Lie}}\left(f_{\lambda}\right)=\left|\lambda\right|^{n^{2}-1}c_{0,V^{\prime}}\int_{Z_{G}\left(F\right)H_{V^{\prime}}\left(F\right)\backslash G\left(F\right)}\int_{\mathfrak{h}_{V^{\prime},0}\left(F\right)}f\left(g^{-1}Xg\right)dXdg
$$
$$
+\left|\lambda\right|^{n^{2}-n}c_{1,V^{\prime}}D^{H_{V^{\prime}}}\left(S_{V^{\prime}}\right)^{1/2}\int_{Z_{G}\left(F\right)H_{S_{V^{\prime}}}(F)\backslash G(F)}\int_{\Sigma_{V^{\prime}}\left(F\right)+S_{V^{\prime}}}\widehat{f}\left(g^{-1}Xg\right)d\mu_{\Sigma_{V^{\prime}}}Xdg.
$$
For $\lambda\in ({\mathcal{O}}_{F}^{\times})^2$, let $\theta_{0,\lambda}\left(X\right)=\theta\left(\lambda^{-1}X\right)$
and $\theta_{\lambda}=\theta_{0,\lambda}\circ\exp^{-1}$. Then $\theta_{f_{\omega,\lambda}}\circ\exp^{-1}=\theta_{\lambda}$. Let $\tilde{\theta}_{\lambda}$ be a transfer of $\theta_{\lambda}$
in ${\mathcal{E}}\left(Z_M\left(F\right)\backslash\tilde{M}\left(F\right),\chi\times{}^\sigma\chi^\vee\right)$.
By shrinking $\omega$ if necessary, we can assume $\tilde{\theta}$
is supported in $\tilde{\Omega}=(\theta_{n}\exp\left(\tilde{\omega}\right))^M$,
where $\tilde{\omega}\subset\mathfrak{m}_{\theta_{n}}\left(F\right)$
is an $M$-excellent open neighborhood of $0$ such that $\tilde{\omega}_{\psi,\mu}$
has a local character expansion on $\theta_{n}\exp\left(\tilde{\omega}\right)$.
We can see that $\tilde{\theta}_{\lambda}\left(\theta_{n}\exp X\right)=\tilde{\theta}\left(\theta_{n}\exp\left(\lambda^{-1}X\right)\right)$, for any $X\in\tilde{\omega}$. Let $\tilde{f}\in {\mathcal{C}}_{\text{scusp}}\left(Z_M\left(F\right)\backslash\tilde{\Omega},\chi\times {}^\sigma\chi^\vee\right)$
such that $\theta_{\tilde{f}}=\tilde{\theta}$.
We set $\tilde{f}_{\lambda}\left(\theta_{n}\exp X\right)=\tilde{f}_{\lambda}\left(\theta_{n}\exp\left(\lambda^{-1}X\right)\right)$, for $X\in\tilde{\omega}$. Then $\theta_{\tilde{f}_{\lambda}}=\tilde{\theta}_\lambda$. By Corollary \ref{4.5}, $WF(\omega_{\psi,\chi})$ is a closure of a minimal nilpotent orbit in $\mathfrak{g}(F)$. Using \cite[Theorem 4.1]{Kon02}, we conclude that the wavefront set of $\tilde{\omega}_{\psi,\mu,\chi}$ is the closure of a minimal nilpotent orbit of $\mathfrak{g}_{\theta_{n},0}\left(F\right)$.
We denote this minimal nilpotent orbit by $\mathfrak{O}_{\psi}$. Then
$$
\underset{i}{\sum}\int_{\mathfrak{g}_{\theta_{n},0}\left(F\right)}f\left(X\right)\left\langle \phi_{i},\tilde{\omega}_{\psi,\mu,\chi}\left(\theta_{n}\exp X\right)\phi_{i}\right\rangle dX
=c_{0}\int_{\mathfrak{g}_{\theta_{n},0}\left(F\right)}f\left(X\right)dX+c_{1}\int_{\mathfrak{O}_\psi}\widehat{f\mid_{\mathfrak{g}_{\theta_{n},0}\left(F\right)}}\left(X\right)dX.
$$
A similar argument as in the untwisted case gives us 
$$
\tilde{J}_{\chi}\left(\tilde{f}_{\lambda}\right)
=\left|\lambda\right|^{n^{2}-1}c_{0}\cdot\int_{Z_M\left(F\right)G_{\theta_{n}}\left(F\right)\backslash M\left(F\right)}\int_{\mathfrak{g}_{\theta_{n},0}\left(F\right)}\tilde{f}\left(m^{-1}\theta_{n}\left(\exp X\right)m\right)dXdm
$$
$$
+\left|\lambda\right|^{n^{2}-n}c_{1}\cdot D^{\tilde{G}}\left(S\right)^{1/2}\,\int_{Z_M\left(F\right)\left(G_{\theta_{n}}\right)_{S}\left(F\right)\backslash M\left(F\right)}\int_{\Sigma\left(F\right)+S}\mathcal{F}_\psi\left(L_{\theta_{n}}\left({}^{m}\tilde{f}\right)\circ\exp\right)\left(X\right)d\mu_{\Sigma}Xdm,
$$
where $S$ is an element in $\mathfrak{O}_{\psi}$ and $\mathfrak{m}_{\theta_{n},0}=\mathfrak{g}_{\theta_{n},0}\oplus^{\perp}\Sigma$
and $\mathcal{F}_\psi$ is the Fourier transform on $\mathcal{S}$$(\mathfrak{m}_{\theta_n,0}(F))$ with respect to $\psi$. By Lemma \ref{10.3}, we have  
$$
J_{\text{qc},\chi}\left(\theta_{\lambda}\right)-\tilde{J}_{qc,\chi}(\tilde{\theta}_\lambda)=J_\chi^\text{Lie}(f_{\omega,\lambda})-\tilde{J}_\chi(\tilde{f}_\lambda)
=\left|\lambda\right|^{n^{2}-1}c_{0,V}\cdot\mu\left(\det V\right)\int_{Z_{G}\left(F\right)H_{V}\left(F\right)\backslash G\left(F\right)}\int_{\mathfrak{h}_{V,0}\left(F\right)}f_{\omega}\left(g^{-1}Xg\right)dXdg
$$
$$
+\left|\lambda\right|^{n^{2}-n}c_{1,V}\cdot\mu\left(\det V\right)D^{H_V}\left(S_{V}\right)^{1/2}\,\underset{t}{\lim}\ \int_{Z_{G}\left(F\right)H_{S_{V}}(F)\backslash G(F)}\int_{\Sigma_{V}\left(F\right)+S_{V}}\widehat{f_{\omega}}\left(g^{-1}Xg\right)d\mu_{\Sigma_{V}}Xdg
$$
$$
+\left|\lambda\right|^{n^{2}-1}c_{0,V^{\prime}}\cdot\mu\left(\det V^{\prime}\right)\int_{Z_{G}\left(F\right)H_{V^{\prime}}\left(F\right)\backslash G\left(F\right)}\int_{\mathfrak{h}_{V^{\prime},0}\left(F\right)}f_{\omega}\left(g^{-1}Xg\right)dXdg
$$
$$
+\left|\lambda\right|^{n^{2}-n}c_{1,V^{\prime}}\cdot\mu\left(\det V^{\prime}\right)D^{H_{V^{\prime}}}\left(S_{V^{\prime}}\right)^{1/2}\int_{Z_{G}\left(F\right)H_{S_{V^{\prime}}}(F)\backslash G(F)}\int_{\Sigma_{V^{\prime}}\left(F\right)+S_{V^{\prime}}}\widehat{f_{\omega}}\left(g^{-1}Xg\right)d\mu_{\Sigma_{V^{\prime}}}Xdg
$$
$$
-\left|\lambda\right|^{n^{2}-1}c_{0}\cdot\int_{Z_M\left(F\right)G_{\theta_{n}}\left(F\right)\backslash M\left(F\right)}\int_{\mathfrak{g}_{\theta_{n},0}\left(F\right)}\tilde{f}\left(m^{-1}\theta_{n}\left(\exp X\right)m\right)dXdm
$$
$$
-\left|\lambda\right|^{n^{2}-n}c_{1}\cdot D^{\tilde{G}}\left(S\right)^{1/2}\,\int_{Z_M\left(F\right)\left(G_{\theta_{n}}\right)_{S}\left(F\right)\backslash M\left(F\right)}\int_{\Sigma\left(F\right)+S}\mathcal{F}_\psi\left(L_{\theta_{n}}\left({}^{m}\tilde{f}\right)\circ\exp\right)\left(X\right)d\mu_{\Sigma}Xdm,
$$
Moreover, for any ${\mathcal{O}}\in\text{Nil}\left(\mathfrak{g}_{0}\right)$, we have $c_{\theta_{\lambda},{\mathcal{O}}}\left(1\right)=\left|\lambda\right|^{\frac{\dim\left({\mathcal{O}}\right)}{2}}c_{\theta,{\mathcal{O}}}\left(1\right)$ and $c_{\tilde{\theta}_{\lambda},{\mathcal{O}}}\left(\theta_{n}\right)=\left|\lambda\right|^{\frac{\dim\left({\mathcal{O}}\right)}{2}}c_{\tilde{\theta},{\mathcal{O}}}\left(\theta_{n}\right)$. By using a basic argument of homogeneity for (\ref{eqn9.2}), we deduce
$$
J_{\text{qc},\chi}\left(\theta\right)=c_{{\mathcal{O}}_{\text{reg}}}\cdot c_{\theta,{\mathcal{O}}_{\text{reg}}}\left(1\right)+\tilde{c}_{{\mathcal{O}}_{\text{reg}}}\cdot c_{\tilde{\theta},{\mathcal{O}}_{\text{reg}}}\left(\theta_{n}\right)+\tilde{J}_{\text{qc},\chi}\left(\tilde{\theta}\right),
$$
for $\theta\in{\mathcal{E}}\left(Z_{G}\left(F\right)\backslash G\left(F\right),\chi\right)$,
where ${\mathcal{O}}_{\text{reg}}$ is the only regular nilpotent orbit
in $\mathfrak{g}_{0}\left(F\right)$. By the twisted endoscopic character identity proved in Theorem \ref{8.1}, together with Proposition \ref{3.2}(i) and its twisted analog, it follows that $c_{\theta,{\mathcal{O}}_{\text{reg}}}\left(1\right)=c_{\tilde{\theta},{\mathcal{O}}_{\text{reg}}}\left(\theta_{n}\right)$.
Therefore, there exists $c_{\chi}\in\mathbb{C}$ such that $J_{\text{qc},\chi}\left(\theta\right)=c_{\chi}\cdot c_{\theta}\left(1\right)+\tilde{J}_{\text{qc},\chi}\left(\tilde{\theta}\right)$, for any $\theta\in{\mathcal{E}}\left(Z_{G}\left(F\right)\backslash G\left(F\right),\chi\right)$.
\end{proof}

\subsection{Proof of Theorem \ref{sec10.2}}\label{sec10.4}
We prove Theorem \ref{sec10.2}. 
\begin{proof}
By Proposition \ref{10.5}, we have 
\begin{equation}\label{eqn9.3}
J_{\text{qc},\chi}\left(\theta\right)=c_{\chi}\cdot c_{\theta}\left(1\right)+\tilde{J}_{\text{qc},\chi}\left(\tilde{\theta}\right).
\end{equation}
for $\theta\in{\mathcal{E}}\left(Z_{G}\left(F\right)\backslash G\left(F\right),\chi\right)$.
There remains to show $c_{\chi}=0$. By substituting $\theta=\theta_{\Pi}$ to linear forms $J_{\text{qc},\chi}$ and $\tilde{J}_{\text{qc},\chi}$ in the above equality, for some irreducible discrete series $\Pi$ of $G(F)$ with central character $\chi$, we have 
$$
J_{\text{qc},\chi}\left(\theta_\Pi\right)= \sum_{\pi\in\Pi_2(G,\chi^{-1})}(\mu\left(\det V\right)m_{V}\left(\bar{\pi}\right)+\mu\left(\det V^{\prime}\right)m_{V^{\prime}}\left(\bar{\pi}\right))
\int_{\Gamma(Z_G(F)\backslash G)_\text{ell}}D^G(x)\theta_{\Pi}(x)\theta_\pi(x)dx
$$
$$
=\mu\left(\det V\right)m_{V}\left(\Pi\right)+\mu\left(\det V^{\prime}\right)m_{V^{\prime}}\left(\Pi\right),
$$
$$
\tilde{J}_{\text{qc},\chi}\left(\theta_{\widetilde{\Pi\times {}^\sigma\Pi^\vee}}\right)= \sum_{\pi\in\Pi_2(\tilde{M},\chi^{-1}\times{}^\sigma\chi)}\epsilon_\psi(\tilde{\pi}^\vee)
\int_{\Gamma(Z_M(F)\backslash \tilde{M})_\text{ell}}D^{\tilde{M}}(\tilde{x})\theta_{\widetilde{\Pi\times {}^\sigma\Pi^\vee}}(\tilde{x})\theta_{\tilde{\pi}}(\tilde{x})d\tilde{x}
$$
$$
=\epsilon_\psi({\widetilde{\Pi\times {}^\sigma\Pi^\vee}})=\chi(-1)^n\omega_{E/F}(-1)^{n(n-1)/2}\epsilon\left(\frac{1}{2},\Pi\times{}^\sigma\Pi^\vee\times\mu^{-1},\psi_E\right),
$$
noting that here we use the orthogonality relations of elliptic representations of connected and twisted reductive groups. From (\ref{eqn9.3}), we obtain 
$$
\mu\left(\det V\right)m_{V}\left(\Pi\right)+\mu\left(\det V^{\prime}\right)m_{V^{\prime}}\left(\Pi\right)=c_{\chi}+\chi(-1)^n\omega_{E/F}\left(-1\right)^{n\left(n-1\right)/2}\epsilon\left(\frac{1}{2},\Pi\times{}^{\sigma}\Pi^{\vee}\times\mu^{-1},\psi_{E}\right).
$$
Theorem \ref{thm1.2}(i) gives us $m_V(\Pi)+m_{V^\prime}(\Pi)=1$. Since $\chi(-1)^n\omega_{E/F}\left(-1\right)^{n\left(n-1\right)/2}\epsilon\left(\frac{1}{2},\Pi\times{}^{\sigma}\Pi^{\vee}\times\mu^{-1},\psi_{E}\right)$ is equal to $\pm\,\omega_{E/F}(-1)^{n/2}$, it follows that $c_\chi\in\{\pm2\,\omega_{E/F}(-1)^{n/2},0\}$. We prove $c_\chi=0$ by contrary. Assume $c_\chi \neq 0$. Without loss of generality, we can assume further $c_\chi=-2\omega_{E/F}(-1)^{n/2}$. The case $c_\chi=2\omega_{E/F}(-1)^{n/2}$ can be treated in the same way to below. We have $m_V(\Pi)=0$, for any $\Pi\in\Pi_2(G,\chi)$, where $V$ is the $n$-dimensional skew-hermitian space over $E$ such that $\mu(\det V)=1$. This gives us
$$
    J_{V,\chi}(f)=\int_{\mathcal{X}(G,\chi^{-1})}\hat{\theta}_f(\pi)m_V(\bar{\pi})d\pi
=\int_{\mathcal{X}_{ind}(G,\chi^{-1})}\hat{\theta}_f(\pi)m_V(\bar{\pi})d\pi,
$$
for any $f\in \mathcal{C}_{scusp}(Z_G(F)\backslash G(F),\chi)$. Since $\hat{\theta}_f(\pi)=\int_{\Gamma(Z_G\backslash G)}D^G(x)\theta_f(x)\theta_\pi(x)dx$ and $\theta_\pi$ vanishes on $\Gamma(Z_G\backslash G)_{\text{ell}}$, for any $\pi \in \mathcal{X}_{ind}(G,\chi^{-1})$, we can see that $\text{Supp}(J_{V,\chi})$ does not contain any elements in $\Gamma(Z_G\backslash G)_{\text{ell}}$. Let $x\in H_V(F)$ be a regular elliptic element in $G(F)$ and $\Omega_x \subseteq G_x(F)$ be a $G$-good neighborhood of $x$. We will show that $x$ lies in the support of $J_{V,\chi}$, which gives a contradiction. We set $\Omega=\Omega_x^G$. Let $f_x\in \mathcal{C}_{scusp}(Z_G(F)\backslash \Omega_x,\chi)$ and $f\in \mathcal{C}_{scusp}(Z_G(F)\backslash \Omega,\chi)$ be a lift of $f_x$ as in Proposition \ref{3.18}. As in the proof of Proposition \ref{7.3}, we have $J_{V,\chi}(f)=J_{\chi}^{H_{V,x}}(f_x)$. It suffices to show $x\in \text{Supp}(J_{\chi}^{H_{V,x}})$. Since $x\in H_V(F)$ is regular elliptic in $G(F)$, it follows that $H_{V,x}(F)$ is isomorphic to the group $E_0^1$ of elements in $E_0$ having $F_0$-norm $1$, where $F_n$ is an $n$-dimensional field extension over $F$ and $E_n=EF_n$. Therefore, the linear form $J_{\text{qc},\chi}^{H_{V,x}}$ can be deduced from the $1$-dimensional case for $E_n^1$. We prove the following lemma.
\begin{lemma}\label{10.6}
    Let $f\in C^\infty_c(E^\times)$ and $W$ be a $1$-dimensional skew-hermitian space over $E$. We have
    $$J^{E^1}_{W}(f)=\frac{1}{2}f(1)+\frac{\mu(\det W)\omega_{E/F}(-2)\gamma_{\psi}(N_{E/F})}{2}\lim_{s\rightarrow 0^+}\int_{E^1}f(x)\mu(1-x)^{-1}|1-x|_E^{s-1/2}dx.$$
\end{lemma}
\begin{proof}
Let $\chi$ be a continuous character of $E^\times$. Applying \cite[Lemme A.1 and (10) in pg. 1360]{BP15} for the character $\chi \times {}^\sigma\chi^{-1}\times \mu^{-1}$, we have 
$$
\epsilon\left(\frac{1}{2},\chi \times {}^\sigma\chi^{-1}\times \mu^{-1},\psi_E\right)=\omega_{E/F}(-2)\gamma_{\psi}(N_{E/F})\lim_{s\rightarrow 0^+}\int_{E^1}(\chi \times {}^\sigma\chi^{-1}\times \mu^{-1})(1-x)|1-x|_E^{s-1/2}dx
$$
$$
=\chi(-1)\omega_{E/F}(-2)\gamma_{\psi}(N_{E/F})\lim_{s\rightarrow 0^+}\int_{E^1}\chi(x)\mu(1-x)^{-1}|1-x|_E^{s-1/2}dx.
$$
The spectral expansion in the case $n=1$ gives us $J^{E^1}_W(f)= \sum_{\chi:E^\times\rightarrow \mathbb{C}^\times} m_W(\chi)\int_{E^\times}f(x)\chi(x)^{-1}dx$. By \cite[Theorem 3.1]{GGP23}, we have
$$
m_W(\chi)=\frac{1+\mu(\det W)\chi(-1)\epsilon\left(\frac{1}{2},\chi \times {}^\sigma\chi^{-1}\times \mu^{-1},\psi_E\right)}{2}.
$$
From the above discussion, the right hand side is equal to
$$
\frac{1}{2}+\frac{\mu(\det W)\omega_{E/F}(-2)\gamma_{\psi}(N_{E/F})}{2}\lim_{s\rightarrow 0^+}\int_{E^1}\chi(x)\mu(1-x)^{-1}|1-x|_E^{s-1/2}dx.
$$
Applying the Fourier inverse formula for $f$, we obtain
$$
J^{E^1}_W(f)= \sum_{\chi:E^\times\rightarrow \mathbb{C}^\times} m_W(\chi)\int_{E^\times}f(x)\chi(x)^{-1}dx
$$
$$
=\frac{1}{2}f(1)+\frac{\mu(\det W)\omega_{E/F}(-2)\gamma_{\psi}(N_{E/F})}{2}\lim_{s\rightarrow 0^+}\int_{E^1}f(x)\mu(1-x)^{-1}|1-x|_E^{s-1/2}dx.
$$
We have finished our proof for Lemma \ref{10.6}.
\end{proof}
It is clear that Lemma \ref{10.6} still holds if we replace $f\in C^\infty_c(E^\times)$ by $C^\infty_c(E^\times/E^{\prime,\times},\chi)$, where $E/E^\prime$ is a finite field extension and $\chi$ is a character of $E^{\prime,\times}$. By applying Lemma \ref{10.6} for $E=E_n$, $F=F_n$ and $W=\text{Res}_{E_n/E}E_n$ and noting that $J^{H_{V,x}}_{V,\chi}(f_x)=J^{E_n^1}_{\text{Res}_{E_n/E}E_n}(f_x)$, we have 
$$
H_{V,x}(F)\backslash\{1\}=E_n^1\backslash\{1\}\subseteq \text{Supp}(J^{E_n^1}_{\text{Res}_{E_n/E}E_n})= \text{Supp}(J^{H_{V,x}}_{V,\chi}).
$$
In particular, $x$ lies in the support of $J^{H_{V,x}}_{V,\chi}$ and hence gives a contradiction. We conclude that $c_\chi=0$. This ends our proof for Theorem \ref{sec10.2}.
\end{proof}

\appendix
\section{Finite multiplicity for the local twisted GGP conjecture}\label{app}
In this appendix, we prove the following theorem by induction on the dimension of skew-hermitian spaces $V$.
\begin{theorem}\label{A.1}
Let $\pi$ be a smooth representation of $G\left(F\right)$ of finite
length. Then 
$$
\dim\text{Hom}_{H}\left(\pi,\omega_{V,\psi,\mu}\right)<\infty.
$$
\end{theorem}
When $\dim V=1$, the above theorem follows from \cite[Theorem 3.1]{GGP23}. We consider the case when $\dim V \geq 2$. In section \ref{sec12.3}, following the strategy in \cite{Del09}, we prove Theorem \ref{A.1} when $\pi$ is supercuspidal. In section \ref{sec12.4}, by using Mackey theory developed in \cite[Section 4]{GGP23} and the induction hypothesis, we prove Theorem \ref{A.1} when $\pi$ is not supercuspidal.

\subsection{Preliminaries}\label{sec12.1}

We recall some notations in \cite{Del09}. Let $A$ be a split torus
of $G$. Let $X_{*}\left(A\right)$ be the group of one-parameter
subgroups of $A$. We fix a uniformizing element $\varpi$ of $F$.
Let $\Lambda\left(A\right)$ be the image of $X_{*}\left(A\right)$
in $A$ via $\lambda\mapsto\lambda\left(\varpi\right)$. By this morphism,
$\Lambda\left(A\right)$ is isomorphic to $X_{*}\left(A\right)$.

Let $\left(P,P^{-}\right)$ be a pair of opposite parabolic subgroups
of $G$ and $M$ be their common Levi subgroup. We denote by $A_{M}$
the split component of its center. Let $U$ and $U^{-}$ be unipotent
radicals of $P$ and $P^{-}$ respectively. Let $\Sigma\left(P,A_{M}\right)$
be the set of roots of $A_{M}$ in $\text{Lie}\left(P\right)$ and
$\Delta\left(P,A\right)$ be the set of simple roots. We set 
$$
A_{M}^{-}\text{ (resp. }A_{M}^{--}\text{) }=\left\{ a\in A_{M}\left(F\right)\mid\ \left|\alpha\left(a\right)\right|_{F}\leq1\text{ (resp. }<1\text{),}\,\alpha\in\Delta\left(P,A\right)\right\} .
$$
We define $A_{M}^{+}$ and $A_{M}^{++}$ similarly by reversing the
inequalities. For $\epsilon>0$, we set 
$$
A_{M}^{-}\left(\epsilon\right)=\left\{ a\in A_{M}\left(F\right)\mid\ \left|\alpha\left(a\right)\right|_{F}\leq\epsilon,\,\alpha\in\Delta\left(P,A\right)\right\} .
$$
Let $M_{0}$ be a minimal Levi subgroup of $G$ and $A_{0}$ be the
split component of $Z_{M_{0}}$. We fix a minimal parabolic subgroup
$P_{0}$ of $G$ with Levi component $M_{0}$. We choose a maximal
compact subgroup $K_{0}$ of $G\left(F\right)$ with relative position
to $M_{0}$. If $P$ is a parabolic subgroup of $G$ containing $A_{0}$,
we denote by $P^{-}$ the opposite parabolic subgroup of $G$ to $P$
containing $A_{0}$ and $M=P\cap P^{-}$. We recall \cite[Proposition 1.4.4]{Cas}.
\begin{proposition}\label{A.2}
There exists a decreasing sequence of compact open subgroups $\left(K_{n}\right)_{n\in\mathbb{N}}$
of $G\left(F\right)$ such that for all $n\in\mathbb{N}^{*}$, $K_{n}$
is normal in $K_{0}$ and for every parabolic subgroup $P$ containing
$P_{0}$, the followings hold
\begin{enumerate}
\item $K_{n}=K_{n,U^{-}}K_{n,M}K_{n,U}$, where $K_{n,U^{-}}=K_{n}\cap U^{-}$,
$K_{n,U}=K_{n}\cap U$, $K_{n,M}=K_{n}\cap M$.
\item For all $a\in A_{M}^{-}$, one has $aK_{n,U}a^{-1}\subset K_{n,U}$
and $aK_{n,U^{-}}a^{-1}\subset K_{n,U^{-}}$.
\item The sequence $K_{n}$ forms a neighborhood basis of the identity in
$G\left(F\right)$.
\end{enumerate}
\end{proposition}
We say $K$ has an Iwahori factorization with respect to $\left(P,P^{-}\right)$
if part (i) and (ii) hold. In this case, we have $\underset{n\in\mathbb{N}}{\bigcup}a^{-n}K_{U}a^{n}=U$, for any $a\in A_{M}^{--}$. Let $\sigma$ be an involution of $G$ defined over $F$ such that
$H=G^{\sigma}$. A parabolic subgroup $P$ of $G$ is called a $\sigma$-parabolic
subgroup if $P^{-}=\sigma\left(P\right)$. Then $M=P\cap\sigma\left(P\right)$
is the $\sigma$-stable Levi subgroup of $P$. By \cite[Proposition 13.4]{HW93},
it follows that $P^{-}H$ is open in $G$ and the restriction of the
modulus character $\delta_{P}$ to $H\left(F\right)$ is trivial.
Moreover, there are finitely many $H\left(F\right)$-conjugacy classes
of $\sigma$-parabolic subgroups of $G$. We denote this set by ${\mathcal{P}}^{\sigma}\left(G\right)$.

We recall the second adjointness theorem. Let $P=MU$ be a parabolic subgroup of $G$ and $P^-=MU^-$ be the opposite parabolic subgroup of $G$ to $P$. We also fix a minimal parabolic subgroup $P_0=M_0U_0$ such that $P_0$ is contained in $P$ and $A_0=A_{M_0}$ contains $A=A_M$. Let $\pi$ be a smooth representation of $G(F)$. We denote by $\pi_U$ the normalized Jacquet module of $\pi$ along $P$ and $j_P$ the projection map from $\pi$ to $\pi_U$. Let $\pi^\vee$ be the smooth dual of $\pi$. We recall the second adjointness theorem in \cite{Ber}.
\begin{theorem}\label{A.3}
    There exists a unique nondegenerate $M$-invariant bilinear form $\langle,\rangle_P$ on $\pi_{U^-}\times \pi_U$ such that for all compact open subgroups $K$ and for all $a\in A^{--}$, we have
    $$
    \delta_P^{1/2}(a^n)\langle j_{P^-}(v^\vee),\pi_U(a^n)j_P(v)\rangle_P=\langle v^\vee, \pi(a^n)v\rangle,\ \forall n\geq n_K(a),    
    $$
    where $n_K(a)\in\mathbb{N}$ depends only on $K$ and $a\in A^{--}$ but not on $\pi$ and $v\in \pi^K$ and $v^\vee\in (\pi^\vee)^K$.
\end{theorem}
We set $\Theta_P = \Delta(P_0\cap M,A_0)$ and $A_0^-(P,<\epsilon)=\{ a\in A_0^-\mid\ |\alpha(a)|_F<\epsilon,\,\forall \alpha \in \Delta(P_0,A_0)\setminus \Theta_P\}$ for $\epsilon > 0$.
\subsection{Constant term of smooth functions}\label{sec12.2}

Let $C^{\infty}\left(G/H,\omega_{V,\psi,\mu}\right)$ be the space
of smooth functions $f$ on $G\left(F\right)$, whose values belong
to $\omega_{V,\psi,\mu}$, satisfying $f\left(gh\right)=\omega_{V,\psi,\mu}\left(h^{-1}\right)f\left(g\right)$, for any $g\in G\left(F\right)$ and $h\in H\left(F\right)$. We consider the left regular representation $L$ of $C^{\infty}\left(G/H,\omega_{V,\psi,\mu}\right)$.
Let $\delta_{e}$ be the Dirac measure in $e$. Let $\left(P,P^{-}\right)$
be a pair of opposite $\sigma$-parabolic subgroups of $G$. Let $\{\phi_i\}_{i\in I}$ be an orthonormal basis for $\omega_{V,\psi,\mu}$.
\begin{definition}\label{A.5}
For $f\in C^{\infty}\left(G/H,\omega_{V,\psi,\mu}\right)$, let $f_{P} \in C^{\infty}\left(M/\left(H\cap M\right),\omega_{V,\psi,\mu}\right)$ characterized by  
$$
\left\langle f_{P}\left(mM\cap H\right),\phi_{i}\right\rangle =\left\langle j_{P^{-}}\left(\delta_{eH}\right),\delta_{P}^{1/2}\left(m\right)\left\langle j_{P}\left(L_{m^{-1}}f\right)\left(\cdot\right),\phi_{i}\right\rangle \right\rangle ,
$$
for $m\in M\left(F\right)$ and $i\in I$. Note that here we regard
$\left\langle j_{P}\left(L_{m^{-1}}f\right)\left(\cdot\right),\phi_{i}\right\rangle $
as a function on $G\left(F\right)$. We call $f_{P}$ the constant
term of $f$ along $P$.
\end{definition}

If $\pi$ is a smooth $G\left(F\right)$-module and
$v\in \pi$ and $\xi\in\text{Hom}_{H}\left(\pi,\omega_{V,\psi,\mu}\right)$,
we denote by $c_{\xi,v}(g)=\xi\left(\pi\left(g^{-1}\right)w\right)$ as a function in $C^{\infty}\left(G/H,\omega_{V,\psi,\mu}\right)$. The map $c_{\xi}:\pi\rightarrow C^{\infty}\left(G/H,\omega_{V,\psi,\mu}\right)$
given by $v\mapsto c_{\xi,v}$ is a morphism of $G\left(F\right)$-modules. We recall some analogs of Proposition 3.14 and Proposition 3.16 in \cite{Del09}.
\begin{proposition}\label{A.6}
If $A_{0}$ is a maximal split torus contained in $M=P\cap P^{-}$ and
$P_{0}\subset P$ is a minimal parabolic subgroup containing $A_{0}$,
we have 
$$
f\left(a^{-1}\right)=\delta_{P}^{1/2}\left(a\right)f_{P}\left(a^{-1}\right),\,\forall a\in A_{0}^{-}\left(P,<\epsilon_{K}\right),
$$
noting that here we identify the Weil representation $\omega_{V,\psi,\mu}$ of $H(F)$ with its restriction on $M(F)\cap H(F)$, which is also the (product of) Weil representations of unitary groups.
\end{proposition}

\begin{proof}
The proof follows from Definition \ref{A.5} and \cite[Lemma 2.1]{Del09}.
\end{proof}
\begin{proposition}\label{A.7}
Let $\left(Q,Q^{-}\right)$
be another pair of opposite parabolic subgroups of $G$ with $Q\subset P,\,Q^{-}\subset P^{-}$
and such that $Q^{-}H$ is open. Denote by $R=Q\cap M$, $R^{-}=Q^{-}\cap M$.
We assume that $R^{-}(M\cap H)$ is open. Then $\left(R,R^{-}\right)$
is a pair of opposite parabolic subgroups of $M$. We have 
\begin{enumerate}
\item 
$$
f_{Q}=\left(f_{P}\right)_{R},\text{ for all }f\in C^{\infty}\left(G/H,\omega_{V,\psi,\mu}\right).
$$
\item 
$$
f_{h\cdot P}\left(hmh^{-1}\right)=\omega_{V,\psi,\mu}\left(h^{-1}\right)\left(L_{h^{-1}}f\right)_{P}\left(m\right),
$$
for all $h\in H\left(F\right),\,m\in M\left(F\right),\,f\in C^{\infty}\left(G/H,\omega_{V,\psi,\mu}\right).$
\end{enumerate}
\end{proposition}

\begin{proof}
The proof follows from Definition \ref{A.5}.
\end{proof}
We define $\left(H,\omega_{V,\psi,\mu}\right)$-cuspidal functions.
\begin{definition}\label{A.8}
A function $f\in C^{\infty}\left(G/H,\omega_{V,\psi,\mu}\right)$
is said to be $\left(H,\omega_{V,\psi,\mu}\right)$-cuspidal if for
any proper $\sigma$-parabolic subgroup $P$ of $G$, one has 
$$
\left(L_{g}f\right)_{P}=0,\text{ for any }g\in G\left(F\right).
$$
We denote by $C^{\infty}\left(G/H,\omega_{V,\psi,\mu}\right)_{\text{cusp}}$
the $G(F)$-space of $\left(H,\omega_{V,\psi,\mu}\right)$-cuspidal
smooth functions from $G\left(F\right)/H\left(F\right)$ to $\omega_{V,\psi,\mu}$.
\end{definition}

\subsection{Finite multiplicity of supercuspidal representations}\label{sec12.3}

In this section, we prove Theorem \ref{A.1} when $\pi$ is a supercuspidal representation. We first recall some basic set-ups in \cite{Del09}. A split torus is said to be $\sigma$-split if its elements are anti-invariant
by $\sigma$. Let $A_{\emptyset}$ be a maximal $\sigma$-split torus
of $G\left(F\right)$. We denote by $M_{\emptyset}$ the centralizer
of $A_{\emptyset}$ in $G$ and $P_{\emptyset}$ a minimal $\sigma$-parabolic
subgroup of $G$ whose $\sigma$-stable Levi component is $M_{\emptyset}$.
Let $A_{\emptyset,G}$ be the maximal $\sigma$-split torus in $A_{G}$.
We denote by $\Lambda^{+}\left(A_{\emptyset}\right)$ the set of $P_{\emptyset}$-dominant
elements in $\Lambda\left(A_{\emptyset}\right)$. Let $\Sigma\left(G,A_{\emptyset}\right)$
be the set of roots of $A_{\emptyset}$ in the Lie algebra of $G$.
We denote by $\Delta\left(G,A_{\emptyset}\right)$ the set of simple
roots of $\Sigma\left(P,A_{\emptyset}\right)$. Given a subset $\Theta\subseteq\Delta\left(G,A_{\emptyset}\right)$,
we denote by $P_{\Theta}$ the parabolic subgroup of $G$ containing
$A_{\emptyset}$ which is corresponding to $\Theta$. It is a $\sigma$-parabolic
subgroup of $G$. 

For $C>0$, we denote by $\Lambda^{+}\left(A_{\emptyset}\right)$
(resp. $\Lambda^{+}\left(A_{\emptyset},C\right)$) the set containing
$a\in\Lambda\left(A_{\emptyset}\right)$ such that $\left|\alpha\left(a\right)\right|_{F}\geq1$
(resp. $C>\left|\alpha\left(a\right)\right|_{F}\geq1$), for all $\alpha\in\Delta\left(G,A_{\emptyset}\right)$.
The following lemma is an analog of \cite[Lemma 4.3]{Del09}.
\begin{lemma}\label{A.10}
Let $K$ be a compact open subgroup of $G\left(F\right)$. Then there
exists a finite set $F_{\emptyset,K}\subset\Lambda^{+}\left(A_{\emptyset}\right)$
such that the restriction of every element of $C^{\infty}\left(G/H,\omega_{V,\psi,\mu}\right)_{\text{cusp}}^{K}$
to $\Lambda^{+}\left(A_{\emptyset}\right)$ is zero outside $F_{\emptyset,K}\Lambda\left(A_{\emptyset}\right)_{G}$.
\end{lemma}

\begin{proof}
We follow closely to the proof of \cite[Lemma 4.3]{Del09}. Let $A_0$ be a maximal split torus contained in $M_\emptyset$, hence contains $A_\emptyset$. Let $P_0$ be a minimal parabolic subgroup of $G$ contained in $P_\theta$ and containing $A_0$. Let $\epsilon>0$ and $C=\epsilon^{-1}$. We set
$$
\Lambda^+(P,>C)=\{a^{-1}\mid\ a\in\Lambda(A_\emptyset)\cap A_0^-(P,<\epsilon)\}.
$$
By Proposition \ref{A.6}, there exists $C>0$ such that for all $P=P_\Theta$, we have $f_{|\Lambda^+(P,>C)}=\delta_P^{1/2}(a)(f_P)_{|\Lambda^+(P,>C)}$, for $f\in C^\infty(G/H,\omega_{V,\psi,\mu)^K}$. When $f$ is cuspidal, by Proposition \ref{A.6}, it follows that $f_P=0$ for all $P=P_\Theta$, where $\Theta$ is equal to $\Delta(G,A_\emptyset)\backslash \{\alpha\}$, for some $\alpha\in \Delta(G,A_\emptyset)$. Moreover, the complement in $\Lambda^+(A_\emptyset)$ of the union of such sets is $\Lambda^+(A_\emptyset,C)$. By \cite[Lemma 4.2]{Del09}, there exists a finite set $F_{\emptyset,K}\subset \Lambda^+(A_\emptyset)$ such that $f_{|\Lambda^+(A_\emptyset)}$ is zero outside $F_{\emptyset,K}\Lambda(A_\emptyset)_G$. This ends the proof of Lemma \ref{A.10}. 
\end{proof}
Let $\left(A_{i}\right)_{i\in I}$ be a set of representatives of
the $H\left(F\right)$-conjugacy classes of maximal $\sigma$-split
tori of $G$. This set is finite. Assume this set contains $A_{\emptyset}$.
The tori $A_{i}$ are conjugated under $G\left(F\right)$. Let $y_{i}\in G\left(F\right)$
satisfying $y_{i}\cdot A_{\emptyset}=A_{i}$. We take $y_{\emptyset}$
to be the identity element $e$. We recall the Cartan decomposition $G\left(F\right)=\underset{i\in I}{\bigcup}\,\Omega\Lambda^{+}\left(A_{\emptyset}\right)y_{i}^{-1}H\left(F\right)$, where $\Omega$ is a certain compact subset of $G\left(F\right)$. We prove an analog of \cite[Theorem 4.4]{Del09}.
\begin{theorem}\label{A.11}
$ $
\begin{enumerate}
\item Let $K$ be a compact open subgroup of $G\left(F\right)$. Let ${\mathcal{X}}$
be a finite family of characters of $\Lambda\left(A_{G}\right)$.
The space $C^{\infty}\left(G/H,\omega_{V,\psi,\mu}\right)_{\text{cusp}}^{K}\left({\mathcal{X}}\right)$
containing elements in $C^{\infty}\left(G/H,\omega_{V,\psi,\mu}\right)_{\text{cusp}}^{K}$
which are of type ${\mathcal{X}}$ under the left regular action of $\Lambda\left(A_{G}\right)$
is finite dimensional.
\item Let $\pi$ be a finite-length supercuspidal representation of $G\left(F\right)$. Then the space $\text{Hom}_{H}\left(\pi,\omega_{V,\psi,\mu}\right)$ is of finite dimensional.
\end{enumerate}
\end{theorem}

\begin{proof}
(i) Since $\Omega$ is compact, there exists $g_{1},\ldots,g_{k}\in\Omega$
such that $\Omega\subset\underset{i}{\cup}Kg_{i}$. Let $f\in C^{\infty}\left(G/H,\omega_{V,\psi,\mu}\right)_{\text{cusp}}^{K}\left({\mathcal{X}}\right)$.
By the Cartan decomposition, it follows that $f$ is determined by
the restriction of $L_{y_{i}g_{j}^{-1}}f$ to $\Lambda^{+}\left(A_{i}\right)$,
when $i$ varies in $I$ and $j=\overline{1,k}$. Moreover, for any
$g\in G\left(F\right)$, the function $L_{g}f$ is $g\cdot K$-invariant.
By \cite[Lemma 4.2]{Del09} and Lemma \ref{A.10} and the properties of functions of type ${\mathcal{X}}$
(see in \cite[(6.3)]{Del09}), there exists a finite subset $X$ of
$G\left(F\right)/H\left(F\right)$ such that $f$ is zero if and only
if it is zero on $X$. Let $x\in X$. It suffices to show that $f\left(x\right)$
belongs to a finite dimensional vector space. We denote by ${\mathcal{X}}$
the set containing restrictions of elements in ${\mathcal{X}}$ to $\Lambda\left(A_{G}\right)\cap H\left(F\right)$.
Let $\omega_{V,\psi,\mu}\left({\mathcal{X}}\right)$ be the subspace of
$\omega_{V,\psi,\mu}$ containing elements of type ${\mathcal{X}}$. Since
$f$ is of type ${\mathcal{X}}$, it takes values in $\omega_{V,\psi,\mu}\left({\mathcal{X}}\right)$.
By using the Howe duality for the dual pair $\left(U_{1}\left(F\right),U\left(V\right)\right)$
(cf. \cite{GT16}), we can see that $\omega_{V,\psi,\mu}\left({\mathcal{X}}\right)$
is an admissible $H(F)$-module. This gives us $\dim\omega_{V,\psi,\mu}^{x^{-1}\cdot K \cap H(F)}\left({\mathcal{X}}\right)<\infty$. Since $f$ is $K$-invariant, it follows that $f\left(x\right)\in\omega_{V,\psi,\mu}^{x^{-1}\cdot K\cap H(F)}\left({\mathcal{X}}\right)$,
which is finite dimensional.

(ii) Since $\pi$ is of finite length, there is a finite
family ${\mathcal{X}}$ of characters of $\Lambda\left(A_{G}\right)$ such
that every element of $\pi$ is of type ${\mathcal{X}}$. We consider the
following map arising from the Frobenius reciprocity  
$$
\begin{array}{ccccc}
\psi_\pi & : & \text{Hom}_{G}\left(\pi,C^{\infty}\left(G/H,\omega_{V,\psi,\mu}\right)_{\text{cusp}}\left({\mathcal{X}}\right)\right) & \longrightarrow & \text{Hom}_{H}\left(\pi,\omega_{V,\psi,\mu}\right)\\
 &  & T & \mapsto & \delta_{eH}\circ T
\end{array}.
$$
We want to prove that the above map is an isomorphism. To be more precise, we will show that 
$$
\xi\in \text{Hom}_{H}\left(\pi,\omega_{V,\psi,\mu}\right) \mapsto c_{\xi}\in \text{Hom}_{G}(\pi,C^{\infty}\left(G/H,\omega_{V,\psi,\mu}\right))
$$
is its inverse. Let $v\in \pi$. It suffices to show that $c_{\xi,v}$ is cuspidal in $C^{\infty}\left(G/H,\omega_{V,\psi,\mu}\right)$. By Definition \ref{A.5}, for any proper $\sigma$-parabolic subgroup $P=MU$ of $G$, we have 
$$
\langle(c_{\xi,v})_P(mM\cap H),\phi_i\rangle = \langle j_{P^-}(\delta_{eH}),\delta_P^{1/2}(m)j_P(\lambda_i(\pi(m^{-1})v))\rangle,
$$
where $\lambda_i$ is the following composition of $G(F)$-modules
$$
v\in\pi\longmapsto c_{\xi,v}\in C^{\infty}(G/H,\omega_{V,\psi,\mu})\longmapsto\langle c_{\xi,v}(\cdot),\phi_{i}\rangle\in C^{\infty}(G).
$$
We denote by $j_P(\lambda_i)$ the $G(F)$-morphism from $\pi_U$ to $C^\infty(G)_U$ arising from taking Jaccquet modules. It follows that 
$$
\langle(c_{\xi,v})_P(mM\cap H),\phi_i\rangle=\langle j_{P^-}(\delta_{eH}),\delta_P^{1/2}(m)j_P(\lambda_i)(\pi_U(m^{-1})j_P(v))\rangle=0,
$$
for any $i\in I$, since $\pi$ is supercuspidal. Therefore, the map $\psi_\pi$ is an isomorphism. Since $\pi$ is of finite length, it is generated by finitely many elements. We denote this finite set
by $Y$. Let $K\subset G\left(F\right)$ be a compact open subgroup
fixing these generators. Let $T$ be an element of $\text{Hom}_{G}\left(\pi,C^{\infty}\left(G/H,\omega_{V,\psi,\mu}\right)_{\text{cusp}}\left({\mathcal{X}}\right)\right)$.
One has $T\left(y\right)\in C^{\infty}\left(G/H,\omega_{V,\psi,\mu}\right)_{\text{cusp}}^{K}\left({\mathcal{X}}\right)$,
for any $y\in Y$. Since $T$ is given by the image of $Y$, it follows from the first part of the theorem that 
$$
\text{Hom}_{G}\left(\pi,C^{\infty}\left(G/H,\omega_{V,\psi,\mu}\right)_{\text{cusp}}\left({\mathcal{X}}\right)\right)\text{ is finite dimensional},
$$
which is to say $\text{Hom}_{H}\left(\pi,\omega_{V,\psi,\mu}\right)$
is finite dimensional.
\end{proof}

\subsection{End of the proof of Theorem \ref{A.1}}\label{sec12.4}
In this section, we finish our proof for Theorem \ref{A.1}. By Langlands classification, any irreducible admissible representation is the irreducible quotient of a parabolically induced representation. Hence, it suffices to prove the following theorem.
\begin{theorem}\label{A.12}
    For a partition $n=a+b$ with $0<a\leq b$, let $V=V_a\oplus V_b$ with $\dim V_a=a$ and $\dim V_b=b$. We consider the maximal parabolic subgroup $P=P_{a,b}=MN$ of $G$ stabilizing $V_a$, with Levi factor $M=M_a\times M_b=\text{Res}_{E/F}(GL(V_a)\times GL(V_b))$. Let $\pi_1$ and $\pi_2$ be finite-length representations of $M_a(F)$ and $M_b(F)$ respectively. We denote $\pi=\pi_1\times \pi_2 = \text{Ind}^G_P(\pi_1\boxtimes \pi_2)$. Then $\text{Hom}_H(\pi,\omega_{V,\psi,\mu})$ is of finite dimensional.
\end{theorem}
\begin{proof}
    We recall the study of orbits of $H(F)$ on the flag variety $G(F)/P(F)$ in \cite[Section 4]{GGP23}. By \cite[Lemma 4.2]{GGP23}, they are represented by the isometry classes of $a$-dimensional $E$-subspaces $X \subset V$, which are parametrized by the dimension $d=\dim(X\cap X^\perp)$ for some $d\leq \min\{a,\text{rank}(V)\}$ and a non-degenerate skew-hermitian form on $X/(X\cap X^\perp)$.
    
    Let $[X]$ be an $H(F)$-orbit in $G(F)/P(F)$, which is represented by an $E$-subspace $X\subset V$ of dimension $a$ and $d=\dim(X\cap X^\perp)$. By Mackey theory, the restriction of $\pi=\pi_1\times \pi_2$ to $H(F)$ has a finite equivariant filtration indexed by the $H(F)$-orbits in $G(F)/P(F)$. We denote the $H(F)$-subquotient of $\pi$ corresponding to $[X]$ by $\pi_X$. It suffices to show that 
    $$
    \dim\text{Hom}_H(\pi_X,\omega_{V,\psi,\mu})<\infty,
    $$
    for any $H(F)$-orbit $X$ in $G(F)/P(F)$. Let $S_X$ be the stabilizer of $[X]$ in $H(F)$ and $N(S_X)$ be its unipotent radical. Observe $S_X/N(S_X) \cong \text{GL}_d(E)\times \text{U}_{a-d}\times \text{U}_{b-d}$. By \cite[Proposition 4.4]{GGP23}, we have
    $$
    \text{Hom}_H(\pi_X,\omega_{V,\psi,\mu})\cong \text{Hom}_{S_X/N(S_X)}
    \left(\delta_{P/S}^{1/2}\cdot (\pi_1)_{d,a-d}\otimes (\pi_2)_{b-d,d},\,\delta_S^{1/2}\cdot |\text{det}_{\text{GL}_d}|^{-1/2}\mu\cdot \omega_{V_n-2d,\psi,\mu} \right),
    $$
    where the representations $(\pi_1)_{d,a-d}$ and $(\pi_2)_{b-d,d}$, as well as the characters $\delta_{P/S}^{1/2}$ and $\delta_S$ are given in \cite[Proposition 4.4]{GGP23}. For simplicity, we assume that $(\pi_1)_{d,a-d}=\pi_1^{d}\boxtimes \pi_1^{a-d}$ and $(\pi_2)_{b-d,d}=\pi_2^{b-d}\boxtimes \pi_2^{d}$, where $\pi_1^{d}$ and $\pi_2^{d}$ are finite-length representations of $\text{GL}_d(E)$ and $\pi_1^{a-d}$ is a finite-length representation of $\text{GL}_{a-d}(E)$ and $\pi_2^{b-d}$ is a finite-length representation of $\text{GL}_{b-d}(E)$. This gives us 
    $$
    \text{Hom}_H(\pi_X,\omega_{V,\psi,\mu})\cong \text{Hom}_{\text{GL}_d}(\delta_{P/S}^{1/2}\cdot \pi_1^{d}\otimes\pi_2^{d},\delta_S^{1/2}\cdot |\text{det}_{\text{GL}_d}|^{-1/2}\mu)
    $$
    $$
    \otimes \text{Hom}_{\text{U}_{a-d}}(\pi_1^{a-d},\omega_{V_{a-d},\psi,\mu}) 
    \otimes \text{Hom}_{\text{U}_{b-d}}(\pi_2^{b-d},\omega_{V_{b-d},\psi,\mu}),
    $$
    noting that $\omega_{V_{n-2d},\psi,\mu}$ is isomorphic to $\omega_{V_{a-d},\psi,\mu}\otimes \omega_{V_{b-d},\psi,\mu}$ as $\text{U}_{a-d}(F)\times\text{U}_{b-d}(F)$-modules. Applying \cite[Theorem 4.5]{Del09} for the symmetric variety $\text{GL}_d(E)\times \text{GL}_d(E)/\text{GL}_d(E)$, we have $\text{Hom}_{\text{GL}_d}(\delta_{P/S}^{1/2}\cdot \pi_1^{d}\otimes\pi_2^{d},\delta_S^{1/2}\cdot |\text{det}_{\text{GL}_d}|^{-1/2}\mu)$ is of finite dimensional. By the induction hypothesis, we have $\text{Hom}_{\text{U}_{a-d}}(\pi_1^{a-d},\omega_{V_{a-d},\psi,\mu})$ and $\text{Hom}_{\text{U}_{b-d}}(\pi_2^{b-d},\omega_{V_{b-d},\psi,\mu})$ are of finite dimensional. Therefore
    $$
    \dim \text{Hom}_H(\pi_X,\omega_{V,\psi,\mu})<\infty.
    $$
\end{proof}

\end{document}